\documentclass[reqno,10pt]{amsart}
\usepackage{amsmath,mathrsfs}
\usepackage{amssymb, amsmath}
\usepackage{graphicx}

\usepackage[colorlinks,
            linkcolor=red,
            anchorcolor=blue,
            citecolor=green
            ]{hyperref}
\usepackage{pdfsync}

\def\inte#1{
\displaystyle\mathop{#1\kern0pt}^\circ }

\let\pa=\partial

\let\f=\frac

\def\pa{\partial}

\def\virgp{\raise 2pt\hbox{,}}
\def\cdotpv{\raise 2pt\hbox{;}}

\def\C{\mathop{\mathbb C\kern 0pt}\nolimits}
\def\DD{\mathop{\mathbb D\kern 0pt}\nolimits}
\def\EE{\mathop{{\mathbb E \kern 0pt}}\nolimits}
\def\K{\mathop{\mathbb K\kern 0pt}\nolimits}
\def\N{\mathop{\mathbb N\kern 0pt}\nolimits}
\def\Q{\mathop{\mathbb Q\kern 0pt}\nolimits}
\def\R{\mathop{\mathbb R\kern 0pt}\nolimits}
\def\SS{\mathop{\mathbb S\kern 0pt}\nolimits}
\def\ZZ{\mathop{\mathbb Z\kern 0pt}\nolimits}
\def\TT{\mathop{\mathbb T\kern 0pt}\nolimits}
\def\P{\mathop{\mathbb P\kern 0pt}\nolimits}

\newcommand{\beq}{\begin{equation}}
\newcommand{\eeq}{\end{equation}}
\newcommand{\ben}{\begin{eqnarray}}
\newcommand{\een}{\end{eqnarray}}
\newcommand{\beno}{\begin{eqnarray*}}
\newcommand{\eeno}{\end{eqnarray*}}

\newtheorem{defi}{Definition}[section]
\newtheorem{thm}{Theorem}[section]
\newtheorem{lem}{Lemma}[section]
\newtheorem{rmk}{Remark}[section]
\newtheorem{col}{Corollary}[section]
\newtheorem{prop}{Proposition}[section]
\renewcommand{\theequation}{\thesection.\arabic{equation}}

\usepackage{color}

\numberwithin{equation}{section}

\begin{document}

\title[Non-cutoff Boltzmann equation in the grazing limit]{Solutions to the non-cutoff Boltzmann equation in the grazing limit}

\author[R.-J. Duan]{Renjun Duan}
\address[R.-J. Duan]{Department of Mathematics, The Chinese University of Hong Kong, Shatin, Hong Kong.}
\email{rjduan@math.cuhk.edu.hk}

\author[L.-B. He]{Ling-Bing He}
\address[L.-B. He]{Department of Mathematical Sciences, Tsinghua University\\
Beijing, 100084,  P.R.~China.}
\email{hlb@tsinghua.edu.cn}

\author[T. Yang]{Tong Yang}
\address[T. Yang]{Department of Mathematics, City University of Hong Kong, Tat Chee Avenue, Kowloon, Hong Kong.}
\email{matyang@cityu.edu.hk}

\author[Y.-L. Zhou]{Yu-long Zhou}
\address[Y.-L. Zhou]{School of Mathematics, Sun Yat-Sen University, Guangzhou, 510275, P.R.~China.}
\email{zhouyulong@mail.sysu.edu.cn}

\begin{abstract}
It is known that in the parameters range $-2 \leq \gamma <-2s$  spectral gap does not exist for the linearized Boltzmann operator without cutoff but it  does for the linearized Landau operator. This paper is devoted to the understanding of the formation of spectral gap in this range through the grazing limit. Precisely, we study the  Cauchy problems of these two classical collisional kinetic equations around global Maxwellians in torus and establish the following results that are uniform in the vanishing grazing parameter $\epsilon$: (i) spectral gap type estimates for the collision operators; (ii) global existence of small-amplitude solutions for  initial data with low regularity; (iii) propagation of regularity in both space and velocity variables as well as velocity moments without smallness; (iv) global-in-time asymptotics of the Boltzmann solution toward the Landau solution at the rate $O(\epsilon)$; (v)  continuous
transition of decay structure   of the Boltzmann operator to the Landau operator.
 In particular, the result in part (v) captures the uniform-in-$\epsilon$ transition of intrinsic optimal time decay structures of solutions that reveals how the spectrum of the linearized non-cutoff Boltzmann equation in the mentioned parameter range changes continuously under the grazing limit.
\end{abstract}



\maketitle

\setcounter{tocdepth}{1}
\tableofcontents





\renewcommand{\theequation}{\thesection.\arabic{equation}}
\setcounter{equation}{0}

\section{Introduction}
It is well known that two classical collisional kinetic equations, Boltzmann equation without angular cutoff and Landau equation  are closely connected through the so-called grazing collision limit. The main purpose in this paper is to
reveal  the continuous
transition of decay structure   of the Boltzmann collision operator
in certain range of parameters from the sub-exponential time decay  to the exponential time decay  in the limit process. We emphasize that this problem is related to the famous spectral gap problem, that is, the linearized Boltzmann operator without cutoff does not have any  spectral gap but the linearized Landau operator  does in the same range.

\subsection{Boltzmann  and  Landau equations in the perturbation framework}
For the setting of the study, we first recall the equations. The Cauchy problem on the Boltzmann equation reads
\begin{equation}\label{Cauchy-Boltzmann} \left\{ \begin{aligned}
&\partial _t F +  v \cdot \nabla_{x} F=Q^{B}(F,F), ~~t > 0, x \in \mathbb{T}^{3}, v \in \R^3,\\
&F|_{t=0} = F_{0},
\end{aligned} \right.
\end{equation}
where $F(t,x,v)\geq 0$ is the density
function of particles with velocity
$v\in\R^3$ at time $t\geq 0$ and position $x \in \mathbb{T}^{3} :=[-\pi,\pi]^{3}$. The Boltzmann collision operator is
\ben \label{Boltzmann-operator}
Q^{B}(g,h)(v):=
\int_{\R^3}\int_{\mathbb{S}^{2}}B(v-v_*,\sigma)\left(g^{\prime}_{*} h^{\prime}-g_{*}h\right)d\sigma dv_*.
\een
Here we have used the standard shorthand notions $h=h(v)$, $g_*=g(v_*)$,
$h'=h(v')$ and $g'_*=g(v'_*)$, where $v'$ and $v_*'$ are given by
\beno
v'=\frac{v+v_{*}}{2}+\frac{|v-v_{*}|}{2}\sigma, \quad
v'_{*}=\frac{v+v_{*}}{2}-\frac{|v-v_{*}|}{2}\sigma, \quad \sigma\in\mathbb{S}^{2}.
\eeno

On the other hand,  with the same initial data as in \eqref{Cauchy-Boltzmann}, the Cauchy problem on the Landau equation reads
\begin{equation}\label{Cauchy-Landau} \left\{ \begin{aligned}
&\partial _t F +  v \cdot \nabla_{x} F=Q^{L}(F,F), ~~t > 0, x \in \mathbb{T}^{3}, v \in\R^3,\\
&F|_{t=0} = F_{0},
\end{aligned} \right.
\end{equation}
where  the Landau operator $Q^{L}(g,h)$ is given  by
\beno
Q^{L}(g,h)(v) :=
\nabla_{v}\cdot\bigg\{\int_{\R^3}a(v-v_{*})[g(v_{*})\nabla_{v}h(v)-\nabla_{v_{*}}g(v_{*})h(v)]dv_{*}\bigg\}.
\eeno
Here
\begin{eqnarray}\label{matrix}
a(z)  = \Lambda |z|^{\gamma+2}(I_{3} - \frac{z \otimes z}{|z|^{2}}),
\end{eqnarray}
where $I_{3}$ is the $3 \times 3$ identity matrix and $\Lambda$ is a positive constant.

In the perturbation framework, that is, by setting  $F(t,x,v)=\mu +\mu^{\frac{1}{2}}f(t,x,v)$ with the normalized global Maxwellians $\mu=\mu(v) := (2\pi)^{-3/2}e^{-|v|^{2}/2}$,  the Cauchy problems \eqref{Cauchy-Boltzmann}  and \eqref{Cauchy-Landau}  are  reduced respectively to
\begin{equation}\label{Cauchy-linearizedBE} \left\{ \begin{aligned}
&\partial_{t}f + v\cdot \nabla_{x} f + \mathcal{L}^{B}f= \Gamma^{B}(f,f), ~~t > 0, x \in \mathbb{T}^{3}, v \in\R^3,\\
&f|_{t=0} = f_{0},
\end{aligned} \right.
\end{equation}
and
\begin{equation}\label{Cauchy-linearizedLE} \left\{ \begin{aligned}
&\partial_{t}f + v\cdot \nabla_{x} f + \mathcal{L}^{L}f= \Gamma^{L}(f,f), ~~t > 0, x \in \mathbb{T}^{3}, v \in\R^3,\\
&f|_{t=0} = f_{0}.
\end{aligned} \right.
\end{equation}
Here the linearized Boltzmann operator $\mathcal{L}^{B}$ and the nonlinear term $\Gamma^{B}$ are given respectively  by
\ben \label{defintion-GammaB-LB}
\Gamma^{B}(g,h):= \mu^{-\frac{1}{2}} Q^{B}(\mu^{\frac{1}{2}}g,\mu^{\frac{1}{2}}h), \quad
\mathcal{L}^{B}g:= -\Gamma^{B}(\mu^{\frac{1}{2}},g) - \Gamma^{B}(g, \mu^{\frac{1}{2}}).
\een
Similarly,   the linearized Landau operator $\mathcal{L}^{L}$ and the nonlinear term $\Gamma^{L}$  are
\beno \Gamma^{L}(g,h) := \mu^{-\frac{1}{2}} Q^{L}(\mu^{\frac{1}{2}}g,\mu^{\frac{1}{2}}h), \quad
\mathcal{L}^{L}g := -\Gamma^{L}(\mu^{\frac{1}{2}},g) - \Gamma^{L}(g, \mu^{\frac{1}{2}}). \eeno

In what follows,  we impose the following assumptions on the Boltzmann kernel $B$ in \eqref{Boltzmann-operator}:

 \begin{itemize}
  \item[$\mathbf{(A1).}$] The Boltzmann kernel $B$ takes the form of
  \beno
B(v-v_{*},\sigma)= C_{B} |v-v_{*}|^{\gamma}b(\cos\theta), ~~-3<\gamma \leq 1, C_{B}>0,
\eeno
where $\cos\theta= \frac{v-v_{*}}{|v-v_{*}|}\cdot \sigma$.

  \item[$\mathbf{(A2).}$] The angular function $b(\cdot)$ is singular in the sense that
 \ben \label{assumption-on-angular-function-b}
C_{b}^{-1} \sin^{-2-2s}\frac{\theta}{2} \leq b(\cos\theta) \leq C_{b} \sin^{-2-2s}\frac{\theta}{2}, ~~0<s<1, C_{b} \geq 1.
\een

  \item[$\mathbf{(A3).}$]
  The parameter $\gamma$ and $s$ satisfy the condition: $\gamma+2s >-1$.

  \item[$\mathbf{(A4).}$]  Without loss of generality, we  assume that $B(v-v_*,\sigma)$ is supported in the interval $0\leq \theta\leq \pi/2$, i.e., $\frac{v-v_*}{|v-v_*|}\cdot\sigma\geq0$, for otherwise $B$ can be replaced by its symmetrized form:
\beno
\overline{B}(v-v_*,\sigma)=|v-v_*|^\gamma\big\{b(\frac{v-v_*}{|v-v_*|}\cdot\sigma)+b(\frac{v-v_*}{|v-v_*|}\cdot(-\sigma))\big\} \mathrm{1}_{\frac{v-v_*}{|v-v_*|}\cdot\sigma\ge0},
\eeno
where $\mathrm{1}_A$ is the characteristic function of a set $A$.
\end{itemize}

\begin{rmk}\label{inverse-power-law}
The above assumptions on $B(v-v_{*},\sigma)$ are motivated by the inverse power law model with   potential $U(r) = r^{-p}, p>1$ where
$s = \frac{1}{p}$ and $\gamma = \frac{p-4}{p}$ satisfying $\gamma+4s=1$.
\end{rmk}


\subsubsection{Mathematical theory on the grazing collisions limit of Boltzmann equation to Landau equation}
In this  subsection, we will briefly review existing mathematical work on the grazing limit of the
Boltzmann equation to the Landau equation.

 Formally the grazing limit means that when one rescales the function of the deviation angle to be  concentrated on the collisions that become grazing, the corresponding Boltzmann equation leads to the Landau equation in the limit. Precisely,  set
\ben \label{definition-angular-part} \epsilon=\sin(\theta_{\max}/2), \quad b^{\epsilon}(\cos \theta):=(1-s)\epsilon^{2s-2}\sin^{-2-2s}(\theta/2)
\mathrm{1}_{\sin(\theta/2) \leq \epsilon},\een
where $\theta_{\max}$ is the maximum deviation angle such that  collisions happen only when $\theta \leq \theta_{\max}$. Then
  the rescaled  Boltzmann kernel $B^{\epsilon}(v-v_*,\sigma) $ and the corresponding collision operator  $Q^{\epsilon}$ are given respectively as follows:
\ben \label{Blotzmann-kernel}
B^{\epsilon}(v-v_*,\sigma) =  \left|v-v_{*}\right|^{\gamma} b^{\epsilon}(\cos \theta)
= \left|v-v_{*}\right|^{\gamma} (1-s) \epsilon^{2s-2}\sin^{-2-2s}(\theta / 2)\mathrm{1}_{\sin(\theta/2) \leq \epsilon},\een
and
\beno Q^{\epsilon}(g,h)(v):=
\int_{\R^3}\int_{\mathbb{S}^{2}}B^{\epsilon}(v-v_*,\sigma)\left(g^{\prime}_{*} h^{\prime}-g_{*}h\right)d\sigma dv_*.
\eeno
At the operator level, it is known that the following asymptotic formula between $Q^{\epsilon}$ and $Q^L$ holds  for suitably smooth functions:
\beno
 \|Q^{\epsilon}(f,f)-Q^{L}(f,f)\|_{L^{2}} \lesssim \epsilon \|f\|_{H^{3}_{\gamma+10}}\|f\|_{H^{5}_{\gamma+10}}.
\eeno
We refer to \cite{desvillettes1992asymptotics,he2014asymptotic,zhou2020refined} for the details.

\bigskip

\noindent $\bullet$ {\bf Weak convergence of the limit.} In the spatially homogeneous case, Arsen'ev-Buryak \cite{arsen1991connection} studied
the convergence of weak solutions of the Boltzmann equation to those of the Landau equation under certain assumptions on the Boltzmann kernel. However, the kernel considered in \cite{arsen1991connection} does not include the inverse power law potential. Goudon \cite{goudon1997boltzmann} proved the convergence of weak solutions  for the inverse power law potential in the case of $\gamma \geq -2$ and  $s \leq \frac{1}{4}$. Note that this range covers the potential $U(r) = r^{-p}$ for $p \geq \frac{4}{3}$ only. Villani \cite{villani1998new} used the symmetry of spherical integrals and introduced a new definition of weak solutions that enables him to show the convergence of weak solutions of the Boltzmann equation to those of the Landau equation by only assuming $\gamma>-4$. Note that  the results in \cite{villani1998new} hold for the Coulomb potential with  $p=1$.

Based on the renormalized solution theory \cite{villani1998new} and the entropy dissipation estimate obtained in \cite{alexandre2000entropy}, important contribution was made by  Alexandre-Villani in \cite{alexandre2004landau}
that gives first  study on the problem in the
spatially inhomogeneous setting.
Thanks to the general setting of weak solutions, the situation in \cite{alexandre2004landau} covers a board class of potentials including the Coulomb interaction.

\medskip

\noindent $\bullet$ {\bf Classical convergence of the limit.} In the spatially homogeneous case, the second author  in \cite{he2014asymptotic}  showed the convergence of \eqref{coboltzmann} to \eqref{Cauchy-Landau} in weighted Sobolev spaces with an explicit rate. More precisely, supposing that $F^{\epsilon}$ and $F$ are solutions to \eqref{coboltzmann} and \eqref{Cauchy-Landau} respectively, it was proved in \cite{he2014asymptotic} that
\beno
\sup_{0 \leq t \leq T} |F^{\epsilon}(t) - F(t)|_{H^{N}} \leq \epsilon C(T, |F_{0}|_{L^{1}_{q(N,l)}}, |F_{0}|_{H^{N+3}_{l}}),
\eeno
for some $T>0$. Here, $T$ can be extended to $\infty$ for $\gamma \geq -2s$, whereas $T<\infty$ is required for $-3<\gamma <-2s$.

\bigskip

In the present work, we are interested in the inverse power law model when the parameters $\gamma$ and $s$ satisfy $-2\le \gamma<-2s$,  because in this setting the linearized Boltzmann collision operator $\mathcal{L}^B$ does not have spectral gap while the linearized Landau operator does have. Correspondingly, this property induces that for the solutions to the nonlinear equations \eqref{Cauchy-linearizedBE} and \eqref{Cauchy-linearizedLE}, one can derive the sub-exponential time decay rate for the Boltzmann equation but the exponential decay rate for the Landau equation.  As we mentioned above, the grazing collision limit bridges these two equations in the limit process. It is then natural to ask whether  one can have a unified framework to show that in the vanishing-in-$\epsilon$ limit process the spectral gap is continuously transferred from non-existence to existence. Unfortunately, so far we have no idea to directly answer this question in  the level of the functional analysis. Thus we resort to finding the continuous transition from the sub-exponential  structure to the exponential  structure by studying the time-decay of solutions.  Mathematically, we are concerned with the rescaled Boltzmann equation in the perturbation framework:
 \begin{equation}\label{coboltzmann} \left\{ \begin{aligned}
&\partial _t F +  v \cdot \nabla_{x} F=Q^{\epsilon}(F,F), ~~t > 0, x \in \mathbb{T}^{3}, v \in\R^3,\\
&F|_{t=0} = F_{0},
\end{aligned} \right.
\end{equation}
as well as the associated Landau equation \eqref{Cauchy-Landau}
 with the same initial data in the limit $\epsilon \rightarrow 0$.

\subsubsection{Mathematical theory on Landau's derivation}     In 1936, Landau derived an effective kinetic equation,  named by the Landau equation (or Fokker-Planck-Landau equation) nowadays, for the charged particles governed by the Coulomb potential in the weak coupling regime. Landau's formal derivation can be found in many books, see \cite{montgomery1964plasma,Lif1981PhysicalKinetics} for instance.
In this situation,  it holds that $\gamma=-3$  in \eqref{matrix}; we refer to \cite{degond1992fokker} for the convergence of the Boltzmann operator to the  Landau operator. At the solution level for the limit from \eqref{coboltzmann} to \eqref{Cauchy-Landau}, we refer readers to \cite{alexandre2004landau} for  convergence of weak solutions as well as
  \cite{he2014well} for  convergence of classical solutions  with an explicit rate $|\ln \epsilon|^{-1}$. We remark that the Boltzmann kernel in \cite{he2014well} is taken as
\beno
\tilde{B}^{\epsilon}(v-v_{*}, \sigma) := |\ln \epsilon|^{-1} |v-v_{*}|^{-3} \sin^{-4}\frac{\theta}{2} \mathrm{1}_{\sin\frac{\theta}{2} \geq \epsilon},
\eeno
and the result holds only locally in time.  Very recently, in the near equilibrium framework, the second and fourth authors in \cite{he2020boltzmann} proved the global-in-time convergence of solutions of \eqref{Cauchy-linearizedBE} with the most singular kernel $\tilde{B}^{\epsilon}$ to solutions of \eqref{Cauchy-linearizedLE}.


%
%
%
%

\subsection{Mathematical setting of the problems:} Let us give a detailed mathematical description on the problems to be discussed
in this paper. We begin with the function spaces.

\smallskip

\noindent $\bullet$ {\bf Function spaces.} We refer to  \cite{alexandre2012boltzmann,GS} and \cite{guo2002landau} on global  well-posedness theories for the Boltzmann
equation without angular cutoff and the Landau equation in weighted Sobolev spaces, respectively.
 In this paper, we will follow the low regularity function space $L^{1}_{k}L^{\infty}_{T}L^{2}$ introduced in \cite{duan2021global} to consider the limit, where $L^1_k$ corresponds to the Weiner algebra over a torus. More precisely, the function space is equipped
with the norm
\beno
\|f\|_{L^{1}_{k}L^{\infty}_{T}L^{2}} := \sum_{k \in \mathbb{Z}^{3}}  \sup_{0 \leq t \leq T} |\hat{f}(t,k,\cdot)|_{L^{2}}.
\eeno
Here $\hat{f}$ is the Fourier transform with respect to $x$. The $|\cdot|_{L^{2}}$ is taken with respect to the variable $v$. Note that in terms of the regularity of $x$ variable on torus, it holds in the formal level that $H^{3/2+\delta}_{x} \hookrightarrow L^{1}_{k} \hookrightarrow L^{\infty}_{x}$. To the best of our knowledge, $L^{1}_{k}L^{\infty}_{T}L^{2}$ seems to be the largest space in which global well-posedness theory for both the non-cutoff Boltzmann equation and the Landau equation can be established via the direct energy method, in contrast with the recent substantial progress in \cite{AMSY} for constructing the $L^2\cap L^\infty$ solutions via the De Giorgi argument.

\smallskip
\noindent $\bullet$ {\bf Some well-known facts.}
Now we list some basic facts on the large time behavior of solutions to both the Boltzmann  and the Landau equations in the space $L^{1}_{k}L^{\infty}_{T}L^{2}$.
Let $f^{L}$ be the solution to the Cauchy problem \eqref{Cauchy-linearizedLE} on the Landau equation. When $\gamma \geq -2$, under suitable smallness assumption on $f_{0}$,
it holds (see Theorem 2.1 in \cite{duan2021global}) that
\ben \label{decay-of-Landau-solution}
\|f^{L}(t)\|_{L^{1}_{k}L^{2}} \lesssim e^{-\lambda t}\|f_{0}\|_{L^{1}_{k}L^{2}} \lesssim e^{-\lambda t}.
\een
See \eqref{general-L1k-norm-for-initial-data} for the precise definition of norm $\| \cdot \|_{L^{1}_{k}L^{2}}$.
The above time  decay  property is consistent with the fact that the linearized Landau operator $\mathcal{L}^{L}$ has a spectral gap if and only if $\gamma \geq -2$. On the other hand, let $f^{B}$ be the solution to the Cauchy problem \eqref{Cauchy-linearizedBE} on the Boltzmann equation. When  $-3<\gamma <-2s$, under suitable smallness assumption on $f_{0}$,
it holds (see Theorem 2.1 in \cite{duan2021global}) that
\ben \label{decay-of-Boltzmann-solution}
\|f^{B}(t)\|_{L^{1}_{k}L^{2}} \lesssim e^{-\lambda t^{\kappa}}\|e^{q \langle v \rangle}f_{0}\|_{L^{1}_{k}L^{2}} \lesssim e^{-\lambda t^{\kappa}},
\een
where  $\kappa = \frac{1}{1+|\gamma+2s|}$, $q>0$, and $\langle v \rangle=\sqrt{1+|v|^{2}}$.
The time decay rate in \eqref{decay-of-Boltzmann-solution} is also consistent with the spectrum
structure of the  linearized Boltzmann operator $\mathcal{L}^{B}$  in the
soft potential regime $\gamma<-2s$ for which there is no  spectrum gap, cf. \cite{YY} and the references therein. To the best of our knowledge, \eqref{decay-of-Landau-solution} and \eqref{decay-of-Boltzmann-solution} provide the optimal decay rate estimates in the existing literatures.
\smallskip

\noindent $\bullet$ {\bf Spectral estimates of the linearized collision operators.}  The spectral gap estimates for the linearized   operators play the important role in the global-in-time well-posedness for the collisional kinetic equation in the perturbation framework. We recall that pioneering work is due to Wang and Uhlenbeck(see \cite{wang1952propagation}) on the Maxwell molecule model $\gamma=0$. Let us denote $\mathcal{L}^{B,\gamma}$ to address that the linearized collision operator $\mathcal{L}^B$ in fact depends on the parameter $\gamma$. In Wang-Uhlenbeck's work, the authors gave the explicit formulas to all the eigenvalues and the associated eigenfunctions to $\mathcal{L}^{B,0}$. As a direct consequence, it implies the so-called spectral gap estimate. To be  precise,  the kernel space of $\mathcal{L}^{B,\gamma}$ and $\mathcal{L}^{L}$ is defined by
\ben \label{kernel-space}
\mathrm{ker}:= \mathrm{span}\{\mu^{\frac{1}{2}}, \mu^{\frac{1}{2}}v_1, \mu^{\frac{1}{2}}v_2,\mu^{\frac{1}{2}}v_3, \mu^{\frac{1}{2}}|v|^2 \}.
\een
Usually, $\mathrm{ker}$ is called the macro-space, and $\mathrm{ker}^{\perp}$ is called micro-space. Wang and Uhlenbeck proved that for any $f\in \mathrm{ker}^{\perp}$,
\ben\label{sepctralgapWU}
\langle \mathcal{L}^{B,0}f, f\rangle \geq \lambda_{e} |f|^{2}_{L^{2}},
\een
where $\lambda_e$ is the first non zero eigenvalue of $\mathcal{L}^{B,0}$ given by
\ben\label{nonzeroeigenvalue} \lambda_e:=\int_0^{\pi/2} b(\cos\theta)\sin\theta(1-\cos\theta)d\theta. \een
Later on, authors in \cite{baranger2005explicit,mouhot2006explicit,mouhotstrain}  proved that the spectral gap estimates can be generalized to the other potentials. It was asserted that  there exists two constant $C_{\gamma}$ and $C_b$ such that for  any $f\in \mathrm{ker}^{\perp}$,
\ben\label{mouhotsepctralgap}
\langle  \mathcal{L}^{B,\gamma}f, f\rangle \geq C_{\gamma} C_b
|f|^{2}{}_{L^{2}_{\gamma/2}}, \een
where \ben\label{CBconstant} C_b:= \inf_{\sigma_1,\sigma_2 \in \SS^2}\int_{\SS^2}\min\{b(\sigma_1\cdot\sigma_3), b(\sigma_2\cdot\sigma_3)\}d\sigma_3.
\een

For the angular function $b^\epsilon$ defined in \eqref{Blotzmann-kernel}, one may check that $\lambda_e\sim 1$ while $C_{b^{\epsilon}}\rightarrow 0$ as $\epsilon$ goes to zero. This shows that the estimate \eqref{sepctralgapWU} is robust in the grazing collisions limit process and thus can be thought as a unified formula for both Boltzmann and Landau collision operators. It also requires us to establish Wang and Uhlenbeck's type estimates for the soft potentials.

\smallskip

\noindent $\bullet$ {\bf Statement of the results.}  It is obvious to see that in the regime $-2 \leq \gamma < -2s$ the time asymptotic behaviors of the solutions described in  \eqref{decay-of-Landau-solution} and \eqref{decay-of-Boltzmann-solution} are different by noticing that the latter is at the sub-exponential decay rate ($0<\kappa < 1$) while the former is at the  exponential decay rate. Since  the grazing collision limit of the Boltzmann equation yields the Landau equation, it is  interesting to find out  whether the transition from sub-exponential decay structure to exponential decay structure occurs in a continuous way through the limit. Furthermore, one may ask whether one can  provide a detailed mathematical description on the time decay structures in the limit process. To answer the above questions, we first    rewrite the rescaled Boltzmann equation  \eqref{coboltzmann} by letting $F=\mu +\mu^{\frac{1}{2}}f$:
\begin{equation}\label{Cauchy-linearizedBE-grazing} \left\{ \begin{aligned}
&\partial_{t}f + v\cdot \nabla_{x} f + \mathcal{L}^{\epsilon}f= \Gamma^{\epsilon}(f,f), ~~t > 0, x \in \mathbb{T}^{3}, v \in\R^3,\\
&f|_{t=0} = f_{0}.
\end{aligned} \right.
\end{equation}
Here the linearized Boltzmann operator $\mathcal{L}^{\epsilon}$ and the nonlinear term $\Gamma^{\epsilon}$ are defined by
\ben\label{DefLep} \Gamma^{\epsilon}(g,h):= \mu^{-\frac{1}{2}} Q^{\epsilon}(\mu^{\frac{1}{2}}g,\mu^{\frac{1}{2}}h), \quad
\mathcal{L}^{\epsilon}g:= -\Gamma^{\epsilon}(\mu^{\frac{1}{2}},g) - \Gamma^{\epsilon}(g, \mu^{\frac{1}{2}}).
\een
From now on, we assume without of loss of generality
the initial perturbation $f_0$ for \eqref{Cauchy-linearizedBE-grazing} and \eqref{Cauchy-linearizedLE} has
zero total mass, momentum and energy:
\beno  \int_{\mathbb{T}^{3} \times \R^3} \mu^{\frac{1}{2}}f_0\phi dxdv=0, \quad \phi(v)=1,v_1,v_2,v_3, |v|^2.
\eeno

Now problems to be discussed  are the global well-posedness of \eqref{Cauchy-linearizedBE-grazing}, the uniform-in-time asymptotic rate in $\epsilon$ between solutions to \eqref{Cauchy-linearizedBE-grazing} and \eqref{Cauchy-linearizedLE}, and the transition from sub-exponential decay(cf.~\eqref{decay-of-Boltzmann-solution}) of \eqref{Cauchy-linearizedBE-grazing} to exponential decay (cf.~\eqref{decay-of-Landau-solution}) of \eqref{Cauchy-linearizedLE} as $\epsilon$ goes to $0$.

\subsection{Main results}
Before stating the main results in this paper, we recall some useful notations.
For a function $f(v)$ on $\mathbb{R}^{3}$, the following norms or semi-norms will be used:

\medskip
\noindent $\bullet$ The bracket $\langle \cdot\rangle$ is defined by $\langle v \rangle :=(1+|v|^2)^{\frac{1}{2}}$.

\noindent $\bullet$ For $N \in \mathbb{N}$ and $l \in \mathbb{R}$, set
\ben \label{Sobolev-norm}
|f|_{H^{N}_{l}} := \sum_{|\beta| \leq N} |\langle v \rangle^{l} \partial_{\beta}f|_{L^{2}},\quad
|f|_{\dot{H}^{N}_{l}} := \sum_{|\beta| = N} |\langle v \rangle^{l} \partial_{\beta}f|_{L^{2}}.
\een
When $N=0$, denote $|f|_{L^{2}_{l}}:=|f|_{H^{0}_{l}}$, and when $l=0$, denote $|f|_{L^{2}}:=|f|_{L^{2}_{0}}$.

\noindent $\bullet$ With the weighted norm $|\cdot|_{\epsilon, l}$ (from the coercivity estimate of $\mathcal{L}^{\epsilon}$ in Theorem \ref{coercivity-structure}) defined in \eqref{norm-definition},  define
\ben \label{coercivity-norm}
|f|_{H^{N}_{\epsilon, l}} := \sum_{|\beta| \leq N} | \partial_{\beta}f|_{\epsilon,l},\quad
|f|_{\dot{H}^{N}_{\epsilon, l}} := \sum_{|\beta| = N} | \partial_{\beta}f|_{\epsilon,l}.
\een
When $N=0$, we write $|f|_{L^{2}_{\epsilon, l}}:=|f|_{H^{0}_{\epsilon, l}}$.

\noindent $\bullet$ For a function $f(t,x,v)$ on $[0, \infty) \times \mathbb{T}^{3} \times \mathbb{R}^{3}$, and a $X$-norm or semi-norm on the velocity variable $v$, we define for $T>0$ and $m \geq 0$ that
\ben \label{general-L1k-norm}
\|f\|_{L^{1}_{k,m}L^{\infty}_{T}X} := \sum_{k \in \mathbb{Z}^{3}}  \langle k \rangle^{m} \sup_{0 \leq t \leq T} |\hat{f}(t,k,\cdot)|_{X},\quad
\|f\|_{L^{1}_{k,m}L^{2}_{T}X} := \sum_{k \in \mathbb{Z}^{3}} \langle k \rangle^{m} \left(\int_{0}^{T} |\hat{f}(t,k,\cdot)|_{X}^{2}
dt \right)^{\frac{1}{2}}.
\een
Here $\hat{f}$ is the Fourier transform with respect to $x$. When $m=0$, denote
$$\|f\|_{L^{1}_{k}L^{\infty}_{T}X}:= \|f\|_{L^{1}_{k,0}L^{\infty}_{T}X},\quad \|f\|_{L^{1}_{k}L^{2}_{T}X}:=\|f\|_{L^{1}_{k,0}L^{2}_{T}X}.
$$

\noindent $\bullet$ In addition, for a function $f(x,v)$ on $\mathbb{T}^{3} \times \mathbb{R}^{3}$, and a $X$-norm or semi-norm on the velocity variable $v$, define for $m \geq 0$ that
\ben \label{general-L1k-norm-for-initial-data}
\|f\|_{L^{1}_{k,m}X} := \sum_{k \in \mathbb{Z}^{3}}  \langle k \rangle^{m}  |\hat{f}(k,\cdot)|_{X}.
\een
When $m=0$,
$\|f\|_{L^{1}_{k}X}:= \|f\|_{L^{1}_{k,0}X}$.

\noindent $\bullet$ For brevity of notations we denote the energy and dissipation functionals for a function $f(t,x,v)$ on $[0, \infty) \times \mathbb{T}^{3} \times \mathbb{R}^{3}$ as
\ben \label{energy-and-dissipation-for-propagation}
E_{T}(f; m,n,l) := \sum_{j=0}^{n} \|f\|_{L^{1}_{k,m+j}L^{\infty}_{T}\dot{H}^{n-j}_{l-j(\gamma+2s)}},\quad
D_{T}^{\epsilon}(f; m,n,l)  := \sum_{j=0}^{n} \|f\|_{L^{1}_{k,m+j}L^{2}_{T}\dot{H}^{n-j}_{\epsilon,\gamma/2+l-j(\gamma+2s)}},
\een
respectively, and the norm on the initial data $f_{0}(x,v)$
as
\ben \label{initial-dependence-for-propagation}
\|f_{0}\|_{m,n,l}:=\sum_{j=0}^{n} \|f_{0}\|_{L^{1}_{k,m+j}H^{n-j}_{l-j(\gamma+2s)}}.
\een


We begin with Wang and Uhlenbeck type spectral gap estimates for the collision operators.

\begin{thm}\label{spectral-gap-soft-inverse-power-law} Let $-3 < \gamma \leq 0, 0<s<1$. Let $B$ satisfy assumptions {\bf(A1, A2 A3, A4)} where {\bf(A2)} can be relaxed to hold when $\theta$ is sufficiently small.  Suppose that $\mathcal{L}^B$ is the linearized collision  operator associated with $B$. Then there exists a constant $C(\gamma, s, \lambda_e)$ depending only on $\gamma,s$ and $\lambda_e$(see \eqref{nonzeroeigenvalue}) such that if $f\in \mathrm{ker}^{\perp}$
\beno
\langle  \mathcal{L}^{B}f, f\rangle \geq C(\gamma,s,\lambda_e) |f|^{2}_{L^{2}_{\gamma/2}}.
\eeno
\end{thm}
\begin{proof}  Let $f\in \mathrm{ker}^{\perp}$ and  $\epsilon_{0}$ be the constant in Theorem \ref{micro-dissipation}. We denote by
$\mathcal{L}^{B}_{\geq \epsilon_{0}}$ and $\mathcal{L}^{B}_{\leq \epsilon_{0}}$ the linearized operator associated to the Boltzmann kernel $B \mathrm{1}_{\sin(\theta/2) \geq \epsilon_{0}}$ and $B \mathrm{1}_{\sin(\theta/2) \leq \epsilon_{0}}$ respectively.

For $\mathcal{L}^{B}_{\geq \epsilon_{0}}$, by \eqref{mouhotsepctralgap}, we have
\beno
\langle  \mathcal{L}^{B}_{\geq \epsilon_{0}} f, f\rangle \geq  C_\gamma C_{b_{\geq \epsilon_{0}}}|f|_{L^2_{\gamma/2}}^2, \eeno
where $b_{\geq \epsilon_{0}}:=b \mathrm{1}_{\sin(\theta/2) \geq \epsilon_{0}}$.

 For $\mathcal{L}^{B}_{\leq \epsilon_{0}}$, thanks to Theorem \ref{micro-dissipation} and  Remark \ref{remarkgap},  we   get
\ben \nonumber
\langle  \mathcal{L}^{B}_{\leq \epsilon_{0}}f, f\rangle \geq     C\bigg(\gamma, s,\int_{\SS^2} \mathrm{1}_{\sin(\theta/2)\leq \epsilon_{0}} b(\cos\theta)\sin^{2}\frac{\theta}{2} d \sigma\bigg) |f|^{2}_{L^{2}_{\gamma/2}}.
\een
Combining together these two estimates, we can get the desired result and then complete the proof.
\end{proof}

\begin{rmk} In comparison with the previous work \cite{baranger2005explicit,mouhot2006explicit,mouhotstrain}, we highlight  dependence of the estimate on  $\lambda_e$.  As a direct application, we successfully extend the Wang-Uhlenbeck's work to the inverse power law interactions, that is, the kernel $B$ verifies assumptions {\bf(A1, A2,   A4)} and the condition $\gamma+4s=1$.
\end{rmk}

\smallskip

With the notations given above, we present the   result  concerning the global well-posedness, propagation of regularity of solutions and the asymptotic rate in term of $\epsilon $ under the grazing limit for the Cauchy problems  \eqref{Cauchy-linearizedBE-grazing} and \eqref{Cauchy-linearizedLE}  on the Boltzmann and the Landau equation, respectively.

\begin{thm}\label{asymptotic-result} Let $-3< \gamma \leq 0$, $\frac{1}{2}<s<1$, and $\gamma+2s > -1$.
There exist $\epsilon_{0}, \delta_{0} > 0$ such that if $0 < \epsilon  \leq \epsilon_{0}$, $\mu + \mu^{\frac{1}{2}}f_{0} \geq 0$ and $\|f_{0}\|_{L^{1}_{k}L^{2}} \leq \delta_{0}$, then the following statements hold.
\begin{enumerate}
\item{\bf (Global well-posedness)} The Cauchy problem
 \eqref{Cauchy-linearizedBE-grazing} for the non-cutoff Boltzmann equation admits a unique global solution $f^\epsilon$ with $\mu + \mu^{\frac{1}{2}}f^\epsilon \geq 0$ and
 \ben \label{lowest-regularity-bounded-by-initial}
  \|f^\epsilon\|_{L^{1}_{k}L^{\infty}_{T}L^{2}} + \|f^\epsilon\|_{L^{1}_{k}L^{2}_{T}L^{2}_{\epsilon,\gamma/2}} \lesssim \|f_{0}\|_{L^{1}_{k}L^{2}},
 \een
 for any $T \geq 0$. As a result, by passing limit $\epsilon \rightarrow 0$,  the Cauchy problem \eqref{Cauchy-linearizedLE} with the same initial data $f_0$ for the Landau equation admits a unique global solution $f^{L}$ satisfying $\mu + \mu^{\frac{1}{2}}f^{L} \geq 0$ and
 \ben \label{lowest-regularity-bounded-by-initial-landau}
  \|f^{L}\|_{L^{1}_{k}L^{\infty}_{T}L^{2}} + \|f^{L}\|_{L^{1}_{k}L^{2}_{T}L^{2}_{\epsilon,\gamma/2}} \lesssim \|f_{0}\|_{L^{1}_{k}L^{2}},
 \een
 for any $T \geq 0$.
\item{\bf (Propagation of regularity and velocity momoment)} Let $n \in \mathbb{N}$ and $m, l \geq 0$. There is a constant $\delta_{m,n,l}$ with $0<\delta_{m,n,l} \leq \delta_{0}$ and a polynomial $P_{n}$ with $P_{n}(0)=0$ such that if initial data satisfy
$\|f_{0}\|_{L^{1}_{k}L^{2}} \leq \delta_{m,n,l}$ and $ \|f_{0}\|_{m,n,l} < \infty,$ then the following statements are valid.
Let $f^{\epsilon}$ ($f^{L}$) be the solution to the Cauchy problem \eqref{Cauchy-linearizedBE-grazing}(problem \eqref{Cauchy-linearizedLE}) with initial data $f_{0}$,
then for any $T \geq 0$, it holds that
\ben \label{general-propagation-f-epsilon}
E_{T}(f^{\epsilon}; m,n,l) + D_{T}^{\epsilon}(f^{\epsilon}; m,n,l) \lesssim  P_{n}(\|f_{0}\|_{m,n,l}).
\een
As a result, by passing limit $\epsilon \rightarrow 0$,
for any $T \geq 0$, it holds that
\ben \label{general-propagation-f-epsilon-landau}
E_{T}(f^{L}; m,n,l) + D_{T}^{0}(f^{L}; m,n,l) \lesssim  P_{n}(\|f_{0}\|_{m,n,l}).
\een
\item{\bf (Asymptotic formula)} Let $f^{\epsilon}$ and $f^{L}$ be the solutions to the Cauchy problems \eqref{Cauchy-linearizedBE-grazing} and \eqref{Cauchy-linearizedLE}, respectively, with the same initial data $f_{0}$ satisfying
 $\|f_{0}\|_{0,3,9} < \infty$ and $\|f_{0}\|_{L^{1}_{k}L^{2}} \leq \delta_{0,3,9}$, then for any $T \geq 0$, it holds that
\ben \label{global-in-time-error}
\|f^{\epsilon} - f^{L}\|_{L^{1}_{k}L^{\infty}_{T}L^{2}} + \|f^{\epsilon} - f^{L}\|_{L^{1}_{k}L^{2}_{T}L^{2}_{0,\gamma/2}} \lesssim \epsilon P_{3}(\|f_{0}\|_{0,3,9})(1+P_{3}(\|f_{0}\|_{0,3,9})).
\een
\end{enumerate}
\end{thm}
Several remarks on Theorem \ref{asymptotic-result} are given as follows.

\begin{rmk}
The  restriction of $s>1/2$ and $\gamma+2s > -1$ on the parameters $s$ and $\gamma$ comes from Theorem \ref{Gamma-full-up-bound} for the upper bound of the nonlinear term $\Gamma^{\epsilon}$. By $\gamma+2s > -1$, the inverse power law potential   is covered, because $\gamma+4s=1$ is satisfied in this case, cf. also  Remark \ref{inverse-power-law}. Since we aim to investigate the  inconsistency of spectrum in the parameter range $-2 \leq \gamma <-2s$, we only focus on the case of $-3< \gamma \leq 0$ in Theorem \ref{asymptotic-result}.
\end{rmk}

\begin{rmk} 
Note  that all the results in Theorem \ref{asymptotic-result} are uniform with respect to the parameter $\epsilon$ and that the smallness assumption is only imposed on $\|f_{0}\|_{L^{1}_{k}L^{2}}$. In particular, in \eqref{general-propagation-f-epsilon} and \eqref{general-propagation-f-epsilon-landau}, we obtain the propagation of the bounds of solutions in the norm $\|\cdot\|_{L^{1}_{k,m}H^{n}_{l}}$ only under the smallness assumption on $\|f_{0}\|_{L^{1}_{k}L^{2}}$ and boundedness on $\|f_{0}\|_{m,n,l}$.  In comparison, \cite{duan2021global} establishes the propagation in norm $\|\cdot\|_{L^{1}_{k,m}L^{2}}$ under the smallness assumption on $\|f_{0}\|_{L^{1}_{k,m}L^{2}}$.
Moreover, the  asymptotic estimate \eqref{global-in-time-error} is global in time and has an  explicit convergence rate $O(\epsilon)$.
\end{rmk}

\begin{rmk} 
By the weak convergence results in \cite{alexandre2004landau} and \cite{villani1998new}, we can directly use \eqref{lowest-regularity-bounded-by-initial} and \eqref{general-propagation-f-epsilon} to derive \eqref{lowest-regularity-bounded-by-initial-landau} and \eqref{general-propagation-f-epsilon-landau}, respectively. This shows that the well-posedness of Boltzmann and Landau equations can be studied in a unified framework.
\end{rmk}

\begin{rmk} 
Theorem \ref{asymptotic-result} does not include the Coulomb potential since $\gamma>-3$ is required. However, we can deal with the Coulomb case using the idea in \cite{he2014asymptotic}. More precisely, we can take the Boltzmann collision kernel with the mathematical choice of $s$ and $\gamma$ by
$$
s=s_\epsilon:=1-\frac{\epsilon}{4},\quad \gamma=\gamma_\epsilon:=-3+\epsilon,
$$
and consider the limit $\epsilon \rightarrow 0$. After all, we need those uniform operator estimates with respect to the parameter $s$ (near $1$) and $\gamma$ (near $-3$) similar to the situation under consideration. Since in the present paper we are mainly concerned with the spectrum inconsistency in the case of $-2 \leq \gamma <-2s$, we leave the Coulomb case for future work.
\end{rmk}

As the main goal of this work, we state the second result revealing  the transition of the decay structure from sub-exponential $e^{-\lambda t^{\kappa}}$ in \eqref{decay-of-Boltzmann-solution} to exponential $e^{-\lambda t}$ in \eqref{decay-of-Landau-solution} under the grazing limit.

\begin{thm}[Transition of decay structure]\label{decay-rate-consistency}
	Let all the assumptions in Theorem \ref{asymptotic-result} be satisfied and further let $-2 \leq \gamma < -2s$. If the positive constants
	 $\lambda$ and  $q$ are chosen such that $\lambda \ll \lambda_{0}$ and $q>2 \lambda$, where $\lambda_{0}>0$ is the constant  in Theorem \ref{micro-dissipation} given later, then
there is a constant $\delta_{1}>0$ such that if $\|e^{q \langle v \rangle}f_{0}\|_{L^{1}_{k}L^{2}} \leq \delta_{1}$,  the solution $f^{\epsilon}$ to the Cauchy problem \eqref{Cauchy-linearizedBE-grazing} for the non-cutoff Boltzmann equation satisfies
\ben \label{nice-decay-connection}
\|f^{\epsilon}(t)\|_{L^{1}_{k}L^{2}} \lesssim  \left(\mathrm{1}_{t \leq T_{\epsilon}} \exp(-\lambda t)
+ \mathrm{1}_{t > T_{\epsilon}} \exp(-\lambda \epsilon^{-2(1-s)\kappa} t^{\kappa}) \right) \|e^{q \langle v \rangle}f_{0}\|_{L^{1}_{k}L^{2}},
\een
for any $t\geq 0$, where
$ T_{\epsilon}:=(\frac{1}{\epsilon})^{\frac{2(1-s)}{|\gamma+2s|}}$ and  $\kappa:=\frac{1}{1+|\gamma+2s|}.$
\end{thm}

We also give some remarks on this result as follows.

\begin{rmk} Theorem \ref{decay-rate-consistency}  shows that  the
transition of  the decay structure is   continuous in the limit process. Note that the key estimate \eqref{nice-decay-connection} is consistent with the sub-exponential decay rate $e^{-\lambda t^{\kappa}}$ in \eqref{decay-of-Boltzmann-solution} and the exponential rate $e^{-\lambda t}$ in \eqref{decay-of-Landau-solution} by additionally taking into account the dependence of the rate on the vanishing grazing parameter $\epsilon$.  Moreover we introduce the time threshold $T_{\epsilon}$ so as to characterize how the transition occurs as $\epsilon\to 0$. Note that estimate on decay rates for the Boltzmann equation with soft potentials from angular cutoff to non-cutoff was  studied in \cite{he2021}.
\end{rmk}

\begin{rmk}  Theorem \ref{decay-rate-consistency} provides a detailed picture on the uniform-in-$\epsilon$ and global-in-time dynamics of solutions to the Cauchy problem \eqref{Cauchy-linearizedBE-grazing} for the non-cutoff Boltzmann equation in the parameter range $-2 \leq \gamma < -2s$. More precisely, the perturbation converges to zero with the exponential decay rate, i.e. $e^{-\lambda t}$, from initial time to $T_\epsilon$.  When time exactly approaches the critical one $T_{\epsilon}$, the convergent rate continuously changes to the sub-exponential decay, i.e. $\exp(-\lambda \epsilon^{-2(1-s)\kappa} t^{\kappa})$, in terms of the definition of $T_\epsilon$. After the transition time,  the solution keeps the sub-exponential decay rate until infinity. Notice $T_\epsilon\to \infty$ as $\epsilon \to 0$, which then recovers the exponential decay for the Cauchy problem  \eqref{Cauchy-linearizedLE}  on the Landau equation  in case $\gamma+2\geq 0$.
\end{rmk}

\begin{rmk}
The exponential velocity weight assumption $\|e^{q \langle v \rangle}f_{0}\|_{L^{1}_{k}L^{2}}\leq \delta_1$ on initial data is used to get the sub-exponential decay for the Cauchy problem \eqref{Cauchy-linearizedBE-grazing} with $\gamma+2s<0$. It is interesting to study what is the transition of the decay structure if  only a finite-order polynomial velocity weight is imposed on the initial data.
\end{rmk}

\begin{rmk} The constructive constant $\lambda$ in fact gives the lower bound of the first non zero eigenvalue for the linearized Landau collision operator $\mathcal{L}^L$ in Theorem \ref{decay-rate-consistency}. In other words, we have explained the formation of the spectral gap through the large time behavior of the semi-group $e^{t\mathcal{L}^{\epsilon}}$ as $\epsilon$ tends to zero. However, understanding the formation of the spectral gap via the spectrum theory is still a fundamental and  more challenging problem. At the moment we are still far from answering this question.
\end{rmk}

\subsection{Strategy of proof}
In this subsection, we outline the strategy in proving Theorem \ref{asymptotic-result} and \ref{decay-rate-consistency}, which would help the readers to have a better understanding of the key ideas.

\subsubsection{Proof of Theorem \ref{asymptotic-result}}

The proofs of both Theorems \ref{asymptotic-result} and \ref{decay-rate-consistency} rely on some subtle analysis of the linear operator $\mathcal{L}^{\epsilon}$ and the nonlinear operator $\Gamma^{\epsilon}$. Referring to \cite{alexandre2000entropy}, the following quantity
\beno
K^{\epsilon}(\xi):= \int_{\mathbb{S}^2} b^\epsilon(\f{\xi}{|\xi|}\cdot \sigma)\min\{ |\xi|^2\sin^2(\theta/2),1\} d\sigma,
\eeno
offers  the velocity regularity in the frequency space.
By Proposition \ref{symbol} originated in \cite[the proof of Theorem 3.1]{he2014asymptotic}),  we have
$K^{\epsilon}(\xi)+ 1 \gtrsim (W^{\epsilon})^{2}(\xi),$ where
\ben\label{characteristic-function}
W^{\epsilon}(y) := \zeta(\epsilon |y|)\langle y\rangle + \big(1-\zeta(\epsilon |y|)\big) \langle \epsilon^{-1} \rangle^{1-s} \langle y\rangle^{s}.
\een
Here the function $\zeta : [0, \infty) \rightarrow [0,1]$ satisfies that
\ben \label{zeta-property}
\zeta \in C^{\infty};\ \zeta(r)=1 \text{ if } r \in [0,1/2];\  \zeta(r)=0 \text{ if } r \in [1, \infty);\ \zeta \text{ is strictly decreasing on } [1/2, 1].
\een
When $\epsilon=0$, we define $W^{0}(y):=\langle y\rangle$.

As in \cite{he2018asymptotic}, we call $W^{\epsilon}(W^{0})$ the characteristic function  associated to  $\mathcal{L}^\epsilon(\mathcal{L}^{L})$. This is because the function $W^{\epsilon}$ is the common weight gain in phase space, frequency space, and anisotropic space.
More precisely, for $l \in \mathbb{R}$ and $\epsilon \geq 0$, we define
\ben \label{norm-definition}
|f|^{2}_{\epsilon,l} :=|W^{\epsilon}((-\Delta_{\mathbb{S}^{2}})^{\frac{1}{2}}) W_{l}f|^{2}_{L^{2}} + |W^{\epsilon}(D)W_{l}f|^2_{L^{2}} + |W^{\epsilon} W_{l}f|^{2}_{L^{2}}. \een
Here $W^{\epsilon}(D)$ is the pseudo-differential operator with symbol $W^{\epsilon}$.
The operator $W^{\epsilon}((-\Delta_{\mathbb{S}^{2}})^{\frac{1}{2}})$ is defined in the way that for $v = r \sigma$ with $r \geq 0$ and $\sigma \in \mathbb{S}^{2}$,
\beno
(W^{\epsilon}((-\Delta_{\mathbb{S}^{2}})^{\frac{1}{2}})f)(v) := \sum_{l=0}^\infty\sum_{m=-l}^{l} W^\epsilon((l(l+1))^{\frac{1}{2}}) Y^{m}_{l}(\sigma)f^{m}_{l}(r),
\eeno
where
$ f^{m}_{l}(r) = \int_{\mathbb{S}^{2}} Y^{m}_{l}(\sigma) f(r \sigma) d\sigma$,
and $Y_l^m, -l\le m\le l$ are the real spherical harmonics satisfying that
$ (-\triangle_{\mathbb{S}^2})Y_l^m=l(l+1)Y_l^m.$

We use the explicitly defined norm $|\cdot|_{\epsilon,\gamma/2}$ in \eqref{norm-definition} to characterize the lower bound of the linear operator $\mathcal{L}^{\epsilon}$ as well as the upper bound of the nonlinear term $\Gamma^{\epsilon}$.

\medskip
\noindent $\bullet$ {\bf Step 1: Coercivity estimate.} We prove that
\ben \label{uniforml2}
\langle \mathcal{L}^{\epsilon}f, f\rangle + |f|^{2}_{L^{2}_{\gamma/2}} \geq \lambda_{1} |f|^{2}_{\epsilon,\gamma/2},
\een
for some constant $\lambda_{1}>0$ independent of $\epsilon$, see Theorem \ref{coercivity-structure}. That is, the uniform-in-$\epsilon$ coercivity estimate of $\mathcal{L}^{\epsilon}$ is obtained by using the norm $|\cdot|_{\epsilon,\gamma/2}$. Note that by  \eqref{norm-definition}, the norm $|\cdot|_{\epsilon,\gamma/2}^{2}$ has three parts
$|W^{\epsilon}((-\Delta_{\mathbb{S}^{2}})^{\frac{1}{2}}) W_{\gamma/2}f|^{2}_{L^{2}}$, $|W^{\epsilon}(D)W_{\gamma/2}f|^2_{L^{2}}$, and $|W^{\epsilon} W_{\gamma/2}f|^{2}_{L^{2}}$. We use some elementary computations to obtain regularity in  phase space $|W^{\epsilon} W_{\gamma/2}f|^{2}_{L^{2}}$ in Proposition \ref{lowerboundpart1}. By the well-known result in \cite{alexandre2000entropy}, we derive frequency space  regularity $|W^{\epsilon}(D)W_{\gamma/2}f|^2_{L^{2}}$ in Lemma \ref{lowerboundpart1-general-g} and Proposition \ref{lowerboundpart1-gamma-eta}. By referring to \cite{he2018sharp}, we gain the anisotropic norm $|W^{\epsilon}((-\Delta_{\mathbb{S}^{2}})^{\frac{1}{2}}) W_{\gamma/2}f|^{2}_{L^{2}}$ in Proposition \ref{lowerboundpart2-general-g} and Proposition \ref{lowerboundpart1-gamma-eta}. Lemma \ref{l-full-estimate-geq-eta} shows that $\langle \mathcal{L}^{\epsilon}f, f\rangle$ is bounded from above  by $|f|^{2}_{\epsilon,\gamma/2}$. The lower and upper bounds together demonstrate $|\cdot|_{\epsilon,\gamma/2}^{2}$ is the right norm to characterize the inner product
$\langle \mathcal{L}^{\epsilon}f, f\rangle$.

\noindent $\bullet$ {\bf Step 2: Spectrum gap type estimate.} Indeed, we are able to prove that for any $f \in \mathrm{ker}^{\perp}$, it holds that
\ben \label{uniforml2-in-perpen}
\langle \mathcal{L}^{\epsilon}f, f\rangle \geq \lambda_{0} |f|^{2}_{\epsilon,\gamma/2},
\een
for some explicitly computable $\lambda_{0}>0$ independent of $\epsilon$, see Theorem \ref{micro-dissipation}.  This is motivated by  Wang-Uhlenbeck's work \cite{wang1952propagation}  on the  explicit spectral gap estimate for the  Maxwell molecule model $\gamma=0$. Our main idea is to reduce the desired estimate to
the case $\gamma=0$  and utilize at the same time   the coercivity estimate \eqref{uniforml2} to get \eqref{uniforml2-in-perpen} for $-3<\gamma<0$, cf. the proof of Theorem \ref{micro-dissipation} for more details.

\noindent$\bullet$ {\bf Step 3: Upper bound of nonlinear term.}
We use the norm $|\cdot|_{\epsilon,\gamma/2}$ to bound the nonlinear term $\Gamma^{\epsilon}$ as
\ben \label{nonlinear-term-upper-bound}
|\langle \Gamma^{\epsilon}(g,h), f\rangle| \leq C |g|_{L^{2}}|h|_{\epsilon,\gamma/2}|f|_{\epsilon,\gamma/2},
\een
for some $C$ independent of $\epsilon$,  see Theorem \ref{Gamma-full-up-bound}. When $\gamma<0$, the relative velocity $v-v_{*}$ has singularity near $0$. For this,  we  consider $|v-v_{*}| \lesssim 1$ and $|v-v_{*}| \gtrsim 1$ separately. When $|v-v_{*}| \gtrsim 1$, it holds that $|v-v_{*}|^{\gamma} \sim \langle v-v_{*} \rangle^{\gamma}$ so that there is no singularity. When $|v-v_{*}| \lesssim 1$, it holds that $\mu^{\frac{1}{2}}(v_{*}) \lesssim \mu^{1/4}(v)$. Hence,  one can make use of $\mu^{\frac{1}{2}}(v_{*})$ in the definition of $\Gamma^{\epsilon}$ to deal with the weight problem.
We call it the weight transferring idea for a general result, cf. Lemma \ref{mu-weight-transfer}. We carry out the idea roughly in Proposition \ref{regularity-part-Q} but fully in Lemma \ref{I-less-1-some-preparation} and Proposition \ref{I-less-eta-upper-bound}.

\noindent$\bullet$ {\bf Step 4: Global well-posedness.}
 With \eqref{uniforml2-in-perpen} and \eqref{nonlinear-term-upper-bound},  we can implement the standard macro-micro decomposition and take advantage of the functional property of the space $L^{1}_{k}L^{\infty}_{T}L^{2}$ to prove global well-posedness of the Boltzmann equation \eqref{Cauchy-linearizedBE-grazing}. Since the procedure is well-established in \cite{duan2021global}, we directly conclude the global well-posedness result in the space $L^{1}_{k}L^{\infty}_{T}L^{2}$ in Theorem \ref{asymptotic-result}.
 We remark that $|g|_{L^{2}}$ in \eqref{nonlinear-term-upper-bound} corresponds to the $L^{2}$ in $L^{1}_{k}L^{\infty}_{T}L^{2}$.

\noindent$\bullet$ {\bf Step 5: Propagation of regularity.}
To obtain propagation of regularity and velocity moments in Theorem \ref{asymptotic-result}, we first derive a commutator estimate between
$\Gamma^{\epsilon}(g, \cdot)$ and weight function $W_{l,q} := \langle v \rangle^{l} \exp(q \langle v \rangle)$ in Lemma \ref{commutatorgamma}. We then derive  a commutator estimate between
$\mathcal{L}^{\epsilon}$ and weight function $W_{l,q}$ in Lemma  \ref{commutator-linear}.
Note that the case $q=0$ represents polynomial weight which is used in Theorem \ref{asymptotic-result}. Our goal is to prove propagation of norm $\|\cdot\|_{L^{1}_{k,m}H^{n}_{l}}$ under smallness of $\|f_{0}\|_{L^{1}_{k}L^{2}}$ and finiteness of $\|f_{0}\|_{m,n,l}$. We first prove in Theorem \ref{lowest-order-propagation} for the case $n=0$ without  velocity derivative. Then the case $n \geq 1$ is proved using the induction argument in Theorem \ref{high-order-propagation}.

\noindent$\bullet$ {\bf Step 6: Asymptotic formula.}
To derive the asymptotic formula between solutions of the Boltzmann and Landau equations in Theorem \ref{asymptotic-result}, we first show the error estimate of $\Gamma^{\epsilon}-\Gamma^{L}$ in Lemma \ref{estimate-operator-difference}. Taking difference between \eqref{Cauchy-linearizedBE-grazing} and \eqref{Cauchy-linearizedLE}, we get an equation for the solution difference $f^{\epsilon} - f^{L}$. We then apply the energy method to the equation to derive \eqref{global-in-time-error} by using the propagation result \eqref{general-propagation-f-epsilon} and error estimate of $\Gamma^{\epsilon}-\Gamma^{L}$.

\subsubsection{Proof of Theorem \ref{decay-rate-consistency}}
We will apply the time-weighted energy method together with the time-velocity splitting technique to establish the transition time-decay structure \eqref{nice-decay-connection} in Theorem  \ref{decay-rate-consistency}. Such approach was initiated by Caflisch \cite{Caf1,Caf2} to treat the spatially homogeneous Boltzmann equation with cutoff soft potentials in torus and later developed by Strain-Guo  \cite{SG-CPDE,SG-08-ARMA} in the spatially inhomogeneous setting as well as by Gressman-Strain \cite{GS} for the non-cutoff case. The key for obtaining the time-decay of solutions is to impose an extra velocity weight on initial data that can be either  polynomial or exponential, inducing the polynomial or sub-exponential rate, respectively.  In what follows we explain the main points in the proof of Theorem \ref{decay-rate-consistency}.

First of all, we apply the Caflisch's idea to determine the time threshold $T_\epsilon$ as it plays the most important role in carrying out the time-velocity splitting technique under the uniform grazing limit. In terms of the energy dissipation norm $|\cdot|_{\epsilon,\gamma/2}$ in \eqref{uniforml2-in-perpen} and \eqref{norm-definition} and the function $W^{\epsilon}$ in \eqref{characteristic-function}, we introduce the following toy model for explanation:
\beno
\partial_t f+v\cdot \nabla_x f+\lambda a_\epsilon (v) f=0,
\eeno
with the constant $\lambda>0$ suitably small and
\beno
a_\epsilon (v):=\zeta(\epsilon |v|) \langle v\rangle^{\gamma+2} + [1-\zeta(\epsilon |v|)]  \frac{\langle v\rangle^{\gamma+2s}}{\epsilon^{2(1-s)}}.
\eeno
The solution is explicitly given by $f(t,x,v)=e^{-\lambda a_\epsilon (v) t}f_0(x-vt,v)$. Assuming that initial data $f_0$ decays in velocity at an exponential rate $\exp (-\lambda \langle v\rangle^\vartheta)$ with $0<\vartheta\leq 2$, one can formally bound $f(t,x,v)$ as
\beno
|f(t,x,v)|\lesssim  \exp (-\lambda b_\epsilon (t,v))\lesssim \exp (-\lambda \inf_v b_\epsilon (t,v)),\quad b_\epsilon (t,v):=a_\epsilon (v)t+\langle v\rangle^\vartheta.
\eeno
To look for a lower bound of $b_\epsilon(t,v)$ in velocity that should depend only on time, we may compute in the parameter range $-2\leq \gamma<-2s$ that
\beno
b_\epsilon(t,v) &=&\zeta (\epsilon |v|) \{\langle v\rangle^{\gamma+2} t+\langle v\rangle^\vartheta\} +[1-\zeta(\epsilon |v|)]\{\langle v\rangle^{\gamma+2s}\frac{t}{\epsilon^{2(1-s)}} +\langle v\rangle^\vartheta\} \\
&\geq & \zeta(\epsilon|v|) t +[1-\zeta(\epsilon |v|)] (\frac{t}{\epsilon^{2(1-s)}})^{\kappa}\\
&\geq & \min \{t,(\frac{t}{\epsilon^{2(1-s)}})^{\kappa}\}=t\mathrm{1}_{t< T_\epsilon} +(\frac{t}{\epsilon^{2(1-s)}})^{\kappa}\mathrm{1}_{t\geq T_\epsilon}.
\eeno
Here we have used the inequalities
\beno
\inf_v\{\langle v\rangle^{\gamma+2} t+\langle v\rangle^\vartheta\} \geq t,\quad \inf_v\{\langle v\rangle^{\gamma+2s}\frac{t}{\epsilon^{2(1-s)}} +\langle v\rangle^\vartheta\}\geq (\frac{t}{\epsilon^{2(1-s)}})^{\kappa},
\eeno
with $\kappa:=\frac{\vartheta}{\vartheta+|\gamma+2s|}$. Moreover, the time threshold $T_\epsilon>0$ has to be chosen such that
$(\frac{t}{\epsilon^{2(1-s)}})^{\kappa}=t$
at $t=T_\epsilon$, implying that
\beno
T_\epsilon= (\frac{1}{\epsilon^{2(1-s)}})^{\frac{\kappa}{1-\kappa}}=(\frac{1}{\epsilon})^{\frac{2\vartheta(1-s)}{|\gamma+2s|}}.
\eeno
In such way the solution $f(t,x,v)$ decays in large time as
\beno
|f(t,x,v)|\lesssim \mathrm{1}_{t \leq T_{\epsilon}} \exp(-\lambda t)
+ \mathrm{1}_{t > T_{\epsilon}} \exp(-\lambda \epsilon^{-2 (1-s)\kappa} t^{\kappa}),
\eeno
whenever $\sup_{x,v} e^{\lambda \langle v\rangle^\vartheta}|f_0(x,v)|<\infty$ holds. Therefore, the transition time-decay structure motivates us to define
\ben \label{auxiliary-function-t}
A_{\epsilon}(t):= \zeta(T_{\epsilon}^{-1}t) t + (1-\zeta(T_{\epsilon}^{-1}t))(\frac{t}{\epsilon^{2(1-s)}})^{\kappa},
\een
and  obtain the energy estimate on $h(t,x,v):=e^{\lambda A_{\epsilon}(t)}f(t,x,v)$ for $\lambda>0$ suitably small. It turns out that after repeating those known energy estimates, it suffices to obtain the uniform bound on
\beno
\sqrt{\lambda} \sum_{k \in \mathbb{Z}^{3}}\left(\int_{0}^{T} A_{\epsilon}^{\prime}(t) \|(1-\zeta)\widehat{h}(t, k)\|_{L^{2}}^{2} d t\right)^{\frac{1}{2}},
\eeno
where the fact that $0\leq A_{\epsilon}^{\prime}(t)\lesssim 1$ and $1-\zeta$ is supported in $|v| \geq \frac{1}{2\epsilon}$
has been used. Then, whenever the velocity-weighted norm $\|\exp(q \langle v \rangle^\vartheta) f(t) \|_{L^{1}_{k}L^{2}}$ is bounded uniformly in time, the time-velocity splitting technique can be applied by separating the time interval into three parts as
\beno
\int_0^Tds=\left(\int_0^{\frac{1}{2\epsilon}} +\int_{\frac{1}{2\epsilon}}^{T_\epsilon}+\int_{T_\epsilon}^T\right)ds
\eeno
for any $T>T_\epsilon$, cf. Section \ref{decay-rate} for more details. Back to propagation of the exponential velocity moments,   using the commutator estimate in Lemma \ref{commutatorgamma} and the upper bound estimate in Theorem \ref{Gamma-full-up-bound}, we have the following weighted upper bound estimate
\ben \label{exponetial-polynomial-commutator-Gamma} |\langle \Gamma^{\epsilon}(g, h), W_{l,q}^{2} f\rangle| \lesssim |W_{l,q} g|_{L^{2}}|W_{l,q}h|_{\epsilon,\gamma/2}|W_{l,q} f|_{\epsilon,\gamma/2},
\een
as shown in Lemma \ref{general-weight-upper-bound}.
Note that \eqref{exponetial-polynomial-commutator-Gamma} is stronger than Lemma 4.1 in \cite{duan2021global}. Therefore, by using the proof of Theorem 2.1 in \cite{duan2021global}, we get the propagation of norm $\|\exp(q \langle v \rangle) f(t) \|_{L^{1}_{k}L^{2}}$ under the smallness assumption on $\|\exp(q \langle v \rangle)f_{0}\|_{L^{1}_{k}L^{2}}$, cf. Theorem \ref{propagation-of-moment}. We remark that for the exponential weight $\exp(q \langle v \rangle^\vartheta)$ we can treat the case $\vartheta=1$ only, cf.~\cite{duan2013stability} and \cite{duan2021global}, due to the specific property of the non-cutoff Boltzmann operator.




\subsection{Usual notations and organization of the paper}
  Denote the multi-index $\beta =(\beta_1,\beta_2,\beta_3)$ with
$|\beta |=\beta _1+\beta _2+\beta _3$. $a\lesssim b$  means that  there is a
generic constant $C$
such that $a\leq Cb$.  The notation $a\sim b$ implies that $a\lesssim b$ and $b\lesssim
a$.
The weight function  $W_l(v):= \langle v\rangle^l $.
We denote $C(\lambda_1,\lambda_2,\cdots, \lambda_n)$ or $C_{\lambda_1,\lambda_2,\cdots, \lambda_n}$  by a constant depending on   parameters $\lambda_1,\lambda_2,\cdots, \lambda_n$.
The notations  $\langle f,g\rangle:= \int_{\R^3}f(v)g(v)dv$ and $(f,g):= \int_{\R^3\times\TT^3} fgdxdv$
are used to denote the two standard inner products
 for $v$ variable and for $x,v$ variables respectively.
As usual, $\mathrm{1}_A$ is the characteristic function of the set $A$. If $A,B$ are two operators, then $[A,B]:= AB-BA$.
Define $|f|_{L \mathrm{log}L}:=\int_{\R^3} |f(v)|\log(1+|f(v)|) d v$.

Finally, the rest of the paper will be organized as follows.
In Section \ref{coercivity-spectral}, we will prove the coercivity estimate given in Theorem \ref{coercivity-structure}
and the spectral gap estimate given in Theorem \ref{micro-dissipation}. In section \ref{upper-bound}, we
will focus on  the upper bound estimate given in Theorem \ref{Gamma-full-up-bound}. In addition, with some commutator estimates, we will also prove upper bound
estimates with polynomial or exponential weights.
The two main theorems will be proved in Section \ref{propagation-asymptotic} and
 Section \ref{decay-rate} respectively. In the Appendix \ref{appendix}, we include some known results that are used in Section \ref{coercivity-spectral}-\ref{decay-rate} for completeness.

\section{Coercivity and spectral gap estimate} \label{coercivity-spectral}
In this section, we will prove coercivity estimate of the linear operator $\mathcal{L}^{\epsilon}$  in Theorem \ref{coercivity-structure} and the spectral gap estimate in Theorem \ref{micro-dissipation}. Unless otherwise specified, the parameter range is $-3 < \gamma \leq 0, 0<s<1$.


In the rest of the paper, we will omit the range of some frequently used variables in the integrals for brevity.
Usually, $\sigma \in \mathbb{S}^{2}, v, v_{*}, u, \xi \in \mathbb{R}^{3}$.
For example, we set $\int (\cdots) d\sigma := \int_{\mathbb{S}^{2}} (\cdots) d\sigma,  \int (\cdots) d\sigma dv dv_{*}
:= \int_{\mathbb{S}^{2} \times \mathbb{R}^{3} \times \mathbb{R}^{3}} (\cdots) d\sigma dv dv_{*}$.
Integration with respect to other variables should be understood in a similar way.
Whenever a new variable appears, we will specify its range once and then omit it thereafter.

\subsection{Elementary results}
In this subsection, we give some preliminaries which will be used frequently in the rest of the paper.
We first list some properties of $W^{\epsilon}$ defined in \eqref{characteristic-function}. Note that $W^{\epsilon}$ is a radical function defined on $\mathbb{R}^{3}$. By the definition \eqref{characteristic-function}, we have
\ben
\label{middle-lower-upper} \langle y \rangle \leq W^{\epsilon}(y) \leq  \langle \epsilon^{-1} \rangle^{1-s} \langle y\rangle^{s}, \text{ if } \frac{1}{2\epsilon} \leq |y| \leq \frac{1}{\epsilon}.
\\
\label{lower-bound-when-small-cf} W^{\epsilon}(y) =   \langle y \rangle, \text{ if } |y| \leq \frac{1}{2\epsilon}.
\\
\label{lower-bound-when-large-cf} W^{\epsilon}(y) =  \langle \epsilon^{-1} \rangle^{1-s} \langle y\rangle^{s}, \text{ if } |y| \geq \frac{1}{\epsilon}.
\\
\label{low-frequency-lb-cf}  W^{\epsilon}(y) \gtrsim  \zeta(\epsilon |y|) \langle y \rangle.
\\
\label{high-frequency-lb-cf}  W^{\epsilon}(y) \gtrsim \left(1-\zeta(\epsilon |y|)\right) \langle \epsilon^{-1} \rangle^{1-s} \langle y \rangle^{s} \gtrsim \left(1-\zeta(\epsilon |y|)\right)  \epsilon^{-1}.
\een
For any $ x, y \in \mathbb{R}^{3}$, one can check that
\ben
\label{non-decreasing-cf} W^{\epsilon}(x) \leq W^{\epsilon}(y), \text{ if } |x| \leq |y|.
\\
\label{separate-into-2-cf} W^{\epsilon}(x-y) \lesssim  W^{\epsilon}(x) W^{\epsilon}(y).
\een

Let us compute an integral regarding to the angular function $b^{\epsilon}$ over the sphere $\mathbb{S}^{2}$.
Recall that $b^{\epsilon}(\cos\theta) =(1-s)\epsilon^{2s-2}\sin^{-2-2s}(\theta/2) \mathrm{1}_{\sin(\theta/2)\leq \epsilon}$. Note that $d\sigma =\sin\theta d\theta d\phi= 4 \sin (\theta/2) d \sin (\theta/2) d\phi$, we have
\ben \label{order-2}
\int b^{\epsilon}(\cos\theta)  \sin^{2}(\theta/2)  d\sigma &=&  4 (1-s)\epsilon^{2s-2}\int_{0}^{\pi} \int_{0}^{2\pi} \mathrm{1}_{\sin(\theta/2)\leq \epsilon}\sin^{1-2s}(\theta/2)  d \sin (\theta/2) d\phi
\nonumber \\&=&8\pi (1-s)\epsilon^{2s-2} \int_{0}^{\epsilon} u^{1-2s} du
=4\pi .
\een

Let us recall cancellation lemma, which is used when one needs to shift regularity between $h$ and $f$ in the inner product $\langle Q(g, h), f \rangle$. By Lemma 1 in \cite{alexandre2000entropy}, we have
\begin{lem}[Cancellation lemma, \cite{alexandre2000entropy}] \label{cancellation-lemma-general-gamma} Recalling $B^{\epsilon}$ in \eqref{Blotzmann-kernel}, then
\beno
\int B^{\epsilon}(v-v_{*},\sigma) g_{*}(h^{\prime}-h) dv d v_{*} d\sigma= C(\epsilon) \int |v-v_{*}|^{\gamma}g_{*} h  dv d v_{*}.
\eeno
where $C(\epsilon)$ is some constant depending on $\epsilon$ and $|C(\epsilon)| \lesssim 1$ thanks to \eqref{order-2}.
\end{lem}

Next, let us present a result regarding to the Riesz potential, whose proof can be found in Lemma 2.7 of \cite{he2018asymptotic}.
\begin{lem}\label{aftercancellation} Set
$
A :=  \int |v-v_{*}|^{\gamma}g_{*} h f dv d v_{*}.
$
Then
\begin{itemize}
\item if $-\frac{3}{2}<\gamma \leq 0$, then
$ |A| \lesssim  (|g|_{L^{2}_{|\gamma|}}+|g|_{L^{1}_{|\gamma|}})|h|_{L^{2}_{\gamma/2}}|f|_{L^{2}_{\gamma/2}};$
\item if $-3<\gamma \leq -\frac{3}{2}$, then for $\eta>0, s_1,s_2 \geq 0$  such that $s_{1} + s_{2}= -\frac{3}{2}-\gamma + \eta$ there holds
\beno
|A| \lesssim C_{\eta}(|g|_{L^1_{|\gamma|}}+|g|_{H^{s_{1}}_{|\gamma|}})|h|_{H^{s_{2}}_{\gamma/2}}|f|_{H^{0}_{\gamma/2}}.
\eeno
\end{itemize}
 \end{lem}
As a result of Lemma \ref{cancellation-lemma-general-gamma} and Lemma \ref{aftercancellation}, we have
\begin{col}\label{cancellation-to-Sob-norm} Fix $\eta>0$, let $s_1,s_2 \geq 0$ verify $s_{1} + s_{2}= \max\{-\frac{3}{2}-\gamma + \eta, 0\}$, then
\beno |\int B^{\epsilon}(v-v_{*},\sigma) g_{*} \left((hf)^{\prime}-hf\right) dv d v_{*} d\sigma| \lesssim C_{\eta}(|g|_{L^1_{|\gamma|}}+|g|_{H^{s_{1}}_{|\gamma|}})|h|_{H^{s_{2}}_{\gamma/2}}|f|_{H^{0}_{\gamma/2}}.
\eeno
\end{col}

\subsection{Coercivity estimate}
We present the coercivity estimate of $\mathcal{L}^{\epsilon}$ in the following theorem.
\begin{thm}\label{coercivity-structure}
There exists a constant $\epsilon_0>0$ such that for $0 < \epsilon \le\epsilon_0$ and any smooth function $f$, there exist is a constant $\lambda_{1}>0$ depending only on $\gamma$, such that
\beno
\langle \mathcal{L}^{\epsilon}f, f\rangle + |f|^{2}_{L^{2}_{\gamma/2}} \geq \lambda_{1} |f|^{2}_{\epsilon,\gamma/2}.
\eeno
\end{thm}
In this subsection, we will prove Theorem \ref{coercivity-structure}.
Our strategy is based on the following relation in the spirit of the triple norm introduced in
\cite{alexandre2012boltzmann} (see the proof of Theorem \ref{coercivity-structure} in subsection \ref{coercivity}):
\ben \label{equivalence-relation} \langle \mathcal{L}^{\epsilon} f,f \rangle+|f|_{L^2_{\gamma/2}}^2
\gtrsim \mathcal{N}^{\epsilon,\gamma}(\mu^{\frac{1}{2}},f) + \mathcal{N}^{\epsilon,\gamma}(f,\mu^{\frac{1}{2}}),
\\ \label{definition-of-functional-N}
\mathcal{N}^{\epsilon,\gamma}(g,h) := \int b^{\epsilon}(\cos\theta)| v-v_{*} |^{\gamma}
g^{2}_{*} (h^{\prime}-h)^{2} d\sigma dv dv_{*}.\een
Thanks to \eqref{equivalence-relation},
to get the coercivity estimate of $\mathcal{L}^{\epsilon}$, it suffices to estimate from below the two functionals
$\mathcal{N}^{\epsilon,\gamma}(\mu^{\frac{1}{2}},f)$  and $ \mathcal{N}^{\epsilon,\gamma}(f,\mu^{\frac{1}{2}})$.
 We will study $ \mathcal{N}^{\epsilon,\gamma}(f,\mu^{\frac{1}{2}})$ in  subsection \ref{gain-weight} and $\mathcal{N}^{\epsilon,\gamma}(\mu^{\frac{1}{2}},f)$ in subsection \ref{gain-S-regularity}, \ref{gain-A-regularity} and \ref{gain-regularity}.
The coercivity estimate is obtained in subsection \ref{coercivity} by utilizing \eqref{equivalence-relation}.
\subsubsection{Gain of weight from $\mathcal{N}^{\epsilon,\gamma}(f,\mu^{\frac{1}{2}}) $} \label{gain-weight}
The functional  $\mathcal{N}^{\epsilon,\gamma}(f,\mu^{\frac{1}{2}})$ yields weight $W^{\epsilon}$ in the phase space as shown in the following proposition.
\begin{prop}\label{lowerboundpart1}
There exists $\epsilon_{0} >0$ such that for  $0< \epsilon \leq \epsilon_{0}$,
\beno
\mathcal{N}^{\epsilon,\gamma}(f,\mu^{\frac{1}{2}}) + |f|^{2}_{L^{2}_{\gamma/2}} \geq  C|W^{\epsilon}f|^{2}_{L^{2}_{\gamma/2}},
\eeno
where $C>0$ is a universal constant.
\end{prop}
\begin{proof}  The proof is divided into four steps.

{\it{Step 1: $16/\pi \leq |v_{*}| \leq \delta/\epsilon.$}} The parameter $0<\delta \leq 1$ will be determined later.
We consider the set $A(\epsilon, \delta) := \{(v_{*},v,\sigma): 16/\pi \leq |v_{*}| \leq \delta/\epsilon, |v| \leq 8/\pi, \sin(\theta/2)
\leq \epsilon\}$.
When $\epsilon \leq \frac{\pi}{16}\delta$, it is easy to check that $A(\epsilon, \delta) $ is  non-empty.
We restrict the integral on the set $A(\epsilon, \delta)$ to get
\begin{eqnarray}\label{reduce-to-set-A}
\mathcal{N}^{\epsilon,\gamma}(f,\mu^{\frac{1}{2}}) \geq \int B^{\epsilon} \mathrm{1}_{A(\epsilon, \delta)} f_{*}^{2}
((\mu^{\frac{1}{2}})^{\prime}-\mu^{\frac{1}{2}})^{2} d\sigma dv dv_{*}.
\end{eqnarray}
Note that $\nabla \mu^{\frac{1}{2}} = -\frac{\mu^{\frac{1}{2}}}{2} v$ and $\nabla^{2} \mu^{\frac{1}{2}} = \frac{\mu^{\frac{1}{2}}}{4} (-2I_{3}+v \otimes v)$. By Taylor expansion,
we have
\beno
\mu^{\frac{1}{2}}(v^{\prime}) - \mu^{\frac{1}{2}}(v) = -\frac{\mu^{\frac{1}{2}}(v)}{2} v \cdot (v^{\prime}-v) + \int_{0}^{1} (1-\kappa)
 (\nabla^{2} \mu^{\frac{1}{2}}) (v(\kappa)):(v^{\prime}-v)\otimes(v^{\prime}-v) d \kappa,
\eeno
where $v(\kappa) = v+\kappa (v^{\prime}-v)$.
Thanks to the fact $(a-b)^{2} \geq \frac{a^{2}}{2} - b^{2}$, we have
\beno
(\mu^{\frac{1}{2}}(v^{\prime}) - \mu^{\frac{1}{2}}(v))^{2} \geq \frac{\mu(v)}{8} |v \cdot (v^{\prime}-v)|^{2} - \int_{0}^{1}  |(\nabla^{2} \mu^{\frac{1}{2}}) (v(\kappa))|^{2}|v^{\prime}-v|^{4} d \kappa.
\eeno
Plugging this into \eqref{reduce-to-set-A} to get
\begin{eqnarray}\label{vsmallvstarsmall}
\mathcal{N}^{\epsilon,\gamma}(f,\mu^{\frac{1}{2}}) &\geq& \frac{1}{8}\int B^{\epsilon} \mathrm{1}_{A(\epsilon, \delta)} \mu(v)|v \cdot (v^{\prime}-v)|^{2} f_{*}^{2}  d\sigma dv dv_{*}
\\&& -  \int B^{\epsilon} \mathrm{1}_{A(\epsilon, \delta)} |(\nabla^{2} \mu^{\frac{1}{2}}) (v(\kappa))|^{2}|v^{\prime}-v|^{4} f_{*}^{2}  d\sigma dv dv_{*} d\kappa  := \frac{1}{8}\mathcal{I}_{1}^{\epsilon} (\delta) - \mathcal{I}_{2}^{\epsilon} (\delta). \nonumber
\end{eqnarray}

To estimate $\mathcal{I}_{1}^{\epsilon} (\delta)$, for fixed $v, v_*$, we introduce an orthonormal basis $(h^{1}_{v,v_{*}},h^{2}_{v,v_{*}}, \frac{v-v_{*}}{|v-v_{*}|})$ such that $d\sigma= \sin\theta d\theta d\phi$. We express $\frac{v^{\prime}-v}{|v^{\prime}-v|}$ and $\frac{v}{|v|}$ using the basis as follows:
\beno
\frac{v^{\prime}-v}{|v^{\prime}-v|} = \cos\frac{\theta}{2}\cos\phi h^{1}_{v,v_{*}} + \cos\frac{\theta}{2}\sin\phi h^{2}_{v,v_{*}} -\sin\frac{\theta}{2} \frac{v-v_{*}}{|v-v_{*}|},
\\
\frac{v}{|v|} = c_{1} h^{1}_{v,v_{*}} + c_{2} h^{2}_{v,v_{*}} + c_{3} \frac{v-v_{*}}{|v-v_{*}|},
\eeno
where $c_{3}=\frac{v}{|v|}\cdot \frac{v-v_{*}}{|v-v_{*}|}$ and $c_{1}, c_{2}$ are  constants independent of $\theta$ and $\phi$. Then we have
\beno
\frac{v}{|v|} \cdot \frac{v^{\prime}-v}{|v^{\prime}-v|}  = c_{1}\cos\frac{\theta}{2}\cos\phi + c_{2}\cos\frac{\theta}{2}\sin\phi - c_{3}\sin\frac{\theta}{2},
\eeno
and thus
\beno
|\frac{v}{|v|} \cdot \frac{v^{\prime}-v}{|v^{\prime}-v|}|^{2}  &=& c^{2}_{1}\cos^{2}\frac{\theta}{2}\cos^{2}\phi + c^{2}_{2}\cos^{2}\frac{\theta}{2}\sin^{2}\phi + c^{2}_{3}\sin^{2}\frac{\theta}{2}
\\ && + 2c_{1}c_{2}\cos^{2}\frac{\theta}{2}\cos\phi\sin\phi - 2c_{3}\cos\frac{\theta}{2}\sin\frac{\theta}{2}(c_{1}\cos\phi + c_{2}\sin\phi).
\eeno
Integrating with respect to $\sigma$, we have
\beno
\int b^{\epsilon}(\cos\theta)\mathrm{1}_{A(\epsilon, \delta)}|v \cdot (v^{\prime}-v)|^{2}d\sigma &=& \int_{0}^{\pi}\int_{0}^{2\pi}b^{\epsilon}(\cos\theta)\sin\theta \mathrm{1}_{A(\epsilon, \delta)}|v \cdot (v^{\prime}-v)|^{2}d\phi d\theta
\\ &\geq& \pi(c^{2}_{1}+c^{2}_{2})|v|^{2}|v-v_{*}|^{2}\mathrm{1}_{B(\epsilon, \delta)},
\eeno
where $B(\epsilon, \delta) = \{(v_{*},v):16/\pi \leq |v_{*}| \leq \delta/\epsilon, |v| \leq 8/\pi\}$. Plugging the above estimate into the definition of
$\mathcal{I}_{1}^{\epsilon} (\delta)$, we get
\beno
\mathcal{I}_{1}^{\epsilon} (\delta) &\geq&  \pi\int (c^{2}_{1}+c^{2}_{2})|v-v_{*}|^{\gamma+2}|v|^{2}\mathrm{1}_{B(\epsilon, \delta)} \mu(v) f_{*}^{2}dv dv_{*}
\\&=& \pi\int (1-(\frac{v}{|v|}\cdot\frac{v_{*}}{|v_{*}|})^{2})|v_{*}|^{2}|v-v_{*}|^{\gamma}|v|^{2} \mathrm{1}_{B(\epsilon, \delta)} \mu(v) f_{*}^{2}dv dv_{*},
\eeno
where we have used the fact $c_1^2+c_2^2+c_3^2=1$ and  $$(1-(\frac{v}{|v|}\cdot\frac{v_{*}}{|v_{*}|})^{2})^{-1}|v-v_{*}|^{2}= (1-c_3^2)^{-1}|v_{*}|^{2}.$$
Note that in the region $B(\epsilon, \delta)$, one has $\frac{1}{2} | v_* | \leq |v-v_{*}| \leq \frac{3}{2} | v_* |$.
Since $\gamma \leq 0$, then
\ben \label{v-vstar-is-vstar-norm}
(\frac{3}{2})^{\gamma} | v_* |^{\gamma} \leq |v-v_{*}|^{\gamma}  \leq  (\frac{1}{2})^{\gamma} | v_* |^{\gamma}.
\een
We then get
\beno
\mathcal{I}_{1}^{\epsilon} (\delta) &\geq&
\pi (\frac{3}{2})^{\gamma}\int (1-(\frac{v}{|v|}\cdot\frac{v_{*}}{|v_{*}|})^{2})|v_{*}|^{\gamma+2} |v|^{2}\mathrm{1}_{B(\epsilon, \delta)} \mu(v) f_{*}^{2}dv dv_{*}
\\&=& \pi (\frac{3}{2})^{\gamma} c_{1}
 \int  | v_{*} |^{\gamma+2}\mathrm{1}_{16/\pi \leq |v_{*}|\leq \delta/\epsilon}  f_{*}^{2} dv_{*},
\eeno
where $c_{1} = \int (1-(\frac{v}{|v|}\cdot\frac{v_{*}}{|v_{*}|})^{2}) |v|^{2} \mu(v) \mathrm{1}_{|v|\leq 8/\pi}dv$ is independent of $v_{*}$.

We now estimate $\mathcal{I}_{2}^{\epsilon} (\delta)$. Recall $B^{\epsilon} = (1-s) \epsilon^{2s-2} \mathrm{1}_{\sin(\theta/2) \leq \epsilon} |v-v_{*}|^{\gamma} \sin^{-2-2s}(\theta/2), |v^{\prime}-v| = |v-v_{*}|\sin(\theta/2)$, we have
\beno
\mathcal{I}_{2}^{\epsilon} (\delta) &=& \int B^{\epsilon} \mathrm{1}_{A(\epsilon, \delta)} |(\nabla^{2} \mu^{\frac{1}{2}}) (v(\kappa))|^{2}|v^{\prime}-v|^{4} f_{*}^{2}  d\sigma dv dv_{*} d\kappa
\\&=& \epsilon^{2s-2}\int \sin^{2-2s}(\theta/2) \mathrm{1}_{A(\epsilon, \delta)} |(\nabla^{2} \mu^{\frac{1}{2}}) (v(\kappa))|^{2}|v-v_{*}|^{\gamma+4} f_{*}^{2}  d\sigma dv dv_{*} d\kappa
\\&\leq& (\frac{3}{2})^{\gamma+4} \epsilon^{2s-2}\int \sin^{2-2s}(\theta/2) \mathrm{1}_{A(\epsilon, \delta)} |(\nabla^{2} \mu^{\frac{1}{2}}) (v(\kappa))|^{2}|v_{*}|^{\gamma+4} f_{*}^{2}  d\sigma dv dv_{*} d\kappa.
\\&=& (\frac{3}{2})^{\gamma+4} \epsilon^{2s-2} \int_{0}^{\pi}\int_{0}^{2\pi} \sin^{2-2s}(\theta/2) \int \mathrm{1}_{A(\epsilon, \delta)} |(\nabla^{2} \mu^{\frac{1}{2}}) (v(\kappa))|^{2}|v_{*}|^{\gamma+4} f_{*}^{2}  \sin\theta d\theta d\phi dv dv_{*} d\kappa.
\eeno
Fix $\kappa, v_{*}, \phi$, in the change of variable $(v,\theta) \rightarrow (v(\kappa),\theta(\kappa))$ where $\theta(\kappa)$ is the angle between $\sigma$  and $v(\kappa) - v_{*}$, we have
\ben 
\nonumber
|\frac{\partial (v(\kappa), \theta(\kappa))}{\partial (v, \theta)}|^{-1} \leq  (1-\frac{\kappa}{2})^{-5} \leq 32 =2^{5},
\\ \label{change-of-variable-angle-1}
\theta/2 \leq \theta(\kappa) \leq \theta,
\sin\theta \leq 2\sin\theta(\kappa),  \sin(\theta(\kappa)/2) \leq \sin(\theta/2) \leq 2\sin(\theta(\kappa)/2),
\\ \label{change-of-variable-angle-2}
\sin(\theta/2) \leq \epsilon \Rightarrow \sin(\theta(\kappa)/2) \leq \epsilon.
\een
Hence, we have $\mathrm{1}_{A(\epsilon, \delta)} \leq \mathrm{1}_{16/\pi \leq |v_{*}| \leq \delta/\epsilon, \sin(\theta(\kappa)/2) \leq \epsilon} $ so that
\beno
\mathcal{I}_{2}^{\epsilon} (\delta)
&\leq& 2^{6}2^{2-2s}(\frac{3}{2})^{\gamma+4}  \epsilon^{2s-2} \int_{0}^{\pi} \int_{0}^{2\pi}  \int \sin^{2-2s}(\theta(\kappa)/2) \mathrm{1}_{16/\pi \leq |v_{*}| \leq \delta/\epsilon, \sin(\theta(\kappa)/2) \leq \epsilon}
\\&&\times |(\nabla^{2} \mu^{\frac{1}{2}}) (v(\kappa))|^{2}|v_{*}|^{\gamma+4} f_{*}^{2}
 \sin\theta(\kappa) d\theta(\kappa) d\phi dv(\kappa) dv_{*} d\kappa
\\&=& 2^{9-2s}\pi(\frac{3}{2})^{\gamma+4}  \epsilon^{2s-2} \left( \int_{0}^{\pi} \sin^{2-2s}(\theta/2) \mathrm{1}_{\sin(\theta/2) \leq \epsilon} \sin\theta d\theta \right)
\\&&\times
\left(\int \mathrm{1}_{16/\pi \leq |v_{*}| \leq \delta/\epsilon} |(\nabla^{2} \mu^{\frac{1}{2}}) (v)|^{2}|v_{*}|^{\gamma+4} f_{*}^{2}   dv dv_{*} \right).
\eeno
Direct computation gives
\beno
\int_{0}^{\pi} \sin^{2-2s}(\theta/2) \mathrm{1}_{\sin(\theta/2) \leq \epsilon} \sin\theta d\theta
= 4 \int_{0}^{\epsilon} u^{3-2s} d u  = \frac{4\epsilon^{4-2s}}{4-2s}.
\eeno
Let $c_{2} = \int   |(\nabla^{2} \mu^{\frac{1}{2}}) (v)|^{2}  dv$. Then we get
\beno
\mathcal{I}_{2}^{\epsilon} (\delta)
&\leq&  2^{11-2s}\pi(\frac{3}{2})^{\gamma+4} (4-2s)^{-1} c_{2}\epsilon^{2}
\left(  \int   \mathrm{1}_{16/\pi \leq |v_{*}| \leq \delta/\epsilon}|v_{*}|^{\gamma+4} f_{*}^{2}  dv_{*} \right)
\\&\leq&  2^{11-2s}\pi(\frac{3}{2})^{\gamma+4} (4-2s)^{-1} c_{2}\delta^{2} \int \mathrm{1}_{16/\pi \leq |v_{*}| \leq \delta/\epsilon}|v_{*}|^{\gamma+2} f_{*}^{2} dv_{*},
\eeno
where we have used   $\epsilon |v_{*}| \leq \delta $.

Plugging the estimates of $\mathcal{I}_{1}^{\epsilon} (\delta)$ and $\mathcal{I}_{2}^{\epsilon} (\delta)$ into \eqref{vsmallvstarsmall}, we get
\beno
\mathcal{N}^{\epsilon,\gamma}(f,\mu^{\frac{1}{2}}) \geq \left(C_{1} -C_{2} \delta^{2}\right)  \int \mathrm{1}_{16/\pi \leq |v_{*}| \leq \delta/\epsilon}|v_{*}|^{\gamma+2} f_{*}^{2} dv_{*},
\eeno
where $C_{1}=2^{-3}\pi (\frac{3}{2})^{\gamma} c_{1}, C_{2}=2^{11-2s}\pi(\frac{3}{2})^{\gamma+4} (4-2s)^{-1} c_{2}$. By choosing $\delta$ such that $C_{2} \delta^{2} = C_{1}/2$, we get
\ben \label{lowerboundvstarsmall}
\mathcal{N}^{\epsilon,\gamma}(f,\mu^{\frac{1}{2}}) \geq \frac{1}{2}C_{1}
 \int \mathrm{1}_{16/\pi \leq |v_{*}| \leq \delta/\epsilon}|v_{*}|^{\gamma+2} f_{*}^{2} dv_{*}.
\een

{\it{Step 2: $|v_{*}| \geq R/\epsilon.$}} Here $R \geq 1$ is a parameter to be determined later. By direct computation, we have
\beno
\mathcal{N}^{\epsilon,\gamma}(f,\mu^{\frac{1}{2}}) &=& \int B^{\epsilon} f_{*}^{2} ((\mu^{\frac{1}{2}})^{\prime}-\mu^{\frac{1}{2}})^{2} d\sigma dv dv_{*}
\\&\geq& \int B^{\epsilon} \mathrm{1}_{4^{-1}R|v_{*}|^{-1} \leq \sin\theta/2 \leq R|v_{*}|^{-1}}
\mathrm{1}_{|v_{*}|\geq R/\epsilon} \mathrm{1}_{|v|\leq 1} f_{*}^{2} ((\mu^{\frac{1}{2}})^{\prime}-\mu^{\frac{1}{2}})^{2} d\sigma dv dv_{*}
\\&\geq&
\int b^{\epsilon} |v-v_{*}|^{\gamma} \mathrm{1}_{4^{-1}R|v_{*}|^{-1} \leq \sin\theta/2 \leq R|v_{*}|^{-1}}\mathrm{1}_{|v_{*}|\geq R/\epsilon} \mathrm{1}_{|v|\leq 1} f_{*}^{2} \mu d\sigma dv dv_{*}
\\&&-2\int b^{\epsilon} |v-v_{*}|^{\gamma} \mathrm{1}_{4^{-1}R|v_{*}|^{-1} \leq \sin\theta/2 \leq R|v_{*}|^{-1}} \mathrm{1}_{|v_{*}|\geq R/\epsilon} \mathrm{1}_{|v|\leq 1} f_{*}^{2}  (\mu^{\frac{1}{2}})^{\prime}\mu^{\frac{1}{2}} d\sigma dv dv_{*}
\\&:=&\mathcal{J}_{1}^{\epsilon}(R)-\mathcal{J}_{2}^{\epsilon}(R).
\eeno
Note that
\ben \label{order-0-restricted-bound}
\int b^{\epsilon} \mathrm{1}_{4^{-1}R|v_{*}|^{-1} \leq \sin\theta/2 \leq R|v_{*}|^{-1}} d\sigma &=& 8\pi (1-s) \epsilon^{2s-2} \int_{4^{-1}R|v_{*}|^{-1}}^{R|v_{*}|^{-1}} u^{-1-2s} du
\\&=& 4\pi \frac{4^{2s}-1}{s}(1-s)R^{-2s}\epsilon^{2s-2}|v_{*}|^{2s} = C_{s} R^{-2s}\epsilon^{2s-2}|v_{*}|^{2s}, \nonumber
\een
where $C_{s} = 4\pi \frac{4^{2s}-1}{s}(1-s)$.
If $\epsilon \leq \frac{1}{2}, |v_{*}|\geq R/\epsilon \geq 2, |v|\leq 1$, we have
\ben \label{v-vstar-is-vstar} \frac{1}{2}|v_{*}| \leq |v-v_{*}| \leq \frac{3}{2}|v_{*}|.\een
Plugging  \eqref{order-0-restricted-bound} into the definition of $\mathcal{J}_{1}^{\epsilon}(R)$ and using \eqref{v-vstar-is-vstar-norm}, we have
\beno
\mathcal{J}_{1}^{\epsilon}(R) &\geq& C_{s} R^{-2s}\epsilon^{2s-2} \int |v-v_{*}|^{\gamma}|v_{*}|^{2s} \mathrm{1}_{|v_{*}|\geq R/\epsilon} \mathrm{1}_{|v|\leq 1} f_{*}^{2} \mu d\sigma dv dv_{*}
\\&\geq& C_{s} (\frac{3}{2})^{\gamma} c_{3} R^{-2s}\epsilon^{2s-2}\int |v_{*}|^{\gamma+2s} \mathrm{1}_{|v_{*}|\geq R/\epsilon} f_{*}^{2}  dv_{*},
\eeno
where $c_{3}= \int \mathrm{1}_{|v|\leq 1} \mu(v)  dv$.

Since $\sin(\theta/2) \geq \epsilon$, there holds $|v^{\prime}|+|v| \geq |v^{\prime}-v| = \sin\frac{\theta}{2}|v-v_{*}|\geq\epsilon|v-v_{*}|\geq \epsilon(|v_{*}|-|v|)$, and then $|v^{\prime}|+(1+\epsilon)|v| \geq \epsilon|v_{*}| \geq R$. Thus,
$R^{2} \leq (|v^{\prime}|+2|v|)^{2} \leq 8 (|v^{\prime}|^{2}+|v|^{2}),$
which implies
\ben \label{relax-cross-term}
\mu^{\prime \frac{1}{2}}\mu^{\frac{1}{2}} = (2\pi)^{-\frac{3}{2}} e^{-\frac{|v^{\prime}|^{2}+|v|^{2}}{4}} \leq (2\pi)^{-\frac{3}{2}}e^{-\frac{|v|^{2}}{8}}e^{-\frac{R^{2}}{2^{6}}}.
\een
Then by \eqref{relax-cross-term}, \eqref{order-0-restricted-bound}, and \eqref{v-vstar-is-vstar},
we have
\beno
\mathcal{J}_{2}^{\epsilon}(R) &=&2\int b^{\epsilon} |v-v_{*}|^{\gamma}  \mathrm{1}_{4^{-1}R|v_{*}|^{-1} \leq \sin\theta/2 \leq R|v_{*}|^{-1}}\mathrm{1}_{|v_{*}|\geq R/\epsilon} \mathrm{1}_{|v|\leq 1} f_{*}^{2}  \mu^{\prime \frac{1}{2}}\mu^{\frac{1}{2}} d\sigma dv dv_{*}
\\ &\leq&2 (2\pi)^{-\frac{3}{2}} e^{-\frac{R^{2}}{2^{6}}}  \int b^{\epsilon} |v-v_{*}|^{\gamma} \mathrm{1}_{4^{-1}R|v_{*}|^{-1} \leq \sin\theta/2 \leq R|v_{*}|^{-1}} \mathrm{1}_{|v_{*}|\geq R/\epsilon} \mathrm{1}_{|v|\leq 1} f_{*}^{2}  e^{-\frac{|v|^{2}}{8}} d\sigma dv dv_{*}
\\ &\leq& 2 (2\pi)^{-\frac{3}{2}} C_{s} e^{-\frac{R^{2}}{2^{6}}}  R^{-2s}\epsilon^{2s-2}
\int |v-v_{*}|^{\gamma} |v_{*}|^{2s}\mathrm{1}_{|v_{*}|\geq R/\epsilon} \mathrm{1}_{|v|\leq 1} f_{*}^{2}  e^{-\frac{|v|^{2}}{8}} dv dv_{*}
\\ &\leq& 2 (2\pi)^{-\frac{3}{2}} C_{s} (\frac{1}{2})^{\gamma} e^{-\frac{R^{2}}{2^{6}}}  R^{-2s}\epsilon^{2s-2}
\int |v_{*}|^{\gamma+2s} \mathrm{1}_{|v_{*}|\geq R/\epsilon} \mathrm{1}_{|v|\leq 1} f_{*}^{2}  e^{-\frac{|v|^{2}}{8}} dv dv_{*}
\\ &=&2 (2\pi)^{-\frac{3}{2}} C_{s} (\frac{1}{2})^{\gamma} c_{4} e^{-\frac{R^{2}}{2^{6}}}  R^{-2s}\epsilon^{2s-2}
\int |v_{*}|^{\gamma+2s} \mathrm{1}_{|v_{*}|\geq R/\epsilon} f_{*}^{2} dv_{*},
\eeno
where $c_{4}= \int \mathrm{1}_{|v|\leq 1} e^{-\frac{|v|^{2}}{8}} dv.$
Combining  the above estimates for $\mathcal{J}_{1}^{\epsilon}(R)$ and $\mathcal{J}_{2}^{\epsilon}(R),$ we arrive at for any $\epsilon \leq \frac{1}{2}, R \geq 1$,
\beno
\mathcal{N}^{\epsilon,\gamma}(f,\mu^{\frac{1}{2}}) \geq  (C_{3} - C_{4} e^{-\frac{R^{2}}{2^{6}}})  R^{-2s}\epsilon^{2s-2} \int | v_{*} |^{\gamma+2s}\mathrm{1}_{|v_{*}|\geq R/\epsilon}  f_{*}^{2} dv_{*},
\eeno
where $C_{3}=C_{s} (\frac{3}{2})^{\gamma} c_{3}, C_{4}= 2 (2\pi)^{-\frac{3}{2}} C_{s} (\frac{1}{2})^{\gamma} c_{4}$. We choose $R \geq 1$ such that $\frac{1}{2}C_{3} \geq C_{4} e^{-\frac{R^{2}}{2^{6}}}$ and arrive at
\ben \label{general-result-large}
\mathcal{N}^{\epsilon,\gamma}(f,\mu^{\frac{1}{2}}) \geq  \frac{1}{2} C_{3}  R^{-2s} \epsilon^{2s-2} \int | v_{*} |^{\gamma+2s}\mathrm{1}_{|v_{*}|\geq R/\epsilon}  f_{*}^{2} dv_{*}.
\een

{\it{Step 3: $16/\pi \leq |v_{*}| \leq R/\epsilon.$}} Here $R$ is the fixed constant in {\it{Step 2}}.
We also recall the fixed constant $\delta$ in {\it{Step 1}}. Since $16/\pi \leq |v_{*}| \leq R/\epsilon = \delta/(\delta R^{-1}\epsilon)$, by \eqref{lowerboundvstarsmall}, we have
\beno
\mathcal{N}^{\delta R^{-1}\epsilon,\gamma}(f,\mu^{\frac{1}{2}}) \geq \frac{1}{2}C_{1}
 \int \mathrm{1}_{16/\pi \leq |v_{*}| \leq R/\epsilon}|v_{*}|^{\gamma+2} f_{*}^{2} dv_{*}.
\eeno
Observe
\beno
b^{\delta R^{-1}\epsilon}(\cos\theta) =  (R/\delta)^{2-2s} (1-s)\epsilon^{2s-2} \sin^{-2-2s}(\theta/2) \mathrm{1}_{\sin(\theta/2) \leq \delta R^{-1}\epsilon} \leq
(R/\delta)^{2-2s} b^{\epsilon}(\cos\theta).
\eeno
From which we get
\ben \label{geq-delta-ep-to-minus-1}
\mathcal{N}^{\epsilon,\gamma}(f,\mu^{\frac{1}{2}}) \geq (R/\delta)^{2s-2} \mathcal{N}^{\delta R^{-1}\epsilon,\gamma}(f,\mu^{\frac{1}{2}})
\geq \frac{1}{2}C_{1}(R/\delta)^{2s-2}  \int \mathrm{1}_{16/\pi \leq |v_{*}| \leq R/\epsilon}|v_{*}|^{\gamma+2} f_{*}^{2} dv_{*}.
\een

{\it{Step 4: To recover weight $W^{\epsilon}$.}}
Combining \eqref{general-result-large}, \eqref{geq-delta-ep-to-minus-1} and $|f|^{2}_{L^{2}_{\gamma/2}} \geq  \int \mathrm{1}_{|v_{*}|\leq 16/\pi}  \langle v_{*} \rangle^{\gamma} f_{*}^{2} dv_{*},$
we arrive at
\beno
\mathcal{N}^{\epsilon,\gamma}(f,\mu^{\frac{1}{2}}) + |f|^{2}_{L^{2}_{\gamma/2}} &\geq& \int \mathrm{1}_{|v_{*}|\leq 16/\pi}  \langle v_{*} \rangle^{\gamma}  f_{*}^{2} dv_{*} +
\frac{1}{4}C_{1}(R/\delta)^{2s-2}\int  | v_{*} |^{\gamma+2}\mathrm{1}_{16/\pi \leq |v_{*}|\leq R/\epsilon}  f_{*}^{2} dv_{*}
\\&& +\frac{1}{4} C_{3}  R^{-2s} \epsilon^{2s-2}\int | v_{*} |^{\gamma+2s}\mathrm{1}_{|v_{*}|\geq R/\epsilon}  f_{*}^{2} dv_{*}.
\eeno
Since  $|v_{*}| \geq 16/\pi \geq 4$, we get $|v_{*}|^{2} \leq 1+ |v_{*}|^{2} \leq \frac{17}{16}|v_{*}|^{2}$, which gives
\beno
| v_{*} |^{\gamma+2} \geq \min \{ 1, (17/16)^{-\gamma/2-1}\} \langle v_{*} \rangle^{\gamma+2},  | v_{*} |^{\gamma+2s} \geq \min \{ 1, (17/16)^{-\gamma/2-s}\} \langle v_{*} \rangle^{\gamma+2s}.
\eeno
Suppose $\epsilon \leq 1/4$, we have $ \epsilon^{2s-2}  \geq (17/16)^{s-1}\langle \epsilon^{-1} \rangle^{2-2s}$.
Therefore,
we get
\beno
\mathcal{N}^{\epsilon,\gamma}(f,\mu^{\frac{1}{2}}) + |f|^{2}_{L^{2}_{\gamma/2}} &\geq& \int \mathrm{1}_{|v_{*}|\leq 16/\pi}  \langle v_{*} \rangle^{\gamma}  f_{*}^{2} dv_{*} +
C_{5} \int  \langle v_{*} \rangle^{\gamma+2} \mathrm{1}_{16/\pi \leq |v_{*}|\leq R/\epsilon}  f_{*}^{2} dv_{*}
\\&& +C_{6}\langle \epsilon^{-1} \rangle^{2-2s}\int \langle v_{*} \rangle^{\gamma+2s}\mathrm{1}_{|v_{*}|\geq R/\epsilon}  f_{*}^{2} dv_{*},
\eeno
where $C_{5}=\frac{1}{4}C_{1}(R/\delta)^{2s-2} \min \{ 1, (17/16)^{-\gamma/2-1}\}, C_{6}= \frac{1}{4} C_{3}  R^{-2s} (17/16)^{s-1} \min \{ 1, (17/16)^{-\gamma/2-s}\}.$

By \eqref{non-decreasing-cf} and \eqref{lower-bound-when-small-cf},
we have $\mathrm{1}_{|v_{*}| \leq 16/\pi}W^{\epsilon}(v_{*}) = \langle v_{*} \rangle \leq W_{1}(16/\pi) =(1+4 (16/\pi)^{2})^{\frac{1}{2}}.$ Then we get
\ben \label{reduce-to-the-exact-form-near-origin}  \mathrm{1}_{|v_{*}| \leq 16/\pi}  \geq  (1+4(16/\pi)^{2})^{-1}\mathrm{1}_{|v_{*}| \leq 16/\pi} (W^{\epsilon})^{2}(v_{*}).\een
In the region $16/\pi \leq |v_{*}|\leq \frac{1}{2}\epsilon^{-1}$, by \eqref{lower-bound-when-small-cf}, we have
\ben \label{reduce-to-the-exact-form}  \langle v_{*} \rangle^{2} = (W^{\epsilon})^{2}(v_{*}). \een
In the region $\frac{1}{2}\epsilon^{-1} \leq |v_{*}|\leq \epsilon^{-1}$, by \eqref{middle-lower-upper},
we have
\ben \label{reduce-to-the-exact-form-2}  \langle v_{*} \rangle^{2} \geq \langle \frac{1}{2}\epsilon^{-1} \rangle^{2-2s} \langle v_{*} \rangle^{2s} \geq (\frac{1}{2})^{2-2s}\langle \epsilon^{-1} \rangle^{2-2s} \langle v_{*} \rangle^{2s}  \geq (\frac{1}{2})^{2-2s} (W^{\epsilon})^{2}(v_{*}). \een
In the region $\epsilon^{-1} \leq |v_{*}|\leq R\epsilon^{-1}$, by \eqref{lower-bound-when-large-cf}, we have
\ben \label{reduce-to-the-exact-form-far-origin} \langle v_{*} \rangle^{2} \geq \langle \epsilon^{-1} \rangle^{2-2s} \langle v_{*} \rangle^{2s} = (W^{\epsilon})^{2}(v_{*}). \een
In the region $|v_{*}| \geq R\epsilon^{-1}$, by \eqref{lower-bound-when-large-cf}, we have
\ben \label{reduce-to-the-exact-form-far-origin-2} \langle \epsilon^{-1} \rangle^{2-2s} \langle v_{*} \rangle^{2s} = (W^{\epsilon})^{2}(v_{*}).\een

Then by \eqref{reduce-to-the-exact-form-near-origin}, \eqref{reduce-to-the-exact-form}, \eqref{reduce-to-the-exact-form-2},  \eqref{reduce-to-the-exact-form-far-origin} and \eqref{reduce-to-the-exact-form-far-origin-2}, we get
\beno
\mathcal{N}^{\epsilon,\gamma}(f,\mu^{\frac{1}{2}}) + |f|^{2}_{L^{2}_{\gamma/2}}&\geq&
 (1+4(16/\pi)^{2})^{-1}  \int \mathrm{1}_{|v_{*}|\leq 16/\pi}  (W^{\epsilon})^{2}(v_{*}) \langle v_{*} \rangle^{\gamma}  f_{*}^{2} dv_{*}
\\&& +  (\frac{1}{2})^{2-2s} C_{5} \int  (W^{\epsilon})^{2}(v_{*}) \langle v_{*} \rangle^{\gamma}\mathrm{1}_{16/\pi < |v_{*}|\leq \epsilon^{-1}}  f_{*}^{2} dv_{*}
\\&& + \min \{C_{5},C_{6}\} \int  (W^{\epsilon})^{2}(v_{*}) \langle v_{*} \rangle^{\gamma}\mathrm{1}_{|v_{*}| > \epsilon^{-1}}  f_{*}^{2} dv_{*}
\\ &\geq& C(\gamma,s)|W^{\epsilon}f|^{2}_{L^{2}_{\gamma/2}},
\eeno
which completes the proof. Here $C(\gamma,s) := \min\{(1+4(16/\pi)^{2})^{-1}, (\frac{1}{2})^{2-2s} C_{5}, C_{6}\}$ is a positive constant depending only on $\gamma, s$. It is easy to check that $C(\gamma,s) \gtrsim 1$ uniformly when $-3 < \gamma \leq 0, 0<s<1$.
\end{proof}

In the following, we show that Proposition \ref{lowerboundpart1} is sharp.
\begin{prop} \label{upperboundpart} The estimate
$ \mathcal{N}^{\epsilon,\gamma}(f,\mu^{\frac{1}{2}}) \lesssim |W^{\epsilon}f|^{2}_{L^{2}_{\gamma/2}}$ holds.
\end{prop}
\begin{proof}
First we have
\beno
\mathcal{N}^{\epsilon,\gamma}(f,\mu^{\frac{1}{2}}) &\lesssim& \int B^{\epsilon} f_{*}^{2} ((\mu^{\frac{1}{4}})^{\prime}-\mu^{\frac{1}{4}})^{2}(\mu^{\prime \frac{1}{2}}+\mu^{\frac{1}{2}}) d\sigma dv dv_{*}
\\&\lesssim& \int  B^{\epsilon} f_{*}^{2} ((\mu^{\frac{1}{4}})^{\prime}-\mu^{\frac{1}{4}})^{2}\mu^{\prime \frac{1}{2}} d\sigma dv dv_{*} + \int  B^{\epsilon} f_{*}^{2} ((\mu^{\frac{1}{4}})^{\prime}-\mu^{\frac{1}{4}})^{2}\mu^{\frac{1}{2}} d\sigma dv dv_{*}
\\&:=& \mathcal{K}^{\epsilon,\gamma}_{1}(f) + \mathcal{K}^{\epsilon,\gamma}_{2}(f).
\eeno
By Taylor expansion, one has
$((\mu^{\frac{1}{4}})^{\prime} - \mu^{\frac{1}{4}})^{2} \lesssim \min\{1,|v-v_{*}|^{2}\theta^{2}\} \sim \min\{1,|v^{\prime}-v_{*}|^{2}\theta^{2}\}.$
By Proposition \ref{symbol} and \eqref{separate-into-2-cf}, we have
$
\int b^{\epsilon}(\cos\theta) \min\{1,|v-v_{*}|^{2}\theta^{2}\} d\sigma \lesssim (W^{\epsilon})^2(|v-v_*|) \lesssim (W^{\epsilon})^{2}(v)(W^{\epsilon})^{2}(v_{*}),
$
which gives
\beno
\mathcal{K}^{\epsilon,\gamma}_{2}(f) \lesssim
\int f_{*}^{2}|v-v_{*}|^{\gamma}(W^{\epsilon})^{2}(v)(W^{\epsilon})^{2}(v_{*})  \mu^{\frac{1}{2}} dv dv_{*} \lesssim \int f_{*}^{2}\langle v_{*}\rangle^{\gamma}(W^{\epsilon})^{2}(v_{*}) dv_{*}= |W^\epsilon f|_{L^2_{\gamma/2}}^2.
\eeno
Here we have used the fact that $\int |v-v_{*}|^{\gamma} \mu^{\frac{1}{4}} dv \lesssim \langle v_{*}\rangle^{\gamma}$.
By the change of variable $v \rightarrow v^{\prime}$, similarly we have $\mathcal{K}^{\epsilon,\gamma}_{1}(f) \lesssim |W^\epsilon f|_{L^2_{\gamma/2}}^2$.
The proof of the lemma is completed.
\end{proof}

\begin{rmk} \label{also-hold-for-a}
By the proof of Proposition \ref{upperboundpart}, for $a \geq \frac{1}{8}$,
the estimate $\mathcal{N}^{\epsilon,\gamma}(f,\mu^{a}) \lesssim |W^{\epsilon}f|^{2}_{L^{2}_{\gamma/2}}$ holds.
\end{rmk}

\subsubsection{Gain of Sobolev regularity from $\mathcal{N}^{\epsilon,0}(g,f)$ } \label{gain-S-regularity}
By Corollary 2.1 and Lemma 3 of \cite{alexandre2000entropy}, and Proposition \ref{symbol}, we have
the following lemma.
\begin{lem}\label{lowerboundpart1-general-g}
Let $g$ be a function such that $|g^{2}|_{L^{1}} \geq \delta >0, |g^{2}|_{L^{1}_{1}} + |g^{2}|_{L \mathrm{log}L} \leq \lambda < \infty$, then there exists $C(\delta, \lambda)$ such that
\beno
\mathcal{N}^{\epsilon,0}(g,f)+ |f|^{2}_{L^{2}} \geq  C(\delta, \lambda)|W^{\epsilon}(D)f|^{2}_{L^{2}} .
\eeno
\end{lem}

\subsubsection{Gain of  anisotropic regularity from $\mathcal{N}^{\epsilon,0}(g,f)$ } \label{gain-A-regularity}
In this part, we derive the anisotropic regularity from  $\mathcal{N}^{\epsilon,0}(g,f)$. To this end, we apply a geometric decomposition in the frequency space. More precisely, we  will use the following decomposition (see \eqref{geo-deco-frequency-space-another} in the proof of Proposition \ref{lowerboundpart2-general-g})
\ben \label{sphere-radius}
\hat{f}(\xi) - \hat{f}(\xi^{+}) &=& \underbrace{\hat{f}(\xi) - \hat{f}(|\xi|\frac{\xi^{+}}{|\xi^{+}|})}_{\text{ spherical part }}+ \underbrace{\hat{f}(|\xi|\frac{\xi^{+}}{|\xi^{+}|}) - \hat{f}(\xi^{+}).}_{\text{ radical part }}
\een

The ``spherical part" gives the anisotropic regularity. Namely,
\begin{lem}\label{spherical-part} Set $\mathcal{A}^{\epsilon}(f) := \int b^{\epsilon}(\frac{\xi}{|\xi|} \cdot \sigma)|\hat{f}(\xi) - \hat{f}(|\xi|\frac{\xi^{+}}{|\xi^{+}|})|^{2} d\xi d\sigma$ where $\xi^{+} = \frac{\xi + |\xi|\sigma}{2}$, then
\beno
\mathcal{A}^{\epsilon}(f) +|f|_{L^2}^2
\sim |W^{\epsilon}((-\Delta_{\mathbb{S}^{2}})^{\frac{1}{2}})f|^{2}_{L^{2}}+|f|^{2}_{L^{2}}.
\eeno
\end{lem}
\begin{proof}
Let $ r= |\xi|, \tau = \xi/|\xi|$ and $\varsigma = \frac{\tau+\sigma}{|\tau+\sigma|}$, then $\frac{\xi}{|\xi|} \cdot \sigma = 2(\tau\cdot\varsigma)^{2} - 1$ and $|\xi|\frac{\xi^{+}}{|\xi^{+}|} = r \varsigma$. In the change of variable $(\xi, \sigma) \rightarrow (r, \tau, \varsigma)$, one has
$
d\xi d\sigma = 4  (\tau\cdot\varsigma) r^{2} dr d \tau d \varsigma.
$
Let $\theta$ be the angle between $\tau$ and $\sigma$, then $2 \sin\frac{\theta}{2} = |\tau-\sigma|$ and thus
\ben \label{b-to-metric}
b^{\epsilon}(\cos\theta) = (1-s)\epsilon^{2s-2} 2^{2+2s} |\tau-\sigma|^{-2-2s} \mathrm{1}_{|\tau-\sigma| \leq 2\epsilon}.
\een
Since $\theta/2$ is the angle between the two unit vectors $\tau$ and $\varsigma$, we get $|\tau - \varsigma| = 2(1-\cos \frac{\theta}{2})$.
It is easy to check that
\beno \frac{1}{2} |\tau-\sigma| \leq |\tau - \varsigma| \leq |\tau-\sigma|. \eeno
From which together with \eqref{b-to-metric}, we get
\beno
b^{\epsilon}(\cos\theta) \geq (1-s)\epsilon^{2s-2} |\tau - \varsigma|^{-2-2s} \mathrm{1}_{|\tau - \varsigma| \leq \epsilon}, b^{\epsilon}(\cos\theta) \leq (1-s) \epsilon^{2s-2} 2^{2+2s} |\tau - \varsigma|^{-2-2s} \mathrm{1}_{|\tau - \varsigma| \leq 2\epsilon}.
\eeno

By \eqref{anisotropic-R-3} in Lemma  \ref{aniso-from-He-Sharp-Bounds}, we have
\beno
\mathcal{A}^{\epsilon}(f) +|f|_{L^2}^2 &=& 4 \int b^{\epsilon}(2(\tau\cdot\varsigma)^{2} - 1)|\hat{f}(r\tau) - \hat{f}(r\varsigma)|^{2} (\tau\cdot\varsigma) r^{2} dr d \tau d \varsigma+|f|_{L^2}^2
\\&\geq& 4(1-s)\epsilon^{2s-2} \int \frac{|\hat{f}(r\tau) - \hat{f}(r\varsigma)|^{2}}{|\tau - \varsigma|^{2+2s}}\mathrm{1}_{|\tau-\varsigma| \leq \epsilon}  r^{2} dr d \tau d \varsigma+|f|_{L^2}^2
\\&\sim& |W^{\epsilon}((-\Delta_{\mathbb{S}^{2}})^{\frac{1}{2}})\hat{f}|^{2}_{L^{2}}+ |\hat{f}|^{2}_{L^{2}}.
\eeno
By Remark \ref{another-factor-equivalent},
we have
\beno
\mathcal{A}^{\epsilon}(f) +|f|_{L^2}^2 &\leq& 4(1-s)\epsilon^{2s-2} 2^{2+2s}\int \frac{|\hat{f}(r\tau) - \hat{f}(r\varsigma)|^{2}}{|\tau - \varsigma|^{2+2s}}\mathrm{1}_{|\tau-\varsigma| \leq 2\epsilon}  r^{2} dr d \tau d \varsigma+|f|_{L^2}^2
\\&\sim& |W^{\epsilon}((-\Delta_{\mathbb{S}^{2}})^{\frac{1}{2}})\hat{f}|^{2}_{L^{2}}+ |\hat{f}|^{2}_{L^{2}}.
\eeno
With the help of Lemma \ref{comWep} and Plancherel's theorem, $|W^{\epsilon}((-\Delta_{\mathbb{S}^{2}})^{\frac{1}{2}})\hat{f}|^{2}_{L^{2}} = |W^{\epsilon}((-\Delta_{\mathbb{S}^{2}})^{\frac{1}{2}})f|^{2}_{L^{2}}$, which completes the proof.
\end{proof}

The ``radical part" in \eqref{sphere-radius} can be bounded by weight $W^{\epsilon}$ gain in the phase and frequency space. Namely,

\begin{lem}\label{gammanonzerotozero}
	Let $
	\mathcal{Z}^{\epsilon,\gamma}(f) := \int b^{\epsilon}(\frac{\xi}{|\xi|}\cdot\sigma)\langle \xi\rangle^{\gamma} |f(|\xi|\frac{\xi^{+}}{|\xi^{+}|}) - f(\xi^{+})|^{2} d\xi d\sigma
	$ with $\xi^{+} = \frac{\xi + |\xi|\sigma}{2}$. Then
	\beno
	\mathcal{Z}^{\epsilon,\gamma}(f) \lesssim |W^{\epsilon}(D)W_{\gamma/2}f|^{2}_{L^{2}} + |W^{\epsilon}W_{\gamma/2}f|^{2}_{L^{2}}.
	\eeno
\end{lem}
\begin{proof} We divide the proof into two steps.
	
{\it Step 1: $\gamma=0$.} As in the proof of Lemma \ref{spherical-part}, with the same change of variable $(\xi, \sigma) \rightarrow (r, \tau, \varsigma)$ with $\xi=r\tau$ and $\varsigma=\f{\sigma+\tau}{|\sigma+\tau|}$, we have
	\beno
	 \mathcal{Z}^{\epsilon, 0}(f) = 4 \int b^{\epsilon}(2(\tau\cdot\varsigma)^{2} - 1)|f(r\varsigma) - f((\tau\cdot\varsigma)r\varsigma)|^{2} (\tau\cdot\varsigma) r^{2} dr d \tau d \varsigma.
	\eeno
Recalling that
\beno
b^{\epsilon}(\frac{\xi}{|\xi|}\cdot\sigma)  = (1-s)\epsilon^{2s-2} (\sin \frac{\theta}{2})^{-2-2s} \mathrm{1}_{\sin \frac{\theta}{2} \leq \epsilon},
\eeno
where $\theta$ be the angle between $\tau$ and $\sigma$. Let $\alpha = \frac{\theta}{2}$ be the angle between $\tau$ and $\sigma$. Let $u = r \varsigma$, then we get $r^{2} dr d \varsigma d \tau= \sin \alpha du  d\alpha d\mathbb{S}$. Therefore we have
	\ben \label{f-to-hat-f-in-future}
	\mathcal{Z}^{\epsilon,0}(f) &=&  8 \pi (1-s)\epsilon^{2s-2} \int_{\R^{3}}\int_{0}^{\pi} (\sin \alpha)^{-1-2s} \mathrm{1}_{\sin \alpha \leq \epsilon}|f(u) - f(u \cos\alpha)|^{2} \cos\alpha d u d\alpha
	\\ \label{Z-epsilon-0-result} &\lesssim& |W^{\epsilon}(D) f|^{2}_{L^{2}} + |W^{\epsilon} f|^{2}_{L^{2}},
	\een
where the last inequality is given by Lemma \ref{a-technical-lemma}.

{\it Step 2: $\gamma<0$.}	We  reduce the case when $\gamma<0$ to the special case $\gamma = 0$. For simplicity, denote $w = |\xi|\frac{\xi^{+}}{|\xi^{+}|}$, then $W_{\gamma}(\xi) = W_{\gamma}(w)$. Hence, we have
	\beno
	&&\langle \xi\rangle^{\gamma} (f(w)-f(\xi^{+}))^{2} \\&=& \{[(W_{\gamma/2}f)(w)-(W_{\gamma/2}f)(\xi^{+})]
	 + (W_{\gamma/2}f)(\xi^{+})(1-W_{\gamma/2}(w)W_{-\gamma/2}(\xi^{+}))\}^{2}
	\\&\leq& 2 [(W_{\gamma/2}f)(\xi^{+})-(W_{\gamma/2}f)(w)]^{2}
	+ 2 |(W_{\gamma/2}f)(\xi^{+})|^{2}|1-W_{\gamma/2}(w)W_{-\gamma/2}(\xi^{+})|^{2}.
	\eeno
	Thus we have
	\beno
	\mathcal{Z}^{\epsilon,\gamma}(f) &\leq& 2\mathcal{Z}^{\epsilon,0}(W_{\gamma/2}f)
	+ 2\int b^{\epsilon}(\frac{\xi}{|\xi|}\cdot\sigma)|(W_{\gamma/2}f)(\xi^{+})|^{2}|1-W_{\gamma/2}(w)W_{-\gamma/2}(\xi^{+})|^{2} d\xi d\sigma
	\\&:=& \mathcal{Z}^{\epsilon,0}(W_{\gamma/2}f) + \mathcal{A}.
	\eeno
	By noticing that
	$
	|W_{\gamma/2}(w)W_{-\gamma/2}(\xi^{+}) - 1| \lesssim \theta^{2},
	$
	we have
	$
	|\mathcal{A}| \lesssim  |W_{\gamma/2}f|^{2}_{L^{2}},
	$
	where the change of variable $\xi \rightarrow \xi^{+}$ has been used.
	The desired result follows from the estimate \eqref{Z-epsilon-0-result} in {\it Step 1}.
\end{proof}
\begin{rmk} \label{hat-f-also-valid}
 With the same notations as in Lemma \ref{gammanonzerotozero}, we also have
	\beno
	\mathcal{Z}^{\epsilon,0}(\hat{f}) \lesssim |W^{\epsilon}(D)f|^{2}_{L^{2}} + |W^{\epsilon}f|^{2}_{L^{2}}.
	\eeno
Indeed, by \eqref{f-to-hat-f-in-future} and Plancherel's theorem, we have
\beno
	\mathcal{Z}^{\epsilon,0}(\hat{f}) &=&  8 \pi (1-s)\epsilon^{2s-2} \int_{\R^{3}}\int_{0}^{\pi} (\sin \alpha)^{-1-2s} \mathrm{1}_{\sin \alpha \leq \epsilon}|\hat{f}(u) - \hat{f}(u \cos\alpha)|^{2} \cos\alpha d u d\alpha
\\&=&  8 \pi (1-s)\epsilon^{2s-2} \int_{\R^{3}}\int_{0}^{\pi} (\sin \alpha)^{-1-2s} \mathrm{1}_{\sin \alpha \leq \epsilon}
|f(u) - f(u / \cos\alpha)|^{2}\cos\alpha d u d\alpha
	\\&\lesssim& |W^{\epsilon}(D) f|^{2}_{L^{2}} + |W^{\epsilon} f|^{2}_{L^{2}},
\eeno
where we have used the change of variable $u \rightarrow u \cos \alpha$ and the estimate \eqref{Z-epsilon-0-result}
in the last inequality.
\end{rmk}

Now we are in a position to get $|W^{\epsilon}((-\Delta_{\mathbb{S}^{2}})^{\frac{1}{2}})f|^{2}_{L^{2}_{\gamma/2}}$ from $\mathcal{N}^{\epsilon,0}(g,f)$.

\begin{prop}\label{lowerboundpart2-general-g}
The following two estimates hold.
\ben\label{anisotropic-regularity-general-g}
\mathcal{N}^{\epsilon,0}(g,f) + |g|^{2}_{L^{2}_{1}}|W^{\epsilon}(D)f|^{2}_{L^{2}}+ |g|^{2}_{L^{2}} |W^{\epsilon}f|^{2}_{L^{2}} \gtrsim  |g|^{2}_{L^{2}} |W^{\epsilon}((-\Delta_{\mathbb{S}^{2}})^{\frac{1}{2}})f|^{2}_{L^{2}}.
\\ \label{anisotropic-regularity-general-g-up-bound}
\mathcal{N}^{\epsilon,0}(g,f) \lesssim  |g|^{2}_{L^{2}} |W^{\epsilon}((-\Delta_{\mathbb{S}^{2}})^{\frac{1}{2}})f|^{2}_{L^{2}} + |g|^{2}_{L^{2}_{1}}|W^{\epsilon}(D)f|^{2}_{L^{2}}+ |g|^{2}_{L^{2}} |W^{\epsilon}f|^{2}_{L^{2}}.
\een
\end{prop}
\begin{proof}
By Bobylev's formula, we have
\beno
\mathcal{N}^{\epsilon,0}(g,f) &=& \frac{1}{(2\pi)^{3}}\int b^{\epsilon}(\frac{\xi}{|\xi|} \cdot \sigma)(\widehat{g^{2}}(0)|\hat{f}(\xi) - \hat{f}(\xi^{+})|^{2} + 2\Re((\widehat{g^{2}}(0) - \widehat{g^{2}}(\xi^{-}))\hat{f}(\xi^{+})\overline{\hat{f}}(\xi)) d\xi d\sigma
\\ &:=& \frac{|g|^{2}_{L^{2}}}{(2\pi)^{3}}\mathcal{I}_{1} + \frac{2}{(2\pi)^{3}}\mathcal{I}_{2},
\eeno
where $\xi^{+} = \frac{\xi+|\xi|\sigma}{2}$ and $\xi^{-} = \frac{\xi-|\xi|\sigma}{2}$.
Thanks to the fact $\widehat{g^{2}}(0) - \widehat{g^{2}}(\xi^{-}) = \int(1-\cos(v \cdot \xi^{-}))g^{2}(v) dv$, we have
\beno
|\mathcal{I}_{2}| &=& |\int b^{\epsilon}(\frac{\xi}{|\xi|} \cdot \sigma) (1-\cos(v \cdot \xi^{-}))g^{2}(v) \Re(\hat{f}(\xi^{+})\overline{\hat{f}}(\xi)) d\sigma d\xi dv |
\\ &\leq& (\int b^{\epsilon}(\frac{\xi}{|\xi|} \cdot \sigma) (1-\cos(v \cdot \xi^{-}))g^{2}(v) |\hat{f}(\xi^{+})|^{2} d\sigma d\xi dv)^{\frac{1}{2}}
\\&& \times (\int b^{\epsilon}(\frac{\xi}{|\xi|} \cdot \sigma) (1-\cos(v \cdot \xi^{-}))g^{2}(v) |\overline{\hat{f}}(\xi)|^{2} d\sigma d\xi dv)^{\frac{1}{2}}.
\eeno
Observe that
$
1-\cos(v \cdot \xi^{-}) \lesssim |v|^{2}|\xi^{-}|^{2} = \frac{1}{4}|v|^{2}|\xi|^{2}|\frac{\xi}{|\xi|} - \sigma|^{2} \sim |v|^{2}|\xi^{+}|^{2}|\frac{\xi^{+}}{|\xi^{+}|} - \sigma|^{2},
$
thus
$
1-\cos(v \cdot \xi^{-}) \lesssim \min\{|v|^{2}|\xi|^{2}|\frac{\xi}{|\xi|} - \sigma|^{2},1\} \sim  \min\{|v|^{2}|\xi^{+}|^{2}|\frac{\xi^{+}}{|\xi^{+}|} - \sigma|^{2},1\}.
$
Note that
$
\frac{\xi}{|\xi|} \cdot \sigma  = 2(\frac{\xi^{+}}{|\xi^{+}|} \cdot \sigma)^{2} - 1,
$
by the change of variable from $\xi$ to $\xi^{+}$, and the property $W^{\epsilon}(|v\|\xi|) \lesssim W^{\epsilon}(|v|)W^{\epsilon}(|\xi|)$, we have
\ben \label{upper-I-2-another}
|\mathcal{I}_{2}| \lesssim \int (W^{\epsilon})^{2}(|v\|\xi|)|\hat{f}(\xi)|^{2}g^{2}(v)dvd\xi
\lesssim |W^{\epsilon}g|^{2}_{L^{2}} |W^{\epsilon}(D)f|^{2}_{L^{2}} \lesssim |g|^{2}_{L^{2}_{1}} |W^{\epsilon}(D)f|^{2}_{L^{2}}.
\een
Now we study the lower bound of $\mathcal{I}_{1}$. By the geometric decomposition
\ben \label{geo-deco-frequency-space-another}
\hat{f}(\xi) - \hat{f}(\xi^{+}) = \hat{f}(\xi) - \hat{f}(|\xi|\frac{\xi^{+}}{|\xi^{+}|})+ \hat{f}(|\xi|\frac{\xi^{+}}{|\xi^{+}|}) - \hat{f}(\xi^{+}),
\een
we have
\beno
\mathcal{I}_{1} &=& \int b^{\epsilon}(\frac{\xi}{|\xi|} \cdot \sigma)|\hat{f}(\xi) - \hat{f}(\xi^{+})|^{2} d\xi d\sigma
\\&\geq& \frac{1}{2} \int b^{\epsilon}(\frac{\xi}{|\xi|} \cdot \sigma)|\hat{f}(\xi) - \hat{f}(|\xi|\frac{\xi^{+}}{|\xi^{+}|})|^{2} d\xi d\sigma
- \int b^{\epsilon}(\frac{\xi}{|\xi|} \cdot \sigma)|\hat{f}(|\xi|\frac{\xi^{+}}{|\xi^{+}|}) - \hat{f}(\xi^{+})|^{2} d\xi d\sigma
\\&:=& \frac{1}{2}\mathcal{I}_{1,1} - \mathcal{I}_{1,2}.
\eeno
By Lemma \ref{spherical-part}, we have
\ben \label{upper-lower-I-11-another}
\mathcal{I}_{1,1}+|f|_{L^2}^2 \sim |W^{\epsilon}((-\Delta_{\mathbb{S}^{2}})^{\frac{1}{2}})f|^{2}_{L^{2}}+|f|^{2}_{L^{2}}.
\een
By  Remark \ref{hat-f-also-valid}, there holds
\ben \label{upper-I-12-another} \mathcal{I}_{1,2} \lesssim |W^{\epsilon}(D)f|^{2}_{L^{2}} + |W^{\epsilon}f|^{2}_{L^{2}}.\een
Combining upper bounds estimates \eqref{upper-I-2-another}, \eqref{upper-I-12-another} and  lower bound estimate \eqref{upper-lower-I-11-another}, we get \eqref{anisotropic-regularity-general-g}.
On the other hand, by $\mathcal{N}^{\epsilon,0}(g,f) \lesssim \mathcal{I}_{1,1} + \mathcal{I}_{1,2} + |\mathcal{I}_{2}|$, one can get \eqref{anisotropic-regularity-general-g-up-bound} by the upper bounds estimates \eqref{upper-I-2-another}, \eqref{upper-lower-I-11-another}, \eqref{upper-I-12-another}.
\end{proof}

\subsubsection{Gain of  anisotropic regularity from $\mathcal{N}^{\epsilon,\gamma}(\mu^{\frac{1}{2}},f)$} \label{gain-regularity}
The strategy is to reduce $\mathcal{N}^{\epsilon,\gamma}$ to $\mathcal{N}^{\epsilon,0}$ so that the estimates in previous parts can be used.

For technical reasons, we  define
\ben \label{N-tilde-epsilon-gamma}
\tilde{\mathcal{N}}^{\epsilon,\gamma}(g,h) := \int_{\R^6\times\mathbb{S}^2} b^{\epsilon}(\cos\theta) \langle v-v_{*} \rangle^{\gamma} g^{2}_{*} (h^{\prime}-h)^{2} d\sigma dv dv_{*}.\een
Moreover, we need to consider the ``velocity regular" version $\tilde{\mathcal{N}}^{\epsilon,\gamma}$ of $\mathcal{N}^{\epsilon,\gamma}$ to
reduce $\tilde{\mathcal{N}}^{\epsilon,\gamma}$ to $\mathcal{N}^{\epsilon,0}$ in the following lemma.
\begin{lem}\label{reduce-gamma-to-0-no-sigularity} Let  $\gamma \in \mathbb{R}$, then
\begin{eqnarray}\label{gamma-to-0-no-sigularity}
 &&\frac{1}{2}C_{1}\mathcal{N}^{\epsilon,0}(W_{-|\gamma|/2}g,W_{\gamma/2}f) - C_{3}|g|^{2}_{L^{2}_{|\gamma/2+1|}}|f|^{2}_{L^{2}_{\gamma/2}}
 \\&\leq& \tilde{\mathcal{N}}^{\epsilon,\gamma}(g,f)  \leq  2C_{2}\mathcal{N}^{\epsilon,0}(W_{|\gamma|/2}g,W_{\gamma/2}f) + 2C_{3}|g|^{2}_{L^{2}_{|\gamma/2+1|}}|f|^{2}_{L^{2}_{\gamma/2}}, \nonumber
\end{eqnarray}
where $C_{1},C_{2},C_{3}$ are constants depending only on $\gamma$ that can be chosen as
some generic constants if $-3 \leq \gamma \leq 0$.
\end{lem}
\begin{proof} Set $F= W_{\gamma/2}f$. By definition, we have
\beno  \tilde{\mathcal{N}}^{\epsilon,\gamma}(g,f) = \int_{\R^6\times\mathbb{S}^2} b^{\epsilon}(\cos\theta) \langle v-v_{*} \rangle^{\gamma}  g^{2}_{*} ((W_{-\gamma/2}F)^{\prime}-W_{-\gamma/2}F)^{2} d\sigma dv dv_{*}.
\eeno
Make the decomposition:
\beno
(W_{-\gamma/2}F)^{\prime}-W_{-\gamma/2}F = (W_{-\gamma/2})^{\prime} (F^{\prime}-F) + F(W_{-\gamma/2}^{\prime}-W_{-\gamma/2}) :=  A + B.
\eeno
From $\frac{1}{2}A^{2} - B^{2} \leq (A+B)^{2} \leq 2A^{2}+2B^{2}$, we get
\ben  \label{lower-upper-bound-basic}
\frac{1}{2}\mathcal{I}_{1}- \mathcal{I}_{2} \leq \tilde{\mathcal{N}}^{\epsilon,\gamma}(g,f) \leq 2(\mathcal{I}_{1} + \mathcal{I}_{2}),
\een
where
\beno
\mathcal{I}_{1} &:=& \int_{\R^6\times\mathbb{S}^2} b^{\epsilon}(\cos\theta) \langle v-v_{*} \rangle^{\gamma}  g^{2}_{*} W_{-\gamma}^{\prime} (F^{\prime}-F)^{2} d\sigma dv dv_{*},\\
\mathcal{I}_{2} &:=& \int_{\R^6\times\mathbb{S}^2} b^{\epsilon}(\cos\theta) \langle v-v_{*} \rangle^{\gamma}  g^{2}_{*}F^{2} (W_{-\gamma/2}^{\prime}-W_{-\gamma/2})^{2} d\sigma dv dv_{*}.
\eeno
Thanks to $|v_{*}-v|\sim|v_{*}-v^{\prime}|, \langle v_{*} \rangle^{-|\gamma|} \lesssim \langle v_{*}-v^{\prime} \rangle^{\gamma} \langle v^{\prime} \rangle^{-\gamma} \lesssim \langle v_{*} \rangle^{|\gamma|}$, we get
\ben \label{mathcal-I-1}
\mathcal{N}^{\epsilon,0}(W_{-|\gamma|/2}g,W_{\gamma/2}f) \lesssim \mathcal{I}_{1} \lesssim  \mathcal{N}^{\epsilon,0}(W_{|\gamma|/2}g,W_{\gamma/2}f).\een
By Taylor expansion, one has $(W_{-\gamma/2}^{\prime}-W_{-\gamma/2})^{2} \lesssim \int \langle v(\kappa) \rangle^{-\gamma-2}|v-v_{*}|^{2}\sin^{2}(\theta/2) d\kappa$. Note that
\beno\langle v-v_{*} \rangle^{\gamma} |v-v_{*}|^{2} \langle v(\kappa) \rangle^{-\gamma-2} \lesssim
 \langle v-v_{*} \rangle^{\gamma+2} \langle v(\kappa) \rangle^{-\gamma-2} \lesssim
 \langle v(\kappa)-v_{*} \rangle^{\gamma+2} \langle v(\kappa) \rangle^{-\gamma-2} \lesssim \langle v_{*} \rangle^{|\gamma+2|}. \eeno
Then by \eqref{order-2},
we get
\ben  \label{mathcal-I-2}
\mathcal{I}_{2} \lesssim  \int g^{2}_{*}\langle v_{*} \rangle^{|\gamma+2|} F^{2} dv dv_{*} \lesssim |g|^{2}_{L^{2}_{|\gamma/2+1|}}|F|^{2}_{L^{2}}.\een
Plugging \eqref{mathcal-I-1} and \eqref{mathcal-I-2} into \eqref{lower-upper-bound-basic}, we get \eqref{gamma-to-0-no-sigularity}. In addition, one can track the proof for the dependence of $C_{1},C_{2},C_{3}$ on $\gamma$.
\end{proof}

We are now ready to give lower bound estimate of $\mathcal{N}^{\epsilon,\gamma}(\mu^{\frac{1}{2}},f)$.

\begin{prop}\label{lowerboundpart1-gamma-eta} The following two estimates are valid.
\ben \label{sobolev-regularity-gamma-eta} \mathcal{N}^{\epsilon,\gamma}(\mu^{\frac{1}{2}},f) +|f|_{L^2_{\gamma/2}}^2\gtrsim |W^\epsilon(D)W_{\gamma/2}f|_{L^2}^2.
\\ \label{anisotropic-regularity-gamma-eta} \mathcal{N}^{\epsilon,\gamma}(\mu^{\frac{1}{2}},f) +  |W^\epsilon(D)W_{\gamma/2}f|_{L^2}^2 + |W^\epsilon W_{\gamma/2}f|_{L^2}^2\gtrsim |W^{\epsilon}((-\Delta_{\mathbb{S}^{2}})^{\frac{1}{2}})W_{\gamma/2}f|^{2}_{L^{2}}. \een
By suitable combination, we have
\ben \label{full-regularity-gamma-eta} \mathcal{N}^{\epsilon,\gamma}(\mu^{\frac{1}{2}},f) + |W^\epsilon W_{\gamma/2}f|_{L^2}^2\gtrsim |W^{\epsilon}((-\Delta_{\mathbb{S}^{2}})^{\frac{1}{2}})W_{\gamma/2}f|^{2}_{L^{2}} +  |W^\epsilon(D)W_{\gamma/2}f|_{L^2}^2. \een
\end{prop}
\begin{proof} Since $\gamma \leq  0$, then $|v-v_{*}|^{\gamma} \geq \langle v-v_{*} \rangle^{\gamma}$, and thus $\mathcal{N}^{\epsilon,\gamma}(g,f) \geq \tilde{\mathcal{N}}^{\epsilon,\gamma}(g,f)$. Then as a direct result of Lemma \ref{reduce-gamma-to-0-no-sigularity}
with $\eta=0, g=\mu^{\frac{1}{2}}$, we have
\ben \label{move-weight-in-side}
\mathcal{N}^{\epsilon,\gamma}(\mu^{\frac{1}{2}},f) +|f|^{2}_{L^{2}_{\gamma/2}}  \gtrsim  \mathcal{N}^{\epsilon,0}(W_{\gamma/2}\mu^{\frac{1}{2}},W_{\gamma/2}f) \geq \mathcal{N}^{\epsilon,0}(W_{-\frac{3}{2}}\mu^{\frac{1}{2}},W_{\gamma/2}f).
\een
Note that $|W_{-3} \mu|_{L^{1}}$ and $|W_{-3} \mu|_{L^{1}_{1}} + |W_{-3} \mu|_{L\mathrm{log}L}$ are generic  constants. Then
according to Lemma \ref{lowerboundpart1-general-g}, we have
\beno \mathcal{N}^{\epsilon,0}(W_{-\frac{3}{2}}\mu^{\frac{1}{2}},W_{\gamma/2}f)
+ |f|^{2}_{L^{2}_{\gamma/2}} \gtrsim |W^{\epsilon}(D)W_{\gamma/2}f|^{2}_{L^{2}}.
\eeno
From which together with \eqref{move-weight-in-side}, we get \eqref{sobolev-regularity-gamma-eta}.

Taking $g=W_{-\frac{3}{2}}\mu^{\frac{1}{2}}$ in \eqref{anisotropic-regularity-general-g} of Proposition \ref{lowerboundpart2-general-g}, we have
\beno \mathcal{N}^{\epsilon,0}(W_{-\frac{3}{2}}\mu^{\frac{1}{2}},W_{\gamma/2}f)
+  |W^\epsilon(D)W_{\gamma/2}f|_{L^2}^2 + |W^\epsilon W_{\gamma/2}f|_{L^2}^2\gtrsim |W^{\epsilon}((-\Delta_{\mathbb{S}^{2}})^{\frac{1}{2}})W_{\gamma/2}f|^{2}_{L^{2}}.
\eeno
From which together with \eqref{move-weight-in-side},
we get \eqref{anisotropic-regularity-gamma-eta}. The proof is completed.
\end{proof}

\subsubsection{Coercivity estimate} \label{coercivity}
We are ready to prove coercivity estimate of $\mathcal{L}^{\epsilon}$ in Theorem \ref{coercivity-structure}.
\begin{proof}[Proof of Theorem \ref{coercivity-structure}] By combining Proposition \ref{lowerboundpart1} and \eqref{full-regularity-gamma-eta} in Proposition \ref{lowerboundpart1-gamma-eta},  we get
\ben \label{two-parts-together-N-N}
&& \mathcal{N}^{\epsilon,\gamma}(\mu^{\frac{1}{2}},f) + \mathcal{N}^{\epsilon,\gamma}(f,\mu^{\frac{1}{2}}) + |f|^{2}_{L^{2}_{\gamma/2}}
\\&\gtrsim&  |W^{\epsilon}((-\Delta_{\mathbb{S}^{2}})^{\frac{1}{2}})W_{\gamma/2}f|_{L^2}^2
+ |W^{\epsilon}(D)W_{\gamma/2}f|_{L^2}^2 + |W^{\epsilon}W_{\gamma/2}f|_{L^2}^2=|f|_{\epsilon,\gamma/2}^2 . \nonumber
\een
Observe that $\mathcal{N}^{\epsilon,\gamma}(\mu^{\frac{1}{2}},f)
+ \mathcal{N}^{\epsilon,\gamma}(f,\mu^{\frac{1}{2}})$ corresponds to the anisotropic norm $|||f|||^{2}$ introduced in \cite{alexandre2012boltzmann}. Recalling \eqref{DefLep}, we have
\ben \label{definition-L1-L2}
\mathcal{L}^{\epsilon} = \mathcal{L}^{\epsilon}_1+\mathcal{L}^{\epsilon}_2, \quad
\mathcal{L}^{\epsilon}_1 g:= -\Gamma^{\epsilon}(\mu^{\frac{1}{2}},g),  \quad \mathcal{L}^{\epsilon}_2 g :=- \Gamma^{\epsilon}(g, \mu^{\frac{1}{2}}).
\een
By Proposition 2.16 in \cite{alexandre2012boltzmann} and Corollary \ref{cancellation-to-Sob-norm}, there holds
\ben \label{L-1-dominate}
\langle \mathcal{L}^{\epsilon}_{1} f,f \rangle \geq \frac{1}{10}\left(\mathcal{N}^{\epsilon,\gamma}(\mu^{\frac{1}{2}},f)
+ \mathcal{N}^{\epsilon,\gamma}(f,\mu^{\frac{1}{2}}) \right) - C |f|_{L^2_{\gamma/2}}^2.
\een
By Lemma \ref{l2-full-estimate-geq-eta} in the next section, we have
\ben \label{L-2-lower}
|\langle \mathcal{L}^{\epsilon}_{2} f,f \rangle| \lesssim |\mu^{\f18}f|_{L^2}^2 \lesssim |f|_{L^2_{\gamma/2}}^2.
\een
Combining \eqref{L-1-dominate}, \eqref{L-2-lower} and \eqref{two-parts-together-N-N} completes  the proof of the theorem.
\end{proof}

\subsection{Dissipation in microscopic space} In this subsection, we consider the dissipative property $\mathcal{L}^{\epsilon}$ in the microscopic space. This is also referred as the ``spectral gap" estimate.
Recall the kernel space $\mathrm{ker}$ defined in \eqref{kernel-space}.
An orthonormal basis of $\mathrm{ker}$ can be chosen as $\{\mu^{\frac{1}{2}}, \mu^{\frac{1}{2}}v_1, \mu^{\frac{1}{2}}v_2,\mu^{\frac{1}{2}}v_3, \mu^{\frac{1}{2}}(|v|^2-3) /\sqrt{6} \}
:= \{e_{j}\}_{1 \leq j \leq 5}$. The projection operator $\mathbb{P}$ on the kernel space
is defined as follows:
\ben\label{DefProj} \mathbb{P}f:=\sum_{j=1}^{5}\langle f, e_{j}\rangle e_{j} =(a+b\cdot v+c|v|^2)\mu^{\frac{1}{2}}, \een
where for $1\le i\le 3$,
\beno
 a=\int_{\R^3} (\frac{5}{2}-\frac{|v|^{2}}{2})\mu^{\frac{1}{2}}fdv, \quad  b_i=\int_{\R^3} v_i\mu^{\frac{1}{2}}fdv, \quad c=\int_{\R^3} (\frac{|v|^2}{6}-\frac{1}{2})\mu^{\frac{1}{2}}fdv.
\eeno

The dissipative property of $\mathcal{L}^{\epsilon}$ in the $\mathrm{ker}^{\perp}$ space is
given in the following theorem.
\begin{thm}\label{micro-dissipation} Let $-3 < \gamma \leq 0$. There are two generic constants $\epsilon_{0}, \lambda_{0}>0$
such that for any $0 < \epsilon \leq \epsilon_{0}$, it holds
\beno \langle  \mathcal{L}^{\epsilon}g, g\rangle \geq \lambda_{0}|g-\mathbb{P}g|^{2}_{\epsilon,\gamma/2}.
\eeno
\end{thm}
\begin{rmk}\label{remarkgap}  We address that from the proof in the below,  $\lambda_0$ depends only on $\gamma$ and the lower bound of $\int b^\epsilon(\cos\theta)\sin^2(\theta/2)d\sigma$.
\end{rmk}

To prove Theorem \ref{micro-dissipation}, we first introduce a special weight function $U_{\delta}$ defined by
\ben\label{specialweightfun} U_{\delta}(v) := (1+ \delta^{2}|v|^{2})^{\frac{1}{2}} \geq \max\{\delta|v|,1\}. \een
Here $\delta$ is a sufficiently small parameter.
We then introduce a new smooth function $\chi$. Recalling the smooth function $\zeta$ in \eqref{zeta-property}.
Let $\chi(\cdot) = \zeta(\cdot/2)$. Then $\chi$ is a smooth function verifying $0 \leq \chi \leq 1, \chi=1$ on $[0,1]$ and $\chi=0$ on  $[2,\infty)$.
Let $\chi_{R}(v):=\chi(|v|/R)$. The following lemma is the Lemma 3.2 in \cite{he2020boltzmann}.

\begin{lem}\label{difference-term-complication} Set $X(\gamma,R,\delta):=\delta^{-\gamma}\left((\chi_{R})^{\prime}(\chi_{R})^{\prime}_{*}(U^{\gamma/2}_{\delta})^{\prime}(U^{\gamma/2}_{\delta})^{\prime}_{*}-
\chi_{R}(\chi_{R})_{*}U^{\gamma/2}_{\delta}(U^{\gamma/2}_{\delta})_{*}\right)^{2}$ with $\gamma \leq 0 < \delta \leq 1 \leq R$, then
\beno
X(\gamma,R,\delta)
\lesssim(\delta^{2}+R^{-2})\theta^{2}\langle v \rangle^{\gamma+2}\langle v_{*} \rangle^{2} \mathrm{1}_{|v|\leq 4R}.
\eeno
\end{lem}

\begin{proof}[Proof of Theorem \ref{micro-dissipation}]
Suppose  $\mathbb{P}g=0$ and then it suffices to   prove
$ \langle  \mathcal{L}^{\epsilon}g, g\rangle \gtrsim  |g|^{2}_{\epsilon,\gamma/2}.
$
In the following, we specify the parameter $\gamma$  in the operator $\mathcal{L}^{\epsilon}$ and denote it by $\mathcal{L}^{\epsilon,\gamma}$.
For brevity, set
\beno J^{\epsilon,\gamma}(g) := 4 \langle  \mathcal{L}^{\epsilon,\gamma}g, g\rangle,\quad\mathbb{A}(f, g):=(f_{*}g + f g_{*} - f^{\prime}_{*}g^{\prime} - f^{\prime} g^{\prime}_{*}), \quad\mathbb{F}(f, g):= \mathbb{A}^{2}(f, g).\eeno
With these notations, we have
 $J^{\epsilon,\gamma}(g) = \int B^{\epsilon} \mathbb{F}(\mu^{\frac{1}{2}}, g) d\sigma dv dv_{*}.$ The proof is divided into four steps.

\noindent{\it Step 1: Localization of $J^{\epsilon,\gamma}(g)$.}
By \eqref{specialweightfun} and the condition $\gamma\le0$,  we get \beno |v - v_{*}|^{-\gamma}
\leq  C_{\gamma} \delta^{\gamma} ((\delta|v|)^{-\gamma}+(\delta|v_{*}|)^{-\gamma})
 \leq 2C_{\gamma} \delta^{\gamma} U^{-\gamma}_{\delta}(v)U^{-\gamma}_{\delta}(v_{*}),  \eeno
which gives
$|v - v_{*}|^{\gamma} \gtrsim \delta^{-\gamma} U^{\gamma}_{\delta}(v)U^{\gamma}_{\delta}(v_{*})$
so that
\beno J^{\epsilon,\gamma}(g) \gtrsim \delta^{-\gamma}\int b^{\epsilon}  \chi^{2}_{R}(\chi^{2}_{R})_{*}U^{\gamma}_{\delta}(U^{\gamma}_{\delta})_{*}\mathbb{F}(\mu^{\frac{1}{2}}, g) d\sigma dv dv_{*}.
\eeno
We include the function $\chi^{2}_{R}(\chi^{2}_{R})_{*}U^{\gamma}_{\delta}(U^{\gamma}_{\delta})_{*}$ inside $\mathbb{F}(\mu^{\frac{1}{2}}, g)$, which leads to  $\mathbb{F}(\chi_{R}U^{\gamma/2}_{\delta}\mu^{\frac{1}{2}}, \chi_{R}U^{\gamma/2}_{\delta}g)$ with some correction terms.
For   simplicity, set $h=\chi_{R} U^{\gamma/2}_{\delta}, f=\mu^{\frac{1}{2}}$,
then
\beno
&&\chi^{2}_{R}(\chi^{2}_{R})_{*}U^{\gamma}_{\delta}(U^{\gamma}_{\delta})_{*}\mathbb{F}(\mu^{\frac{1}{2}}, g)
\\ &=&
h^{2}_{*}h^{2}\mathbb{F}(f, g)
=
\left(h h_{*}\left(f_{*}g + f g_{*}\right) -
h h_{*}\left(f^{\prime}_{*}g^{\prime} + f^{\prime} g^{\prime}_{*}\right)\right)^{2}
 \\&=& \left(h h_{*}\left(f_{*}g + f g_{*}\right) -
h^{\prime} h^{\prime}_{*}\left(f^{\prime}_{*}g^{\prime} + f^{\prime} g^{\prime}_{*}\right)
+ \left(h^{\prime} h^{\prime}_{*}-
h h_{*}\right)
\left(f^{\prime}_{*}g^{\prime} + f^{\prime} g^{\prime}_{*}\right)
\right)^{2}
\nonumber \\  &\geq&  \frac{1}{2}  \left(h h_{*}\left(f_{*}g + f g_{*}\right) -
h^{\prime} h^{\prime}_{*}\left(f^{\prime}_{*}g^{\prime} + f^{\prime} g^{\prime}_{*}\right) \right)^{2} -\left(h^{\prime} h^{\prime}_{*}-
h h_{*}\right)^{2}
\left(f^{\prime}_{*}g^{\prime} + f^{\prime} g^{\prime}_{*}\right)^{2}
\\  &=&  \frac{1}{2}  \mathbb{F}(h f, h g)
-\left(h^{\prime} h^{\prime}_{*}-
h h_{*}\right)^{2}
\left(f^{\prime}_{*}g^{\prime} + f^{\prime} g^{\prime}_{*}\right)^{2}.
\eeno
Hence,  we get
\ben \label{move-inside-sym} J^{\epsilon,\gamma}(g) &\gtrsim&
\frac{1}{2}\delta^{-\gamma}
\int b^{\epsilon}  \mathbb{F}(\chi_{R}U^{\gamma/2}_{\delta}\mu^{\frac{1}{2}}, \chi_{R}U^{\gamma/2}_{\delta}g) d\sigma dv dv_{*}
\\  && - \delta^{-\gamma}\int b^{\epsilon} \left(h^{\prime} h^{\prime}_{*}-
h h_{*}\right)^{2}
\left(f^{\prime}_{*}g^{\prime} + f^{\prime} g^{\prime}_{*}\right)^{2} d\sigma dv dv_{*}. \nonumber
\een
We now move $\chi_{R}U^{\gamma/2}_{\delta}$ before $\mu^{\frac{1}{2}}$ out of $\mathbb{F}(\chi_{R}U^{\gamma/2}_{\delta}\mu^{\frac{1}{2}}, \chi_{R}U^{\gamma/2}_{\delta}g)$,
which leads to $\mathbb{F}(\mu^{\frac{1}{2}}, \chi_{R}U^{\gamma/2}_{\delta}g)$ with some correction terms. That is,
\ben \label{move-outside-asym} \mathbb{F}(\chi_{R}U^{\gamma/2}_{\delta}\mu^{\frac{1}{2}}, \chi_{R}U^{\gamma/2}_{\delta}g)  &=& \mathbb{A}^{2}(\chi_{R}U^{\gamma/2}_{\delta}\mu^{\frac{1}{2}}, \chi_{R}U^{\gamma/2}_{\delta}g)
\nonumber \\&=&\left( \mathbb{A}(\mu^{\frac{1}{2}}, \chi_{R}U^{\gamma/2}_{\delta}g)  - \mathbb{A}\big((1-\chi_{R}U^{\gamma/2}_{\delta})\mu^{\frac{1}{2}}, \chi_{R}U^{\gamma/2}_{\delta}g\big) \right)^{2}
\nonumber \\&\geq&\frac{1}{2} \mathbb{A}^{2}(\mu^{\frac{1}{2}}, \chi_{R}U^{\gamma/2}_{\delta}g)
-\mathbb{A}^{2}\big((1-\chi_{R}U^{\gamma/2}_{\delta})\mu^{\frac{1}{2}}, \chi_{R}U^{\gamma/2}_{\delta}g\big)
\nonumber \\&=&\frac{1}{2}\mathbb{F}(\mu^{\frac{1}{2}}, \chi_{R}U^{\gamma/2}_{\delta}g) -  \mathbb{F}((1-\chi_{R}U^{\gamma/2}_{\delta})\mu^{\frac{1}{2}}, \chi_{R}U^{\gamma/2}_{\delta}g).
\een
By symmetry, we have
\ben \label{sym-term}\int b^{\epsilon} \left(h^{\prime} h^{\prime}_{*}-
h h_{*}\right)^{2}
\left(f^{\prime}_{*}g^{\prime} + f^{\prime} g^{\prime}_{*}\right)^{2} d\sigma dv dv_{*} \leq 4\int b^{\epsilon} \left(h^{\prime} h^{\prime}_{*}-
h h_{*}\right)^{2}f^{2}_{*}g^{2} d\sigma dv dv_{*}.
 \een
Thanks  to \eqref{move-inside-sym}, \eqref{move-outside-asym} and \eqref{sym-term}, we get
\ben \label{separate-key-part} J^{\epsilon,\gamma}(g) &\gtrsim&
\frac{1}{4}\delta^{-\gamma}
\int b^{\epsilon} \mathbb{F}(\mu^{\frac{1}{2}}, \chi_{R}U^{\gamma/2}_{\delta}g) d\sigma dv dv_{*}
\\  && - \frac{1}{2}\delta^{-\gamma}\int b^{\epsilon} \mathbb{F}((1-\chi_{R}U^{\gamma/2}_{\delta})\mu^{\frac{1}{2}}, \chi_{R}U^{\gamma/2}_{\delta}g) d\sigma dv dv_{*}
\nonumber
\\  && - 4\delta^{-\gamma}\int b^{\epsilon} \left(h^{\prime} h^{\prime}_{*}-
h h_{*}\right)^{2} f^{2}_{*}g^{2} d\sigma dv dv_{*}
:= \frac{1}{4}J_{1} - \frac{1}{2}J_{2} -  4J_{3}. \nonumber
\een

\noindent{\it Step 2: Estimates of $J_i(i=1,2,3)$.} We will give the estimates term by term.

{\it \underline{Lower bound of $J_1$.}}
 We claim that for $\epsilon \leq 16^{-1}R^{-1}$ and some generic constant $C$,
\ben \label{estimate-J-1}  J_{1} \gtrsim \delta^{-\gamma}|g|^{2}_{L^{2}_{\gamma/2}}-C(\delta^{2}+R^{-2})|g|^{2}_{\epsilon,\gamma/2}.  \een
By Wang-Uhlenbeck \cite{wang1952propagation}, for any function $F$, it holds
\beno \langle  \mathcal{L}^{\epsilon,0}F, F\rangle \gtrsim  (\int b^{\epsilon}(\cos\theta) \sin^{2}(\theta/2) d \sigma)|(\mathbb{I}-\mathbb{P})F|^{2}_{L^{2}},\eeno
where $\mathbb{I}$ is the identity operator.
Thanks to \eqref{order-2}, there is a generic constant $c_{0}$ such that
\ben \label{eta-pure-lower-bound} \langle  \mathcal{L}^{\epsilon,0}F, F\rangle \geq c_{0}|(\mathbb{I}-\mathbb{P})F|^{2}_{L^{2}}. \een
Applying \eqref{eta-pure-lower-bound} with $F=\chi_{R}U^{\gamma/2}_{\delta}g$,
and using $(a-b)^{2} \geq a^{2}/2-b^{2}$, we have
\beno J_{1}&=&\delta^{-\gamma}
\int b^{\epsilon} \mathbb{F}(\mu^{\frac{1}{2}}, \chi_{R}U^{\gamma/2}_{\delta}g) d\sigma dv dv_{*}
= 4 \delta^{-\gamma}\langle  \mathcal{L}^{\epsilon,0}\chi_{R}U^{\gamma/2}_{\delta}g, \chi_{R}U^{\gamma/2}_{\delta}g\rangle
\\&\gtrsim& \delta^{-\gamma}|(\mathbb{I}-\mathbb{P})(\chi_{R}U^{\gamma/2}_{\delta}g)|^{2}_{L^{2}}
\\&\gtrsim& \frac{1}{2}\delta^{-\gamma}|\chi_{R}U^{\gamma/2}_{\delta}g|^{2}_{L^{2}} - \delta^{-\gamma}|\mathbb{P}(\chi_{R}U^{\gamma/2}_{\delta}g)|^{2}_{L^{2}}
\\&\gtrsim& \frac{1}{4}\delta^{-\gamma}|U^{\gamma/2}_{\delta}g|^{2}_{L^{2}} - \frac{1}{2}
\delta^{-\gamma}|(1-\chi_{R})U^{\gamma/2}_{\delta}g|^{2}_{L^{2}}
-\delta^{-\gamma}|\mathbb{P}(\chi_{R}U^{\gamma/2}_{\delta}g)|^{2}_{L^{2}}
 := J_{1,1}- J_{1,2} - J_{1,3}.
\eeno
\begin{itemize}
\item Since $\delta \leq 1$ and $\gamma \leq 0$, then $U^{\gamma/2}_{\delta} \geq W_{\gamma/2}$, which yields the leading term
\ben \label{leading-term} J_{1,1} \gtrsim \delta^{-\gamma}|g|^{2}_{L^{2}_{\gamma/2}}.\een
\item
Thanks to the fact $\delta^{-\gamma}U^{\gamma}_{\delta} \leq W_{\gamma}$ and $1-\chi_{R}(v)=0$ when $|v|\leq R$,  we have
\ben \label{large-velocity-part-1} J_{1,2} &=& \frac{1}{2}
\delta^{-\gamma}|(1-\chi_{R})U^{\gamma/2}_{\delta}g|^{2}_{L^{2}}
\lesssim |(1-\chi_{R})W_{\gamma/2}g|^{2}_{L^{2}}
 \\&\lesssim&  |\mathrm{1}_{|v| \geq R} \zeta(\epsilon \cdot) W_{\gamma/2}g|^{2}_{L^{2}}  +
 |(1- \zeta(\epsilon \cdot))W_{\gamma/2}g|^{2}_{L^{2}} \nonumber
\\&\lesssim&  R^{-2}| \zeta(\epsilon \cdot) W_{\gamma/2+1}g|^{2}_{L^{2}}  +
 \epsilon^{2}|(1- \zeta(\epsilon \cdot)) \epsilon^{-2} W_{\gamma/2}g|^{2}_{L^{2}}
\lesssim (R^{-2}+\epsilon^{2})|W^{\epsilon}W_{\gamma/2}g|^{2}_{L^{2}}, \nonumber\een
where we have used \eqref{low-frequency-lb-cf} and \eqref{high-frequency-lb-cf} in the last inequality.
By the assumption $\epsilon \leq 16^{-1}R^{-1}$, we have
\ben \label{large-velocity-part-12} J_{1,2} \lesssim R^{-2}|W^{\epsilon}W_{\gamma/2}g|^{2}_{L^{2}}.\een
\item
We now estimate $J_{1,3}$. Recalling \eqref{DefProj} for the definition of $\mathbb{P}$
and  by the condition $\mathbb{P}g = 0$, we have
\beno  \mathbb{P}(\chi_{R}U^{\gamma/2}_{\delta}g) &=& \sum_{i=1}^{5} e_{i}\int e_{i}\chi_{R}U^{\gamma/2}_{\delta}g dv
= \sum_{i=1}^{5} e_{i}\int e_{i}(\chi_{R}U^{\gamma/2}_{\delta}-1)g dv .\eeno
Observing
\ben\label{lclfact1} 1-\chi_{R}U^{\gamma/2}_{\delta}\lesssim 1-\chi_R+\delta |v|\chi_{R}, \een
and thus $e_{i}(1-\chi_{R}U^{\gamma/2}_{\delta}) \lesssim (\delta+R^{-1}) \mu^{\frac{1}{4}}$, we have
$ \big|\int e_{i}(\chi_{R}U^{\gamma/2}_{\delta}-1)g dv\big| \lesssim (\delta+R^{-1}) |\mu^{\f18}g|_{L^{2}},$
which gives
\ben \label{large-velocity-part-13} J_{1,3} = \delta^{-\gamma}|\mathbb{P}(\chi_{R}U^{\gamma/2}_{\delta}g)|^{2}_{L^{2}} \lesssim (\delta^{2}+R^{-2}) |\mu^{\f18}g|^{2}_{L^{2}} \lesssim (\delta^{2}+R^{-2}) |g|^{2}_{L^{2}_{\gamma/2}}.  \een
 \end{itemize}
Combining the estimates \eqref{leading-term}, \eqref{large-velocity-part-12} and \eqref{large-velocity-part-13} gives \eqref{estimate-J-1}.

{\it \underline{Upper bound of $J_2$.}}
For simplicity, by setting $f_{\gamma}= (1-\chi_{R}U^{\gamma/2}_{\delta})\mu^{\frac{1}{2}}, g_{\gamma}=\chi_{R}U^{\gamma/2}_{\delta}g$, we get
\ben \label{J2-part}
J_{2} &=&
\delta^{-\gamma}\int b^{\epsilon} \mathbb{F}((1-\chi_{R}U^{\gamma/2}_{\delta})\mu^{\frac{1}{2}}, \chi_{R}U^{\gamma/2}_{\delta}g) d\sigma dv dv_{*}
=\delta^{-\gamma}\int b^{\epsilon} \mathbb{F}(f_{\gamma}, g_{\gamma}) d\sigma dv dv_{*}
\nonumber\\&\lesssim& \delta^{-\gamma}\int b^{\epsilon} (f_{\gamma}^{2})_{*}(g_{\gamma}^{\prime}-g_{\gamma})^{2} d\sigma dv dv_{*} + \delta^{-\gamma}\int b^{\epsilon} (g_{\gamma}^{2})_{*}(f_{\gamma}^{\prime}-f_{\gamma})^{2} d\sigma dv dv_{*}
:= J_{2,1} + J_{2,2}.
\een
Thanks to \eqref{lclfact1}, we have
\ben \label{f-gamma-upper}
(f_{\gamma}^{2})_{*} = ((1-\chi_{R}U^{\gamma/2}_{\delta})\mu^{\frac{1}{2}})^{2}_{*} \lesssim (\delta^{2}+R^{-2}) \mu^{\frac{1}{2}}_{*}.
\een
Plugging \eqref{f-gamma-upper} into $J_{2,1},$ we have
\ben \label{J21}
J_{2,1}&\lesssim& (\delta^{2}+R^{-2})\delta^{-\gamma}\int b^{\epsilon}  \mu^{\frac{1}{2}}_{*} (g_{\gamma}^{\prime}-g_{\gamma})^{2} d\sigma dv dv_{*}
\\&=& (\delta^{2}+R^{-2})\delta^{-\gamma} \mathcal{N}^{\epsilon,0}(\mu^{\frac{1}{4}},g_{\gamma})
\lesssim (\delta^{2}+R^{-2})\delta^{-\gamma}|\chi_{R}U^{\gamma/2}_{\delta}g|^{2}_{\epsilon,\gamma/2} \lesssim (\delta^{2}+R^{-2})|g|^{2}_{\epsilon,\gamma/2},\nonumber
\een
where we have used \eqref{anisotropic-regularity-general-g-up-bound}
and  Lemma \ref{operatorcommutator1} with  $\Phi=\delta^{-\gamma/2}\chi_{R}U^{\gamma/2}_{\delta}\in S^{\gamma/2}_{1,0}$ and $M=W^\epsilon\in S^1_{1,0}$.

By Taylor expansion up to order 1,
$ f_{\gamma}^{\prime}-f_{\gamma}  = \int_{0}^{1} (\nabla f_{\gamma})(v(\kappa))\cdot (v^{\prime}-v)d\kappa.$ From which together with
\beno |\nabla f_{\gamma}| &=&| \nabla ((1-\chi_{R}U^{\gamma/2}_{\delta})\mu^{\frac{1}{2}})|
= |(1-\chi_{R}U^{\gamma/2}_{\delta}) \nabla \mu^{\frac{1}{2}}
- U^{\gamma/2}_{\delta}\mu^{\frac{1}{2}} \nabla \chi_{R} - \chi_{R} \mu^{\frac{1}{2}} \nabla U^{\gamma/2}_{\delta} |
\\&\lesssim& \mu^{\f18}(\delta + R^{-1}),
\eeno
we get
\ben \label{f-gamma-by-taylor}
|f_{\gamma}^{\prime}-f_{\gamma}|^{2}  \lesssim
(\delta^{2}+R^{-2})\theta^{2} \int_{0}^{1} \mu^{\frac{1}{4}}(v(\kappa))|(v(\kappa)-v_{*})|^{2}d\kappa.\een
Since $R \leq 16^{-1}\epsilon^{-1}$, by the change $v \rightarrow v(\kappa)$, and the fact \eqref{lower-bound-when-small-cf},
we have
\ben \label{J22}
J_{2,2}&\lesssim& (\delta^{2}+R^{-2})\delta^{-\gamma}\int b^{\epsilon} \theta^{2} (\chi_{R}U^{\gamma/2}_{\delta}g)^{2}_{*} \mu^{\frac{1}{4}}(v(\kappa))|v(\kappa)-v_{*}|^{2} d\sigma dv(\kappa) dv_{*} d\kappa
\nonumber\\&\lesssim& (\delta^{2}+R^{-2})|\chi_{R}W_{\gamma/2+1}g|^{2}_{L^{2}}
\lesssim (\delta^{2}+R^{-2})|W_{\gamma/2}W^{\epsilon}g|^{2}_{L^{2}}
\lesssim (\delta^{2}+R^{-2})|g|^{2}_{\epsilon,\gamma/2}.
\een
Plugging the estimates \eqref{J21} and \eqref{J22} into \eqref{J2-part}, we get
\ben \label{estimate-J-2}
J_{2} \lesssim (\delta^{2}+R^{-2})|g|^{2}_{\epsilon,\gamma/2}.
\een

{\it \underline{Upper bound of $J_3$.}}
By Lemma \ref{difference-term-complication}, we have
$
\delta^{-\gamma}\left(h^{\prime} h^{\prime}_{*}-
h h_{*}\right)^{2}
\lesssim(\delta^{2}+R^{-2})\theta^{2}\langle v \rangle^{\gamma+2}\langle v_{*} \rangle^{2}\mathrm{1}_{|v|\leq 4R}. $
Since $8R \leq \frac{1}{2}\epsilon^{-1}$, by \eqref{lower-bound-when-small-cf},  we have,
\ben \label{estimate-J-3}
J_{3} &=& \delta^{-\gamma}\int b^{\epsilon} \left(h^{\prime} h^{\prime}_{*}-
h h_{*}\right)^{2} \mu_{*}g^{2} d\sigma dv dv_{*}
\lesssim (\delta^{2}+R^{-2})
\int b^{\epsilon} \theta^{2}\langle v_{*} \rangle^{2}\langle v \rangle^{\gamma+2}\mu_{*}\mathrm{1}_{|v|\leq 4R} g^{2} d\sigma dv dv_{*}
\nonumber\\&\lesssim& (\delta^{2}+R^{-2})|\mathrm{1}_{|\cdot|\leq 4R} W_{\gamma/2+1}g|^{2}_{L^{2}} \lesssim (\delta^{2}+R^{-2})|W_{\gamma/2}W^{\epsilon}g|^{2}_{L^{2}}
\leq (\delta^{2}+R^{-2})|g|^{2}_{\epsilon,\gamma/2}.
\een

{\it Step 3: Case $- \frac{7}{4} \leq  \gamma \leq 0$.}
Plugging the estimates of $J_{1}$ in \eqref{estimate-J-1},  $J_{2}$ in \eqref{estimate-J-2}, $J_{3}$ in \eqref{estimate-J-3} into \eqref{separate-key-part},  for $\epsilon \leq 16^{-1}R^{-1}, 0<\delta<1$, we get
\beno   J^{\epsilon,\gamma}(g) \gtrsim \delta^{-\gamma}|g|^{2}_{L^{2}_{\gamma/2}}-C(\delta^{2}+R^{-2})|g|^{2}_{\epsilon,\gamma/2}. \eeno
Choosing $R = \delta^{-1}$, for some universal constants $C_{1}, C_{2}$, we have
\ben \label{key-estimate-by-He}   J^{\epsilon,\gamma}(g) \geq C_{1}\delta^{-\gamma}|g|^{2}_{L^{2}_{\gamma/2}}-C_{2}\delta^{2}|g|^{2}_{\epsilon,\gamma/2}. \een
By the coercivity estimate in Theorem \ref{coercivity-structure}, for some universal constants $C_{3}, C_{4}$, we have
\ben \label{known-estimate}    J^{\epsilon,\gamma}(g) \geq C_{3}|g|^{2}_{\epsilon,\gamma/2}- C_{4}|g|^{2}_{L^{2}_{\gamma/2}}.\een
Multiplying \eqref{known-estimate} by $C_{5}\delta^{2}$ and adding the resulting inequality to \eqref{key-estimate-by-He}, we get
\beno   (1+C_{5}\delta^{2}) J^{\epsilon,\gamma}(g) \geq (
C_{1}\delta^{-\gamma}-C_{4}C_{5}\delta^{2})|g|^{2}_{L^{2}_{\gamma/2}}+(C_{3}C_{5}-C_{2})\delta^{2}|g|^{2}_{\epsilon,\gamma/2}.  \eeno
Firstly, we  take $C_{5}$ large enough such that $C_{3}C_{5}-C_{2} \geq C_{2}$, for example let $C_{5} = 2C_{2}/C_{3}$.
Then we take $\delta$ small enough such that $C_{1}\delta^{-\gamma}-C_{4}C_{5}\delta^{2} \geq 0$, for example, let $\delta=\left(\frac{C_{1}}{C_{4}C_{5}}\right)^{1/(2+\gamma)} = \left(\frac{C_{1}C_{3}}{2C_{4}C_{2}}\right)^{1/(2+\gamma)}$. Then we get
\beno   J^{\epsilon,\gamma}(g) \geq C_{2}\delta^{2}|g|^{2}_{\epsilon,\gamma/2} = C_{2}\left(\frac{C_{1}C_{3}}{2C_{4}C_{2}}\right)^{2/(2+\gamma)}|g|^{2}_{\epsilon,\gamma/2}, \eeno
for any $0 < \epsilon \leq 16^{-1}R^{-1}  = 16^{-1}\left(\frac{C_{1}C_{3}}{2C_{4}C_{2}}\right)^{2/(2+\gamma)}$. Without loss of generality, we may assume $\frac{C_{1}C_{3}}{2C_{4}C_{2}} \leq 1$, which gives for $-\frac{7}{4} \leq  \gamma \leq 0$,
\ben \label{conclusion1-gamma-minus2}   J^{\epsilon,\gamma}(g) \geq C_{2}\left(\frac{C_{1}C_{3}}{2C_{4}C_{2}}\right)^{8}|g|^{2}_{\epsilon,\gamma/2}, \een
Note that $C_{2}\left(\frac{C_{1}C_{3}}{2C_{4}C_{2}}\right)^{8}$ depends only on the parameter $\gamma$ and is a universal constant when $-\frac{7}{4} \leq \gamma \leq 0$.

{\it Step 4: Case $-3< \gamma < -\frac{7}{4}$.}
In this case, we take $-\frac{7}{4} \leq \alpha, \beta<0$ such that $\alpha+\beta=\gamma$. Replacing $b^{\epsilon}$ by  $b^{\epsilon}|v-v_{*}|^{\alpha}$ and $\gamma$ by $\beta$, similar to \eqref{separate-key-part}, we get
\ben \label{separate-gamma-very-small} J^{\epsilon,\gamma}(g) &\gtrsim&
\frac{1}{4}\delta^{-\beta}
\int b^{\epsilon}|v-v_{*}|^{\alpha} \mathbb{F}(\mu^{\frac{1}{2}}, \chi_{R}U^{\beta/2}_{\delta}g) d\sigma dv dv_{*}
\\  && - \frac{1}{2}\delta^{-\beta}\int b^{\epsilon}|v-v_{*}|^{\alpha} \mathbb{F}((1-\chi_{R}U^{\beta/2}_{\delta})\mu^{\frac{1}{2}}, \chi_{R}U^{\beta/2}_{\delta}g) d\sigma dv dv_{*}
\nonumber\\  && -  4\delta^{-\beta}\int b^{\epsilon}|v-v_{*}|^{\alpha} \left(h^{\prime} h^{\prime}_{*}-
h h_{*}\right)^{2} \mu_{*}g^{2} d\sigma dv dv_{*}
:= \frac{1}{4} J^{\alpha,\beta}_{1}- \frac{1}{2} J^{\alpha,\beta}_{2} - 4 J^{\alpha,\beta}_{3},
\nonumber
\een
where $h :=\chi_{R} U^{\beta/2}_{\delta}$.

 {\it \underline{ Lower bound of $J^{\alpha,\beta}_{1}$.}}
Since $-\frac{7}{4} \leq \alpha<0$, we can use previous estimate \eqref{conclusion1-gamma-minus2} to get
\beno J^{\alpha,\beta}_{1} = \delta^{-\beta} J^{\epsilon,\alpha}(\chi_{R}U^{\beta/2}_{\delta}g) = \delta^{-\beta} J^{\epsilon,\alpha}((\mathbb{I}-\mathbb{P})\chi_{R}U^{\beta/2}_{\delta}g) \gtrsim \delta^{-\beta}
|W_{\alpha/2}(\mathbb{I}-\mathbb{P})(\chi_{R}U^{\beta/2}_{\delta}g)|^{2}_{L^{2}}.
\eeno
Using $(a-b)^{2} \geq a^{2}/2 - b^{2}$, we get
\ben \label{out-truncation-projection-2}   J^{\alpha,\beta}_{1}
&\gtrsim& \frac{1}{4}\delta^{-\beta}|W_{\alpha/2}U^{\beta/2}_{\delta}g|^{2}_{L^{2}} -
\frac{1}{2}\delta^{-\beta}|W_{\alpha/2}(1-\chi_{R})U^{\beta/2}_{\delta}g|^{2}_{L^{2}}
-\delta^{-\beta}|W_{\alpha/2}\mathbb{P}(\chi_{R}U^{\beta/2}_{\delta}g)|^{2}_{L^{2}}
 \nonumber \\&:=& J^{\alpha,\beta}_{1,1}- J^{\alpha,\beta}_{1,2} - J^{\alpha,\beta}_{1,3}.
\een
Thanks to $U_{\delta} \leq W$, one has $U^{\beta/2}_{\delta} \geq W_{\beta/2}$ and
\ben J^{\alpha,\beta}_{1,1} \label{J-al-be-11}
\gtrsim \delta^{-\beta}|W_{\alpha/2}W_{\beta/2}g|^{2}_{L^{2}} =  \delta^{-\beta}|g|^{2}_{L^{2}_{\gamma/2}}.
\een
Thanks to $\delta^{-\beta}U^{\beta}_{\delta} \leq W_{\beta}$, similar to \eqref{large-velocity-part-1}  and \eqref{large-velocity-part-12}, we have
\ben \label{J-al-be-12}
J^{\alpha,\beta}_{1,2} \lesssim |W_{\alpha/2}(1-\chi_{R})W_{\beta/2}g|^{2}_{L^{2}} \lesssim R^{-2}|W^{\epsilon}W_{\gamma/2}g|^{2}_{L^{2}}.\een
Similar to \eqref{large-velocity-part-13}, we get
\ben \label{J-al-be-13}
J^{\alpha,\beta}_{1,3} = \delta^{-\beta}|W_{\alpha/2}\mathbb{P}(\chi_{R}U^{\beta/2}_{\delta}g)|^{2}_{L^{2}}  \lesssim (\delta^{2}+R^{-2}) |\mu^{\f18}g|^{2}_{L^{2}} \lesssim (\delta^{2}+R^{-2}) |g|^{2}_{L^{2}_{\gamma/2}}.  \een
Plugging \eqref{J-al-be-11}, \eqref{J-al-be-12}, \eqref{J-al-be-13} into \eqref{out-truncation-projection-2},
we get
\ben \label{J-al-be-1}  J^{\alpha,\beta}_{1} \gtrsim \delta^{-\beta}|g|^{2}_{L^{2}_{\gamma/2}}-C(\delta^{2}+R^{-2})|g|^{2}_{\epsilon,\gamma/2}. \een

 {\it \underline{ Upper bound of $J^{\alpha,\beta}_{2}$.}} Now we estimate
\beno J^{\alpha,\beta}_{2}= \delta^{-\beta}\int b^{\epsilon}|v-v_{*}|^{\alpha} \mathbb{F}((1-\chi_{R}U^{\beta/2}_{\delta})\mu^{\frac{1}{2}}, \chi_{R}U^{\beta/2}_{\delta}g) d\sigma dv dv_{*}. \eeno
For simplicity, set $f_{\beta}= (1-\chi_{R}U^{\beta/2}_{\delta})\mu^{\frac{1}{2}}, g_{\beta}=\chi_{R}U^{\beta/2}_{\delta}g$. Then we get
\ben \label{J-al-be-2-into-2}
J^{\alpha,\beta}_{2} &\lesssim& \delta^{-\beta}\int b^{\epsilon} |v-v_{*}|^{\alpha} (f_{\beta}^{2})_{*}(g_{\beta}^{\prime}-g_{\beta})^{2} d\sigma dv dv_{*} + \delta^{-\beta}\int b^{\epsilon} |v-v_{*}|^{\alpha} (g_{\beta}^{2})_{*}(f_{\beta}^{\prime}-f_{\beta})^{2} d\sigma dv dv_{*}
\nonumber \\&:=& J^{\alpha,\beta}_{2,1} + J^{\alpha,\beta}_{2,2}.
\een
Similar to \eqref{f-gamma-upper}, we get
$ f_{\beta}^{2}=((1-\chi_{R}U^{\beta/2}_{\delta})\mu^{\frac{1}{2}})^{2} \lesssim (\delta^{2}+R^{-2}) \mu^{\frac{1}{2}}. $
From which together with $\delta^{-\beta}U^{\beta/2}_{\delta} \leq W_{\beta/2}$, we get
\ben \label{J-al-be-2-into-2-1}
J^{\alpha,\beta}_{2,1}&\lesssim& (\delta^{2}+R^{-2})\delta^{-\beta}\int b^{\epsilon}|v-v_{*}|^{\alpha}  \mu^{\frac{1}{2}}_{*} (g_{\beta}^{\prime}-g_{\beta})^{2} d\sigma dv dv_{*}
=(\delta^{2}+R^{-2})\delta^{-\beta}\mathcal{N}^{\epsilon,\alpha}(\mu^{\frac{1}{4}},g_{\beta})
\nonumber \\&\lesssim& (\delta^{2}+R^{-2})\delta^{-\beta}|\chi_{R}U^{\beta/2}_{\delta}g|^{2}_{\epsilon,\alpha/2} \lesssim (\delta^{2}+R^{-2})|g|^{2}_{\epsilon,\gamma/2},
\een
where we have used Corollary \ref{functional-N-epsilon-gamma} and
Lemma \ref{operatorcommutator1} with $\Phi=W_{\alpha/2}\delta^{-\beta/2}U^{\beta/2}_{\delta}\chi_{R}, M=W^\epsilon$.
Similar to \eqref{f-gamma-by-taylor}, we have
\beno  |f_{\beta}^{\prime}-f_{\beta}|^{2}  \lesssim
(\delta^{2}+R^{-2})\theta^{2} \int_{0}^{1} \mu^{\frac{1}{4}}(v(\kappa))|v(\kappa)-v_{*}|^{2}d\kappa.\eeno
Thanks to $|v-v_{*}| \sim |v(\kappa)-v_{*}|$, since $2 R \leq \frac{1}{2}\epsilon^{-1}$, by the change of
variable $v \rightarrow v(\kappa)$, we have
\ben \label{J-al-be-2-into-2-2}
J^{\alpha,\beta}_{2,2}&\lesssim& (\delta^{2}+R^{-2})\delta^{-\beta}\int b^{\epsilon} \theta^{2} (\chi_{R}U^{\beta/2}_{\delta}g)^{2}_{*} \mu^{\frac{1}{4}}(v(\kappa))|v(\kappa)-v_{*}|^{2+\alpha} d\sigma dv(\kappa) dv_{*} d\kappa.
\nonumber \\&\lesssim& (\delta^{2}+R^{-2})|\chi_{R}W_{\gamma/2+1}g|^{2}_{L^{2}} \lesssim (\delta^{2}+R^{-2})|g|^{2}_{\epsilon,\gamma/2}.
\een
Plugging \eqref{J-al-be-2-into-2-1} and \eqref{J-al-be-2-into-2-2} into \eqref{J-al-be-2-into-2}, we get
\ben \label{J-al-be-2-into-2-final}
J^{\alpha,\beta}_{2} \lesssim (\delta^{2}+R^{-2})|g|^{2}_{\epsilon,\gamma/2}.
\een

 {\it \underline{ Upper bound of $J^{\alpha,\beta}_{3}$.}}
Recall $J^{\alpha,\beta}_{3} = \delta^{-\beta}\int b^{\epsilon}|v-v_{*}|^{\alpha} \left(h^{\prime} h^{\prime}_{*}-
h h_{*}\right)^{2} \mu_{*}g^{2} d\sigma dv dv_{*} $.
By Lemma \ref{difference-term-complication}, we have
\beno\delta^{-\beta}\left(h^{\prime} h^{\prime}_{*}-
h h_{*}\right)^{2} = X(\beta,R,\delta)
 \lesssim (\delta^{2}+R^{-2})\theta^{2}\langle v_{*} \rangle^{2}\langle v \rangle^{\beta+2}\mathrm{1}_{|v|\leq 4R}. \eeno
Thanks to $\int |v-v_{*}|^{\alpha}\langle v_{*} \rangle^{2} \mu_{*} d v_{*} \lesssim \langle v \rangle^{\alpha}$, since $8 R \leq \frac{1}{2}\epsilon^{-1}$,
we get
\ben \label{J-al-be-3-up}
J^{\alpha,\beta}_{3} &\lesssim& (\delta^{2}+R^{-2}) \int b^{\epsilon}|v-v_{*}|^{\alpha} \theta^{2}\langle v_{*} \rangle^{2}\langle v \rangle^{\beta+2}\mu_{*} \mathrm{1}_{|v|\leq 4R} g^{2} d\sigma dv dv_{*}
\nonumber \\&\lesssim& (\delta^{2}+R^{-2})|\mathrm{1}_{|\cdot|\leq 4R} W_{\gamma/2+1}g|^{2}_{L^{2}} \lesssim (\delta^{2}+R^{-2})|g|^{2}_{\epsilon,\gamma/2}.\een

Plugging the estimates \eqref{J-al-be-1},
 \eqref{J-al-be-2-into-2-final} and \eqref{J-al-be-3-up} into \eqref{separate-gamma-very-small},
we get
\beno   J^{\epsilon,\gamma}(g) \gtrsim \delta^{-\beta}|g|^{2}_{L^{2}_{\gamma/2}}-C(\delta^{2}+R^{-2})|g|^{2}_{\epsilon,\gamma/2}. \eeno
Choosing $R = \delta^{-1}$, for some universal constants $C_{6},C_{7}$,
we get
\beno   J^{\epsilon,\gamma}(g) \geq C_{6}\delta^{-\beta}|g|^{2}_{L^{2}_{\gamma/2}}-C_{7}\delta^{2}|g|^{2}_{\epsilon,\gamma/2}. \eeno
Together with coercivity estimate \eqref{known-estimate}, thanks to $- 7/4 \leq \beta <0$,
by a similar argument as in {\it Step 3}, similar to \eqref{conclusion1-gamma-minus2}, we get for $-3< \gamma < - 7/4$,
\beno   J^{\epsilon,\gamma}(g) \geq C_{7}\left(\frac{C_{6}C_{3}}{2C_{4}C_{7}}\right)^{2/(2+\beta)}|g|^{2}_{\epsilon,\gamma/2}
\geq C_{7}\left(\frac{C_{6}C_{3}}{2C_{4}C_{7}}\right)^{8}|g|^{2}_{\epsilon,\gamma/2},
\eeno
for any $0 < \epsilon \leq \min\bigg\{  16^{-1}\left(\frac{C_{1}C_{3}}{2C_{4}C_{2}}\right)^{8}, 16^{-1}\left(\frac{C_{6}C_{3}}{2C_{4}C_{7}}\right)^{8}\bigg\}$.
Note that $C_{4}$ depends on $\gamma$. And this completes the proof of the theorem.
\end{proof}

\section{Upper bound estimate} \label{upper-bound}
Unless otherwise specified, in this section the parameters $\gamma$ and $s$ satisfy $-3 < \gamma \leq 0, \frac{1}{2}<s<1, \gamma+2s>-1$.
In particular, we specify the parameter $\gamma$ in $\Gamma^{\epsilon}$ and use $ \Gamma^{\epsilon,\gamma}$.
We derive the uniform upper bound of the nonlinear term $ \Gamma^{\epsilon,\gamma}$ given
in the following theorem.
\begin{thm}\label{Gamma-full-up-bound} Let $-3< \gamma \leq 0, \frac{1}{2}<s<1, \gamma+2s > -1$. It holds
\ben \label{upper-bound-Gamma}
\langle \Gamma^{\epsilon,\gamma}(g,h), f \rangle \lesssim |g|_{L^{2}}|h|_{\epsilon,\gamma/2}|f|_{\epsilon,\gamma/2}. \een
\end{thm}

Note that the estimate \eqref{upper-bound-Gamma} matches perfectly with Theorem \ref{micro-dissipation}, which enables us to establish the well-posedness theory for Cauchy problem \eqref{Cauchy-linearizedBE-grazing} near equilibrium. We remark that the parameter constraint $s>1/2, \gamma+2s>-1$, is a condition of Lemma 1.1 in \cite{he2014asymptotic}.

When $\gamma<0$, the relative velocity $v-v_{*}$ has singularity near $0$, which creates some difficulty to  obtain
the $L^{2}$ norm for the position $g$ in \eqref{upper-bound-Gamma}.
To deal with the singularity, we separate the kernel $B^{\epsilon} = B^{\epsilon,\gamma,<} + B^{\epsilon,\gamma,>}$, where
$B^{\epsilon,\gamma,<} := \zeta(|v-v_{*}|)B^{\epsilon}, B^{\epsilon,\gamma,>} := (1- \zeta(|v-v_{*}|))B^{\epsilon}$. We recall that the function $\zeta$ is defined in \eqref{zeta-property}. We call $|v-v_{*}| \leq 1$ (support of $\zeta(|v-v_{*}|)$) the singular region and $|v-v_{*}| \geq 1/2$ (support of $1-\zeta(|v-v_{*}|)$) the regular region.

 We associate $Q^{\epsilon,\gamma,>}$ with kernel $B^{\epsilon,\gamma,>}$ and
 denote $\mathcal{L}^{\epsilon,\gamma,>}, \mathcal{L}^{\epsilon,\gamma,>}_{1}, \mathcal{L}^{\epsilon,\gamma,>}_{2}, \Gamma^{\epsilon,\gamma,>}$
 correspondingly. Without ambiguity,
we explicitly define  the Boltzmann operator $Q^{\epsilon,\gamma,>}$ as follows:
\ben \label{Q-ep-ga-geq-eta}
Q^{\epsilon,\gamma,>}(g,h)(v):=
\int B^{\epsilon,\gamma,>}(v-v_*,\sigma) (g^{\prime}_{*} h^{\prime}-g_{*} h )d\sigma dv_*.
\een
Similar in \eqref{DefLep}, we define
\ben\label{Gamma-ep-eta-ga}\Gamma^{\epsilon,\gamma,>}(g,h):= \mu^{-\frac{1}{2}} Q^{\epsilon,\gamma,>}(\mu^{\frac{1}{2}}g,\mu^{\frac{1}{2}}h),
\\ \label{L-ep-eta-ga}
\mathcal{L}^{\epsilon,\gamma,>}g:= -\Gamma^{\epsilon,\gamma,>}(\mu^{\frac{1}{2}},g) - \Gamma^{\epsilon,\gamma,>}(g, \mu^{\frac{1}{2}}),
\\ \label{L1-ep-eta-ga}
\mathcal{L}^{\epsilon,\gamma,>}_{1}g:=  -\Gamma^{\epsilon,\gamma,>}(\mu^{\frac{1}{2}},g),
\\ \label{L2-ep-eta-ga} \mathcal{L}^{\epsilon,\gamma,>}_{2}g:=  - \Gamma^{\epsilon,\gamma,>}(g, \mu^{\frac{1}{2}}).
\een
Obviously
\beno
\Gamma^{\epsilon,\gamma,>}(g,h) &=& \mu^{-\frac{1}{2}} Q^{\epsilon,\gamma,>}(\mu^{\frac{1}{2}}g,\mu^{\frac{1}{2}}h)
\\&=& \int B^{\epsilon,\gamma,>}(v-v_*,\sigma) \mu^{\frac{1}{2}}_{*} \left(g^{\prime}_{*} h^{\prime}-g_{*} h\right) d\sigma dv_*
\\&=& \int B^{\epsilon,\gamma,>}(v-v_*,\sigma) \left((\mu^{\frac{1}{2}}g)^{\prime}_{*} h^{\prime}-(\mu^{\frac{1}{2}}g)_{*} h\right)d\sigma dv_*
\\&&+ \int B^{\epsilon,\gamma,>}(v-v_*,\sigma) \left(\mu^{\frac{1}{2}}_{*}-(\mu^{\frac{1}{2}})^{\prime}_{*}\right)g^{\prime}_{*} h^{\prime} d\sigma dv_*
\\&=& Q^{\epsilon,\gamma,>}(\mu^{\frac{1}{2}}g,h) + I^{\epsilon,\gamma,>}(g,h),
\eeno
where
\ben \label{I-ep-ga-geq-eta}
I^{\epsilon,\gamma,>}(g,h) :=\int B^{\epsilon,\gamma,>}(v-v_*,\sigma) \left(\mu^{\frac{1}{2}}_{*}-(\mu^{\frac{1}{2}})^{\prime}_{*}\right)g^{\prime}_{*} h^{\prime} d\sigma dv_{*} .
\een
We use kernel $B^{\epsilon,\gamma,<}(v-v_*,\sigma)$ for the Boltzmann operator $Q^{\epsilon,\gamma,<}$ in the same way as in \eqref{Q-ep-ga-geq-eta}. As in \eqref{Gamma-ep-eta-ga},  \eqref{L-ep-eta-ga}, \eqref{L1-ep-eta-ga}, \eqref{L2-ep-eta-ga}, we define
$\Gamma^{\epsilon,\gamma,<}(g,h), \mathcal{L}^{\epsilon,\gamma,<}g, \mathcal{L}^{\epsilon,\gamma,<}_{1}g, \mathcal{L}^{\epsilon,\gamma,<}_{2}g$ correspondingly.
In addition, $I^{\epsilon,\gamma,<}(g,h)$ is defined using the kernel $B^{\epsilon,\gamma,<}(v-v_*,\sigma)$  as in \eqref{I-ep-ga-geq-eta}. With these notations in hand, we have
\ben
\label{Gamma-ep-ga-into-IQ}
\Gamma^{\epsilon,\gamma}(g,h) &=& Q^{\epsilon,\gamma}(\mu^{\frac{1}{2}}g,h) + I^{\epsilon,\gamma}(g,h),
\\
\label{Gamma-ep-ga-geq-eta-into-IQ}
\Gamma^{\epsilon,\gamma,>}(g,h) &=& Q^{\epsilon,\gamma,>}(\mu^{\frac{1}{2}}g,h) + I^{\epsilon,\gamma,>}(g,h),
\\
\label{Gamma-ep-ga-leq-eta-into-IQ}
\Gamma^{\epsilon,\gamma,<}(g,h) &=& Q^{\epsilon,\gamma,<}(\mu^{\frac{1}{2}}g,h)+ I^{\epsilon,\gamma,<}(g,h),
\\
\label{Q-ep-ga-sep-eta}
Q^{\epsilon,\gamma}(g,h) &=& Q^{\epsilon,\gamma,>}(g,h) + Q^{\epsilon,\gamma,<}(g,h),
\\
\label{I-ep-ga-sep-eta}
I^{\epsilon,\gamma}(g,h) &=& I^{\epsilon,\gamma,>}(g,h) + I^{\epsilon,\gamma,<}(g,h).
\een

We will give the estimates in the following propositions and theorems as shown in the following table.
\begin{table}[!htbp]
\centering
\caption{Summary of results} \label{ResultSummary} 
\begin{tabular}{cc}
\hline
Functionals &Proposition or Theorem \\
\hline
$\langle Q^{\epsilon,\gamma,>}(g,h), f\rangle$ & Proposition \ref{ubqepsilonnonsingular}\\
$\langle I^{\epsilon,\gamma,>}(g,h), f\rangle$ & Proposition \ref{upforI-ep-ga-et}\\
$\langle Q^{\epsilon,\gamma,<}(\mu^{\frac{1}{2}}g,h), f\rangle$ & Proposition \ref{regularity-part-Q}\\
$\langle I^{\epsilon,\gamma,<}(g,h), f\rangle$ & Proposition \ref{I-less-eta-upper-bound}\\
$\langle \Gamma^{\epsilon,\gamma,>}(g,h), f\rangle$ & Theorem \ref{upGammagh-geq-eta}\\
$\langle Q^{\epsilon,\gamma}(\mu^{\frac{1}{2}}g,h), f\rangle$ & Theorem \ref{Q-full-up-bound}\\
$\langle I^{\epsilon,\gamma}(g,h), f\rangle$ & Theorem \ref{upforI-total}\\
$\langle \Gamma^{\epsilon,\gamma}(g,h), f\rangle$ & Theorem \ref{Gamma-full-up-bound}\\
\hline
\end{tabular}
\end{table}

The theorems in table \ref{ResultSummary} can be derived from the propositions. First, by \eqref{Gamma-ep-ga-geq-eta-into-IQ}, Proposition \ref{ubqepsilonnonsingular} and Proposition \ref{upforI-ep-ga-et} give Theorem \ref{upGammagh-geq-eta}.
By \eqref{Q-ep-ga-sep-eta}, Proposition \ref{ubqepsilonnonsingular} and Proposition \ref{regularity-part-Q} give Theorem \ref{Q-full-up-bound}. Similarly, by \eqref{I-ep-ga-sep-eta},
 Proposition \ref{upforI-ep-ga-et} and Proposition \ref{I-less-eta-upper-bound} give Theorem \ref{upforI-total}. By \eqref{Gamma-ep-ga-into-IQ}, Theorem \ref{Q-full-up-bound} and Theorem \ref{upforI-total} give Theorem \ref{Gamma-full-up-bound}.

Therefore, it remains to prove Propositions \ref{ubqepsilonnonsingular}-\ref{I-less-eta-upper-bound}
in this section.

\subsection{Upper bound in the regular region}
In this subsection, we will give the upper bound for the nonlinear term $\Gamma^{\epsilon,\gamma,>}(g,h)$.
Thanks to \eqref{Gamma-ep-ga-geq-eta-into-IQ}, we have
\ben \label{Gamma-ep-ga-geq-eta-into-IQ-inner}
\langle \Gamma^{\epsilon,\gamma,>}(g,h), f\rangle =  \langle Q^{\epsilon,\gamma,>}(\mu^{\frac{1}{2}}g,h), f\rangle +
\langle I^{\epsilon,\gamma,>}(g,h), f\rangle.
\een
We first consider $\langle Q^{\epsilon,\gamma,>}(g,h), f\rangle$ in subsection \ref{Q-regular}.
and then $\langle I^{\epsilon,\gamma,>}(g,h), f\rangle$ in subsection \ref{I-regular}.
\subsubsection{Upper bound of $Q^{\epsilon,\gamma,>}$} \label{Q-regular}
We give the upper bound of $ Q^{\epsilon,\gamma,>}$ in the following proposition.
\begin{prop}\label{ubqepsilonnonsingular} The estimate
$ |\langle Q^{\epsilon,\gamma,>}(g,h), f\rangle| \lesssim |g|_{L^{1}_{|\gamma|+2}}|h|_{\epsilon,\gamma/2}|f|_{\epsilon,\gamma/2}$ holds.
\end{prop}
\begin{proof} Define the translation operator $T_{v_{*}}$ by $(T_{v_{*}}f)(v) = f(v_{*}+v)$.
By geometric decomposition, we have
$\langle Q^{\epsilon,\gamma,>}(g,h), f\rangle = \mathcal{A}_{\text{r}} + \mathcal{A}_{\text{s}}$,
where
 \beno \mathcal{A}_{\text{r}} := \int b^{\epsilon}(\frac{u}{|u|}\cdot\sigma)|u|^{\gamma}(1-\zeta(|u|))g_{*} (T_{v_{*}}h)(u) ((T_{v_{*}}f)(u^{+})-(T_{v_{*}}f)(|u|\frac{u^{+}}{|u^{+}|})) d\sigma dv_{*} du,
\\
\mathcal{A}_{\text{s}}:=\int b^{\epsilon}(\frac{u}{|u|}\cdot\sigma)|u|^{\gamma}(1-\zeta(|u|))g_{*} (T_{v_{*}}h)(u)((T_{v_{*}}f)(|u|\frac{u^{+}}{|u^{+}|})- (T_{v_{*}}f)(u)) d\sigma dv_{*} du.\eeno
Note that "r" and "s" are referred to "radical" and "spherical" respectively.

{\it Step 1: Estimate of $\mathcal{A}_{\text{r}}$.}
By Lemma \ref{crosstermsimilar} and Remark \ref{exact-cross-term}, we have
\ben \label{A-r-upper-bound}
|\mathcal{A}_{\text{r}}| &\lesssim& \int |g_{*}| (|W^{\epsilon}W_{\gamma/2}T_{v_{*}}h|_{L^{2}}+|W^{\epsilon}(D)W_{\gamma/2}T_{v_{*}}h|_{L^{2}})
\\&&\times(|W^{\epsilon}W_{\gamma/2}T_{v_{*}}f|_{L^{2}}+|W^{\epsilon}(D)W_{\gamma/2}T_{v_{*}}f|_{L^{2}}) dv_{*}.\nonumber
\een
Thanks to \eqref{separate-into-2-cf}, for $u \in \mathbb{R}^{3}$, we have
\ben \label{translation-out-cf}
|W^{\epsilon}T_{u}f|_{L^{2}} \lesssim W^{\epsilon}(u) |W^{\epsilon}f|_{L^{2}}.
\een
For $u \in \mathbb{R}^{3}, l \in \mathbb{R}$, $
(T_{u} W^{l})(v) = \langle v + u \rangle^{l} \lesssim C(l)\langle u \rangle^{|l|} \langle v \rangle^{l}.$
As a result, we have
\ben \label{translation-out-weight}
|T_{u}f|_{L^{2}_{l}} \lesssim \langle u \rangle^{|l|} |f|_{L^{2}_{l}}.
\een
By \eqref{translation-out-cf} and \eqref{translation-out-weight}, we have
\begin{eqnarray}\label{tvstartonovstar1}
|W^{\epsilon}W_{\gamma/2}T_{v_{*}}h|_{L^{2}} \lesssim W^{\epsilon}(v_{*})W_{|\gamma|/2}(v_{*})|W^{\epsilon}W_{\gamma/2}h|_{L^{2}}\lesssim W_{|\gamma|/2+1}(v_{*})|W^{\epsilon}W_{\gamma/2}h|_{L^{2}}.
\end{eqnarray}
Since $W^{\epsilon} \in S^{1}_{1,0}, W_{\gamma/2} \in S^{\gamma/2}_{1,0}$, by Lemma \ref{operatorcommutator1}, we have
\begin{eqnarray}\label{tvstartonovstar2}
|W^{\epsilon}(D)W_{\gamma/2}T_{v_{*}}h|_{L^{2}} &\lesssim& |W_{\gamma/2}W^{\epsilon}(D)T_{v_{*}}h|_{L^{2}} + |T_{v_{*}}h|_{H^{0}_{\gamma/2-1}}
\\&=& |W_{\gamma/2}T_{v_{*}}W^{\epsilon}(D)h|_{L^{2}} + |T_{v_{*}}h|_{H^{0}_{\gamma/2-1}}
\nonumber\\&\lesssim& W_{|\gamma|/2}(v_{*})(|W_{\gamma/2}W^{\epsilon}(D)h|_{L^{2}} + |h|_{L^{2}_{\gamma/2-1}})
\nonumber\\&\lesssim& W_{|\gamma|/2}(v_{*})|W^{\epsilon}(D)W_{\gamma/2}h|_{L^{2}}, \nonumber
\end{eqnarray}
where we have used the fact $T_{v_{*}}$ and $W^{\epsilon}(D)$ are commutable, inequality \eqref{translation-out-weight} and Lemma \ref{operatorcommutator1}.
Plugging \eqref{tvstartonovstar1} and \eqref{tvstartonovstar2} into \eqref{A-r-upper-bound}, we have
\beno
|\mathcal{A}_{\text{r}}| \lesssim | g|_{L^{1}_{|\gamma|+2}}( |W^{\epsilon}(D)W_{\gamma/2}h|_{L^{2}} + |W^{\epsilon}W_{\gamma/2}h|_{L^{2}})
( |W^{\epsilon}(D)W_{\gamma/2}f|_{L^{2}} + |W^{\epsilon}W_{\gamma/2}f|_{L^{2}}).
\eeno

{\it Step 2: Estimate of $\mathcal{A}_{\text{s}}$.}
Let $u = r \tau$ and $\varsigma = \frac{\tau+\sigma}{|\tau+\sigma|}$, then $\frac{u}{|u|} \cdot \sigma = 2(\tau\cdot\varsigma)^{2} - 1$ and $|u|\frac{u^{+}}{|u^{+}|} = r \varsigma$. In the change of variable $(u, \sigma) \rightarrow (r, \tau, \varsigma)$, one has
$
du d\sigma = 4  (\tau\cdot\varsigma) r^{2} dr d \tau d \varsigma.
$
Then
\beno
\mathcal{A}_{\text{s}} &=& 4 \int r^\gamma (1-\zeta(r))b^{\epsilon}(2(\tau\cdot\varsigma)^{2} - 1)(T_{v_*}h)(r\tau)\big((T_{v_*}f)(r\varsigma) - (T_{v_*}f) (r\tau)\big) (\tau\cdot\varsigma) r^{2} dr d \tau d \varsigma d v_{*}\\
&=& 2 \int r^\gamma(1-\zeta(r))b^{\epsilon}(2(\tau\cdot\varsigma)^{2} - 1)\big((T_{v_*}h)(r\tau) - (T_{v_*}h) (r\varsigma)\big)\\ &&\times \big((T_{v_*}f)(r\varsigma) - (T_{v_*}f) (r\tau)\big) (\tau\cdot\varsigma) r^{2} dr d \tau d \varsigma d v_{*}\\&=&
-\frac{1}{2}\int b^{\epsilon}(\frac{u}{|u|}\cdot\sigma)|u|^{\gamma}(1-\zeta(|u|))g_{*} ((T_{v_{*}}h)(|u|\frac{u^{+}}{|u^{+}|})-(T_{v_{*}}h)(u))\\ &&\times
 ((T_{v_{*}}f)(|u|\frac{u^{+}}{|u^{+}|})-(T_{v_{*}}f)(u)) d\sigma dv_{*} du.
\eeno
Then by Cauchy-Schwartz inequality and the fact $|u|^{\gamma}(1-\zeta(|u|)) \lesssim \langle u \rangle^{\gamma}$, we have
\beno
|\mathcal{A}_{\text{s}}| &\lesssim& \{\int b^{\epsilon}(\frac{u}{|u|}\cdot\sigma)\langle u \rangle^{\gamma}|g_{*}| ((T_{v_{*}}h)( |u|\frac{u^{+}}{|u^{+}|})-(T_{v_{*}}h)( u))^{2} d\sigma dv_{*} du\}^{\frac{1}{2}}
\\&& \times \{\int b^{\epsilon}(\frac{u}{|u|}\cdot\sigma)\langle u \rangle^{\gamma}|g_{*}|
((T_{v_{*}}f)(|u|\frac{u^{+}}{|u^{+}|})-(T_{v_{*}})f(u))^{2} d\sigma dv_{*} du\}^{\frac{1}{2}}
:= (\mathcal{A}_{\text{s}}(h))^{\frac{1}{2}}(\mathcal{A}_{\text{s}}(f))^{\frac{1}{2}}.
\eeno
Note that $\mathcal{A}_{\text{s}}(h)$ and $\mathcal{A}_{\text{s}}(f)$ have exactly the same structure. It suffices to estimate  $\mathcal{A}_{\text{s}}(f)$.
Since
\beno
((T_{v_{*}}f)(|u|\frac{u^{+}}{|u^{+}|})-(T_{v_{*}}f)(u))^{2} \leq 2 ((T_{v_{*}}f)(|u|\frac{u^{+}}{|u^{+}|})-(T_{v_{*}}f)(u^{+}))^{2} + 2 ((T_{v_{*}}f)(u^{+})-(T_{v_{*}}f)(u))^{2},
\eeno
we have
\beno
\mathcal{A}_{\text{s}}(f) &\lesssim& \int b^{\epsilon}(\frac{u}{|u|}\cdot\sigma)\langle u \rangle^{\gamma}|g_{*}| ((T_{v_{*}}f)(|u|\frac{u^{+}}{|u^{+}|})-(T_{v_{*}}f)(u^{+}))^{2} d\sigma dv_{*} du
\\&&+\int b^{\epsilon}(\frac{u}{|u|}\cdot\sigma)\langle u \rangle^{\gamma}|g_{*}| ((T_{v_{*}}f)(u^{+})-(T_{v_{*}}f)(u))^{2} d\sigma dv_{*} du
:= \mathcal{A}_{\text{s},1}(f)+ \mathcal{A}_{\text{s},2}(f).
\eeno
By Lemma \ref{gammanonzerotozero}, and the facts (\ref{tvstartonovstar1}) and (\ref{tvstartonovstar2}), we have
\beno
\mathcal{A}_{\text{s},1}(f) \lesssim \int |g_{*}| \mathcal{Z}^{\epsilon,\gamma}(T_{v_{*}}f) dv_{*} \lesssim| g|_{L^{1}_{|\gamma|+2}}(|W^{\epsilon}(D)W_{\gamma/2}f|^{2}_{L^{2}}+|W^{\epsilon}W_{\gamma/2}f|^{2}_{L^{2}}).
\eeno
Recalling the notation in \eqref{N-tilde-epsilon-gamma}, we observe that $\mathcal{A}_{\text{s},2}(f) =  \tilde{\mathcal{N}}^{\epsilon,\gamma}(\sqrt{|g|}, f)$.
By Lemma \ref{reduce-gamma-to-0-no-sigularity}, we have
\beno
\tilde{\mathcal{N}}^{\epsilon,\gamma}(\sqrt{|g|},f)  \lesssim \mathcal{N}^{\epsilon,0}(W_{|\gamma|/2} \sqrt{|g|},W_{\gamma/2}f) + |g|_{L^{1}_{|\gamma+2|}}|f|^{2}_{L^{2}_{\gamma/2}}. \nonumber
\eeno
Then by \eqref{anisotropic-regularity-general-g-up-bound} in Proposition \ref{lowerboundpart2-general-g}, we get
\beno
\mathcal{N}^{\epsilon,0}(W_{|\gamma|/2} \sqrt{|g|},W_{\gamma/2}f) \lesssim  |W_{|\gamma|/2} \sqrt{|g|}|^{2}_{L^{2}_{1}}|f|^{2}_{\epsilon,\gamma/2} \lesssim |g|_{L^{1}_{|\gamma|+2}}|f|^{2}_{\epsilon,\gamma/2},
\eeno
which gives
$
\mathcal{A}_{\text{s},2}(f) \lesssim  |g|_{L^{1}_{|\gamma|+2}}|f|^{2}_{\epsilon,\gamma/2}.
$
Combining the estimates for  $\mathcal{A}_{\text{s},1}(f)$ and $\mathcal{A}_{\text{s},2}(f)$, we get
$
\mathcal{A}_{\text{s}}(f) \lesssim |g|_{L^{1}_{|\gamma|+2}}|f|^{2}_{\epsilon,\gamma/2}.
$
Then we have
\beno
|\mathcal{A}_{\text{s}}| \lesssim (\mathcal{A}_{\text{s}}(h))^{\frac{1}{2}}(\mathcal{A}_{\text{s}}(f))^{\frac{1}{2}} \lesssim |g|_{L^{1}_{|\gamma|+2}}|h|_{\epsilon,\gamma/2}|f|_{\epsilon,\gamma/2}.
\eeno
Then the proof of the proposition is completed by  the estimates of $\mathcal{A}_{\text{r}}$ and $\mathcal{A}_{\text{s}}$.
\end{proof}

\subsubsection{Upper bound of $I^{\epsilon,\gamma,>}$} \label{I-regular}
We now turn to the upper bound estimate of the term $\langle I^{\epsilon,\gamma,>}(g,h), f\rangle$ that is given in the following proposition.

\begin{prop}\label{upforI-ep-ga-et}
The estimate
$ |\langle I^{\epsilon,\gamma,>}(g,h), f\rangle|  \lesssim  |g|_{L^{2}}|h|_{\epsilon,\gamma/2}|W^{\epsilon}f|_{L^{2}_{\gamma/2}}$ holds.
\end{prop}
\begin{proof} Notice that
$
(\mu^{\frac{1}{2}})_{*}^{\prime} - (\mu^{\frac{1}{2}})_{*} = ((\mu^{\frac{1}{4}})^{\prime}_{*} - \mu^{\frac{1}{4}}_{*})^{2} + 2\mu^{\frac{1}{4}}_{*}((\mu^{\frac{1}{4}})^{\prime}_{*} - \mu^{\frac{1}{4}}_{*})
$
and $h = (h-h^{\prime}) + h^{\prime}$, we have
\ben \label{I-into-3-parts}
\langle I^{\epsilon,\gamma,>}(g,h), f\rangle &=& \int B^{\epsilon,\gamma,>}((\mu^{\frac{1}{2}})_{*}^{\prime} - (\mu^{\frac{1}{2}})_{*}) g_{*} h f^{\prime} d\sigma dv_{*} dv
\nonumber \\&=& \int B^{\epsilon,\gamma,>}((\mu^{\f18})_{*}^{\prime} + (\mu^{\f18})_{*})^{2}((\mu^{\f18})_{*}^{\prime} - (\mu^{\f18})_{*})^{2}g_{*} h f^{\prime} d\sigma dv_{*} dv
\nonumber \\&&+ 2 \int B^{\epsilon,\gamma,>}((\mu^{\frac{1}{4}})^{\prime}_{*} - \mu^{\frac{1}{4}}_{*})(\mu^{\frac{1}{4}}g)_{*}(h-h^{\prime})f^{\prime} d\sigma dv_{*}dv
\nonumber \\&&+ 2 \int B^{\epsilon,\gamma,>}((\mu^{\frac{1}{4}})^{\prime}_{*} - \mu^{\frac{1}{4}}_{*})(\mu^{\frac{1}{4}}g)_{*} h^{\prime} f^{\prime} d\sigma dv_{*}dv
:= \mathcal{I}_{1} + \mathcal{I}_{2} + \mathcal{I}_{3}.
\een

We divide the  proof into three steps for  $\mathcal{I}_{1}, \mathcal{I}_{2}$ and $\mathcal{I}_{3}$ defined in \eqref{I-into-3-parts} respectively.
In the proof, the following estimate is used often:
\ben \label{mu-prodcut-square}
((\mu^{\f18})_{*}^{\prime} - (\mu^{\f18})_{*})^{2} \lesssim \min\{1,|v-v_{*}|^{2}\sin^{2}\frac{\theta}{2}\} \sim \min\{1,|v^{\prime}-v_{*}|^{2}\sin^{2}\frac{\theta}{2}\}\sim \min\{1,|v-v^{\prime}_{*}|^{2}\sin^{2}\frac{\theta}{2}\}.
\een

{\it Step 1: Estimate of $\mathcal{I}_{1}$.}
Recall that
\beno
\mathcal{I}_{1} = \int B^{\epsilon,\gamma,>}((\mu^{\f18})_{*}^{\prime} + (\mu^{\f18})_{*})^{2}((\mu^{\f18})_{*}^{\prime} - (\mu^{\f18})_{*})^{2}g_{*} h f^{\prime} d\sigma dv_{*} dv.
\eeno
Since $|v-v_{*}| \geq 1/2$, we have
\ben \label{v-minus-vstar-lower-no-singularity}
|v-v_{*}|^{\gamma}  \sim \langle v -v_{*} \rangle^{\gamma}. \een
By \eqref{v-minus-vstar-lower-no-singularity} and Cauchy-Schwartz inequality, we have
\beno
|\mathcal{I}_{1}| &\lesssim& \{\int b^{\epsilon}(\cos\theta)\langle v-v_{*}\rangle^{\gamma}((\mu^{\f18})_{*}^{\prime} + (\mu^{\f18})_{*})^{2}((\mu^{\f18})_{*}^{\prime} - (\mu^{\f18})_{*})^{2}g^{2}_{*} h^{2}  d\sigma dv_{*} dv\}^{\frac{1}{2}}
\\&&\times \{\int b^{\epsilon}(\cos\theta)\langle v-v_{*} \rangle^{\gamma}((\mu^{\f18})_{*}^{\prime} + (\mu^{\f18})_{*})^{2}((\mu^{\f18})_{*}^{\prime} - (\mu^{\f18})_{*})^{2} f^{\prime 2} d\sigma dv_{*} dv\}^{\frac{1}{2}}
:= (\mathcal{I}_{1,1})^{\frac{1}{2}} (\mathcal{I}_{1,2})^{\frac{1}{2}}.
\eeno

\underline{Estimate of $\mathcal{I}_{1,1}$.}
We claim that
\ben \label{claim-1-L2} \mathcal{A} := \int b^{\epsilon}(\cos\theta)\langle v-v_{*}\rangle^{\gamma}((\mu^{\f18})_{*}^{\prime} + (\mu^{\f18})_{*})^{2}((\mu^{\f18})_{*}^{\prime} - (\mu^{\f18})_{*})^{2}  d\sigma \lesssim (W^{\epsilon})^{2}(v)\langle v\rangle^{\gamma},\een
which yields
$\mathcal{I}_{1,1} \lesssim  |g|^{2}_{L^{2}}|W^{\epsilon}h|^{2}_{L^{2}_{\gamma/2}}$. Now we prove \eqref{claim-1-L2}. Since
$((\mu^{\f18})_{*}^{\prime} + (\mu^{\f18})_{*})^{2} \leq 2(\mu^{\frac{1}{4}})^{\prime}_{*} + 2\mu^{\frac{1}{4}}_{*}$, we have
\beno \mathcal{A} &\lesssim& \int b^{\epsilon}(\cos\theta)\langle v-v_{*}\rangle^{\gamma}\mu^{\frac{1}{4}}_{*}((\mu^{\f18})_{*}^{\prime} - (\mu^{\f18})_{*})^{2}  d\sigma
\\&&+ \int b^{\epsilon}(\cos\theta)\langle v-v_{*}\rangle^{\gamma}(\mu^{\frac{1}{4}})^{\prime}_{*}((\mu^{\f18})_{*}^{\prime} - (\mu^{\f18})_{*})^{2}  d\sigma
:= \mathcal{A}_{1}+\mathcal{A}_{2}.
\eeno
By \eqref{mu-prodcut-square} and Proposition \ref{symbol},  one has
\ben \label{estimate-of-A1-gives-weight-v}
\mathcal{A}_{1} \lesssim \langle v-v_{*}\rangle^{\gamma}\mu^{\frac{1}{4}}_{*}(W^{\epsilon})^{2}(v-v_{*}) \lesssim \langle v \rangle^{\gamma} \langle v_{*}\rangle^{|\gamma|} \mu^{\frac{1}{4}}_{*} (W^{\epsilon})^{2}(v) (W^{\epsilon})^{2}(v_{*}) \lesssim (W^{\epsilon})^{2}(v)\langle v\rangle^{\gamma},\een
where we have used \eqref{separate-into-2-cf}.
As for $\mathcal{A}_{2}$, since $|v-v_{*}| \sim |v-v^{\prime}_{*}|$ so that $\langle v-v_{*}\rangle^{\gamma} \lesssim \langle v-v^{\prime}_{*}\rangle^{\gamma} \lesssim \langle v\rangle^{\gamma}\langle v^{\prime}_{*}\rangle^{|\gamma|}$, we have
 $\mathcal{A}_{2} \lesssim \langle v\rangle^{\gamma} \int b^{\epsilon}(\cos\theta) (\mu^{\f18})_{*}^{\prime} \min\{1,|v-v_{*}|^{2}\sin^{2}\frac{\theta}{2}\} d\sigma$.
If $|v-v_{*}|\geq 10|v|$, then it holds $|v^{\prime}_{*}| = |v^{\prime}_{*}-v+v|\geq |v^{\prime}_{*}-v| -|v| \geq (\frac{1}{\sqrt{2}} - \frac{1}{10})|v-v_{*}| \geq \frac{1}{5}|v-v_{*}|$, and thus $(\mu^{\f18})_{*}^{\prime} \lesssim \mu^{\frac{1}{200}}(v-v_{*})$ that gives
\beno \mathcal{A}_{2} \lesssim \langle v\rangle^{\gamma}\mu^{\frac{1}{200}}(v-v_{*}) (W^{\epsilon})^{2}(v-v_{*}) \lesssim \langle v\rangle^{\gamma}.\eeno
If $|v-v_{*}|\leq 10|v|$, by Proposition \ref{symbol}, we have
\beno \mathcal{A}_{2} \lesssim \langle v\rangle^{\gamma} \int b^{\epsilon}(\cos\theta)  \min\{1,|v|^{2}\sin^{2}\frac{\theta}{2}\} d\sigma  \lesssim (W^{\epsilon})^{2}(v)\langle v\rangle^{\gamma}.\eeno
Combining the estimates of $\mathcal{A}_{1}$ and $\mathcal{A}_{2}$ gives \eqref{claim-1-L2}.

\underline{Estimate of $\mathcal{I}_{1,2}$.}
By the change of variables $(v,v_{*}) \rightarrow (v^{\prime},v_{*}^{\prime})$ and   $(v,v_{*},\sigma) \rightarrow (v_{*},v,-\sigma)$, using \eqref{v-minus-vstar-lower-no-singularity} and $\gamma \leq 0$, we have
\beno
\mathcal{I}_{1,2} \lesssim \int b^{\epsilon}(\cos\theta)|v-v_{*}|^{\gamma}((\mu^{\f18})^{\prime} - \mu^{\f18})^{2} f_{*}^{2} d\sigma dv_{*} dv
= \mathcal{N}^{\epsilon,\gamma}(f, \mu^{\f18}) \lesssim  |W^{\epsilon}f|^{2}_{L^{2}_{\gamma/2}},
\eeno
where we have used the estimate in  Remark \ref{also-hold-for-a}.

Combining the estimates of $\mathcal{I}_{1,1}$ and $\mathcal{I}_{1,2}$, we have
\ben \label{I-1-final}
\mathcal{I}_{1} \lesssim |g|_{L^{2}}|W^{\epsilon}h|_{L^{2}_{\gamma/2}}|W^{\epsilon}f|_{L^{2}_{\gamma/2}}.
\een

{\it Step 2: Estimate of $\mathcal{I}_{2}$.} By Cauchy-Schwartz inequality, we have
\beno
\mathcal{I}_{2} &=& 2\int B^{\epsilon,\gamma,>}((\mu^{\frac{1}{4}})^{\prime}_{*} - \mu^{\frac{1}{4}}_{*})(\mu^{\frac{1}{4}}g)_{*}(h-h^{\prime})f^{\prime}
d\sigma dv_{*} dv
\\&\lesssim&  \bigg(\int B^{\epsilon,\gamma,>}|(\mu^{\frac{1}{4}}g)_{*}|(h-h^{\prime})^{2}d\sigma dv_{*} dv\bigg)^{\frac{1}{2}}
\\&&\times \bigg(\int B^{\epsilon,\gamma,>}((\mu^{\frac{1}{4}})^{\prime}_{*} - \mu^{\frac{1}{4}}_{*})^{2}|(\mu^{\frac{1}{4}}g)_{*}|f^{\prime 2}d\sigma dv_{*} dv\bigg)^{\frac{1}{2}}
:= (\mathcal{I}_{2,1})^{\frac{1}{2}}(\mathcal{I}_{2,2})^{\frac{1}{2}}.
\eeno

\underline{Estimate of $\mathcal{I}_{2,1}$.}
Notice that
$
(h-h^{\prime})^{2} = (h^{2})^{\prime} - h^{2} - 2h(h^{\prime}-h),
$
we have
\beno
\mathcal{I}_{2,1} = \int B^{\epsilon,\gamma,>}(\mu^{\frac{1}{4}}g)_{*}((h^{2})^{\prime} - h^{2})d\sigma dv_{*} dv - 2 \langle Q^{\epsilon}(|\mu^{\frac{1}{4}}g|, h), h\rangle.
\eeno
By cancellation lemma and \eqref{v-minus-vstar-lower-no-singularity}, we get
\beno |\int B^{\epsilon,\gamma,>}(\mu^{\frac{1}{4}}g)_{*}((h^{2})^{\prime} - h^{2})d\sigma dv_{*} dv|\lesssim \int \langle v-v_{*} \rangle^{\gamma} | (\mu^{\frac{1}{4}}g)_{*}h^{2} | dv dv_{*} \lesssim |\mu^{\f18}g|_{L^{2}}|h|^{2}_{L^{2}_{\gamma/2}}.\eeno
By Proposition \ref{ubqepsilonnonsingular}, we have
\beno |\langle Q^{\epsilon,\gamma,>}(\mu^{\frac{1}{4}}g, h), h\rangle|  \lesssim |\mu^{\frac{1}{4}}g|_{L^{1}_{|\gamma|+2}}|h|^{2}_{\epsilon,\gamma/2} \lesssim |\mu^{\f18}g|_{L^{2}}|h|_{\epsilon,\gamma/2}^2.\eeno
Then the above two estimates give
\beno
|\mathcal{I}_{2,1}| \lesssim |\mu^{\f18}g|_{L^{2}}|h|_{\epsilon,\gamma/2}^2.
\eeno

\underline{Estimate of $\mathcal{I}_{2,2}$.}
Using the change of variable $v \rightarrow v^{\prime}$ and the estimate \eqref{estimate-of-A1-gives-weight-v} of $\mathcal{A}_{1}$, we have
$
\mathcal{I}_{2,2} \lesssim |\mu^{\f18}g|_{L^{2}}|W^{\epsilon}f|^{2}_{L^{2}_{\gamma/2}}.
$

Putting together the estimates of $\mathcal{I}_{2,1}$ and $\mathcal{I}_{2,2}$, we get
\ben \label{I-2-final}
|\mathcal{I}_{2}| \lesssim
| \mu^{\f18}g|_{L^{2}}|h|_{\epsilon,\gamma/2}|W^{\epsilon}f|_{L^{2}_{\gamma/2}}.
\een

{\it Step 3: Estimate of $\mathcal{I}_{3}$.}
By the change of variables $(v,v_{*}) \rightarrow (v^{\prime},v_{*}^{\prime})$ and   $(v,v_{*},\sigma) \rightarrow (v_{*},v,-\sigma)$,
\beno
\mathcal{I}_{3} =  2\int B^{\epsilon,\gamma,>}(\mu^{\frac{1}{4}} - (\mu^{\frac{1}{4}})^{\prime})(\mu^{\frac{1}{4}}g)^{\prime} h_{*}f_{*} d\sigma dv_{*}dv.
\eeno
For notational convenience, let $E_{1} = \{(v,v_{*},\sigma): |v_{*}| \geq 1/\epsilon, \sin(\theta/2) \leq |v_{*}|^{-1}\}, E_{2} = \{(v,v_{*},\sigma): |v_{*}| \geq 1/\epsilon, |v_{*}|^{-1} \leq \sin(\theta/2) \leq \epsilon\}, E_{3} = \{(v,v_{*},\sigma):  |v_{*}| \leq 1/\epsilon\}$. Then $\mathcal{I}_{3}$ can be decomposed into three parts:
 $ \mathcal{I}_{3,i}$ corresponding to $E_{i}$ for $i=1,2,3$.

 \underline{Estimate of $\mathcal{I}_{3,1}$} By Taylor expansion, one has
\beno
\mu^{\frac{1}{4}} - (\mu^{\frac{1}{4}})^{\prime} = (\nabla \mu^{\frac{1}{4}})(v^{\prime})\cdot(v-v^{\prime}) + \int_{0}^{1} (1-\kappa) [(\nabla^{2} \mu^{\frac{1}{4}})(v(\kappa)):(v-v^{\prime})\otimes(v-v^{\prime})] d\kappa,
\eeno
where $v(\kappa) = v^{\prime} + \kappa(v-v^{\prime})$.
Observe that for any fixed $v_{*}$, it holds
\beno
\int B^{\epsilon,\gamma,>} \mathrm{1}_{|v_{*}| \geq 1/\epsilon, \sin(\theta/2) \leq |v_{*}|^{-1}}(\nabla \mu^{\frac{1}{4}})(v^{\prime})\cdot(v-v^{\prime})(\mu^{\frac{1}{4}}g)^{\prime}  d\sigma dv = 0,
\eeno
which gives
\beno
|\mathcal{I}_{3,1}| &=&  |\int_{E_{3}\times[0,1]} B^{\epsilon,\gamma,>}
\mathrm{1}_{|v_{*}| \geq 1/\epsilon,  \sin(\theta/2) \leq |v_{*}|^{-1}}
\\&&\times (1-\kappa)[(\nabla^{2} \mu^{\frac{1}{4}})(v(\kappa)):(v-v^{\prime})\otimes(v-v^{\prime})] (\mu^{\frac{1}{4}}g)^{\prime} h_{*}f_{*} d\kappa d\sigma dv_{*}dv|
\\&\lesssim& \int b^{\epsilon}(\cos\theta) \langle v^{\prime}-v_{*} \rangle ^{\gamma+2}\sin^{2}\frac{\theta}{2}
\mathrm{1}_{|v_{*}| \geq 1/\epsilon, \sin(\theta/2) \leq |v_{*}|^{-1}}|(\mu^{\frac{1}{4}}g)^{\prime} h_{*}f_{*}|  d\sigma dv_{*}dv^{\prime}
\\&\lesssim& \epsilon^{2s-2}\int \langle v_{*} \rangle^{\gamma+2s} \mathrm{1}_{|v_{*}| \geq 1/\epsilon}|(\mu^{\frac{1}{8}}g)^{\prime} h_{*}f_{*}|   dv_{*}dv^{\prime} \lesssim |\mu^{\f{1}{16}}g|_{L^{2}}|W^{\epsilon}h|_{L^{2}_{\gamma/2}}|W^{\epsilon}f|_{L^{2}_{\gamma/2}},
\eeno
where we have used the fact $|v^{\prime}-v_{*}| \sim |v-v_{*}|$, the estimate \eqref{v-minus-vstar-lower-no-singularity}, the change of variable $v \rightarrow v^{\prime}$, the estimate $\int b^{\epsilon}(\cos\theta) \sin^{2}\frac{\theta}{2}
\mathrm{1}_{\sin(\theta/2) \leq |v_{*}|^{-1}} d\sigma \lesssim \epsilon^{2s-2} |v_{*}|^{2s-2} \sim \epsilon^{2s-2} \langle v_{*} \rangle^{2s-2}$, $\langle v^{\prime}-v_{*} \rangle^{\gamma+2} \lesssim \langle v_{*} \rangle^{\gamma+2}\langle v^{\prime} \rangle ^{|\gamma+2|}$ and $\langle v^{\prime} \rangle^{|\gamma+2|}(\mu^{\frac{1}{4}})^{\prime} \lesssim (\mu^{\frac{1}{8}})^{\prime}$.

 \underline{Estimate of $\mathcal{I}_{3,2}$.} By the estimate $\int b^{\epsilon}(\cos\theta) \sin^{2}\frac{\theta}{2}
\mathrm{1}_{|v_{*}|^{-1} \leq \sin(\theta/2) \leq \epsilon} d\sigma \lesssim \epsilon^{2s-2} |v_{*}|^{2s} \sim \epsilon^{2s-2} \langle v_{*} \rangle^{2s}$ and the similar argument used for $\mathcal{I}_{3,1}$, we get
\beno
|\mathcal{I}_{3,2}| &\lesssim& \int b^{\epsilon}(\cos\theta) \mathrm{1}_{|v_{*}| \geq 1/\epsilon,  |v_{*}|^{-1} \leq \sin(\theta/2) \leq \epsilon}
\langle v^{\prime}-v_{*} \rangle^{\gamma} |(\mu^{\frac{1}{4}}g)^{\prime} h_{*}f_{*}| d\sigma dv_{*}dv^{\prime}
\\&\lesssim& \epsilon^{2s-2}\int \langle v_{*} \rangle^{\gamma+2s} \mathrm{1}_{|v_{*}| \geq 1/\epsilon}|(\mu^{\frac{1}{8}}g)^{\prime} h_{*}f_{*}|   dv_{*}dv^{\prime} \lesssim |\mu^{\f{1}{16}}g|_{L^{2}}|W^{\epsilon}h|_{L^{2}_{\gamma/2}}|W^{\epsilon}f|_{L^{2}_{\gamma/2}}.
\eeno

\underline{Estimate of $\mathcal{I}_{3,3}$.}
By the estimate \eqref{order-2}
and the similar argument used for $\mathcal{I}_{3,1}$, we get
\beno
|\mathcal{I}_{3,3}| &\lesssim& \int b^{\epsilon}(\cos\theta)\langle v^{\prime}-v_{*}\rangle^{\gamma+2}\sin^{2}\frac{\theta}{2}\mathrm{1}_{|v_{*}| \leq 1/\epsilon }|(\mu^{\frac{1}{4}}g)^{\prime} h_{*}f_{*}|  d\sigma dv_{*}dv^{\prime}
\\&\lesssim& \int \langle v_{*} \rangle^{\gamma+2} \mathrm{1}_{|v_{*}| \leq 1/\epsilon}|(\mu^{\frac{1}{8}}g)^{\prime} h_{*}f_{*}|   dv_{*}dv^{\prime} \lesssim |\mu^{\f{1}{16}}g|_{L^{2}}|W^{\epsilon}h|_{L^{2}_{\gamma/2}}|W^{\epsilon}f|_{L^{2}_{\gamma/2}}.
\eeno

The estimates of $\mathcal{I}_{3,1}, \mathcal{I}_{3,2}$ and $\mathcal{I}_{3,3}$ give
\ben \label{I-3-final}
|\mathcal{I}_{3}| \lesssim |\mu^{\f18}g|_{L^{2}}|W^{\epsilon}h|_{L^{2}_{\gamma/2}}|W^{\epsilon}f|_{L^{2}_{\gamma/2}}.
\een

The proof of the proposition is completed  by combining \eqref{I-1-final}, \eqref{I-2-final} and \eqref{I-3-final}.
\end{proof}

\subsubsection{Upper bound of $ \Gamma^{\epsilon,\gamma,>}(g,h)$}
Recalling \eqref{Gamma-ep-ga-geq-eta-into-IQ-inner},
by Proposition \ref{ubqepsilonnonsingular} and Proposition \ref{upforI-ep-ga-et}, noting that $|\mu^{\frac{1}{2}} g|_{L^{1}_{|\gamma|+2}} \lesssim |g|_{L^{2}}$,
we have
\begin{thm}\label{upGammagh-geq-eta}
The estimate
$
|\langle \Gamma^{\epsilon,\gamma,>}(g,h), f\rangle| \lesssim  |g|_{L^{2}}|h|_{\epsilon,\gamma/2}|f|_{\epsilon, \gamma/2}
$ holds.
\end{thm}

\subsection{Upper bound in the singular region}
By \eqref{Gamma-ep-ga-leq-eta-into-IQ}, we have
\beno
\langle \Gamma^{\epsilon,\gamma,<}(g,h), f \rangle = \langle Q^{\epsilon,\gamma,<}(\mu^{\frac{1}{2}}g,h), f \rangle + \langle I^{\epsilon,\gamma,<}(g,h), f \rangle.
\eeno
We will give estimates on  $\langle Q^{\epsilon,\gamma,<}(\mu^{\frac{1}{2}}g,h), f \rangle$  and $\langle I^{\epsilon,\gamma,<}(g,h), f \rangle$ in Sections 3.2.1 and 3.2.2 respectively.
\subsubsection{Upper bound of $Q^{\epsilon,\gamma,<}$} Recall Lemma 1.1 from \cite{he2014asymptotic},
\ben \label{He-JSP-upper-bound}
|\langle Q^{\epsilon,\gamma}(g,h) , f \rangle| \lesssim (|g|_{L^{1}_{N_{1}}} + |g|_{L^{2}_{N_{1}}}) |W^{\epsilon}(D)W_{N_{2}}h|_{L^{2}}|W^{\epsilon}(D)W_{N_{3}}f|_{L^{2}},
\een
where $N_{1}, N_{2}, N_{3}$  satisfying $N_{1} \geq |N_{2}|+|N_{3}|$ and $N_{2} + N_{3} \geq \gamma +2$.
Note that estimate \eqref{He-JSP-upper-bound} requires $\gamma+2$ order weight on the latter two
functions,
while  \eqref{upper-bound-Gamma}  allows only $\gamma$ order.
On the other hand, we only need to consider $\langle Q^{\epsilon,\gamma,<}(\mu^{\frac{1}{2}}g,h), f \rangle$
for which there is a factor $\mu^{\frac{1}{2}}$ for $g$.
In addition, when $|v-v_{*}| \leq 1$ as in $Q^{\epsilon,\gamma,<}$, one has $\mu(v_{*}) \lesssim \mu^{\frac{1}{2}}(v)$, which means
weight can be exchanged between $v$ and $v_{*}$. By the above observation, we can apply \eqref{He-JSP-upper-bound} to get the
 following upper bound  on $ Q^{\epsilon,\gamma,<}$.
\begin{prop} \label{regularity-part-Q}
The estimate
$
\langle Q^{\epsilon,\gamma,<}(\mu^{\frac{1}{2}}g,h) , f \rangle \lesssim  |\mu^{\f38}g|_{L^{2}}|W^{\epsilon}(D)\mu^{\frac{1}{64}}h|_{L^{2}}|W^{\epsilon}(D)\mu^{\frac{1}{64}}f|_{L^{2}}
$ holds.
\end{prop}
\begin{proof}
We omit the detail  of the proof for brevity because we will use the weight exchange idea in Lemma \ref{I-less-1-some-preparation} and Proposition \ref{I-less-eta-upper-bound}
by using Lemma \ref{mu-weight-transfer}. With the weight exchange idea and the proof of \eqref{He-JSP-upper-bound}
in \cite{he2014asymptotic}, the proof for this proposition is straightforward.
\end{proof}
As a result of Proposition \ref{ubqepsilonnonsingular} and Proposition \ref{regularity-part-Q}, we have
the following theorem.
\begin{thm}\label{Q-full-up-bound} The estimate
$ |\langle Q^{\epsilon,\gamma}(\mu^{\frac{1}{2}}g,h), f\rangle| \lesssim |\mu^{\f38}g|_{L^{2}}|h|_{\epsilon,\gamma/2}|f|_{\epsilon,\gamma/2}$ holds.
\end{thm}

We turn to derive the upper bound on $\mathcal{N}^{\epsilon,\gamma}(\mu^{\frac{1}{2}}, f)$ by applying Theorem \ref{Q-full-up-bound}.
\begin{col} \label{functional-N-epsilon-gamma} The estimate
$\mathcal{N}^{\epsilon,\gamma}(\mu^{a}, f) \lesssim |f|^{2}_{\epsilon,\gamma/2}$ holds for any $a \geq \frac{1}{8}$.
\end{col}
\begin{proof} By \eqref{definition-of-functional-N} and the identity $(f^{\prime}-f)^{2} = (f^{2})^{\prime} - f^{2} - 2 f (f^{\prime}-f)$, we have
\beno
\mathcal{N}^{\epsilon,\gamma}(\mu^{a}, f) = - 2 \langle Q^{\epsilon,\gamma}(\mu^{2a}, f), f\rangle + \int B^{\epsilon} \mu^{2a}_{*} ((f^{2})^{\prime} - f^{2}) d\sigma dv_{*}dv.
\eeno
Since $2a \geq \frac{1}{4}$, then by Theorem \ref{Q-full-up-bound} and Corollary \ref{cancellation-to-Sob-norm}, the estimate follows directly.
\end{proof}

\subsubsection{Upper bound of $I^{\epsilon,\gamma,<}(g,h)$}
The weight exchange idea in Proposition \ref{regularity-part-Q} is based on the following more general result.
\begin{lem} \label{mu-weight-transfer} For $\kappa, \iota \in [0,1]$,
let $v(\kappa) := v + \kappa(v^{\prime} - v), v_{*}(\iota) := v_{*} + \iota(v^{\prime}_{*} - v_{*})$. Suppose $a,b,c,d \in \mathbb{R}$ with $a+b+c+d>0$. If $|v-v_{*}| \leq 1$, then
\beno
\mu^{a}(v) \mu^{b}(v_{*}) \mu^{c}(v(\kappa)) \mu^{d}(v_{*}(\iota)) \leq (2\pi)^{-\frac{3}{2}(a+b+c+d)} \exp (\frac{1}{2}C(a,b,c,d)),
\\
C(a,b,c,d) := \left(a(\frac{4a}{e}-1)\right)^{+}
+ \left(b(\frac{4b}{e}-1)\right)^{+} + \left(c(\frac{4c}{e}-1)\right)^{+} + \left(d(\frac{4d}{e}-1)\right)^{+}.
\eeno
with $e:=a+b+c+d, A^{+}:=\max\{A, 0\}$.
\end{lem}
\begin{proof} Since $|v-v_{*}| \leq 1$, we have $|v-v(\kappa)| \leq 1, |v-v_{*}(\iota)| \leq 1$.
We assume without loss of generality $a>0$.
Let $e=a+b+c+d$,  without loss of generality, we assume $e=1$. Otherwise, one may consider $\frac{a}{e}, \frac{b}{e}, \frac{c}{e}, \frac{d}{e}$. Now $a+b+c+d=1$ and
\beno
&& a |v|^{2} + b |v_{*}|^{2} + c |v(\kappa)|^{2} + d |v_{*}(\iota)|^{2}
\\&=& \frac{1}{4}|v|^{2} +
(\frac{1}{4}-b)|v|^{2} + b |v_{*}|^{2} + (\frac{1}{4}-c)|v|^{2} + c |v(\kappa)|^{2} + (\frac{1}{4}-d)|v|^{2} + d |v_{*}(\iota)|^{2}.
\eeno
To estimate $(\frac{1}{4}-b)|v|^{2} + b |v_{*}|^{2}$, note that
 for any $ 0 \leq  \alpha < 1$,  it holds
\ben \label{general-inequality}
|x|^{2} \geq \alpha |y|^{2}- \frac{\alpha}{1-\alpha}|x-y|^{2}.
\een
If $b \leq 0$, by taking $\alpha = -b(\frac{1}{4}-b)^{-1}$ and  \eqref{general-inequality}, we have
\beno
(\frac{1}{4}-b)|v|^{2} + b |v_{*}|^{2} \geq 4b(\frac{1}{4}-b) |v-v_{*}|^{2} \geq 4b(\frac{1}{4}-b).
\eeno
If $0< b< \frac{1}{4}$, it is obvious that
\beno
(\frac{1}{4}-b)|v|^{2} + b |v_{*}|^{2} \geq 0.
\eeno
If $b\geq \frac{1}{4}$, by taking $\alpha = (b-\frac{1}{4})b^{-1}$  and  \eqref{general-inequality}, we have
\beno
(\frac{1}{4}-b)|v|^{2} + b |v_{*}|^{2} \geq 4b(\frac{1}{4}-b) |v-v_{*}|^{2} \geq -4b(b-\frac{1}{4}).
\eeno
In summary, we get
\beno
(\frac{1}{4}-b)|v|^{2} + b |v_{*}|^{2} \geq - \left(b(4b-1)\right)^{+}.
\eeno
Since $|v-v(\kappa)| \leq 1, |v-v_{*}(\iota)| \leq 1$, similarly, we have
\beno
(\frac{1}{4}-c)|v|^{2} + c |v(\kappa)|^{2} \geq - \left(c(4c-1)\right)^{+}, \quad (\frac{1}{4}-d)|v|^{2} + d |v_{*}(\iota)|^{2} \geq - \left(d(4d-1)\right)^{+}.
\eeno
Therefore
\beno
a |v|^{2} + b |v_{*}|^{2} + c |v(\kappa)|^{2} + d |v_{*}(\iota)|^{2} \geq - \left(b(4b-1)\right)^{+} - \left(c(4c-1)\right)^{+} - \left(d(4d-1)\right)^{+}.
\eeno
In the general case when $e \neq 1$ or  the case when $a \leq 0$, we have
\beno
&& a |v|^{2} + b |v_{*}|^{2} + c |v(\kappa)|^{2} + d |v_{*}(\iota)|^{2}
\\&\geq& - \left(a(4\frac{a}{e}-1)\right)^{+}
- \left(b(4\frac{b}{e}-1)\right)^{+} - \left(c(4\frac{c}{e}-1)\right)^{+} - \left(d(4\frac{d}{e}-1)\right)^{+}.
\eeno
By noting $\mu(v) = (2\pi)^{-\frac{3}{2}} \exp(-|v|^{2}/2)$, we have the desired estimate.
\end{proof}

To keep the proof of Proposition \ref{I-less-eta-upper-bound} in a reasonable length, we prepare some estimate in advance in the following Lemma \ref{I-less-1-some-preparation}. First, define
\ben \label{definition-mathcal-X}
\mathcal{X}(G, H, F) :=  \int B^{\epsilon,\gamma,<}((\mu^{\frac{1}{2}})_{*}^{\prime} - (\mu^{\frac{1}{2}})_{*})(\mu^{-\frac{1}{16}}G)_{*} \mu^{-\frac{1}{16}}H (\mu^{-\frac{1}{16}}F)^{\prime} d\sigma dv_{*} dv.
\een
According to \eqref{I-euquals-X}, we have $\langle I^{\epsilon,\gamma,<}(g,h), f\rangle = \mathcal{X}(G, H, F)$ if we set $G = \mu^{\frac{1}{16}}g, H = \mu^{\frac{1}{16}}h, F = \mu^{\frac{1}{16}}f$. We have two decompositions on  $\mathcal{X}(G, H, F)$.
The first is
\ben \label{decomposition-1}
\mathcal{X}(G, H, F) &=& \mathcal{A}(G, H, F) + \mathcal{B}(G, H, F),
\\ \nonumber
\mathcal{A}(G, H, F) &:=&  \int B^{\epsilon,\gamma,<}((\mu^{\frac{1}{2}})_{*}^{\prime} - (\mu^{\frac{1}{2}})_{*})(\mu^{-\frac{1}{16}}G)_{*} \left(\mu^{-\frac{1}{16}}H- (\mu^{-\frac{1}{16}}H)^{\prime}\right)(\mu^{-\frac{1}{16}}F)^{\prime} d\sigma dv_{*} dv,
\\ \label{definition-mathcal-B}
\mathcal{B}(G, H, F) &:=&  \int B^{\epsilon,\gamma,<}((\mu^{\frac{1}{2}})_{*}^{\prime} - (\mu^{\frac{1}{2}})_{*})(\mu^{-\frac{1}{16}}G)_{*} (\mu^{-1/8}H F)^{\prime} d\sigma dv_{*} dv.
\een
Note that decomposition \eqref{decomposition-1} uses regularity of $H$ since $\mathcal{A}(G, H, F)$ contains $\mu^{-\frac{1}{16}}H- (\mu^{-\frac{1}{16}}H)^{\prime}$.
The second one is
\ben \label{decomposition-2}
\mathcal{X}(G, H, F) &=& \mathcal{C}(G, H, F) + \mathcal{D}(G, H, F),
\\ \nonumber
\mathcal{C}(G, H, F) &:=&  \int B^{\epsilon,\gamma,<}((\mu^{\frac{1}{2}})_{*}^{\prime} - (\mu^{\frac{1}{2}})_{*})(\mu^{-\frac{1}{16}}G)_{*} \mu^{-\frac{1}{16}}H
 \left((\mu^{-\frac{1}{16}}F)^{\prime}- \mu^{-\frac{1}{16}}F\right) d\sigma dv_{*} dv,
\\ \label{definition-mathcal-D}
\mathcal{D}(G, H, F) &:=& \int B^{\epsilon,\gamma,<}((\mu^{\frac{1}{2}})_{*}^{\prime} - (\mu^{\frac{1}{2}})_{*})(\mu^{-\frac{1}{16}}G)_{*} \mu^{-\frac{1}{16}}H  \mu^{-\frac{1}{16}}F d\sigma dv_{*} dv.
\een
Note that decomposition \eqref{decomposition-2} uses  regularity of $F$ since $\mathcal{C}(G, H, F)$ contains $(\mu^{-\frac{1}{16}}F)^{\prime}- \mu^{-\frac{1}{16}}F$.

We now give some rough estimates of $\mathcal{X}(G, H, F)$ in the following lemma. Based on this, a refined estimate will be given in Proposition \ref{I-less-eta-upper-bound}.

\begin{lem} \label{I-less-1-some-preparation}
Let $j$ be an integer satisfying $2^{j} \geq \frac{1}{4\epsilon}$, then the following estimates hold.
\ben \label{estimate-mathcal-X-1}
|\mathcal{X}(G, H, F)| &\lesssim& |G|_{L^{2}}|H|_{H^{1}}|F|_{L^{2}},
\\ \label{estimate-mathcal-X-2}
|\mathcal{X}(G, H, F)| &\lesssim& |G|_{L^{2}}|H|_{L^{2}}|F|_{H^{1}}+|G|_{L^{2}}|H|_{H^{s}}|F|_{H^{s}},
\\ \label{estimate-mathcal-X-3}
|\mathcal{X}(G, H, F)| &\lesssim& \epsilon^{2s-2} 2^{(2s-2)j} |G|_{L^{2}}|H|_{H^{1}}|F|_{L^{2}} +
\epsilon^{2s-2} 2^{(2s-1)j} |G|_{L^{2}}|H|_{L^{2}}|F|_{L^{2}}.
\een
\end{lem}
\begin{proof}
The following is divided into 3 parts.

{\it {Part 1: Proof of \eqref{estimate-mathcal-X-1}.}} We first estimate $\mathcal{A}(G, H, F)$.
By applying Taylor expansion to $(\mu^{\frac{1}{2}})_{*}^{\prime} - (\mu^{\frac{1}{2}})_{*}$ and $\mu^{-\frac{1}{16}}H- (\mu^{-\frac{1}{16}}H)^{\prime}$ up to order 1, we get
\ben \label{mu-order-1-expansion}
|(\mu^{\frac{1}{2}})^{\prime}_{*}-\mu^{\frac{1}{2}}_{*}| = |\int_{0}^{1} (\nabla \mu^{\frac{1}{2}})(v_{*}(\iota)) \cdot (v^{\prime}-v) d\iota| \lesssim
\sin(\theta/2)|v-v_{*}| \int_{0}^{1} |\mu^{\frac{1}{4}}(v_{*}(\iota))| d\iota.
\\ \nonumber
|\mu^{-\frac{1}{16}}H- (\mu^{-\frac{1}{16}}H)^{\prime}|
= |\int_{0}^{1} (\nabla \mu^{-\frac{1}{16}}H )(v(\kappa)) \cdot (v^{\prime}-v) d\kappa|
\\ \lesssim
\sin(\theta/2)|v-v_{*}| \int_{0}^{1} \mu^{-\frac{1}{8}}(v(\kappa))(|H (v(\kappa))|+|\nabla H (v(\kappa))|) d\kappa. \nonumber
\een
This implies
\ben \label{absorb-mu-weight}
|\mathcal{A}(G, H, F)| &\lesssim&
\int B^{\epsilon,\gamma,<}\sin^{2}(\theta/2)|v-v_{*}|^{2} \mu^{\frac{1}{4}}(v_{*}(\iota))|(\mu^{-\frac{1}{16}}G)_{*}|\mu^{-\frac{1}{8}}(v(\kappa))
\nonumber \\&&\times(|H (v(\kappa))|+|\nabla H (v(\kappa))|)|(\mu^{-\frac{1}{16}}F)^{\prime}| d\sigma dv_{*} dv d \iota d \kappa
\nonumber \\&\lesssim&
\int B^{\epsilon,\gamma,<}\sin^{2}(\theta/2)|v-v_{*}|^{2} |G_{*}| (|H (v(\kappa))|+|\nabla H (v(\kappa))|) |F^{\prime}| d\sigma dv_{*} dv  d \kappa,
\een
where we have used Lemma \ref{mu-weight-transfer} in the last inequality. By Cauchy-Schwarz inequality and the change of variable $(v,\theta) \rightarrow (v(\kappa),\theta(\kappa))$, we get
\ben \label{Cauchy-Schwarz-plus-change}
|\mathcal{A}(G, H, F)|  &\lesssim& \left( \int B^{\epsilon,\gamma,<}\sin^{2}(\theta/2)|v-v_{*}|^{2} |G_{*}| (|H (v(\kappa))|^{2}+|\nabla H (v(\kappa))|^{2}) d\sigma dv_{*} dv  d \kappa
\right)^{\frac{1}{2}}
\nonumber\\&& \times \left(\int B^{\epsilon,\gamma,<}\sin^{2}(\theta/2)|v-v_{*}|^{2} |G_{*}| |F^{\prime}|^{2} d\sigma dv_{*} dv
\right)^{\frac{1}{2}}
\nonumber\\&\lesssim&
\bigg(\int b^{\epsilon}(\cos\theta)\sin^{2}(\theta/2)|v(\kappa)-v_{*}|^{\gamma+2} \mathrm{1}_{|v(\kappa)-v_{*}| \leq 1} |G_{*}|
\nonumber \\&&\quad\quad \times(|H (v(\kappa))|^{2}+|\nabla H (v(\kappa))|^{2}) \sin \theta  d\theta(\kappa) d \phi dv_{*} dv(\kappa)  d \kappa
\bigg)^{\frac{1}{2}}
\nonumber\\&& \times \left(\int b^{\epsilon}(\cos\theta)\sin^{2}(\theta/2)|v^{\prime}-v_{*}|^{\gamma+2} \mathrm{1}_{|v^{\prime}-v_{*}| \leq 1} |G_{*}| |F^{\prime}|^{2} \sin \theta  d\theta^{\prime} d \phi  dv_{*} dv^{\prime}
\right)^{\frac{1}{2}}\!\!.
\een
Here $\theta(\kappa)$ is the angle between $\sigma$  and $v(\kappa)-v_{*}$ and $\theta^{\prime} = \frac{\theta}{2}$ is the angle between $\sigma$  and $v^{\prime}-v_{*}$.
By noting $b^{\epsilon}(\cos\theta) = (1-s) \epsilon^{2s-2} \sin^{-2-2s}(\theta/2) \mathrm{1}_{\sin(\theta/2) \leq \epsilon}$, and the relation \eqref{change-of-variable-angle-1}, \eqref{change-of-variable-angle-2}, we have
\ben \label{compute-angular-integral}
\int_{0}^{\pi} b^{\epsilon}(\cos\theta)\sin^{2}(\theta/2) \sin\theta
d\theta(\kappa) &\leq&  2 \int_{0}^{\pi} (1-s) \epsilon^{2s-2} \sin^{-2s}(\theta/2) \mathrm{1}_{\sin(\theta/2) \leq \epsilon} \sin\theta(\kappa)
d\theta(\kappa)
\nonumber\\&\leq&  8 \int_{0}^{\pi} (1-s) \epsilon^{2s-2} \sin^{1-2s}(\theta(\kappa)/2) \mathrm{1}_{\sin(\theta(\kappa)/2) \leq \epsilon}
d \sin (\theta(\kappa)/2)
\nonumber \\&=& 8 \int_{0}^{\epsilon} (1-s) \epsilon^{2s-2} t^{1-2s}
d t = 4.
\een
Hence.
\beno
|\mathcal{A}(G, H, F)| &\lesssim&
\left(\int |v-v_{*}|^{\gamma+2} \mathrm{1}_{|v-v_{*}| \leq 1} |G_{*}| (|H|^{2} + |\nabla H |^{2}) dv_{*} dv
\right)^{\frac{1}{2}}
\\&& \times \left(\int |v-v_{*}|^{\gamma+2} \mathrm{1}_{|v-v_{*}| \leq 1} |G_{*}| |F|^{2}  dv_{*} dv
\right)^{\frac{1}{2}}.
\eeno
Since $\gamma >-3$, for any $v \in \mathbb{R}^{3}$, we have
\ben \label{G-l2-kills-singularity}
\int |v-v_{*}|^{\gamma+2} \mathrm{1}_{|v-v_{*}| \leq 1} |G_{*}|  dv_{*} \leq |G|_{L^{2}} \left(\int |v-v_{*}|^{2\gamma+4} \mathrm{1}_{|v-v_{*}| \leq 1} dv_{*}\right)^{\frac{1}{2}} \lesssim |G|_{L^{2}},
\een
which yields
\ben \label{estimate-mathcal-A}
|\mathcal{A}(G, H, F)| \lesssim |G|_{L^{2}}|H|_{H^{1}}|F|_{L^{2}}.
\een
Similarly, we also have
\ben \label{estimate-mathcal-C}
|\mathcal{C}(G, H, F)| \lesssim |G|_{L^{2}}|H|_{L^{2}}|F|_{H^{1}}.
\een
By Taylor expansion and $v^{\prime}_{*}-v_{*} = v-  v^{\prime}$,
\ben \label{Taylor-to-mu-order-2}
 (\mu^{\frac{1}{2}})^{\prime}_{*}-\mu^{\frac{1}{2}}_{*} = (\nabla \mu^{\frac{1}{2}})_{*} \cdot (v-  v^{\prime}) + \int_{0}^{1} (1-\kappa)  (\nabla^{2}\mu^{\frac{1}{2}})(v_{*}(\kappa)):(v^\prime_{*} - v_{*})\otimes (v^\prime_{*} - v_{*}) d\kappa. \een
Plugging this into \eqref{definition-mathcal-B}, we have $\mathcal{B}(G, H, F)= \mathcal{B}_{1}(G, H, F) + \mathcal{B}_{2}(G, H, F)$, where
\beno
\mathcal{B}_{1}(G, H, F) &:=& \int B^{\epsilon,\gamma,<} (\mu^{-\frac{1}{16}}G)_{*}(\mu^{-\frac{1}{8}}H F)^{\prime} (\nabla \mu^{\frac{1}{2}})_{*} \cdot (v-  v^{\prime}) d\sigma dv_{*} dv,
\\
\mathcal{B}_{2}(G, H, F) &:=& \int B^{\epsilon,\gamma,<} (\mu^{-\frac{1}{16}}G)_{*}(\mu^{-\frac{1}{8}}H F)^{\prime}
\\&&\times\left(\int_{0}^{1} (1-\kappa)  (\nabla^{2}\mu^{\frac{1}{2}})(v_{*}(\kappa)):(v^\prime_{*} - v_{*})\otimes (v^\prime_{*} - v_{*}) d\kappa\right) d\sigma dv_{*} dv.
\eeno
Note that for fixed $v_{*}$, one has $\int B^{\epsilon,\gamma,<} (\mu^{-\frac{1}{8}}H F)^{\prime} (v-  v^{\prime}) d\sigma  dv  = 0$, which gives
$
\mathcal{B}_{1}(G, H, F) = 0.
$
By using the arguments in  \eqref{absorb-mu-weight}, \eqref{Cauchy-Schwarz-plus-change}, \eqref{compute-angular-integral} and \eqref{G-l2-kills-singularity}, we get
\ben \label{estimate-mathcal-B}
|\mathcal{B}(G, H, F)|&=&|\mathcal{B}_{2}(G, H, F)|
\nonumber \\&\lesssim& \int B^{\epsilon,\gamma,<} \sin^{2}(\theta/2)|v-v_{*}|^{2} |G_{*}H^{\prime} F^{\prime}| d\sigma dv_{*} dv \lesssim |G|_{L^{2}}|H|_{L^{2}}|F|_{L^{2}}.
\een
Combining \eqref{estimate-mathcal-A} and \eqref{estimate-mathcal-B} and noting \eqref{decomposition-1}, we have  \eqref{estimate-mathcal-X-1}.

{\it {Part 2: Proof of \eqref{estimate-mathcal-X-2}.}}
Plugging \eqref{Taylor-to-mu-order-2} into  \eqref{definition-mathcal-D} gives  $\mathcal{D}(G, H, F)= \mathcal{D}_{1}(G, H, F) + \mathcal{D}_{2}(G, H, F)$, where
\beno
\mathcal{D}_{1}(G, H, F) &:=& \int B^{\epsilon,\gamma,<} (\mu^{-\frac{1}{16}}G)_{*} \mu^{-\frac{1}{16}}H  \mu^{-\frac{1}{16}}F (\nabla \mu^{\frac{1}{2}})_{*} \cdot (v-  v^{\prime}) d\sigma dv_{*} dv,
\\
\mathcal{D}_{2}(G, H, F) &:=& \int B^{\epsilon,\gamma,<} (\mu^{-\frac{1}{16}}G)_{*} \mu^{-\frac{1}{16}}H  \mu^{-\frac{1}{16}}F  \\&&\times(1-\kappa)  (\nabla^{2}\mu^{\frac{1}{2}})(v_{*}(\kappa)):(v^\prime_{*} - v_{*})\otimes (v^\prime_{*} - v_{*}) d\kappa d\sigma dv_{*} dv.
\eeno
By the symmetry of the $\sigma$ integral and \eqref{order-2}, we have
\ben \label{symmetry-of-sigma-integral}
\int b^{\epsilon} (v- v^{\prime}) d\sigma = (v-v_{*}) \int b^{\epsilon} \sin^{2}(\theta/2) d\sigma = 4 \pi (v-v_{*}),
\een
which gives
\beno
|\mathcal{D}_{1}(G, H, F)| &=& 4\pi |\int |v-v_{*}|^{\gamma} \zeta(|v-v_{*}|) (\mu^{-\frac{1}{16}}G)_{*} \mu^{-\frac{1}{16}}H  \mu^{-\frac{1}{16}}F (\nabla \mu^{\frac{1}{2}})_{*} \cdot (v-  v_{*})  dv_{*} dv|
\\&\lesssim&
\int |v-v_{*}|^{\gamma+1} \mathrm{1}_{|v-v_{*}| \leq 1} |(\mu^{\frac{1}{16}}G)_{*} H  F| dv_{*} dv
\\&\lesssim&
\int |v-v_{*}|^{-2s} \mathrm{1}_{|v-v_{*}| \leq 1} |(\mu^{\frac{1}{16}}G)_{*} H  F| dv_{*} dv,
\eeno
where we have used  $\gamma+2s > -1$.
By Hardy inequality, we have $\int |v-v_{*}|^{-2s} H^{2} dv \lesssim |H|_{H^{s}}$ and thus
\beno
|\mathcal{D}_{1}(G, H, F)| \lesssim |G|_{L^{2}} |H|_{H^{s}} |F|_{H^{s}}.
\eeno
Similar to $\mathcal{B}_{2}(G, H, F)$, we get $|\mathcal{D}_{2}(G, H, F)| \lesssim |G|_{L^{2}} |H|_{L^{2}} |F|_{L^{2}}.$
Therefore,
\ben \label{estimate-mathcal-D}
|\mathcal{D}(G, H, F)| \lesssim |G|_{L^{2}} |H|_{H^{s}} |F|_{H^{s}}.
\een
By combining \eqref{estimate-mathcal-C} and \eqref{estimate-mathcal-D} and noting \eqref{decomposition-2}, we have \eqref{estimate-mathcal-X-2}.

{\it {Part 3: Proof of \eqref{estimate-mathcal-X-3}.}} Let $C_{j}(v) := \min \{2^{-j}|v-v_{*}|^{-1} , \epsilon \}$.
We decompose
\ben \label{A-2-leq-geq}
\mathcal{A}(G,H,F) = \mathcal{A}_{\leq  }(G,H,F) + \mathcal{A}_{\geq  }(G,H,F),
\een
where $\mathcal{A}_{\leq  }(G,H,F)$ stands for the integration over
$\sin(\theta/2) \in [0,  C_{j}(v)]$ and  $\mathcal{A}_{\geq  }(G,H,F)$ stands for the integration over
$\sin(\theta/2) \in [C_{j}(v), \epsilon]$.  For $\mathcal{A}_{\leq  }(G,H,F)$, similar to \eqref{absorb-mu-weight}, we get
\beno
|\mathcal{A}_{\leq  }(G,H,F)| &\lesssim&
\int B^{\epsilon,\gamma,<}\sin^{2}(\theta/2) \mathrm{1}_{\sin(\theta/2) \leq C_{j}(v)}|v-v_{*}|^{2} |G_{*}|
 \\&&\times(|H (v(\kappa))|+|\nabla H (v(\kappa))|) |F^{\prime}| d\sigma dv_{*} dv  d \kappa.
\eeno
By Cauchy-Schwarz inequality and the change of variable $(v,\theta) \rightarrow (v(\kappa),\theta(\kappa))$, similar to \eqref{Cauchy-Schwarz-plus-change},
we get
\ben \label{Cauchy-Schwarz-plus-change-2}
|\mathcal{A}_{\leq  }(G,H,F)|  &\lesssim&
\bigg(\int b^{\epsilon}(\cos\theta)\sin^{2}(\theta/2)\mathrm{1}_{\sin(\theta/2) \leq C_{j}(v(\kappa))}|v(\kappa)-v_{*}|^{\gamma+2} \mathrm{1}_{|v(\kappa)-v_{*}| \leq 1} |G_{*}|
\nonumber \\&&\quad\quad \times(|H (v(\kappa))|^{2}+|\nabla H (v(\kappa))|^{2}) \sin \theta  d\theta(\kappa) d \phi dv_{*} dv(\kappa)  d \kappa
\bigg)^{\frac{1}{2}}
\nonumber\\&& \times \bigg(\int b^{\epsilon}(\cos\theta)\sin^{2}(\theta/2)\mathrm{1}_{\sin(\theta/2) \leq C_{j}(v^{\prime})}|v^{\prime}-v_{*}|^{\gamma+2} \mathrm{1}_{|v^{\prime}-v_{*}| \leq 1}
\nonumber \\&&\quad\quad |G_{*}| |F^{\prime}|^{2} \sin \theta  d\theta^{\prime} d \phi  dv_{*} dv^{\prime}
\bigg)^{\frac{1}{2}},
\een
where we have used the fact $C_{j}(v) \leq C_{j}(v(\kappa))$ since $|v(\kappa)-v_{*}| \leq |v-v_{*}|$ for $\kappa \in [0,1]$. Similar to \eqref{compute-angular-integral}, we have
\beno
&&\int_{0}^{\pi} b^{\epsilon}(\cos\theta)\sin^{2}(\theta/2) \mathrm{1}_{\sin(\theta/2) \leq C_{j}(v(\kappa))} \sin\theta
d\theta(\kappa) \\&\leq&  8 \int_{0}^{2^{-j}|v(\kappa)-v_{*}|^{-1}} (1-s) \epsilon^{2s-2} t^{1-2s}
d t = 4 \times \epsilon^{2s-2} 2^{(2s-2)j} |v(\kappa)-v_{*}|^{2s-2}.
\eeno
Plugging this into \eqref{Cauchy-Schwarz-plus-change-2} gives
\beno
|\mathcal{A}_{\leq  }(G,H,F)| &\lesssim& \epsilon^{2s-2} 2^{(2s-2)j}
\left(\int |v-v_{*}|^{\gamma+2s} \mathrm{1}_{|v-v_{*}| \leq 1} |G_{*}| (|H|^{2} + |\nabla H |^{2}) dv_{*} dv
\right)^{\frac{1}{2}}
\\&& \times \left(\int |v-v_{*}|^{\gamma+2s} \mathrm{1}_{|v-v_{*}| \leq 1} |G_{*}| |F|^{2}  dv_{*} dv
\right)^{\frac{1}{2}}.
\eeno
Since $\gamma+2s >-1$, similar to \eqref{G-l2-kills-singularity},
for any $v \in \mathbb{R}^{3}$, we have
\ben \label{G-l2-kills-singularity-2}
\int |v-v_{*}|^{\gamma+2s} \mathrm{1}_{|v-v_{*}| \leq 1} |G_{*}|  dv_{*} \lesssim |G|_{L^{2}},
\een
which yields
\ben \label{estimate-mathcal-A-leq}
|\mathcal{A}_{\leq  }(G,H,F)| \lesssim \epsilon^{2s-2} 2^{(2s-2)j} |G|_{L^{2}}|H|_{H^{1}}|F|_{L^{2}}.
\een
For $\mathcal{A}_{\geq  }(G,H,F)$, by \eqref{mu-order-1-expansion},
we get
\beno
|\mathcal{A}_{\geq  }(G,H,F)| &\lesssim&
\int B^{\epsilon,\gamma,<}\mathrm{1}_{C_{j}(v) \leq \sin(\theta/2) \leq \epsilon}\sin(\theta/2)|v-v_{*}| \mu^{\frac{1}{4}}(v_{*}(\iota))|(\mu^{-\frac{1}{16}}G)_{*}
\\ && \times
\left(|\mu^{-\frac{1}{16}}H|+ |(\mu^{-\frac{1}{16}}H)^{\prime}|\right)|(\mu^{-\frac{1}{16}}F)^{\prime}|
\\&\lesssim&
\int B^{\epsilon,\gamma,<}\mathrm{1}_{C_{j}(v) \leq \sin(\theta/2) \leq \epsilon}\sin(\theta/2)|v-v_{*}| |G_{*}|\left(|H|+ |(H)^{\prime}|\right)|(F)^{\prime}|,
\eeno
where we have used  Lemma \ref{mu-weight-transfer} in the last inequality. By Cauchy-Schwarz inequality and the change of variable $(v,\theta) \rightarrow (v(\kappa),\theta(\kappa))$, similar to \eqref{Cauchy-Schwarz-plus-change},
we get
\ben \label{Cauchy-Schwarz-plus-change-3}
&& |\mathcal{A}_{\geq  }(G,H,F)| \\ &\lesssim& \bigg[
\bigg(\int b^{\epsilon}(\cos\theta)\sin(\theta/2)\mathrm{1}_{C_{j}(v) \leq \sin(\theta/2) \leq \epsilon}|v-v_{*}|^{\gamma+1} \mathrm{1}_{|v-v_{*}| \leq 1} |G_{*}| |H|^{2} \sin \theta  d\theta d \phi  dv_{*} dv \bigg)^{\frac{1}{2}}
\nonumber\\&& +\bigg(\int b^{\epsilon}(\cos\theta)\sin(\theta/2)\mathrm{1}_{C_{j}(v) \leq \sin(\theta/2) \leq \epsilon}|v^{\prime}-v_{*}|^{\gamma+1} \mathrm{1}_{|v^{\prime}-v_{*}| \leq 1} |G_{*}| |H^{\prime}|^{2} \sin \theta  d\theta^{\prime} d \phi  dv_{*} dv^{\prime} \bigg)^{\frac{1}{2}}
\bigg]
\nonumber\\&& \times \left(\int b^{\epsilon}(\cos\theta)\sin(\theta/2)\mathrm{1}_{C_{j}(v) \leq \sin(\theta/2) \leq \epsilon}|v^{\prime}-v_{*}|^{\gamma+1} \mathrm{1}_{|v^{\prime}-v_{*}| \leq 1} |G_{*}| |F^{\prime}|^{2} \sin \theta  d\theta^{\prime} d \phi  dv_{*} dv^{\prime}
\right)^{\frac{1}{2}}. \nonumber
\een
Similar to \eqref{compute-angular-integral}, we have
\ben \label{compute-angular-integral-3}
&&\int_{0}^{\pi} b^{\epsilon}(\cos\theta)\sin(\theta/2) \mathrm{1}_{C_{j}(v) \leq \sin(\theta/2) \leq \epsilon} \sin\theta
d\theta
\nonumber \\&=&  4 \int_{C_{j}(v)}^{\epsilon} (1-s) \epsilon^{2s-2} t^{-2s}
d t = 4 \frac{1-s}{2s-1} \epsilon^{2s-2} ( (C_{j}(v))^{1-2s} - \epsilon^{1-2s}  )
\nonumber \\&\leq& 4 \frac{1-s}{2s-1} \epsilon^{2s-2} 2^{(2s-1)j} |v-v_{*}|^{2s-1},
\een
where we have used $2s>1$ and  $(C_{j}(v))^{1-2s} - \epsilon^{1-2s} \leq 2^{(2s-1)j} |v-v_{*}|^{2s-1}$ in the last inequality. Note that $C_{j}(v) = \min \{2^{-j}|v-v_{*}|^{-1} , \epsilon \} \geq \{2^{-j}|v^{\prime}-v_{*}|^{-1}2^{-1/2}  , \epsilon \} :=C^{\prime}_{j}(v^{\prime})$ and  $\theta^{\prime} = \frac{\theta}{2}$. Thus
\beno
C_{j}(v) \leq \sin(\theta/2) \leq \epsilon \Rightarrow C^{\prime}_{j}(v^{\prime})  \leq
\sin (\theta^{\prime})  \leq \epsilon.
\eeno
By this, similar to \eqref{compute-angular-integral} and \eqref{compute-angular-integral-3}
we have
\ben \label{compute-angular-integral-4}
&&\int_{0}^{\pi} b^{\epsilon}(\cos\theta)\sin(\theta/2) \mathrm{1}_{C_{j}(v) \leq \sin(\theta/2) \leq \epsilon} \sin\theta
d\theta^{\prime}
\nonumber \\&\leq&  4 \int_{C^{\prime}_{j}(v^{\prime})}^{\epsilon} (1-s) \epsilon^{2s-2} t^{-2s}
d t \lesssim \frac{1-s}{2s-1}  \epsilon^{2s-2} 2^{(2s-1)j} |v^{\prime}-v_{*}|^{2s-1}.
\een
Plugging \eqref{compute-angular-integral-3} and \eqref{compute-angular-integral-4} into \eqref{Cauchy-Schwarz-plus-change-3} gives
\beno
|\mathcal{A}_{\geq  }(G,H,F)| &\lesssim&
\epsilon^{2s-2} 2^{(2s-1)j} \bigg(\int |v-v_{*}|^{\gamma+2s} \mathrm{1}_{|v-v_{*}| \leq 1} |G_{*}| |H|^{2} dv_{*} dv \bigg)^{\frac{1}{2}}
\nonumber\\&& \times \left(\int |v-v_{*}|^{\gamma+2s} \mathrm{1}_{|v-v_{*}| \leq 1} |G_{*}| |F|^{2}  dv_{*} dv
\right)^{\frac{1}{2}}.
\eeno
By \eqref{G-l2-kills-singularity-2}, we obtain
\ben \label{estimate-mathcal-A-geq}
|\mathcal{A}_{\geq  }(G,H,F)| \lesssim
\epsilon^{2s-2} 2^{(2s-1)j} |G|_{L^{2}}|H|_{L^{2}}|F|_{L^{2}}.
\een
By combining  \eqref{estimate-mathcal-A-leq}, \eqref{estimate-mathcal-A-geq} and \eqref{estimate-mathcal-B}, and by noting \eqref{A-2-leq-geq}, \eqref{decomposition-1} and $\epsilon^{2s-2} 2^{(2s-1)j} \gtrsim \epsilon^{-1} \geq 1$ because $2^{j} \gtrsim \epsilon^{-1}$ and $2s>1$,
we obtain \eqref{estimate-mathcal-X-3}.
\end{proof}

We give the estimate of $\langle I^{\epsilon,\gamma,<}(g,h), f\rangle$ in the following proposition.
\begin{prop} \label{I-less-eta-upper-bound}
For suitable functions $g, h$ and $f$, it holds
\beno \langle I^{\epsilon,\gamma,<}(g,h), f\rangle  \lesssim |\mu^{\frac{1}{16}}g|_{L^{2}}|W^{\epsilon}(D)\mu^{\frac{1}{16}}h|_{L^{2}}|W^{\epsilon}(D)\mu^{\frac{1}{16}}f|_{L^{2}}. \eeno
\end{prop}
\begin{proof}
Recall
\beno
\langle I^{\epsilon,\gamma,<}(g,h), f\rangle =  \int B^{\epsilon,\gamma,<}((\mu^{\frac{1}{2}})_{*}^{\prime} - \mu^{\frac{1}{2}}_{*})g_{*} h f^{\prime} d\sigma dv_{*} dv.
\eeno
Let $G = \mu^{\frac{1}{16}}g, H = \mu^{\frac{1}{16}}h, F = \mu^{\frac{1}{16}}f$. By   \eqref{definition-mathcal-X}, we have
\ben \label{I-euquals-X}
\langle I^{\epsilon,\gamma,<}(g,h), f\rangle =  \int B^{\epsilon,\gamma,<}((\mu^{\frac{1}{2}})_{*}^{\prime} - (\mu^{\frac{1}{2}})_{*})(\mu^{-\frac{1}{16}}G)_{*} \mu^{-\frac{1}{16}}H (\mu^{-\frac{1}{16}}F)^{\prime} d\sigma dv_{*} dv = \mathcal{X}(G, H, F).
\een
By the function $\zeta$ in \eqref{zeta-property}, define $\mathfrak{F}_{\zeta} f := \zeta(\epsilon|D|)f$ and $\mathfrak{F}^{\zeta} f := f - \zeta(\epsilon|D|)f$. We then decompose
\beno
 \mathcal{X}(G, H, F)  =  \mathcal{X}(G, \mathfrak{F}_{\zeta}H, F) +  \mathcal{X}(G, \mathfrak{F}^{\zeta}H, \mathfrak{F}_{\zeta}F)
 +  \mathcal{X}(G, \mathfrak{F}^{\zeta}H, \mathfrak{F}^{\zeta}F).
\eeno
By \eqref{estimate-mathcal-X-1} and \eqref{low-frequency-lb-cf}, we have
\ben \label{X-1-up}
|\mathcal{X}(G, \mathfrak{F}_{\zeta}H, F)|  \lesssim |G|_{L^{2}}|\mathfrak{F}_{\zeta}H|_{H^{1}}|F|_{L^{2}}
\lesssim  |G|_{L^{2}}|W^{\epsilon}(D)H|_{L^{2}}|F|_{L^{2}}.
 \een
By \eqref{estimate-mathcal-X-2}, \eqref{low-frequency-lb-cf} and \eqref{high-frequency-lb-cf}, we have
\ben \label{X-2-up}
|\mathcal{X}(G, \mathfrak{F}^{\zeta}H, \mathfrak{F}_{\zeta}F)|  &\lesssim& |G|_{L^{2}}|\mathfrak{F}^{\zeta}H|_{L^{2}}|\mathfrak{F}_{\zeta}F|_{H^{1}}+|G|_{L^{2}}
|\mathfrak{F}^{\zeta}H|_{H^{s}}|\mathfrak{F}_{\zeta}F|_{H^{s}}
\nonumber \\&\lesssim&  |G|_{L^{2}}|W^{\epsilon}(D)H|_{L^{2}}|W^{\epsilon}(D)F|_{L^{2}}.
 \een

For $\mathcal{X}(G, \mathfrak{F}^{\zeta}H, \mathfrak{F}^{\zeta}F)$, by \eqref{decomposition-1}, the dyadic decomposition \eqref{dyadic-decomposition-def},
\eqref{estimate-mathcal-X-3}, Cauchy-Schwartz inequality and \eqref{high-frequency-lb-cf}, we have
\ben \nonumber
|\mathcal{X}(G,\mathfrak{F}^{\zeta}H,\mathfrak{F}^{\zeta}F)|
&=& |\sum_{j \geq [-\log_{2}\epsilon] -2} \mathcal{X}(G,\varphi_{j}(D)\mathfrak{F}^{\zeta}H,\mathfrak{F}^{\zeta}F)|
\\ \nonumber &\lesssim&
|G|_{L^{2}}|\mathfrak{F}^{\zeta}F|_{L^{2}} \sum_{j \geq [-\log_{2}\epsilon] -2  } \epsilon^{2s-2} 2^{(2s-1)j}|\varphi_{j}(D)\mathfrak{F}^{\zeta}H|_{L^{2}}
\\ \nonumber & \lesssim& |G|_{L^{2}}|F|_{L^{2}}
\big(
\sum_{j \geq [-\log_{2}\epsilon] -2  } \epsilon^{2s-2} 2^{2sj}|\varphi_{j}(D)\mathfrak{F}^{\zeta}H|^{2}_{L^{2}}  \big)^{\frac{1}{2}}
\big(
\sum_{j \geq [-\log_{2}\epsilon] -2  } \epsilon^{2s-2}  2^{(2s-2)j} \big)^{\frac{1}{2}}
\\ \label{X-3-up} &\lesssim& |G|_{L^{2}}|W^{\epsilon}(D)H|_{L^{2}}|F|_{L^{2}}.
\een
Combining  \eqref{X-1-up}, \eqref{X-2-up} and \eqref{X-3-up} completes  the proof.
\end{proof}

\subsubsection{Upper bound of $I^{\epsilon,\gamma}$}
Combining  Proposition \ref{upforI-ep-ga-et} and  Proposition \ref{I-less-eta-upper-bound} gives
the following theorem.
\begin{thm}\label{upforI-total}
The estimate  $|\langle I^{\epsilon,\gamma}(g,h) , f \rangle|  \lesssim |g|_{L^{2}}|h|_{\epsilon,\gamma/2}|f|_{\epsilon,\gamma/2}$ holds.
\end{thm}

\subsection{Upper bound of the nonlinear term}
Furthermore, by recalling
 \eqref{Gamma-ep-ga-into-IQ} and by Theorem \ref{Q-full-up-bound} and  Theorem \ref{upforI-total}, we conclude the proof for  Theorem \ref{Gamma-full-up-bound}.

Taking $g=\mu^{\frac{1}{2}}$ in Theorem \ref{Gamma-full-up-bound} and recalling $\mathcal{L}^{\epsilon,\gamma}_{1}h =-\Gamma^{\epsilon,\gamma}(\mu^{\frac{1}{2}},h)$ in \eqref{definition-L1-L2}, we have
\begin{col}\label{upgammamuff1-full}  The estimate  $|\langle \mathcal{L}^{\epsilon,\gamma}_{1}h, f\rangle| \lesssim |h|_{\epsilon,\gamma/2}|f|_{\epsilon,\gamma/2}$ holds.
\end{col}
Recalling $\mathcal{L}^{\epsilon,\gamma}_{2}h =-\Gamma^{\epsilon,\gamma}(h,\mu^{\frac{1}{2}})$ in \eqref{definition-L1-L2}, we have the following lemma.
\begin{lem} \label{l2-full-estimate-geq-eta} The estimate  $|\langle \mathcal{L}^{\epsilon,\gamma}_{2}h , f\rangle| \lesssim |\mu^{\f18}h|_{L^{2}}|\mu^{\f18}f|_{L^{2}}$ holds.
\end{lem}
For brevity, we omit the proof of Lemma \ref{l2-full-estimate-geq-eta}. In fact, with  Corollary \ref{cancellation-to-Sob-norm},
one can refer to Lemma 2.15 in \cite{alexandre2012boltzmann} for proving Lemma \ref{l2-full-estimate-geq-eta}.

Noting $\mathcal{L}^{\epsilon,\gamma}h = \mathcal{L}^{\epsilon,\gamma}_{1}h + \mathcal{L}^{\epsilon,\gamma}_{2}h$,
by Corollary \ref{upgammamuff1-full} and Lemma \ref{l2-full-estimate-geq-eta}, we have the following lemma.
\begin{lem} \label{l-full-estimate-geq-eta}  The estimate  $|\langle \mathcal{L}^{\epsilon,\gamma}h , f\rangle| \lesssim |h|_{\epsilon,\gamma/2}|f|_{\epsilon,\gamma/2}$ holds.
\end{lem}

\subsection{Weighted upper bound of the nonlinear term}
In this subsection, we give upper bound estimate of $\Gamma^{\epsilon,\gamma}$ with weight.

We will  consider both polynomial and exponential weights together.
For $l, q \geq 0$, let
\beno
W_{l,q}(v) := \langle v \rangle^{l} \exp(q \langle v \rangle).
\eeno
Since $\langle v + u \rangle \leq \langle v \rangle + \langle u \rangle$ and $\langle v + u \rangle \leq \langle v \rangle \langle u \rangle$, we have
\ben \label{factor-out}
W_{l,q}(v+u) \leq W_{l,q}(v)W_{l,q}(u).
\een
In addition, the following estimates hold:
\beno
\nabla W_{l,q} &=& l W_{l,q} \langle v \rangle^{-2} v + q W_{l,q} \langle v \rangle^{-1} v,
\\
\nabla^{2} W_{l,q} &=& l W_{l,q} \langle v \rangle^{-2} I_{3} + q W_{l,q} \langle v \rangle^{-1} I_{3}
- 2 l W_{l,q} \langle v \rangle^{-4} v \otimes v - q W_{l,q} \langle v \rangle^{-3} v \otimes v
\\&&+ l^{2} W_{l,q} \langle v \rangle^{-4} v \otimes v + q^{2} W_{l,q} \langle v \rangle^{-2} v \otimes v + 2 l q W_{l,q} \langle v \rangle^{-3} v \otimes v.
\eeno
Hence
\ben \label{derivative-order-1}
|\nabla W_{l,q}| &\lesssim& (l+q) W_{l,q},
\\ \label{derivative-order-2}
\nabla^{2} W_{l,q} &\lesssim& (l^{2}+q^{2}+l+q) W_{l,q}.
\een

We first estimate the commutator $[Q^{\epsilon}(\mu^{\frac{1}{2}}g, \cdot), W_{l,q}]$.
\begin{lem}\label{commutatorQepsilon} Let $l, q \geq 0$.  It holds
		 \beno |\langle Q^{\epsilon}(\mu^{\frac{1}{2}}g,W_{l,q}h)-W_{l,q}Q^{\epsilon}(\mu^{\frac{1}{2}}g,h), f\rangle| \lesssim |\mu^{\frac{1}{16}}g|_{L^{2}}|W_{l,q}h|_{\epsilon,\gamma/2}|f|_{\epsilon,\gamma/2}. \eeno
\end{lem}
\begin{proof}
Note that
\beno &&\langle Q^{\epsilon}(\mu^{\frac{1}{2}}g,W_{l,q}h)-W_{l,q}Q^{\epsilon}(\mu^{\frac{1}{2}}g,h), f\rangle  = \int B^{\epsilon}(W_{l,q}-W^{\prime}_{l,q})\mu_{*}^{\frac{1}{2}}g_{*} h f^{\prime} d\sigma dv_{*} dv
\\&=& \int B^{\epsilon}(W_{l,q}-W^{\prime}_{l,q})\mu_{*}^{\frac{1}{2}}g_{*} h (f^{\prime}-f) d\sigma dv_{*} dv
 + \int B^{\epsilon}(W_{l,q}-W^{\prime}_{l,q})\mu_{*}^{\frac{1}{2}}g_{*} h f d\sigma dv_{*} dv
:=\mathcal{A}_{1} + \mathcal{A}_{2}. \eeno

{\it Step 1: Estimate of $\mathcal{A}_{1}$.}
By Cauchy-Schwartz inequality, we have
\beno |\mathcal{A}_{1}| \leq \{\int B^{\epsilon} \mu_{*}^{\frac{1}{2}}(f^{\prime}-f)^{2} d\sigma dv_{*} dv\}^{\frac{1}{2}}
\{\int B^{\epsilon}(W_{l,q}-W^{\prime}_{l,q})^{2}\mu_{*}^{\frac{1}{2}}g^{2}_{*} h^{2}  d\sigma dv_{*} dv\}^{\frac{1}{2}}
:=(\mathcal{A}_{1,1})^{\frac{1}{2}}(\mathcal{A}_{1,2})^{\frac{1}{2}}. \eeno
Note that $\mathcal{A}_{1,1} = \mathcal{N}^{\epsilon,\gamma}(\mu^{\frac{1}{4}},f) \lesssim |f|^{2}_{\epsilon,\gamma/2}$ by Corollary \ref{functional-N-epsilon-gamma}.
It remains to estimate $\mathcal{A}_{1,2}$.
By Taylor expansion,
\beno
W^{\prime}_{l,q} - W_{l,q} = \int_{0}^{1} \nabla W_{l,q}(v(\kappa)) \cdot (v^{\prime} - v) d \kappa.
\eeno
Since $|v(\kappa)| \leq |v| + |v_{*}|$, together with \eqref{derivative-order-1} and \eqref{factor-out}, we have $
|\nabla W_{l,q}(v(\kappa))| \lesssim W_{l,q}(v)W_{l,q}(v_{*}) $ and thus
\beno 
|W^{\prime}_{l,q} - W_{l,q}| \lesssim W_{l,q}(v)W_{l,q}(v_{*}) |v-v_{*}|\sin \frac{\theta}{2}.
\eeno
By \eqref{factor-out} and  $|v^{\prime}| \leq |v|+|v_{*}|$ , we also have $
|W_{l,q}-W^{\prime}_{l,q}| \lesssim W_{l,q}(v)W_{l,q}(v_{*})$. Combining the above two estimates
gives
\beno
|W^{\prime}_{l,q} - W_{l,q}|^{2} \lesssim W_{l,q}^{2}(v)W_{l,q}^{2}(v_{*}) \min\{|v-v_{*}|^{2} \sin^{2}\frac{\theta}{2},1\}.
\eeno
By this and Proposition \ref{symbol}, we obtain
\ben \nonumber
\int B^{\epsilon}(W_{l,q}-W^{\prime}_{l,q})^{2}\mu_{*}^{\frac{1}{2}} d\sigma
&\lesssim& \mathrm{1}_{|v-v_{*}| \geq 1} \langle v - v_{*}\rangle^{\gamma} W_{l,q}^{2}(v)W_{l,q}^{2}(v_{*}) (W^{\epsilon})^{2}(v-v_{*}) \mu_{*}^{\frac{1}{2}}
\\ \nonumber &&+
\mathrm{1}_{|v-v_{*}| \leq 1} |v - v_{*}|^{\gamma+2} W_{l,q}^{2}(v)W_{l,q}^{2}(v_{*}) \mu_{*}^{\frac{1}{2}}
\\ \label{sigma-integral-Wlq-square}
&\lesssim&  \mathrm{1}_{|v-v_{*}| \geq 1} \langle v\rangle^{\gamma} W_{l,q}^{2}(v) (W^{\epsilon})^{2}(v) \mu_{*}^{\frac{1}{8}} + \mathrm{1}_{|v-v_{*}| \leq 1} |v - v_{*}|^{-1} \mu^{\frac{1}{8}}\mu_{*}^{\frac{1}{8}}.
\een
Here, when $|v-v_{*}| \geq 1$, we  use $\langle v - v_{*}\rangle^{\gamma} \lesssim \langle v \rangle^{\gamma} \langle v_{*} \rangle^{|\gamma|}$ and $W^{\epsilon}(v-v_{*}) \lesssim W^{\epsilon}(v)W^{\epsilon}(v_{*})$ by \eqref{separate-into-2-cf}. When $|v-v_{*}| \leq 1$, we use Lemma \ref{mu-weight-transfer} to get $\mu_{*}^{\frac{1}{2}} \lesssim \mu^{\frac{1}{6}}\mu_{*}^{\frac{1}{6}}$. Then in both cases, the additional weights can be absorbed by the exponential decay in $\mu$. Plugging \eqref{sigma-integral-Wlq-square} into  $\mathcal{A}_{1,2}$ gives
\ben \nonumber
\mathcal{A}_{1,2} \lesssim \int \langle v\rangle^{\gamma} W_{l,q}^{2}(v) (W^{\epsilon})^{2}(v) \mu_{*}^{\frac{1}{8}}  g^{2}_{*} h^{2}  dv_{*} dv + \int  |v - v_{*}|^{-1} \mu^{\frac{1}{8}}\mu_{*}^{\frac{1}{8}} g^{2}_{*} h^{2}  dv_{*} dv
\\ \label{A-1-2-two-parts} \lesssim |\mu^{\frac{1}{16}}g|^{2}_{L^{2}} |W_{\gamma/2}W^{\epsilon} W_{l,q} h|^{2}_{L^{2}} + |\mu^{\frac{1}{16}}g|^{2}_{L^{2}} |\mu^{\frac{1}{16}}h|^{2}_{H^{\frac{1}{2}}} \lesssim |\mu^{\frac{1}{16}}g|_{L^{2}}^{2} |W_{l,q} h|_{\epsilon,\gamma/2}^{2},
\een
where we have used  Hardy inequality and $|\cdot|_{\epsilon, \gamma/2} \geq |\cdot|_{H^{1/2}_{\gamma/2}}$ because  $s \geq \frac{1}{2}$. Combining the estimates for $\mathcal{A}_{1,1}$ and $\mathcal{A}_{1,2}$,  we have
\ben \label{result-A1}
|\mathcal{A}_{1}| \lesssim |\mu^{\frac{1}{16}}g|_{L^{2}} |W_{l,q} h|_{\epsilon,\gamma/2} |f|_{\epsilon,\gamma/2}.
\een

{\it Step 2: Estimate of $\mathcal{A}_{2}$.} We want to show that
\ben \label{sigma-integral-Wlq-order-1}
|\int B^{\epsilon}(W^{\prime}_{l,q}-W_{l,q})\mu_{*}^{\frac{1}{2}} d\sigma|
\lesssim  \mathrm{1}_{|v-v_{*}| \geq 1} \langle v\rangle^{\gamma}  (W^{\epsilon})^{2}(v) W_{l,q}(v) \mu_{*}^{\frac{1}{8}} + \mathrm{1}_{|v-v_{*}| \leq 1} |v - v_{*}|^{\gamma+1} \mu^{\frac{1}{8}}\mu_{*}^{\frac{1}{8}}.
\een
By Taylor expansion, one has
\ben \label{order-2-Taylor-to-Wlq}
W^{\prime}_{l,q} - W_{l,q} = (\nabla W_{l,q})(v)\cdot(v^{\prime}-v) +\int_{0}^{1}(1-\kappa)(\nabla^{2}W_{l,q})(v(\kappa)):(v^{\prime}-v)\otimes(v^{\prime}-v)d\kappa. \een
We first consider the case $|v-v_{*}| \leq 1$.
By \eqref{order-2-Taylor-to-Wlq},  \eqref{symmetry-of-sigma-integral}, \eqref{derivative-order-1}, \eqref{derivative-order-2} and \eqref{factor-out} and Lemma \ref{mu-weight-transfer} with $|v-v_{*}| \leq 1$,
we have
\ben \nonumber
&&|\int B^{\epsilon}(W^{\prime}_{l,q}-W_{l,q})\mu_{*}^{\frac{1}{2}} d\sigma|
\\ \nonumber &\lesssim&
|\int B^{\epsilon}(\nabla W_{l,q})(v)\cdot(v^{\prime}-v)\mu_{*}^{\frac{1}{2}} d\sigma| +
\int B^{\epsilon}|(\nabla^{2}W_{l,q})(v(\kappa))| |v^{\prime}-v|^{2}\mu_{*}^{\frac{1}{2}} d\kappa d\sigma
\\ \label{case-leq-1} &\lesssim& |v-v_{*}|^{\gamma+1} \mu_{*}^{\frac{1}{2}} W_{l,q}(v) W_{l,q}(v_{*}) \lesssim |v - v_{*}|^{\gamma+1} \mu^{\frac{1}{8}}\mu_{*}^{\frac{1}{8}}.
\een
We next consider the case $|v-v_{*}| \geq 1$. 
Similar to \eqref{case-leq-1}, since
 $|v-v_{*}| \sim \langle v-v_{*}\rangle$, we have
\ben \label{case-geq-1-general}
|\int B^{\epsilon}(W^{\prime}_{l,q}-W_{l,q})\mu_{*}^{\frac{1}{2}} d\sigma|
\lesssim \langle v-v_{*}\rangle^{\gamma+2} \mu_{*}^{\frac{1}{2}} W_{l,q}(v) W_{l,q}(v_{*}).
\een
If  $|v| \leq \epsilon^{-1}$, then $W^{\epsilon}(v) \gtrsim \langle v \rangle$. By \eqref{case-geq-1-general},
we have directly
\ben \label{case-geq-1-general-sub1}
|\int B^{\epsilon}(W^{\prime}_{l,q}-W_{l,q})\mu_{*}^{\frac{1}{2}} d\sigma|
\lesssim \langle v \rangle^{\gamma+2} \mu_{*}^{\frac{1}{8}} W_{l,q}(v) \lesssim \langle v \rangle^{\gamma} (W^{\epsilon})^{2}(v) W_{l,q}(v) \mu_{*}^{\frac{1}{8}}.
\een
If  $|v| > \epsilon^{-1}, |v-v_{*}| \leq \epsilon^{-1}$, then $W^{\epsilon}(v) \gtrsim \epsilon^{s-1}\langle v \rangle^{s}$. By \eqref{case-geq-1-general}, we  have
\ben \label{case-geq-1-general-sub2}
|\int B^{\epsilon}(W^{\prime}_{l,q}-W_{l,q})\mu_{*}^{\frac{1}{2}} d\sigma|
\lesssim \epsilon^{2s-2} \langle v -v_{*} \rangle^{\gamma+2s} \mu_{*}^{\frac{1}{4}} W_{l,q}(v) \lesssim \langle v \rangle^{\gamma} (W^{\epsilon})^{2}(v) W_{l,q}(v) \mu_{*}^{\frac{1}{8}}.
\een
It remains to consider the last case  $|v| > \epsilon^{-1}, |v-v_{*}| \geq \epsilon^{-1}$. We divide the angle $\theta$ into two parts:
\beno
\int B^{\epsilon}(W^{\prime}_{l,q}-W_{l,q})\mu_{*}^{\frac{1}{2}} d\sigma = \mathcal{B}_{1} + \mathcal{B}_{2},
\\
\mathcal{B}_{1} := \int B^{\epsilon}\mathrm{1}_{\sin\frac{\theta}{2} \leq |v-v_{*}|^{-1}}(W^{\prime}_{l,q}-W_{l,q})\mu_{*}^{\frac{1}{2}} d\sigma,
\\ \mathcal{B}_{2} := \int B^{\epsilon} \mathrm{1}_{\sin\frac{\theta}{2} \geq |v-v_{*}|^{-1}} (W^{\prime}_{l,q}-W_{l,q})\mu_{*}^{\frac{1}{2}} d\sigma.
\eeno
For $\mathcal{B}_{1}$,
by using the expansion \eqref{order-2-Taylor-to-Wlq}, similar to \eqref{case-leq-1},  since $\int B^{\epsilon} \mathrm{1}_{\sin\frac{\theta}{2} \leq |v-v_{*}|^{-1}} \sin^{2}\frac{\theta}{2} d\sigma \lesssim \epsilon^{2s-2} |v -v_{*}|^{\gamma+2s-2}$,
we have
\beno
|\mathcal{B}_{1}| \lesssim \epsilon^{2s-2}  |v -v_{*}|^{\gamma+2s} \mu_{*}^{\frac{1}{2}} W_{l,q}(v) W_{l,q}(v_{*}).
\eeno
For $\mathcal{B}_{2}$, by  the fact that  $
|W_{l,q}-W^{\prime}_{l,q}| \lesssim W_{l,q}(v)W_{l,q}(v_{*})$, since $\int B^{\epsilon} \mathrm{1}_{\sin\frac{\theta}{2} \geq |v-v_{*}|^{-1}} d\sigma \lesssim \epsilon^{2s-2} |v -v_{*}|^{\gamma+2s}$,
we have
\beno
|\mathcal{B}_{2}| \lesssim \epsilon^{2s-2}  |v -v_{*}|^{\gamma+2s} \mu_{*}^{\frac{1}{2}} W_{l,q}(v) W_{l,q}(v_{*}).
\eeno
Similar to \eqref{case-geq-1-general-sub2},  the estimates  $\mathcal{B}_{1}$ and $\mathcal{B}_{2}$ give
\ben \label{case-geq-1-general-sub3}
|\int B^{\epsilon}(W^{\prime}_{l,q}-W_{l,q})\mu_{*}^{\frac{1}{2}} d\sigma|
\lesssim \epsilon^{2s-2}  |v -v_{*}|^{\gamma+2s} \mu_{*}^{\frac{1}{2}} W_{l,q}(v) W_{l,q}(v_{*}) \lesssim \langle v \rangle^{\gamma} (W^{\epsilon})^{2}(v) W_{l,q}(v) \mu_{*}^{\frac{1}{8}}.
\een
By combining \eqref{case-leq-1},  \eqref{case-geq-1-general-sub1}, \eqref{case-geq-1-general-sub2} and \eqref{case-geq-1-general-sub3}, we have \eqref{sigma-integral-Wlq-order-1}. Then \eqref{sigma-integral-Wlq-order-1} implies
\beno
|\mathcal{A}_{2}| \lesssim \int  \langle v\rangle^{\gamma}  (W^{\epsilon})^{2}(v) W_{l,q}(v) \mu_{*}^{\frac{1}{8}} |g_{*} h f| dv_{*} dv + \int \mathrm{1}_{|v - v_{*}| \leq 1} |v - v_{*}|^{\gamma+1} \mu^{\frac{1}{8}}\mu_{*}^{\frac{1}{8}} |g_{*} h f| dv_{*} dv.
\eeno
Obviously,
\beno
 &&\int  \langle v\rangle^{\gamma}  (W^{\epsilon})^{2}(v) W_{l,q}(v) \mu_{*}^{\frac{1}{8}} |g_{*} h f| dv_{*} dv
  \\&\lesssim& |\mu^{\frac{1}{8}} g|_{L^{1}} |W^{\epsilon} W_{l,q} h|_{L^{2}_{\gamma/2}} |W^{\epsilon} f|_{L^{2}_{\gamma/2}} \lesssim |\mu^{\frac{1}{16}}g|_{L^{2}} |W_{l,q} h|_{\epsilon,\gamma/2} |f|_{\epsilon,\gamma/2}.
\eeno
By Cauchy-Schwartz inequality, when $\gamma > -3$, similar to \eqref{A-1-2-two-parts}, we have
\beno
&& \int |v - v_{*}|^{\gamma+1} \mu^{\frac{1}{8}}\mu_{*}^{\frac{1}{8}} |g_{*} h f| dv_{*} dv
\\ &\lesssim&
(\int |v - v_{*}|^{-1} \mu^{\frac{1}{8}}\mu_{*}^{\frac{1}{8}} |g_{*} h|^{2} dv_{*} dv)^{\frac{1}{2}} (\int |v - v_{*}|^{\gamma} \mu^{\frac{1}{8}}\mu_{*}^{\frac{1}{8}}  f^{2} dv_{*} dv)^{\frac{1}{2}}
\lesssim |\mu^{\frac{1}{16}}g|_{L^{2}} |W_{l,q} h|_{\epsilon,\gamma/2} |f|_{\epsilon,\gamma/2}.
\eeno
Then the above two estimates give
\ben \label{result-A2}
|\mathcal{A}_{2}| \lesssim |\mu^{\frac{1}{16}}g|_{L^{2}} |W_{l,q} h|_{\epsilon,\gamma/2} |f|_{\epsilon,\gamma/2}.
\een
Finally,  \eqref{result-A1} and \eqref{result-A2} complete the proof.
\end{proof}
The next lemma gives estimate of the commutator $[I^{\epsilon,\gamma}(g, \cdot), W_{l,q}] $.
\begin{lem}\label{commutatorforI} Let $l, q \geq 0$. If $-2 \leq \gamma \leq 0$, it holds
		\beno
		|\langle [I^{\epsilon,\gamma}(g, \cdot), W_{l,q}]h, f \rangle
| \lesssim |\mu^{1/32}g|_{L^{2}} |W_{l,q} h|_{\epsilon,\gamma/2}|W^{\epsilon}f|_{L^{2}_{\gamma/2}} + |W_{l,q} g|_{L^{2}}|W_{l,q} h|_{L^{2}_{\gamma/2}}|W^{\epsilon}f|_{L^{2}_{\gamma/2}}.
		\eeno	
If $q =0$, it holds
\beno
		|\langle [I^{\epsilon,\gamma}(g, \cdot), W_{l,0}]h, f \rangle
| \lesssim |g|_{L^{2}} |W_{l,0} h|_{\epsilon,\gamma/2}|W^{\epsilon}f|_{L^{2}_{\gamma/2}}.
		\eeno
\end{lem}
\begin{proof}
By recalling $I^{\epsilon,\gamma}$ in \eqref{I-ep-ga-sep-eta}, the structure \eqref{I-ep-ga-geq-eta}, and the identity
$(\mu^{\frac{1}{2}})_{*}^{\prime} - \mu^{\frac{1}{2}}_{*} =((\mu^{\frac{1}{4}})_{*}^{\prime} - \mu^{\frac{1}{4}}_{*})^{2}+2\mu^{\frac{1}{4}}_{*}((\mu^{\frac{1}{4}})_{*}^{\prime} - \mu^{\frac{1}{4}}_{*})$, we have
\beno
\langle [I^{\epsilon,\gamma}(g, \cdot), W_{l,q}]h, f \rangle &=& \int B^{\epsilon}((\mu^{\frac{1}{2}})_{*}^{\prime} - \mu^{\frac{1}{2}}_{*}) (W_{l,q}-W^{\prime}_{l,q})g_{*} h f^{\prime} d\sigma dv_{*} dv
\\&=&  \int B^{\epsilon}((\mu^{\frac{1}{4}})_{*}^{\prime} - \mu^{\frac{1}{4}}_{*})^{2}(W_{l,q}-W^{\prime}_{l,q})g_{*} h f^{\prime} d\sigma dv_{*} dv
\\&&
 + 2 \int B^{\epsilon}\mu^{\frac{1}{4}}_{*}((\mu^{\frac{1}{4}})_{*}^{\prime} - \mu^{\frac{1}{4}}_{*})(W_{l,q}-W^{\prime}_{l,q})g_{*} h f^{\prime} d\sigma dv_{*} dv
:= \mathcal{A}_{1} + 2\mathcal{A}_{2}.
\eeno

{\it Step 1: Estimate of $\mathcal{A}_{1}$.}
By Cauchy-Schwartz inequality, we have
\beno |\mathcal{A}_{1}| &\leq& \{\int B^{\epsilon} ((\mu^{\frac{1}{4}})^{\prime}_{*} - \mu^{\frac{1}{4}}_{*})^{2} f^{\prime 2} d\sigma dv_{*} dv\}^{\frac{1}{2}}
\\&&\times\{\int B^{\epsilon}((\mu^{\frac{1}{4}})^{\prime}_{*} - \mu^{\frac{1}{4}}_{*})^{2}(W_{l,q}-W^{\prime}_{l,q})^{2}g^{2}_{*} h^{2}  d\sigma dv_{*} dv\}^{\frac{1}{2}}
:=(\mathcal{A}_{1,1})^{\frac{1}{2}}(\mathcal{A}_{1,2})^{\frac{1}{2}}. \eeno
By the change of variables $(v,v_{*}) \rightarrow (v_{*}^{\prime},v^{\prime})$ and Remark \ref{also-hold-for-a}, we have
\ben \label{estimate-of-A11} \mathcal{A}_{1,1} = \int B^{\epsilon} ((\mu^{\frac{1}{4}})^{\prime} - \mu^{\frac{1}{4}})^{2} f^{2}_{*} d\sigma dv_{*} dv
\lesssim |W^{\epsilon}f|^{2}_{L^{2}_{\gamma/2}}. \een
Since  $((\mu^{\frac{1}{4}})^{\prime}_{*} - \mu^{\frac{1}{4}}_{*})^{2} = ((\mu^{\frac{1}{8}})_{*}^{\prime}  + \mu^{\frac{1}{8}}_{*})^{2}((\mu^{\frac{1}{8}})_{*}^{\prime}  - \mu^{\frac{1}{8}}_{*})^{2} \leq 2 ((\mu^{\frac{1}{4}})^{\prime}_{*} + \mu^{\frac{1}{4}}_{*})((\mu^{\frac{1}{8}})_{*}^{\prime}  - \mu^{\frac{1}{8}}_{*})^{2}$, we have
\beno \mathcal{A}_{1,2} &\lesssim& \int B^{\epsilon}\mu^{\frac{1}{4}}_{*}((\mu^{\frac{1}{8}})_{*}^{\prime}  - \mu^{\frac{1}{8}}_{*})^{2}(W_{l,q}-W^{\prime}_{l,q})^{2}g^{2}_{*} h^{2}  d\sigma dv_{*} dv \\&&+ \int B^{\epsilon}(\mu^{\frac{1}{4}})^{\prime}_{*}((\mu^{\frac{1}{8}})^{\prime}_{*} - \mu^{\frac{1}{8}}_{*})^{2}(W_{l,q}-W^{\prime}_{l,q})^{2}g^{2}_{*} h^{2}  d\sigma dv_{*} dv
:= \mathcal{A}_{1,2,1} + \mathcal{A}_{1,2,2}.\eeno
We first estimate $\mathcal{A}_{1,2,2}$.
Referring to \cite{duan2013stability}(more  precisely, eq.(2.10) on page 170), we get
\ben \label{weight-to-vstarprime}
|W_{l,q}-W^{\prime}_{l,q}| \lesssim |v^{\prime} -v| \langle v \rangle^{-1} \langle v_{*}^{\prime} \rangle^{2}
W_{l,q}(v) W_{l,q}(v_{*}).
\een
This together with the assumption that $\gamma \geq -2$ give
\beno
&& \int B^{\epsilon}(\mu^{\frac{1}{4}})^{\prime}_{*}((\mu^{\frac{1}{8}})^{\prime}_{*} - \mu^{\frac{1}{8}}_{*})^{2}(W_{l,q}-W^{\prime}_{l,q})^{2}  d\sigma
\\&\lesssim&  \langle v \rangle^{-2} W^{2}_{l,q}(v) W^{2}_{l,q}(v_{*}) \int |v-v_{*}^{\prime}|^{\gamma+2} (\mu^{\frac{1}{4}})^{\prime}_{*} \langle v_{*}^{\prime} \rangle^{2} b^{\epsilon}(\cos\theta) \sin^{2} \frac{\theta}{2}  d\sigma
\lesssim \langle v \rangle^{\gamma} W^{2}_{l,q}(v) W^{2}_{l,q}(v_{*})
\eeno
so that $ \mathcal{A}_{1,2,2} \lesssim |W_{l,q} g|^{2}_{L^{2}}|W_{l,q} h|^{2}_{L^{2}_{\gamma/2}}.$ If $q=0$, by the proof of (2.84) in \cite{he2018asymptotic}, we have
\beno
\int B^{\epsilon}(\mu^{\frac{1}{4}})^{\prime}_{*}((\mu^{\frac{1}{8}})^{\prime}_{*} - \mu^{\frac{1}{8}}_{*})^{2}(W_{l,0}-W^{\prime}_{l,0})^{2} d\sigma \lesssim \langle v \rangle^{2l+\gamma},
\eeno
which implies $ \mathcal{A}_{1,2,2} \lesssim |g|^{2}_{L^{2}}|W_{l,0} h|^{2}_{L^{2}_{\gamma/2}}.$

Similar to the estimate of $\mathcal{A}_{1,2}$ in
{\it{Step 1}} of  Lemma \ref{commutatorQepsilon}, we obtain
$\mathcal{A}_{1,2,1} \lesssim |\mu^{\frac{1}{16}}g|^{2}_{L^{2}} |W_{l,q} h|^{2}_{\epsilon,\gamma/2}$.

In summary,  the estimates of $\mathcal{A}_{1,2,1}$ and $\mathcal{A}_{1,2,2}$ imply that
 $\mathcal{A}_{1,2} \lesssim |W_{l,q} g|^{2}_{L^{2}}|W_{l,q} h|^{2}_{L^{2}_{\gamma/2}}$ when $q>0$, and
 $\mathcal{A}_{1,2} \lesssim |g|^{2}_{L^{2}}|W_{l,0} h|^{2}_{L^{2}_{\gamma/2}}$ when $q=0$.
Hence,  we conclude that when $q>0$,
\beno |\mathcal{A}_{1}| \lesssim |W_{l,q} g|_{L^{2}}|W_{l,q} h|_{L^{2}_{\gamma/2}}|W^{\epsilon}f|_{L^{2}_{\gamma/2}}.\eeno
When $q=0$,
\beno |\mathcal{A}_{1}| \lesssim |g|_{L^{2}}|W_{l,0} h|_{L^{2}_{\gamma/2}}|W^{\epsilon}f|_{L^{2}_{\gamma/2}}.\eeno

{\it Step 2: Estimate of $\mathcal{A}_{2}$.} By Cauchy-Schwartz inequality, we have
\beno |\mathcal{A}_{2}| &\leq& \{\int B^{\epsilon}((\mu^{\frac{1}{4}})^{\prime}_{*} - \mu^{\frac{1}{4}}_{*})^{2} f^{\prime 2} d\sigma dv_{*} dv\}^{\frac{1}{2}}
\\&&\times\{\int B^{\epsilon}\mu^{\frac{1}{2}}_{*}(W_{l,q}-W^{\prime}_{l,q})^{2} g^{2}_{*} h^{2}  d\sigma dv_{*} dv\}^{\frac{1}{2}}
:=(\mathcal{A}_{2,1})^{\frac{1}{2}}(\mathcal{A}_{2,2})^{\frac{1}{2}}. \eeno
Note that $\mathcal{A}_{2,1}=\mathcal{A}_{1,1}$, then by \eqref{estimate-of-A11}, we have  $\mathcal{A}_{2,1}  \lesssim |W^\epsilon f|_{L^2_{\gamma/2}}^2.$ Similar to the estimate of $\mathcal{A}_{1,2}$ in
{\it{Step 1}} of  Lemma \ref{commutatorQepsilon}, we get
\beno \mathcal{A}_{2,2} \lesssim |\mu^{\frac{1}{16}}g|^{2}_{L^{2}} |W_{l,q} h|^{2}_{\epsilon,\gamma/2}.\eeno
The estimates of $\mathcal{A}_{2,1}$ and $\mathcal{A}_{2,2}$ give
\beno |\mathcal{A}_{2}| \lesssim |\mu^{\frac{1}{16}}g|_{L^{2}} |W_{l,q} h|_{\epsilon,\gamma/2}|W^{\epsilon}f|_{L^{2}_{\gamma/2}}.\eeno
Then the proof of the lemma is completed  by combining the estimates of $\mathcal{A}_{1}$ and $\mathcal{A}_{2}$.
\end{proof}

Recalling \eqref{Gamma-ep-ga-into-IQ}, the  following lemma  is a direct consequence of Lemma \ref{commutatorQepsilon} and Lemma \ref{commutatorforI}.
\begin{lem}\label{commutatorgamma}
Let $l, q \geq 0$. If $-2 \leq \gamma \leq 0$, then
\beno |\langle \Gamma^{\epsilon}(g,W_{l,q}h)-W_{l,q} \Gamma^{\epsilon}(g,h), f\rangle| \lesssim |\mu^{1/32}g|_{L^{2}}|W_{l,q}h|_{\epsilon,\gamma/2}|f|_{\epsilon,\gamma/2}
+ |W_{l,q} g|_{L^{2}}|W_{l,q} h|_{L^{2}_{\gamma/2}}|W^{\epsilon}f|_{L^{2}_{\gamma/2}}.\eeno
If $q = 0$,  then
\beno |\langle \Gamma^{\epsilon}(g,W_{l,0}h)-W_{l,0} \Gamma^{\epsilon}(g,h), f\rangle| \lesssim |g|_{L^{2}}|W_{l,0}h|_{\epsilon,\gamma/2}|f|_{\epsilon,\gamma/2}.\eeno
\end{lem}

Then  Lemma \ref{commutatorgamma} and Theorem \ref{Gamma-full-up-bound} give the following
lemma.
\begin{lem}\label{general-weight-upper-bound}
Let $l, q \geq 0$. If $-2 \leq \gamma \leq 0$, then
\beno |\langle \Gamma^{\epsilon}(g, h), W_{l,q}^{2} f\rangle| \lesssim |W_{l,q} g|_{L^{2}}|W_{l,q}h|_{\epsilon,\gamma/2}|W_{l,q} f|_{\epsilon,\gamma/2}. \eeno
If $q = 0$,  then
\ben \label{polynomial-commutator-nonlinear}
|\langle \Gamma^{\epsilon}(g, h), W_{l,0}^{2} f\rangle| \lesssim |g|_{L^{2}}|W_{l,0}h|_{\epsilon,\gamma/2}|W_{l,0}f|_{\epsilon,\gamma/2}.\een
\end{lem}

Recalling $\mathcal{L}^{\epsilon,\gamma}h = \mathcal{L}^{\epsilon,\gamma}_{1}h + \mathcal{L}^{\epsilon,\gamma}_{2}h = -\Gamma^{\epsilon}(\mu^{\frac{1}{2}}, h) + \mathcal{L}^{\epsilon,\gamma}_{2}h$,
the following lemma follows from  \eqref{polynomial-commutator-nonlinear} and Lemma \ref{l2-full-estimate-geq-eta}.
\begin{lem}\label{commutator-linear-g-h}
Let $l \geq 0$. The estimate $|\langle \mathcal{L}^{\epsilon,\gamma}h, W_{l,0}^{2}f\rangle| \lesssim |W_{l,0}h|_{\epsilon,\gamma/2}|W_{l,0}f|_{\epsilon,\gamma/2}$ holds.
\end{lem}

In the following, we give an estimate of the commutator between $\mathcal{L}^{\epsilon,\gamma}$ and $W_{l,q}$ as a special case.
\begin{lem}\label{commutator-linear}
Let $l, q \geq 0, -2 \leq \gamma \leq 0$. The estimate $|\langle [\mathcal{L}^{\epsilon,\gamma}, W_{l,q}]f, W_{l,q}f\rangle| \lesssim |W_{l,q}f|_{L^{2}_{\gamma/2}}^{2}$ holds.
\end{lem}
\begin{proof} Recall $\mathcal{L}^{\epsilon,\gamma} = \mathcal{L}^{\epsilon,\gamma}_{1} + \mathcal{L}^{\epsilon,\gamma}_{2}$, where $\mathcal{L}^{\epsilon,\gamma}_{1}f=  - \Gamma^{\epsilon,\gamma}(\mu^{\frac{1}{2}},f), \mathcal{L}^{\epsilon,\gamma}_{2}f= - \Gamma^{\epsilon,\gamma}(f,\mu^{\frac{1}{2}})$. Direct computation gives
\beno
\langle [\mathcal{L}^{\epsilon,\gamma}_{1}, W_{l,q}]f, W_{l,q}f\rangle &=& \int B^{\epsilon} \mu^{\frac{1}{2}}_{*}
(\mu^{\frac{1}{2}})^{\prime}_{*} f f^{\prime} W_{l,q}( W_{l,q} - W_{l,q}^{\prime})  d v d v_{*} d\sigma
\\&=& \frac{1}{2} \int B^{\epsilon} \mu^{\frac{1}{2}}_{*}
(\mu^{\frac{1}{2}})^{\prime}_{*} f f^{\prime} ( W_{l,q} - W_{l,q}^{\prime})^{2}  d v d v_{*} d\sigma.
\eeno
By using
change of variable $(v,v_{*}) \rightarrow (v^{\prime},v^{\prime}_{*})$, one has
\beno
|\langle [\mathcal{L}^{\epsilon,\gamma}_{1}, W_{l,q}]f, W_{l,q}f\rangle| \leq  \frac{1}{2} \int B^{\epsilon} \mu^{\frac{1}{2}}_{*}
(\mu^{\frac{1}{2}})^{\prime}_{*} f^{2} ( W_{l,q} - W_{l,q}^{\prime})^{2}  d v d v_{*} d\sigma.
\eeno
Then \eqref{weight-to-vstarprime} gives
\beno
\int B^{\epsilon} \mu^{\frac{1}{2}}_{*}
(\mu^{\frac{1}{2}})^{\prime}_{*} ( W_{l,q} - W_{l,q}^{\prime})^{2}  d\sigma \lesssim
|v -v_{*}|^{\gamma+2} \langle v \rangle^{-2}
W_{l,q}^{2}(v) \mu^{\frac{1}{2}}_{*}
\int b^{\epsilon} (\cos\theta) \sin^{2}(\theta/2) d\sigma,
\eeno
which implies
\beno
|\langle [\mathcal{L}^{\epsilon,\gamma}_{1}, W_{l,q}]f, W_{l,q}f\rangle| \lesssim   \int |v -v_{*}|^{\gamma+2} \langle v \rangle^{-2}
W_{l,q}^{2}(v) \mu^{\frac{1}{2}}_{*} f^{2}   d v d v_{*}
 \lesssim |W_{l,q}f|^{2}_{L^{2}_{\gamma/2}}.
\eeno
By Lemma \ref{l2-full-estimate-geq-eta}, we have
\beno
|\langle [\mathcal{L}^{\epsilon,\gamma}_{2}, W_{l,q}]f, W_{l,q}f\rangle|
 \lesssim |\mu^{\frac{1}{16}}f|^{2}_{L^{2}}.
\eeno
Then the above two estimates complete  the proof.
\end{proof}

\section{Propagation of regularity and asymptotic formula} \label{propagation-asymptotic}
In this section, we will give the proof to Theorem \ref{asymptotic-result}.  With coercivity estimate in Theorem \ref{coercivity-structure}, spectral gap estimate in Theorem \ref{micro-dissipation}, upper bound estimate in Theorem \ref{Gamma-full-up-bound}, we can derive the global well-posedness result \eqref{lowest-regularity-bounded-by-initial} in Theorem \ref{asymptotic-result} as in \cite{duan2021global}.

Then it remains to show the  propagation of regularity and asymptotic formula. We will derive propagation of regularity in section \ref{propagation} and
asymptotic formula in section \ref{asymptotic}.

\subsection{Propagation of regularity}\label{propagation}
In this subsection, we prove \eqref{general-propagation-f-epsilon} as stated in Theorem \ref{high-order-propagation}. We recall  \eqref{Sobolev-norm}, \eqref{coercivity-norm}, \eqref{general-L1k-norm}, \eqref{general-L1k-norm-for-initial-data},  \eqref{energy-and-dissipation-for-propagation} and \eqref{initial-dependence-for-propagation}
about  the norms used.

We first consider propagation of spatial regularity with polynomial weight, i.e. the norm $\|\cdot\|_{L^{1}_{k,m}L^{2}_{l}}$.
\begin{thm} \label{lowest-order-propagation}
Let $m, l \geq 0$. Suppose $f$ is a solution to the Boltzmann equation \eqref{Cauchy-linearizedBE-grazing} with initial data $f_{0}$ satisfying
$ \|f_{0}\|_{L^{1}_{k,m}L^{2}_{l}} < \infty.
$
There is a constant $\delta>0$ such that if
$
\|f_{0}\|_{L^{1}_{k}L^{2}} < \delta,
$
then
\ben \label{a-priori-q-weight-x-derivative}
\|f\|_{L^{1}_{k,m}L^{\infty}_{T}L^{2}_{l}} + \|f\|_{L^{1}_{k,m}L^{2}_{T}L^{2}_{\epsilon,l+\gamma/2}} \lesssim
 \|f_{0}\|_{L^{1}_{k,m}L^{2}_{l}}(1 +\|f_{0}\|_{L^{1}_{k,m}L^{2}_{l}}).
\een
\end{thm}

Note that we will show propagation of norm $\|\cdot\|_{L^{1}_{k,m}L^{2}_{l}}$ only under smallness assumption on $\|f_{0}\|_{L^{1}_{k}L^{2}}$ and finiteness on $\|f_{0}\|_{L^{1}_{k,m}L^{2}_{l}}$.
\begin{proof} Consider
\ben \label{Boltzmann-equation-epsilon}
 \partial_{t}f + v \cdot \nabla_{x} f + \mathcal{L}^{\epsilon}f = \Gamma^{\epsilon}(f,f), \een
with the initial condition $f_{0}$.
For simplicity, denote $\mathcal{H}:=\Gamma^{\epsilon}(f,f)$.

Recall \eqref{lowest-regularity-bounded-by-initial} as
\ben \label{a-priori-no-weight-no-derivative}
\|f\|_{L^{1}_{k}L^{\infty}_{T}L^{2}} + \|f\|_{L^{1}_{k}L^{2}_{T}L^{2}_{\epsilon,\gamma/2}}
\lesssim \|f_{0}\|_{L^{1}_{k}L^{2}}.
\een
The proof of the theorema is divided into three steps where $\|\cdot\|_{L^{1}_{k,m}L^{2}}, \|\cdot\|_{L^{1}_{k}L^{2}_{l}}, \|\cdot\|_{L^{1}_{k,m}L^{2}_{l}}$ are considered respectively.

{\it{Step 1: $\|\cdot\|_{L^{1}_{k,m}L^{2}}.$}}
Following the proof of  Theorem 5.1 in \cite{duan2021global}, we have
\ben
\|[a,b,c]\|_{L^{1}_{k,m}L^{2}_{T}} &\lesssim& \|(\mathbb{I}-\mathbb{P})f\|_{L^{1}_{k,m}L^{2}_{T}L^{2}_{\epsilon,\gamma/2}} + \|f\|_{L^{1}_{k,m}L^{\infty}_{T}L^{2}}
+ \|f_{0}\|_{L^{1}_{k,m}L^{2}}
\nonumber \\&&+ \sum_{j}\sum_{k \in \mathbb{Z}^{3}} \left( \int_{0}^{T} |\langle \widehat{\mathcal{H}} \langle k \rangle^{m},  P_{j}\mu^{\frac{1}{2}} \rangle|^{2} d t \right)^{\frac{1}{2}}, \label{macro-no-weight-x-derivative}
\een
where $\{P_{j}\}_{j}$ is a set of polynomials with degree $\leq 4$.

Taking Fourier transform of \eqref{Boltzmann-equation-epsilon} with respect to $x$, at mode $k \in \mathbb{Z}^{3}$, we have
\ben \label{equation-after-Fourier}
 \partial_{t} \widehat{f}(k) + \mathrm{i} v \cdot k  \widehat{f}(k) + \mathcal{L}^{\epsilon}  \widehat{f}(k)= \widehat{\mathcal{H}}(k). \een
Taking inner product with $\langle k \rangle^{2m} \widehat{f}$, similar to Eq.(3.7) in
\cite{duan2021global}, we have
\ben \label{micro-no-weight-x-derivative}
\|f\|_{L^{1}_{k,m}L^{\infty}_{T}L^{2}} + \|(\mathbb{I}-\mathbb{P})f\|_{L^{1}_{k,m}L^{2}_{T}L^{2}_{\epsilon,\gamma/2}} \lesssim
\|f_{0}\|_{L^{1}_{k,m}L^{2}} + \sum_{k \in \mathbb{Z}^{3}} \left( \int_{0}^{T} |\langle \widehat{\mathcal{H}}(k) \langle k \rangle^{m}, \langle k \rangle^{m}\widehat{f} \rangle| d t \right)^{\frac{1}{2}}.
\een
By suitably combining  \eqref{macro-no-weight-x-derivative} and \eqref{micro-no-weight-x-derivative}, and noting
\beno
\|[a,b,c]\|_{L^{1}_{k,m}L^{2}_{T}} \sim \|\mathbb{P}f\|_{L^{1}_{k,m}L^{2}_{T}L^{2}_{\epsilon,\gamma/2}},
\eeno
we have
\ben \nonumber
&&\|f\|_{L^{1}_{k,m}L^{\infty}_{T}L^{2}} + \|f\|_{L^{1}_{k,m}L^{2}_{T}L^{2}_{\epsilon,\gamma/2}}
\\&\lesssim&
 \|f_{0}\|_{L^{1}_{k,m}L^{2}}  + \sum_{j}\sum_{k \in \mathbb{Z}^{3}} \left( \int_{0}^{T} |\langle \widehat{\mathcal{H}} \langle k \rangle^{m}, P_{j}\mu^{\frac{1}{2}} \rangle|^{2} d t \right)^{\frac{1}{2}} + \sum_{k \in \mathbb{Z}^{3}} \left( \int_{0}^{T} |\langle \widehat{\mathcal{H}}(k) \langle k \rangle^{m}, \langle k \rangle^{m}\widehat{f} \rangle| d t \right)^{\frac{1}{2}}\!\!.
 \label{macro-micro-no-weight-x-derivative}
\een

Recalling $\widehat{\mathcal{H}}(k)=\sum_{p \in \mathbb{Z}^{3}} \Gamma^{\epsilon}(\widehat{f}(k-p), \widehat{f}(p))$, by Theorem \ref{Gamma-full-up-bound} and
\ben \label{factor-out-into-2-sum}
\langle k \rangle^{m} \lesssim \langle k-p \rangle^{m} + \langle p \rangle^{m},
\een
we have
\beno
\sum_{k \in \mathbb{Z}^{3}} \left( \int_{0}^{T} |\langle \langle k \rangle^{m} \sum_{p \in \mathbb{Z}^{3}} \Gamma^{\epsilon}(\widehat{f}(k-p), \widehat{f}(p)), P_{j}\mu^{\frac{1}{2}} \rangle|^{2} d t \right)^{\frac{1}{2}}
\\ \lesssim  \|f\|_{L^{1}_{k,m}L^{\infty}_{T}L^{2}} \|f\|_{L^{1}_{k}L^{2}_{T}L^{2}_{\epsilon,\gamma/2}} +
\|f\|_{L^{1}_{k}L^{\infty}_{T}L^{2}} \|f\|_{L^{1}_{k,m}L^{2}_{T}L^{2}_{\epsilon,\gamma/2}},
\eeno
and (similar to the estimate of upper bound of Eq.(3.8) in \cite{duan2021global})
\beno
&&\sum_{k \in \mathbb{Z}^{3}} \left( \int_{0}^{T} |\langle \langle k \rangle^{m} \sum_{p \in \mathbb{Z}^{3}} \Gamma^{\epsilon}(\widehat{f}(k-p), \widehat{f}(p)), \langle k \rangle^{m} \widehat{f} \rangle| d t \right)^{\frac{1}{2}}
\\ &\lesssim& \eta \|f\|_{L^{1}_{k,m}L^{2}_{T}L^{2}_{\epsilon,\gamma/2}} + \frac{1}{4 \eta} \|f\|_{L^{1}_{k,m}L^{\infty}_{T}L^{2}} \|f\|_{L^{1}_{k}L^{2}_{T}L^{2}_{\epsilon,\gamma/2}} + \frac{1}{4 \eta} \|f\|_{L^{1}_{k}L^{\infty}_{T}L^{2}} \|f\|_{L^{1}_{k,m}L^{2}_{T}L^{2}_{\epsilon,\gamma/2}}.
\eeno

Plugging the above two inequalities into \eqref{macro-micro-no-weight-x-derivative}, for $0<\eta \leq 1$, we get
\beno
&&\|f\|_{L^{1}_{k,m}L^{\infty}_{T}L^{2}} + \|f\|_{L^{1}_{k,m}L^{2}_{T}L^{2}_{\epsilon,\gamma/2}}
\\&\lesssim& \|f_{0}\|_{L^{1}_{k,m}L^{2}} + \eta \|f\|_{L^{1}_{k,m}L^{2}_{T}L^{2}_{\epsilon,\gamma/2}}
+ \frac{1}{\eta} \|f\|_{L^{1}_{k,m}L^{\infty}_{T}L^{2}} \|f\|_{L^{1}_{k}L^{2}_{T}L^{2}_{\epsilon,\gamma/2}} + \frac{1}{\eta} \|f\|_{L^{1}_{k}L^{\infty}_{T}L^{2}} \|f\|_{L^{1}_{k,m}L^{2}_{T}L^{2}_{\epsilon,\gamma/2}}.
\eeno
By choosing $\eta$ small, recalling \eqref{a-priori-no-weight-no-derivative}, under smallness assumption on $\|f_{0}\|_{L^{1}_{k}L^{2}}$, we arrive at
\ben \label{a-priori-no-weight-x-derivative}
\|f\|_{L^{1}_{k,m}L^{\infty}_{T}L^{2}} + \|f\|_{L^{1}_{k,m}L^{2}_{T}L^{2}_{\epsilon,\gamma/2}}
\lesssim \|f_{0}\|_{L^{1}_{k,m}L^{2}}.
\een

{\it{Step 2: $\|\cdot\|_{L^{1}_{k}L^{2}_{l}}.$}}
We now consider propagation of polynomial moments. Starting from \eqref{equation-after-Fourier},  taking inner product with $W_{2l}\widehat{f}$, similar to Eq.(3.7) in \cite{duan2021global}, we have
\ben \label{micro-q-weight-no-derivative}
\|f\|_{L^{1}_{k}L^{\infty}_{T}L^{2}_{l}} + \|(\mathbb{I}-\mathbb{P})W_{l}f\|_{L^{1}_{k}L^{2}_{T}L^{2}_{\epsilon,\gamma/2}} \lesssim
\|f_{0}\|_{L^{1}_{k}L^{2}_{l}}
+ \sum_{k \in \mathbb{Z}^{3}} \left( \int_{0}^{T} |\langle \widehat{\mathcal{H}}, W_{l} \widehat{f} \rangle| d t \right)^{\frac{1}{2}}.
\een
In this step, $\widehat{\mathcal{H}}(k) = [\mathcal{L}^{\epsilon}, W_{l}]\widehat{f}(k) + W_{l}\sum_{p \in \mathbb{Z}^{3}} \Gamma^{\epsilon}(\widehat{f}(k-p), \widehat{f}(p))$.
By \eqref{DefProj} and \eqref{a-priori-no-weight-no-derivative}, we have
\ben \label{macro-q-weight-no-derivative}
\|\mathbb{P}W_{l}f\|_{L^{1}_{k}L^{2}_{T}L^{2}_{\epsilon,\gamma/2}} \lesssim \|f\|_{L^{1}_{k}L^{2}_{T}L^{2}_{\epsilon,\gamma/2}}
\lesssim \|f_{0}\|_{L^{1}_{k}L^{2}}.
\een
A suitable combination of \eqref{micro-q-weight-no-derivative} and \eqref{macro-q-weight-no-derivative} gives
\ben
\|f\|_{L^{1}_{k}L^{\infty}_{T}L^{2}_{l}} + \|f\|_{L^{1}_{k}L^{2}_{T}L^{2}_{\epsilon,l+\gamma/2}}
 \lesssim
 \|f_{0}\|_{L^{1}_{k}L^{2}_{l}}
 + \sum_{k \in \mathbb{Z}^{3}} \left( \int_{0}^{T} |\langle \widehat{\mathcal{H}}, W_{l}\widehat{f} \rangle| d t \right)^{\frac{1}{2}}.
 \label{macro-micro-q-weight-no-derivative}
\een

For the term involving $[\mathcal{L}^{\epsilon}, W_{l}]\widehat{f}(k)$, by Lemma \ref{commutator-linear}, we get
\beno
\sum_{k \in \mathbb{Z}^{3}} \left( \int_{0}^{T} |\langle [\mathcal{L}^{\epsilon}, W_{l}]\widehat{f}, W_{l}\widehat{f} \rangle| d t \right)^{\frac{1}{2}}
\lesssim \|f\|_{L^{1}_{k}L^{2}_{T}L^{2}_{l+\gamma/2}}.
\eeno
Since $|\cdot|_{\epsilon,l} \geq |\cdot|_{L^{2}_{l+s}} \geq |\cdot|_{L^{2}_{l+1/2}}$, we have
$
|f|_{L^{2}_{l+\gamma/2}} \leq  \eta |f|_{L^{2}_{\epsilon, l+\gamma/2}} + C(\eta,l) |f|_{L^{2}_{\epsilon, \gamma/2}}
$
which gives
\ben \label{commutator-weight-and-linear-op-no-m}
\sum_{k \in \mathbb{Z}^{3}} \left( \int_{0}^{T} |\langle [\mathcal{L}^{\epsilon}, W_{l}]\widehat{f},  W_{l}\widehat{f} \rangle| d t \right)^{\frac{1}{2}}
 \lesssim  \eta \|f\|_{L^{1}_{k}L^{2}_{T}L^{2}_{\epsilon,l+\gamma/2}} + C(\eta,l) \|f\|_{L^{1}_{k}L^{2}_{T}L^{2}_{\epsilon,\gamma/2}}.
\een

For the term involving $W_{l}\sum_{p \in \mathbb{Z}^{3}} \Gamma^{\epsilon}(\widehat{f}(k-p), \widehat{f}(p))$,
by \eqref{polynomial-commutator-nonlinear} in Lemma  \ref{general-weight-upper-bound}, we get
\ben \nonumber
&& \sum_{k \in \mathbb{Z}^{3}} \left( \int_{0}^{T} |\langle W_{l} \sum_{p \in \mathbb{Z}^{3}} \Gamma^{\epsilon}(\widehat{f}(k-p), \widehat{f}(p)), W_{l} \widehat{f} \rangle| d t \right)^{\frac{1}{2}}
\\ \label{non-linear-term-q-weight} &\lesssim& \eta \|f\|_{L^{1}_{k}L^{2}_{T}L^{2}_{\epsilon,l+\gamma/2}} + \frac{1}{4 \eta} \|f\|_{L^{1}_{k}L^{\infty}_{T}L^{2}} \|f\|_{L^{1}_{k}L^{2}_{T}L^{2}_{\epsilon,l+\gamma/2}}.
\een
Plugging \eqref{commutator-weight-and-linear-op-no-m} and \eqref{non-linear-term-q-weight} into \eqref{macro-micro-q-weight-no-derivative},
by choosing $\eta$ small and using \eqref{a-priori-no-weight-no-derivative}, under the smallness assumption on $\|f_{0}\|_{L^{1}_{k}L^{2}}$,
we have
\ben \label{a-priori-q-weight-no-derivative}
\|f\|_{L^{1}_{k}L^{\infty}_{T}L^{2}_{l}} + \|f\|_{L^{1}_{k}L^{2}_{T}L^{2}_{\epsilon,l+\gamma/2}} \lesssim
 \|f_{0}\|_{L^{1}_{k}L^{2}_{l}} + \|f\|_{L^{1}_{k}L^{2}_{T}L^{2}_{\epsilon,\gamma/2}} \lesssim  \|f_{0}\|_{L^{1}_{k}L^{2}_{l}},
\een
where we have used \eqref{a-priori-no-weight-no-derivative} in the last inequality.

{\it{Step 3: $\|\cdot\|_{L^{1}_{k,m}L^{2}_{l}}.$}} We now show the  propagation of spatial regularity with polynomial moment. Starting from \eqref{equation-after-Fourier}, taking inner product with $\langle k \rangle^{2m} W_{2l}\widehat{f}$, similar to Eq.(3.7) in \cite{duan2021global}, we have
\ben  \label{micro-q-weight-x-derivative}
\|f\|_{L^{1}_{k,m}L^{\infty}_{T}L^{2}_{l}} + \|(\mathbb{I}-\mathbb{P})W_{l}f\|_{L^{1}_{k,m}L^{2}_{T}L^{2}_{\epsilon,\gamma/2}} \lesssim
\|f_{0}\|_{L^{1}_{k,m}L^{2}_{l}}
+ \sum_{k \in \mathbb{Z}^{3}} \left( \int_{0}^{T} |\langle \widehat{\mathcal{H}}, \langle k \rangle^{2m} W_{l} \widehat{f} \rangle| d t \right)^{\frac{1}{2}}.
\een
In this step, $\widehat{\mathcal{H}}(k) = [\mathcal{L}^{\epsilon}, W_{l}]\widehat{f}(k) + W_{l}\sum_{p \in \mathbb{Z}^{3}} \Gamma^{\epsilon}(\widehat{f}(k-p), \widehat{f}(p))$.
By \eqref{DefProj} and \eqref{a-priori-no-weight-x-derivative}, we have
\ben \label{macro-q-weight-x-derivative}
\|\mathbb{P}W_{l}f\|_{L^{1}_{k,m}L^{2}_{T}L^{2}_{\epsilon,\gamma/2}} \lesssim \|f\|_{L^{1}_{k,m}L^{2}_{T}L^{2}_{\epsilon,\gamma/2}} \lesssim \|f_{0}\|_{L^{1}_{k,m}L^{2}}.
\een
A suitable combination of \eqref{micro-q-weight-x-derivative} and \eqref{macro-q-weight-x-derivative} gives
\ben
\|f\|_{L^{1}_{k,m}L^{\infty}_{T}L^{2}_{l}} + \|f\|_{L^{1}_{k,m}L^{2}_{T}L^{2}_{\epsilon,l+\gamma/2}}
 \lesssim
 \|f_{0}\|_{L^{1}_{k,m}L^{2}_{l}} + \sum_{k \in \mathbb{Z}^{3}} \left( \int_{0}^{T} |\langle \widehat{\mathcal{H}}, \langle k \rangle^{2m} W_{2l}\widehat{f} \rangle| d t \right)^{\frac{1}{2}}.
 \label{macro-micro-m-q-weight-no-derivative}
\een

Similar to \eqref{commutator-weight-and-linear-op-no-m}, we have
\ben \label{commutator-weight-and-linear-op}
\sum_{k \in \mathbb{Z}^{3}} \left( \int_{0}^{T} |\langle [\mathcal{L}^{\epsilon}, W_{l}]\widehat{f}, \langle k \rangle^{2m} W_{l}\widehat{f} \rangle| d t \right)^{\frac{1}{2}}
 \lesssim  \eta \|f\|_{L^{1}_{k,m}L^{2}_{T}L^{2}_{\epsilon,l+\gamma/2}} + C(\eta,l) \|f\|_{L^{1}_{k,m}L^{2}_{T}L^{2}_{\epsilon,\gamma/2}}.
\een

For the term involving $W_{l}\sum_{p \in \mathbb{Z}^{3}} \Gamma^{\epsilon}(\widehat{f}(k-p), \widehat{f}(p))$,
by \eqref{polynomial-commutator-nonlinear} in Lemma  \ref{general-weight-upper-bound}  with \eqref{factor-out-into-2-sum},
we get
\ben \nonumber
&&\sum_{k \in \mathbb{Z}^{3}} \left( \int_{0}^{T} |\langle \langle k \rangle^{m} W_{l} \sum_{p \in \mathbb{Z}^{3}} \Gamma^{\epsilon}(\widehat{f}(k-p), \widehat{f}(p)), \langle k \rangle^{m} W_{l} \widehat{f} \rangle| d t \right)^{\frac{1}{2}}
\\ \label{weight-nonlinear-term} &\lesssim& \eta \|f\|_{L^{1}_{k,m}L^{2}_{T}L^{2}_{\epsilon,l+\gamma/2}}
+ \frac{1}{4 \eta} \|f\|_{L^{1}_{k}L^{\infty}_{T}L^{2}} \|f\|_{L^{1}_{k,m}L^{2}_{T}L^{2}_{\epsilon,l+\gamma/2}} + \frac{1}{4 \eta} \|f\|_{L^{1}_{k,m}L^{\infty}_{T}L^{2}} \|f\|_{L^{1}_{k}L^{2}_{T}L^{2}_{\epsilon,l+\gamma/2}}\!\!.
\een
Plugging \eqref{commutator-weight-and-linear-op}  and \eqref{weight-nonlinear-term} into \eqref{macro-micro-m-q-weight-no-derivative},
by choosing $\eta$ small, recalling \eqref{a-priori-no-weight-no-derivative}, under the smallness assumption on $\|f_{0}\|_{L^{1}_{k}L^{2}}$,
we have
\beno
\|f\|_{L^{1}_{k,m}L^{\infty}_{T}L^{2}_{l}} + \|f\|_{L^{1}_{k,m}L^{2}_{T}L^{2}_{\epsilon,l+\gamma/2}} \lesssim
 \|f_{0}\|_{L^{1}_{k,m}L^{2}_{l}} + \|f\|_{L^{1}_{k,m}L^{2}_{T}L^{2}_{\epsilon,\gamma/2}} +  \|f\|_{L^{1}_{k,m}L^{\infty}_{T}L^{2}} \|f\|_{L^{1}_{k}L^{2}_{T}L^{2}_{\epsilon,l+\gamma/2}}.
\eeno
By  \eqref{a-priori-no-weight-x-derivative}, \eqref{a-priori-q-weight-no-derivative} and  \eqref{a-priori-q-weight-x-derivative} we complete the proof of the theorem.
\end{proof}

We now turn to the propagation of spatial and velocity regularity with polynomial moment.
Taking $v$ derivative $\partial_{\beta}$ of \eqref{equation-after-Fourier}, we get
\beno \partial_{t} \partial_{\beta} \widehat{f}(k) + \mathrm{i} v \cdot k \partial_{\beta} \widehat{f}(k) + \mathcal{L}^{\epsilon} \partial_{\beta} \widehat{f}(k)= \mathrm{i} [v \cdot k, \partial_{\beta}] \widehat{f}(k) +
[\mathcal{L}^{\epsilon},\partial_{\beta}] \widehat{f}(k) + \partial_{\beta} \widehat{\Gamma^{\epsilon}(f,f)}(k). \eeno

Taking inner product with $\langle k \rangle^{2m} \langle v \rangle^{2l} \partial_{\beta}\widehat{f}$ where $\beta \in \mathbb{Z}^{3}, m, l \geq 0$,
similar to Eq.(3.7) in \cite{duan2021global}, we have
\ben \nonumber
&&\|\partial_{\beta} f\|_{L^{1}_{k,m}L^{\infty}_{T}L^{2}_{l}} + \|(\mathbb{I}-\mathbb{P})W_{l}\partial_{\beta}f\|_{L^{1}_{k,m}L^{2}_{T}L^{2}_{\epsilon,\gamma/2}}
\\ \label{micro-q-weight-x-v-derivative} &\lesssim&
\|\partial_{\beta}f_{0}\|_{L^{1}_{k,m}L^{2}_{l}}
 + \sum_{k \in \mathbb{Z}^{3}} \left( \int_{0}^{T} |\langle \widehat{\mathcal{H}}, \langle k \rangle^{2m}
\langle v \rangle^{l} \partial_{\beta}\widehat{f} \rangle| d t \right)^{\frac{1}{2}},
\een
where \ben \label{general-right-hand-side}
\widehat{\mathcal{H}}(k) = \mathrm{i} W_{l} [v \cdot k, \partial_{\beta}] \widehat{f}(k) + [\mathcal{L}^{\epsilon}, W_{l}] \partial_{\beta} \widehat{f}(k)
+ W_{l} [\mathcal{L}^{\epsilon},\partial_{\beta}] \widehat{f}(k)
+ W_{l} \partial_{\beta}\sum_{p \in \mathbb{Z}^{3}} \Gamma^{\epsilon}(\widehat{f}(k-p), \widehat{f}(p)).\een

By \eqref{DefProj} and integrating by parts, it holds
\ben \label{macro-q-weight-x-v-derivative}
\|\mathbb{P}W_{l}\partial_{\beta}f\|_{L^{1}_{k,m}L^{2}_{T}L^{2}_{\epsilon,\gamma/2}} \lesssim \|f\|_{L^{1}_{k,m}L^{2}_{T}L^{2}_{\epsilon,\gamma/2}}
\lesssim \|f_{0}\|_{L^{1}_{k,m}L^{2}},
\een
where we have used \eqref{a-priori-no-weight-x-derivative} in the last inequality.

A suitable combination of \eqref{micro-q-weight-x-v-derivative} and \eqref{macro-q-weight-x-v-derivative} gives
\ben
\|\partial_{\beta}f\|_{L^{1}_{k,m}L^{\infty}_{T}L^{2}_{l}} + \|\partial_{\beta}f\|_{L^{1}_{k,m}L^{2}_{T}L^{2}_{\epsilon,l+\gamma/2}} \lesssim
 \|f_{0}\|_{L^{1}_{k,m}H^{|\beta|}_{l}}  + \sum_{k \in \mathbb{Z}^{3}} \left( \int_{0}^{T} |\langle \widehat{\mathcal{H}},
 \langle k \rangle^{2m} \langle v \rangle^{l} \partial_{\beta}\widehat{f} \rangle| d t \right)^{\frac{1}{2}}.
 \label{macro-micro-q-weight-x-v-derivative-general}
\een

We now estimate the last term in \eqref{macro-micro-q-weight-x-v-derivative-general}. Recalling \eqref{general-right-hand-side}, we will estimate term by term.

For the term involving $\mathrm{i} [v \cdot k, \partial_{\beta}] \widehat{f}(k)$, if $\beta=(\beta_{1},\beta_{2},\beta_{3})$, then
\beno
[v \cdot k, \partial_{\beta}] \widehat{f}(k)= - \sum_{j=1}^{3} \beta_{j}  k_{j} \partial_{\beta-e_{j}}\widehat{f}(k).
\eeno
Hence
\ben \label{commutator-transport-general}
&& \sum_{k \in \mathbb{Z}^{3}} \left( \int_{0}^{T} |\langle [v \cdot k, \partial_{\beta}] \widehat{f}(k),
\langle k \rangle^{2m} \langle v \rangle^{2l} \partial_{\beta}\widehat{f} \rangle| d t \right)^{\frac{1}{2}}
\nonumber \\ &\lesssim&
\sum_{k \in \mathbb{Z}^{3}} \left( \int_{0}^{T} \sum_{j=1}^{3} \beta_{j} |\langle k \rangle^{m+1}
\partial_{\beta-e_{j}} \widehat{f}(k)|_{L^{2}_{l-(\gamma/2+s)}}
|\langle k \rangle^{m} \partial_{\beta} \widehat{f}(k)|_{L^{2}_{l+\gamma/2+s}}
 d t \right)^{\frac{1}{2}}
\nonumber  \\ &\lesssim& \eta \|\partial_{\beta}f\|_{L^{1}_{k,m}L^{2}_{T}L^{2}_{\epsilon,l+\gamma/2}}
 + \frac{1}{\eta} \|f\|_{L^{1}_{k,m+1}L^{2}_{T} \dot{H}^{N}_{\epsilon,\gamma/2+l-(\gamma+2s)}}.
\een

Similar to \eqref{commutator-weight-and-linear-op}, we have
\ben \label{commutator-weight-and-linear-v-derivative}
&&\sum_{k \in \mathbb{Z}^{3}} \left( \int_{0}^{T} |\langle [\mathcal{L}^{\epsilon}, W_{l}]\partial_{\beta}\widehat{f}, \langle k \rangle^{2m} W_{l}\partial_{\beta}\widehat{f} \rangle| d t \right)^{\frac{1}{2}}
\nonumber \\ &\lesssim&  \eta \|\partial_{\beta}f\|_{L^{1}_{k,m}L^{2}_{T}L^{2}_{\epsilon,l+\gamma/2}} + C(\eta,l) \|\partial_{\beta}f\|_{L^{1}_{k,m}L^{2}_{T}L^{2}_{\epsilon,\gamma/2}}.
\een

By Lemma \ref{commutator-linear-g-h}, we get
\ben \label{commutator-with-linear-q-weight-v-derivative}
&&\sum_{k \in \mathbb{Z}^{3}} \left( \int_{0}^{T} |\langle W_{l}[\mathcal{L}^{\epsilon},\partial_{\beta}] \widehat{f}(k), \langle k \rangle^{2m} W_{l} \partial_{\beta} \widehat{f} \rangle| d t \right)^{\frac{1}{2}}
\nonumber \\ &\lesssim& \sum_{k \in \mathbb{Z}^{3}} \left( \int_{0}^{T} \sum_{\beta_{1}<\beta}|\langle k \rangle^{m} \partial_{\beta_{1}}\widehat{f}|_{L^{2}_{\epsilon, l+\gamma/2}}|\langle k \rangle^{m} \partial_{\beta}\widehat{f}|_{L^{2}_{\epsilon,l+\gamma/2}}
 d t \right)^{\frac{1}{2}}
\nonumber \\ &\lesssim&  \eta \|\partial_{\beta}f\|_{L^{1}_{k,m}L^{2}_{T}L^{2}_{\epsilon,l+\gamma/2}} + \frac{1}{\eta} \|f\|_{L^{1}_{k,m}L^{2}_{T}H^{N}_{\epsilon, l+\gamma/2}}.
\een

For the term involving $\partial_{\beta}\sum_{p \in \mathbb{Z}^{3}} \Gamma^{\epsilon}(\widehat{f}(k-p), \widehat{f}(p))$,
by \eqref{polynomial-commutator-nonlinear} in Lemma  \ref{general-weight-upper-bound}, with \eqref{factor-out-into-2-sum},
we get
\ben \label{non-linear-term-general}
\sum_{k \in \mathbb{Z}^{3}} \left( \int_{0}^{T} |\langle \partial_{\beta}
\sum_{p \in \mathbb{Z}^{3}} \Gamma^{\epsilon}(\widehat{f}(k-p),\widehat{f}(p)),
\langle k \rangle^{2m} \langle v \rangle^{2l} \partial_{\beta} \widehat{f} \rangle| d t \right)^{\frac{1}{2}}
 \lesssim \eta \|\partial_{\beta}f\|_{L^{1}_{k,m}L^{2}_{T}L^{2}_{\epsilon,l+\gamma/2}}
\nonumber \\ + \frac{1}{4 \eta} \|f\|_{L^{1}_{k}L^{\infty}_{T}L^{2}} \|\partial_{\beta}f\|_{L^{1}_{k,m}L^{2}_{T}L^{2}_{\epsilon,l+\gamma/2}}
+ \frac{1}{4 \eta} \|f\|_{L^{1}_{k,m}L^{\infty}_{T}L^{2}} \|\partial_{\beta}f\|_{L^{1}_{k}L^{2}_{T}L^{2}_{\epsilon,l+\gamma/2}}
\nonumber \\+ \frac{1}{4 \eta} \|\partial_{\beta}f\|_{L^{1}_{k}L^{\infty}_{T}L^{2}} \|f\|_{L^{1}_{k,m}L^{2}_{T}L^{2}_{\epsilon,l+\gamma/2}}
+ \frac{1}{4 \eta} \|\partial_{\beta}f\|_{L^{1}_{k,m}L^{\infty}_{T}L^{2}} \|f\|_{L^{1}_{k}L^{2}_{T}L^{2}_{\epsilon,l+\gamma/2}}
\nonumber \\+ \frac{1}{4 \eta} \|f\|_{L^{1}_{k}L^{\infty}_{T}H^{N}} \|f\|_{L^{1}_{k,m}L^{2}_{T}H^{N}_{\epsilon,l+\gamma/2}}
+ \frac{1}{4 \eta} \|f\|_{L^{1}_{k,m}L^{\infty}_{T}H^{N}} \|f\|_{L^{1}_{k}L^{2}_{T}H^{N}_{\epsilon,l+\gamma/2}}.
\een

With the above preparation, we are ready to prove the propagation of both spatial and velocity regularity with polynomial weight in the following theorem.

\begin{thm} \label{high-order-propagation}
Let $n \in \mathbb{N}, m, l \geq 0$. Suppose $f$ is a solution to the Boltzmann equation \eqref{Cauchy-linearizedBE-grazing} with initial data $f_{0}$ verifying $ \|f_{0}\|_{m,n,l} < \infty.$
There is a constant $\delta>0$ and a polynomial $P_{n}$ with $P_{n}(0)=0$ such that if
$
\|f_{0}\|_{L^{1}_{k}L^{2}} < \delta,
$
then
\ben \label{general-propagation}
E_{T}(f; m,n,l) + D_{T}^{\epsilon}(f; m,n,l) \lesssim  P_{n}(\|f_{0}\|_{m,n,l}).
\een
\end{thm}
\begin{proof}
Recall the notations $E_{T}(f; m,n,l), D_{T}^{\epsilon}(f; m,n,l)$ in \eqref{energy-and-dissipation-for-propagation},
$\|f_{0}\|_{m,n,l}$ in \eqref{initial-dependence-for-propagation}.
We will prove \eqref{general-propagation} by induction. First, by \eqref{a-priori-q-weight-x-derivative},
we see that \eqref{general-propagation} is valid for $n=0$ with $P_{0}(x)=x(1+x)$.

Let $N \geq 0$ be an integer.
Let us assume that \eqref{general-propagation} is valid for any $0 \leq n \leq N$ and $m, l \geq 0$. We will prove the above statement
\eqref{general-propagation}  is also valid for $n=N+1, m, l \geq 0$. To be clear, we fix two parameters $m_{*}, l_{*} \geq 0$ and prove
\eqref{general-propagation} for $n=N+1, m=m_{*}, l= l_{*}$.

We  concentrate on $\|\cdot\|_{L^{1}_{k,m_{*}}\dot{H}^{N+1}_{l_{*}}}$.
We divide the proof in four steps for the estimation on
$\|\cdot\|_{L^{1}_{k}\dot{H}^{N+1}}$,  $\|\cdot\|_{L^{1}_{k,m_{*}}\dot{H}^{N+1}}$,
 $\|\cdot\|_{L^{1}_{k}\dot{H}^{N+1}_{l_{*}}}$  and $\|\cdot\|_{L^{1}_{k,m_{*}}\dot{H}^{N+1}_{l_{*}}}$ respectively.

{\it {Step 1. $\|\cdot\|_{L^{1}_{k}\dot{H}^{N+1}}$.}}
We start from \eqref{macro-micro-q-weight-x-v-derivative-general} by taking $|\beta|=N+1, m=l=0$. In this case,
$[\mathcal{L}^{\epsilon}, W_{l}]=0$ in \eqref{general-right-hand-side}.

Plugging \eqref{commutator-transport-general}, \eqref{commutator-with-linear-q-weight-v-derivative} and \eqref{non-linear-term-general} for the case $m=l=0$
into \eqref{macro-micro-q-weight-x-v-derivative-general}, taking sum over $|\beta|=N+1$,
by choosing $\eta$ small and under the smallness assumption on $\|f_{0}\|_{L^{1}_{k}L^{2}}$ that implies
$\|f\|_{L^{1}_{k}L^{\infty}_{T}L^{2}}$ and $\|f\|_{L^{1}_{k}L^{2}_{T}L^{2}_{\epsilon,\gamma/2}}$ are small,
and by using \eqref{general-propagation} with $n=N, m=1, l=-(\gamma+2s)$,
we have
\ben \nonumber
&&\|f\|_{L^{1}_{k}L^{\infty}_{T}\dot{H}^{N+1}} + \|f\|_{L^{1}_{k}L^{2}_{T}\dot{H}^{N+1}_{\epsilon,\gamma/2}}
\\ &\lesssim&
 \|f_{0}\|_{L^{1}_{k}H^{N+1}}  + P_{N}(\|f_{0}\|_{1,N,-(\gamma+2s)}) (1+ P_{N}(\|f_{0}\|_{1,N,-(\gamma+2s)})). \label{pure-highest-order}
\een
Let us denote $P^{0,N+1,0}(x) := P_{N}(x)(1+P_{N}(x))$.
Adding \eqref{pure-highest-order} and \eqref{general-propagation} with $n=N, m=1,  l=-(\gamma+2s)$, we have
\ben
E_{T}(f; 0,N+1, 0) + D_{T}^{\epsilon}(f; 0,N+1, 0)
 \lesssim P^{0,N+1,0}(\|f_{0}\|_{0,N+1,0}). \label{a-priori-no-weight-v-derivative-order-N+1}
\een

{\it {Step 2. $\|\cdot\|_{L^{1}_{k, m_{*}}\dot{H}^{N+1}}$.}}
We start from \eqref{macro-micro-q-weight-x-v-derivative-general} by taking $|\beta|=N+1, m=m_{*}, l=0$. In this case,
$[\mathcal{L}^{\epsilon}, W_{l}]=0$ in \eqref{general-right-hand-side}.

Plugging \eqref{commutator-transport-general}, \eqref{commutator-with-linear-q-weight-v-derivative} and \eqref{non-linear-term-general} for the case $m=m_{*}, l=0$
into \eqref{macro-micro-q-weight-x-v-derivative-general}, taking sum over $|\beta|=N+1$,
by choosing $\eta$ small and under the smallness assumption on $\|f_{0}\|_{L^{1}_{k}L^{2}}$ that implies
$\|f\|_{L^{1}_{k}L^{\infty}_{T}L^{2}}$ and $\|f\|_{L^{1}_{k}L^{2}_{T}L^{2}_{\epsilon,\gamma/2}}$ are small,
using \eqref{general-propagation} with $n=N, m=m_{*}+1, l=-(\gamma+2s)$ and \eqref{a-priori-no-weight-v-derivative-order-N+1},
we have
\ben \label{pure-m-N+1}
\|f\|_{L^{1}_{k,m_{*}}L^{\infty}_{T}\dot{H}^{N+1}} + \|f\|_{L^{1}_{k,m_{*}}L^{2}_{T}\dot{H}^{N+1}_{\epsilon,\gamma/2}}
 \lesssim \|f_{0}\|_{L^{1}_{k,m_{*}}H^{N+1}} + P_{0}(\|f_{0}\|_{m_{*},0,0}) P^{0,N+1,0}(\|f_{0}\|_{0,N+1,0})
\nonumber   \\  + P_{N}(\|f_{0}\|_{m_{*}+1,N,-(\gamma+2s)}) (1+ P_{N}(\|f_{0}\|_{m_{*}+1,N,-(\gamma+2s)})).
\een
Let us define $P^{m_{*},N+1,0}(x):= P_{0}(x) P^{0,N+1,0}(x) + P_{N}(x) (1+ P_{N}(x))$.
Adding \eqref{pure-m-N+1} and \eqref{general-propagation} with $n=N, m=m_{*}+1, l=-(\gamma+2s)$, we have
\ben
E_{T}(f; m_{*},N+1, 0) + D_{T}^{\epsilon}(f; m_{*},N+1, 0)
 \lesssim P^{m_{*},N+1,0}(\|f_{0}\|_{m_{*},N+1,0}). \label{a-priori-no-weight-x-v-derivative}
\een

{\it {Step 3. $\|\cdot\|_{L^{1}_{k}\dot{H}^{N+1}_{l_{*}}}$.}}
We  start from \eqref{macro-micro-q-weight-x-v-derivative-general} by taking $|\beta|=N+1, m=0, l=l_{*}$.

Plugging \eqref{commutator-transport-general}, \eqref{commutator-weight-and-linear-v-derivative},  \eqref{commutator-with-linear-q-weight-v-derivative} and \eqref{non-linear-term-general} for the case $m=0, l=l_{*}$
into \eqref{macro-micro-q-weight-x-v-derivative-general}, taking sum over $|\beta|=N+1$,
by choosing $\eta$ small and under the smallness assumption on $\|f_{0}\|_{L^{1}_{k}L^{2}}$ that implies
$\|f\|_{L^{1}_{k}L^{\infty}_{T}L^{2}}$ is small,
using \eqref{general-propagation} with $n=N, m=1, l=l_{*}-(\gamma+2s)$ and \eqref{a-priori-no-weight-v-derivative-order-N+1},
we have
\ben \label{pure-N+1-q}
\|f\|_{L^{1}_{k}L^{\infty}_{T}\dot{H}^{N+1}_{l_{*}}} + \|f\|_{L^{1}_{k}L^{2}_{T}\dot{H}^{N+1}_{\epsilon,l_{*}+\gamma/2}}
\lesssim  \|f_{0}\|_{L^{1}_{k}H^{N+1}_{l_{*}}} +P_{0}(\|f_{0}\|_{0,0,l_{*}}) P^{0,N+1,0}(\|f_{0}\|_{0,N+1,0})
\nonumber  \\ + P_{N}(\|f_{0}\|_{1,N,l_{*}-(\gamma+2s)})(1+P_{N}(\|f_{0}\|_{1,N,l_{*}-(\gamma+2s)})).
\een
Let us define $P^{0,N+1,l_{*}}(x):= P_{0}(x) P^{0,N+1,0}(x) + P_{N}(x) (1+ P_{N}(x))$.
Adding \eqref{pure-N+1-q} and \eqref{general-propagation} with $n=N, m=1, l=l_{*}-(\gamma+2s)$, we have
\ben
E_{T}(f; 0,N+1,l_{*}) + D_{T}^{\epsilon}(f; 0,N+1,l_{*})
 \lesssim P^{0,N+1,l_{*}}(x)(\|f_{0}\|_{0,N+1,l_{*}}). \label{a-priori-q-weight-v-derivative-order-N+1}
\een

{\it {Step 4. $\|\cdot\|_{L^{1}_{k,m_{*}}\dot{H}^{N+1}_{l_{*}}}$.}}
We  start from \eqref{macro-micro-q-weight-x-v-derivative-general} by taking $|\beta|=N+1, m=m_{*}, l=l_{*}$.

Plugging \eqref{commutator-transport-general}, \eqref{commutator-weight-and-linear-v-derivative} \eqref{commutator-with-linear-q-weight-v-derivative} and \eqref{non-linear-term-general} for the case $m=m_{*}, l=l_{*}$
into \eqref{macro-micro-q-weight-x-v-derivative-general}, taking sum over $|\beta|=N+1$,
by choosing $\eta$ small and under the smallness assumption on $\|f_{0}\|_{L^{1}_{k}L^{2}}$ that implies
$\|f\|_{L^{1}_{k}L^{\infty}_{T}L^{2}}$ is small, and by using \eqref{a-priori-no-weight-v-derivative-order-N+1},
\eqref{a-priori-no-weight-x-v-derivative}, \eqref{a-priori-q-weight-v-derivative-order-N+1} and \eqref{general-propagation} with
$n=N, m=m_{*}+1, l=l_{*}-(\gamma+2s)$,
we have
\ben \label{pure-m-N+1-q}
\|f\|_{L^{1}_{k,m_{*}}L^{\infty}_{T}\dot{H}^{N+1}_{l_{*}}}
+ \|f\|_{L^{1}_{k,m_{*}}L^{2}_{T}\dot{H}^{N+1}_{\epsilon,l_{*}+\gamma/2}} \lesssim
 \|f_{0}\|_{L^{1}_{k,m_{*}}H^{N+1}_{l_{*}}}  + P_{0}(\|f_{0}\|_{m_{*},0,l_{*}})P^{0,N+1,0}(\|f_{0}\|_{0,N+1,0})
\nonumber \\+ P_{0}(\|f_{0}\|_{m_{*},0,0})P^{0,N+1,l_{*}}(\|f_{0}\|_{0,N+1,l_{*}})
 + P_{0}(\|f_{0}\|_{0,0,l_{*}})P^{m_{*},N+1,0}(\|f_{0}\|_{m_{*},N+1,0})
\nonumber \\ + P_{N}(\|f_{0}\|_{m_{*}+1,N,l_{*}-(\gamma+2s)})(1+P_{N}(\|f_{0}\|_{m_{*}+1,N,l_{*}-(\gamma+2s)})).
\een
Define
\beno
P_{N+1}(x):= P_{0}(x)P^{0,N+1,l_{*}}(x)
 + P_{0}(x)P^{0,N+1,0}(x)
 + P_{0}(x)P^{m_{*},N+1,0}(x)
 + P_{N}(x)(1+P_{N}(x)).
\eeno
Note that $P_{N+1}$ is independent of $m_{*}, l_{*}$.
Summing \eqref{pure-m-N+1-q} and \eqref{general-propagation} with
$n=N, m=m_{*}+1, l=l_{*}-(\gamma+2s)$ gives
\beno
E_{T}(f; m_{*},N+1,l_{*}) + D_{T}^{\epsilon}(f; m_{*},N+1,l_{*})
 \lesssim P_{N+1}(\|f_{0}\|_{m_{*},N+1,l_{*}}).
\eeno
And this completes the proof of the theorem.
\end{proof}

\subsection{Global asymptotics} \label{asymptotic}
We will prove \eqref{global-in-time-error} in this subsection.
We first give an estimate on the operator $\Gamma^{\epsilon}-\Gamma^{L}$.
\begin{lem}\label{estimate-operator-difference} For suitable functions $g,h,f$, there holds
\beno|\langle (\Gamma^{\epsilon}-\Gamma^{L})(g,h), f \rangle| \lesssim \epsilon |g|_{H^{3}}|h|_{H^{3}_{9+\gamma/2}}|f|_{L^{2}_{\gamma/2}}.\eeno
\end{lem}
\begin{proof} Note that
\beno
\langle (\Gamma^{\epsilon}-\Gamma^{L})(g,h), f \rangle = \langle Q^{\epsilon}(\mu^{\frac{1}{2}}g, \mu^{\frac{1}{2}}h) - Q^{L}(\mu^{\frac{1}{2}}g, \mu^{\frac{1}{2}}h) , \mu^{-\frac{1}{2}}f \rangle.
\eeno
By the proof in \cite{zhou2020refined}, it holds
\beno
|\langle Q^{\epsilon}(G, H) - Q^{L}(G, H) , F \rangle| \lesssim \epsilon |G|_{H^{3}_{8+\gamma}}|H|_{H^{3}_{6+\gamma/2}}|F|_{L^{2}_{\gamma/2}}.
\eeno
By using the fact
\ben \label{order-3-more}
|\nabla^{3} (\mu^{\frac{1}{2}} h)| \lesssim \langle v \rangle^{3} \mu^{\frac{1}{2}}(|h|+|\nabla h|+|\nabla^{2} h|+|\nabla^{3} h|),
\een
and the proof in \cite{zhou2020refined}, the estimate in the lemma follows.
Note that the additional $3$ weight (from $6+\gamma/2$ to $9+\gamma/2$) for the  function $h$ comes from $\langle v \rangle^{3}$ in \eqref{order-3-more}. On the other hand,  the factor $\mu^{\frac{1}{2}}$ before $g$ absorbs any polynomial weight.
\end{proof}
We are ready to prove \eqref{global-in-time-error} in Theorem \ref{asymptotic-result}.
\begin{proof}[Proof of \eqref{global-in-time-error}]
Let $f^{\epsilon}$ and $f^{L}$ be the solutions to  \eqref{Cauchy-linearizedBE-grazing} and \eqref{Cauchy-linearizedLE} respectively with the initial data $f_0$. Set $F^{\epsilon}_{R} :=  \epsilon^{-1} (f^{\epsilon}-f^{L})$, then it solves
\beno \partial_{t}F^{\epsilon}_{R} + v \cdot \nabla_{x} F^{\epsilon}_{R} + \mathcal{L}^{L}F^{\epsilon}_{R} = \epsilon^{-1} [(\mathcal{L}^{L}-\mathcal{L}^{\epsilon})f^{\epsilon}+(\Gamma^{\epsilon}-\Gamma^{L})(f^{\epsilon},f^{L})]
+\Gamma^{\epsilon}(f^{\epsilon},F^{\epsilon}_{R})+\Gamma^{L}(F^{\epsilon}_{R},f^{L}). \eeno
For simplicity, we denote the right hand side by $\mathcal{H}$,
\beno
\mathcal{H} := \epsilon^{-1} [(\mathcal{L}^{L}-\mathcal{L}^{\epsilon})f^{\epsilon}+(\Gamma^{\epsilon}-\Gamma^{L})(f^{\epsilon},f^{L})]
+\Gamma^{\epsilon}(f^{\epsilon},F^{\epsilon}_{R})+\Gamma^{L}(F^{\epsilon}_{R},f^{L}).
\eeno
Taking Fourier transform with respect to $x$, for the mode $k \in \mathbb{Z}^{3}$, we have
\ben \label{error-equation} \partial_{t} \widehat{F^{\epsilon}_{R}}(k) + \mathrm{i} v \cdot k \widehat{F^{\epsilon}_{R}}(k) + \mathcal{L}^{L} \widehat{F^{\epsilon}_{R}}(k)= \widehat{\mathcal{H}}(k),  \een
where
\ben \label{definition-of-hat-H}
\widehat{\mathcal{H}}(k) &:=&\epsilon^{-1} (\mathcal{L}^{L}-\mathcal{L}^{\epsilon}) \widehat{f^{\epsilon}}(k)
+ \sum_{p \in \mathbb{Z}^{3}} \epsilon^{-1}(\Gamma^{\epsilon}-\Gamma^{L})( \widehat{f^{\epsilon}}(k-p), \widehat{f^{L}}(p))
\nonumber \\&&+ \sum_{p \in \mathbb{Z}^{3}} \Gamma^{\epsilon}(\widehat{f^{\epsilon}}(k-p), \widehat{F^{\epsilon}_{R}}(p))+ \sum_{p \in \mathbb{Z}^{3}}\Gamma^{L}( \widehat{F^{\epsilon}_{R}}(k-p), \widehat{f^{L}}(p)).
\een

We divide the proof into three steps.

{\it{Step 1: Macroscopic part.}}
For the estimate of $[a,b,c]$  of   $F^{\epsilon}_{R} $ defined by \eqref{DefProj}, by Theorem 5.1 in \cite{duan2021global}, since $F^{\epsilon}_{R}(0)=0$, it  holds
\beno
\|[a,b,c]\|_{L^{1}_{k}L^{2}_{T}} \lesssim \|(\mathbb{I}-\mathbb{P})F^{\epsilon}_{R}\|_{L^{1}_{k}L^{2}_{T}L^{2}_{0,\gamma/2}} + \|F^{\epsilon}_{R}\|_{L^{1}_{k}L^{\infty}_{T}L^{2}} + \sum_{j}\sum_{k \in \mathbb{Z}^{3}} \left( \int_{0}^{T} |\langle \widehat{\mathcal{H}}(k), P_{j}\mu^{\frac{1}{2}} \rangle|^{2} d t \right)^{\frac{1}{2}},
\eeno
where $\{P_{j}\}_{j}$ is a set of polynomials with degree $\leq 4$.

Note that $\widehat{\mathcal{H}}(k)$ has four terms in \eqref{definition-of-hat-H}.
For the term $\sum_{p \in \mathbb{Z}^{3}} \Gamma^{\epsilon}(\widehat{f^{\epsilon}}(k-p), \widehat{F^{\epsilon}_{R}}(p))$, by Theorem \ref{Gamma-full-up-bound}, we have
\beno
\sum_{k \in \mathbb{Z}^{3}} \left( \int_{0}^{T} |\langle \sum_{p \in \mathbb{Z}^{3}} \Gamma^{\epsilon}(\widehat{f^{\epsilon}}(k-p), \widehat{F^{\epsilon}_{R}}(p)), P_{j}\mu^{\frac{1}{2}} \rangle|^{2} d t \right)^{\frac{1}{2}}
 \lesssim  \|f^{\epsilon}\|_{L^{1}_{k}L^{\infty}_{T}L^{2}} \|F^{\epsilon}_{R}\|_{L^{1}_{k}L^{2}_{T}L^{2}_{0,\gamma/2}}.
\eeno
Similarly, for the term $\sum_{p \in \mathbb{Z}^{3}} \Gamma^{L}(\widehat{F^{\epsilon}_{R}}(k-p), \widehat{f^{L}}(p))$,
we have
\beno
\sum_{k \in \mathbb{Z}^{3}} \left( \int_{0}^{T} |\langle \sum_{p \in \mathbb{Z}^{3}} \Gamma^{L}(\widehat{F^{\epsilon}_{R}}(k-p), \widehat{f^{L}}(p)), P_{j}\mu^{\frac{1}{2}} \rangle|^{2} d t \right)^{\frac{1}{2}} \lesssim \|F^{\epsilon}_{R}\|_{L^{1}_{k}L^{\infty}_{T}L^{2}} \|f^{L}\|_{L^{1}_{k}L^{2}_{T}L^{2}_{0,\gamma/2}}.
\eeno
For the term
$\sum_{p \in \mathbb{Z}^{3}} \epsilon^{-1}(\Gamma^{\epsilon}-\Gamma^{L})( \widehat{f^{\epsilon}}(k-p), \widehat{f^{L}}(p))$, by Lemma \ref{estimate-operator-difference},
we have
\beno
\sum_{k \in \mathbb{Z}^{3}} \left( \int_{0}^{T} |\langle \sum_{p \in \mathbb{Z}^{3}} \epsilon^{-1}(\Gamma^{\epsilon}-\Gamma^{L})( \widehat{f^{\epsilon}}(k-p), \widehat{f^{L}}(p)), P_{m}\mu^{\frac{1}{2}} \rangle|^{2} d t \right)^{\frac{1}{2}}
 \lesssim  \|f^{\epsilon}\|_{L^{1}_{k}L^{\infty}_{T}H^{3}} \|f^{L}\|_{L^{1}_{k}L^{2}_{T}H^{3}_{9+\gamma/2}}.
\eeno
Similarly, for the term $\epsilon^{-1} (\mathcal{L}^{L}-\mathcal{L}^{\epsilon}) \widehat{f^{\epsilon}}(k)$, by Lemma \ref{estimate-operator-difference},
we have
\beno
\sum_{k \in \mathbb{Z}^{3}} \left( \int_{0}^{T} |\langle \epsilon^{-1} (\mathcal{L}^{L}-\mathcal{L}^{\epsilon}) \widehat{f^{\epsilon}}(k), P_{m}\mu^{\frac{1}{2}} \rangle|^{2} d t \right)^{\frac{1}{2}}
\lesssim  \|f^{\epsilon}\|_{L^{1}_{k}L^{2}_{T}H^{3}_{9+\gamma/2}}.
\eeno
In summary, we have
\ben
\|[a,b,c]\|_{L^{1}_{k}L^{2}_{T}} \leq C_{1}\|(\mathbb{I}-\mathbb{P})F^{\epsilon}_{R}\|_{L^{1}_{k}L^{2}_{T}L^{2}_{0,\gamma/2}} + C_{1}\|F^{\epsilon}_{R}\|_{L^{1}_{k}L^{\infty}_{T}L^{2}} +
C_{1}\|f^{\epsilon}\|_{L^{1}_{k}L^{\infty}_{T}L^{2}} \|F^{\epsilon}_{R}\|_{L^{1}_{k}L^{2}_{T}L^{2}_{0,\gamma/2}}
\nonumber \\+C_{1}\|F^{\epsilon}_{R}\|_{L^{1}_{k}L^{\infty}_{T}L^{2}} \|f^{L}\|_{L^{1}_{k}L^{2}_{T}L^{2}_{0,\gamma/2}}
+ C_{1}\|f^{\epsilon}\|_{L^{1}_{k}L^{\infty}_{T}H^{3}} \|f^{L}\|_{L^{1}_{k}L^{2}_{T}H^{3}_{9+\gamma/2}}
+C_{1}\|f^{\epsilon}\|_{L^{1}_{k}L^{2}_{T}H^{3}_{9+\gamma/2}}. \label{mirco-part}
\een

{\it{Step 2: Microscopic part.}} From \eqref{error-equation},
taking inner product with $\widehat{F^{\epsilon}_{R}}(k)$, similar to Eq.(3.7) in \cite{duan2021global}, with $F^{\epsilon}_{R}(0)=0$,
we have
\beno
\|F^{\epsilon}_{R}\|_{L^{1}_{k}L^{\infty}_{T}L^{2}} + \|(\mathbb{I}-\mathbb{P})F^{\epsilon}_{R}\|_{L^{1}_{k}L^{2}_{T}L^{2}_{0,\gamma/2}} \lesssim
\sum_{k \in \mathbb{Z}^{3}} \left( \int_{0}^{T} |\langle \widehat{\mathcal{H}}(k), \widehat{F^{\epsilon}_{R}}(k) \rangle| d t \right)^{\frac{1}{2}}.
\eeno

We estimate $\widehat{\mathcal{H}}(k)$ term by term as follows.
For the term $\sum_{p \in \mathbb{Z}^{3}} \Gamma^{\epsilon}(\widehat{f^{\epsilon}}(k-p), \widehat{F^{\epsilon}_{R}}(p))$, by Theorem \ref{Gamma-full-up-bound}, we have
\beno
&&\sum_{k \in \mathbb{Z}^{3}} \left( \int_{0}^{T} |\langle \sum_{p \in \mathbb{Z}^{3}} \Gamma^{\epsilon}(\widehat{f^{\epsilon}}(k-p), \widehat{F^{\epsilon}_{R}}(p)), \widehat{F^{\epsilon}_{R}}(k) \rangle| d t \right)^{\frac{1}{2}}
\\ &\lesssim& \eta \|F^{\epsilon}_{R}\|_{L^{1}_{k}L^{2}_{T}L^{2}_{0,\gamma/2}} + \frac{1}{4 \eta} \|f^{\epsilon}\|_{L^{1}_{k}L^{\infty}_{T}L^{2}} \|F^{\epsilon}_{R}\|_{L^{1}_{k}L^{2}_{T}L^{2}_{0,\gamma/2}}.
\eeno
Similarly, for the term $\sum_{p \in \mathbb{Z}^{3}} \Gamma^{L}(\widehat{F^{\epsilon}_{R}}(k-p), \widehat{f^{L}}(p))$,
we have
\beno
&&\sum_{k \in \mathbb{Z}^{3}} \left( \int_{0}^{T} |\langle \sum_{p \in \mathbb{Z}^{3}} \Gamma^{L}(\widehat{F^{\epsilon}_{R}}(k-p), \widehat{f^{L}}(p)), \widehat{F^{\epsilon}_{R}}(k) \rangle| d t \right)^{\frac{1}{2}}
\\ &\lesssim& \eta \|F^{\epsilon}_{R}\|_{L^{1}_{k}L^{2}_{T}L^{2}_{0,\gamma/2}} + \frac{1}{4 \eta} \|F^{\epsilon}_{R}\|_{L^{1}_{k}L^{\infty}_{T}L^{2}} \|f^{L}\|_{L^{1}_{k}L^{2}_{T}L^{2}_{0,\gamma/2}}.
\eeno
For the term
$\sum_{p \in \mathbb{Z}^{3}} \epsilon^{-1}(\Gamma^{\epsilon}-\Gamma^{L})( \widehat{f^{\epsilon}}(k-p), \widehat{f^{L}}(p))$, by Lemma \ref{estimate-operator-difference},
we have
\beno
&&\sum_{k \in \mathbb{Z}^{3}} \left( \int_{0}^{T} |\langle \sum_{p \in \mathbb{Z}^{3}} \epsilon^{-1}(\Gamma^{\epsilon}-\Gamma^{L})( \widehat{f^{\epsilon}}(k-p), \widehat{f^{L}}(p)), \widehat{F^{\epsilon}_{R}}(k) \rangle| d t \right)^{\frac{1}{2}}
\\ &\lesssim& \eta \|F^{\epsilon}_{R}\|_{L^{1}_{k}L^{2}_{T}L^{2}_{0,\gamma/2}} + \frac{1}{4 \eta} \|f^{\epsilon}\|_{L^{1}_{k}L^{\infty}_{T}H^{3}} \|f^{L}\|_{L^{1}_{k}L^{2}_{T}H^{3}_{9+\gamma/2}}.
\eeno
Finally, for the term $\epsilon^{-1} (\mathcal{L}^{L}-\mathcal{L}^{\epsilon}) \widehat{f^{\epsilon}}(k)$, by Lemma \ref{estimate-operator-difference},
we have
\beno
\sum_{k \in \mathbb{Z}^{3}} \left( \int_{0}^{T} |\langle \epsilon^{-1} (\mathcal{L}^{L}-\mathcal{L}^{\epsilon}) \widehat{f^{\epsilon}}(k), \widehat{F^{\epsilon}_{R}}(k) \rangle| d t \right)^{\frac{1}{2}}
\lesssim \eta \|F^{\epsilon}_{R}\|_{L^{1}_{k}L^{2}_{T}L^{2}_{0,\gamma/2}} + \frac{1}{4 \eta}  \|f^{\epsilon}\|_{L^{1}_{k}L^{2}_{T}H^{3}_{9+\gamma/2}}.
\eeno
Combining the above estimates gives
\ben
\|F^{\epsilon}_{R}\|_{L^{1}_{k}L^{\infty}_{T}L^{2}} + \|(\mathbb{I}-\mathbb{P})F^{\epsilon}_{R}\|_{L^{1}_{k}L^{2}_{T}L^{2}_{0,\gamma/2}} \leq
\eta \|F^{\epsilon}_{R}\|_{L^{1}_{k}L^{2}_{T}L^{2}_{0,\gamma/2}} +
\frac{C_{2}}{\eta}\|f^{\epsilon}\|_{L^{1}_{k}L^{\infty}_{T}L^{2}} \|F^{\epsilon}_{R}\|_{L^{1}_{k}L^{2}_{T}L^{2}_{0,\gamma/2}}
\nonumber \\+\frac{C_{2}}{\eta}\|F^{\epsilon}_{R}\|_{L^{1}_{k}L^{\infty}_{T}L^{2}} \|f^{L}\|_{L^{1}_{k}L^{2}_{T}L^{2}_{0,\gamma/2}}
+ \frac{C_{2}}{\eta}\|f^{\epsilon}\|_{L^{1}_{k}L^{\infty}_{T}H^{3}} \|f^{L}\|_{L^{1}_{k}L^{2}_{T}H^{3}_{9+\gamma/2}}
+ \frac{C_{2}}{\eta}\|f^{\epsilon}\|_{L^{1}_{k}L^{2}_{T}H^{3}_{9+\gamma/2}}. \label{marco-part}
\een

{\it{Step 3: Micro-Macro components.}}
The combination $\eqref{mirco-part} \times \frac{1}{2C_{1}} + \eqref{marco-part}$ gives
\beno
&&\frac{1}{2}\|F^{\epsilon}_{R}\|_{L^{1}_{k}L^{\infty}_{T}L^{2}} + \frac{1}{2}\|(\mathbb{I}-\mathbb{P})F^{\epsilon}_{R}\|_{L^{1}_{k}L^{2}_{T}L^{2}_{0,\gamma/2}} + \frac{1}{2C_{1}} \|[a,b,c]\|_{L^{1}_{k}L^{2}_{T}}
\\ &\leq& \eta \|F^{\epsilon}_{R}\|_{L^{1}_{k}L^{2}_{T}L^{2}_{0,\gamma/2}}
 +
(\frac{C_{2}}{\eta} + \frac{1}{2})\|f^{\epsilon}\|_{L^{1}_{k}L^{\infty}_{T}L^{2}} \|F^{\epsilon}_{R}\|_{L^{1}_{k}L^{2}_{T}L^{2}_{0,\gamma/2}}
+(\frac{C_{2}}{\eta} + \frac{1}{2})\|F^{\epsilon}_{R}\|_{L^{1}_{k}L^{\infty}_{T}L^{2}} \|f^{L}\|_{L^{1}_{k}L^{2}_{T}L^{2}_{0,\gamma/2}}
 \\ &&+ (\frac{C_{2}}{\eta} + \frac{1}{2})\|f^{\epsilon}\|_{L^{1}_{k}L^{\infty}_{T}H^{3}} \|f^{L}\|_{L^{1}_{k}L^{2}_{T}H^{3}_{9+\gamma/2}}
+ (\frac{C_{2}}{\eta} + \frac{1}{2})\|f^{\epsilon}\|_{L^{1}_{k}L^{2}_{T}H^{3}_{9+\gamma/2}}.
\eeno
Note that
\beno
\frac{1}{2}\|(\mathbb{I}-\mathbb{P})F^{\epsilon}_{R}\|_{L^{1}_{k}L^{2}_{T}L^{2}_{0,\gamma/2}} + \frac{1}{2C_{1}} \|[a,b,c]\|_{L^{1}_{k}L^{2}_{T}}
\geq c_{1} \|F^{\epsilon}_{R}\|_{L^{1}_{k}L^{2}_{T}L^{2}_{0,\gamma/2}}.
\eeno
Then by choosing $\eta = \frac{c_{1}}{2}$, we have
\beno
&&\frac{1}{2}\|F^{\epsilon}_{R}\|_{L^{1}_{k}L^{\infty}_{T}L^{2}} +  \frac{c_{1}}{2}\|F^{\epsilon}_{R}\|_{L^{1}_{k}L^{2}_{T}L^{2}_{0,\gamma/2}}
 \\ &\leq&
(\frac{2C_{2}}{c_{1}} + \frac{1}{2})\|f^{\epsilon}\|_{L^{1}_{k}L^{\infty}_{T}L^{2}} \|F^{\epsilon}_{R}\|_{L^{1}_{k}L^{2}_{T}L^{2}_{0,\gamma/2}}
+(\frac{2C_{2}}{c_{1}} + \frac{1}{2})\|F^{\epsilon}_{R}\|_{L^{1}_{k}L^{\infty}_{T}L^{2}} \|f^{L}\|_{L^{1}_{k}L^{2}_{T}L^{2}_{0,\gamma/2}}
\\&&+ (\frac{2C_{2}}{c_{1}} + \frac{1}{2})\|f^{\epsilon}\|_{L^{1}_{k}L^{\infty}_{T}H^{3}} \|f^{L}\|_{L^{1}_{k}L^{2}_{T}H^{3}_{9+\gamma/2}}
+ (\frac{2C_{2}}{c_{1}} + \frac{1}{2})\|f^{\epsilon}\|_{L^{1}_{k}L^{2}_{T}H^{3}_{9+\gamma/2}}.
\eeno

Under the smallness assumption on $\|f_{0}\|_{L^{1}_{k}L^{2}}$, by \eqref{lowest-regularity-bounded-by-initial},
we have
\beno
&&\|F^{\epsilon}_{R}\|_{L^{1}_{k}L^{\infty}_{T}L^{2}} + \|F^{\epsilon}_{R}\|_{L^{1}_{k}L^{2}_{T}L^{2}_{0,\gamma/2}}
\\&\lesssim&
\|f^{\epsilon}\|_{L^{1}_{k}L^{2}_{T}H^{3}_{9+\gamma/2}} + \|f^{\epsilon}\|_{L^{1}_{k}L^{\infty}_{T}H^{3}} \|f^{L}\|_{L^{1}_{k}L^{2}_{T}H^{3}_{9+\gamma/2}}
\lesssim P_{3}(\|f_{0}\|_{0,3,9})(1+P_{3}(\|f_{0}\|_{0,3,9})),
\eeno
where we have used  the propagation estimates  \eqref{general-propagation-f-epsilon} and \eqref{general-propagation-f-epsilon-landau} with $m=0,n=3,l=9$. Note that
 $F^{\epsilon}_{R} =  \epsilon^{-1} (f^{\epsilon}-f^{L})$ and this gives \eqref{global-in-time-error}.
\end{proof}

\section{Propagation of moment and decay transition} \label{decay-rate}
With Lemma \ref{general-weight-upper-bound}
and the proof of Theorem 2.1 in \cite{duan2021global}, we have the propagation of moment
stated in the following theorem.
\begin{thm}\label{propagation-of-moment} Under the assumptions in Theorem \ref{asymptotic-result},
	let $l,q \geq 0, -2 \leq \gamma \leq 0$. There is a constant $\delta_{1}>0$, such that if
$\|W_{l,q}f_{0}\|_{L^{1}_{k}L^{2}} \leq \delta_{1},$
then the solution $f^{\epsilon}$ to the Boltzmann equation \eqref{Cauchy-linearizedBE-grazing} satisfies
\beno
\|W_{l,q}f^{\epsilon}\|_{L^{1}_{k}L^{\infty}_{T}L^{2}} \lesssim \|W_{l,q}f_{0}\|_{L^{1}_{k}L^{2}}.
\eeno
\end{thm}

We are now ready  to prove Theorem \ref{decay-rate-consistency}.

\begin{proof}[Proof of Theorem \ref{decay-rate-consistency}] Consider
\ben \label{solution-into-equqation}
\partial_{t} f^{\epsilon} + v \cdot \nabla_{x} f^{\epsilon} + \mathcal{L}^{\epsilon} f^{\epsilon} = \Gamma^{\epsilon}(f^{\epsilon}, f^{\epsilon}).
\een
Recall $T_{\epsilon}=(\frac{1}{\epsilon})^{\frac{2(1-s)}{|\gamma+2s|}}, \kappa=\frac{1}{1+|\gamma+2s|}$,
and the function $A_{\epsilon}(t)$ defined in \eqref{auxiliary-function-t} as
\beno
A_{\epsilon}(t) = \zeta(T_{\epsilon}^{-1}t) t + (1-\zeta(T_{\epsilon}^{-1}t))(\frac{t}{\epsilon^{2(1-s)}})^{\kappa}.
\eeno
Here the function $\zeta$ is defined in \eqref{zeta-property}. Multiplying \eqref{solution-into-equqation} by $g(t):= \exp(\lambda A_{\epsilon}(t))$, with $h^{\epsilon}(t) := g(t) f^{\epsilon}(t)$, we have
\beno
\partial_{t} h^{\epsilon} + v \cdot \nabla_{x} h^{\epsilon} + \mathcal{L}^{\epsilon} h^{\epsilon} = \Gamma^{\epsilon}(f^{\epsilon},h^{\epsilon}) + \lambda A_{\epsilon}^{\prime}(t) h.
\eeno
Similar to Eq.(6.3) in \cite{duan2021global},
we get
\beno
&&\sum_{k \in \mathbb{Z}^{3}} \sup _{0 \leq t \leq T}\|\widehat{h}(t, k)\|_{L^{2}} + \sqrt{\lambda_{0}}\sum_{k \in \mathbb{Z}^{3}}\left(\int_{0}^{T}|\widehat{h}(t, k)|_{\epsilon,\gamma/2}^{2} d t\right)^{\frac{1}{2}}  \\
\\&\lesssim& \sum_{k \in \mathbb{Z}^{3}} \|\widehat{f}_{0}(k)\|_{L^{2}} + \sqrt{\lambda} \sum_{k \in \mathbb{Z}^{3}}\left(\int_{0}^{T} A_{\epsilon}^{\prime}(t) \|\widehat{h}(t, k)\|_{L^{2}}^{2} d t\right)^{\frac{1}{2}} .
\eeno

Now we use dissipation and propagation of moment to cancel the last term. Note $0 \leq A_{\epsilon}^{\prime}(t) \lesssim 1$ and
\beno
|\widehat{h}(t, k)|_{\epsilon,\gamma/2}^{2} \geq |\langle v \rangle^{1+\gamma/2} \zeta \widehat{h}(t, k)|_{L^{2}}^{2} \geq |\zeta \widehat{h}(t, k)|_{L^{2}}^{2},
\eeno
which gives for $\lambda/\lambda_{0} \ll 1$,
\ben \nonumber
&&\sum_{k \in \mathbb{Z}^{3}} \sup _{0 \leq t \leq T}\|\widehat{h}(t, k)\|_{L^{2}} + \sqrt{\lambda_{0}}\sum_{k \in \mathbb{Z}^{3}}\left(\int_{0}^{T}|\widehat{h}(t, k)|_{\epsilon,\gamma/2}^{2} d t\right)^{\frac{1}{2}}
\\&\lesssim& \sum_{k \in \mathbb{Z}^{3}} \|\widehat{f}_{0}(k)\|_{L^{2}} + \sqrt{\lambda} \sum_{k \in \mathbb{Z}^{3}}\left(\int_{0}^{T} A_{\epsilon}^{\prime}(t) \|(1-\zeta)\widehat{h}(t, k)\|_{L^{2}}^{2} d t\right)^{\frac{1}{2}} . \label{starting-point-of-decay}
\een

We claim that for any $T>0$,
\ben \label{conclusion}
\sum_{k \in \mathbb{Z}^{3}} \sup _{0 \leq t \leq T}\|\widehat{h}(t, k)\|_{L^{2}} + \sqrt{\lambda_{0}} \sum_{k \in \mathbb{Z}^{3}}\left(\int_{0}^{T}|\widehat{h}(t, k)|_{\epsilon,\gamma/2}^{2} d t\right)^{\frac{1}{2}}  \lesssim \|W_{0,q}f_{0}\|_{L^{1}_{k}L^{2}}.
\een

Note that the support of $(1-\zeta)$ is $|v| \geq \frac{1}{2\epsilon}$. We prove the Claim by considering three cases.

{\it { Case 1, $T \leq \frac{1}{2\epsilon}$.}} In this case,   $t \leq T \leq \frac{1}{2\epsilon} \leq |v|$. With the fact
that  $A_{\epsilon}(t) \leq 2t$,
by the propagation of moment in Theorem \ref{propagation-of-moment}, we have
\ben
&& \sum_{k \in \mathbb{Z}^{3}}\left(\int_{0}^{T} A_{\epsilon}^{\prime}(t) \|(1-\zeta)\widehat{h}(t, k)\|_{L^{2}}^{2} d t\right)^{\frac{1}{2}}
\nonumber \\ &\lesssim&
\sum_{k \in \mathbb{Z}^{3}}\left(\int_{0}^{T} \exp(4\lambda t-2qt)\|(1-\zeta)\exp(q\langle v \rangle)\widehat{f}(t, k)\|_{L^{2}}^{2} d t\right)^{\frac{1}{2}}
\nonumber \\ &\lesssim& C(q,\lambda)
\sum_{k \in \mathbb{Z}^{3}} \|(1-\zeta)\exp(q\langle v \rangle)\widehat{f}(t, k)\|_{L^{\infty}_{T}L^{2}}^{2}
\lesssim C(q,\lambda) \|\exp(q\langle v \rangle)f_{0}\|_{L^{1}_{k}L^{2}}, \label{T-case1}
\een
because $q>2\lambda$. Plugging \eqref{T-case1} into \eqref{starting-point-of-decay},
we have \eqref{conclusion} for $T \leq \frac{1}{2\epsilon}$.

{\it { Case 2, $\frac{1}{2\epsilon} \leq T \leq T_{\epsilon}$.}} In view of \eqref{T-case1}, we only need to consider
\beno
\sum_{k \in \mathbb{Z}^{3}}\left(\int_{\frac{1}{2\epsilon}}^{T} A_{\epsilon}^{\prime}(t) \|(1-\zeta)\widehat{h}(t, k)\|_{L^{2}}^{2} d t\right)^{\frac{1}{2}}.
\eeno

The domain of $\frac{1}{2 \epsilon} \leq t \leq T, v \in \mathbb{R}^{3}$ can be divided into two parts:
\beno
\mathrm{D}_{1} :=\{ \langle v \rangle \leq T_{\epsilon}\}, \quad \mathrm{D}_{2} :=\{ \langle v \rangle > T_{\epsilon}\}.
\eeno
In $\mathrm{D}_{1}$, recalling $T_{\epsilon}=(\frac{1}{\epsilon})^{\frac{2(1-s)}{|\gamma+2s|}}$, we have
$\langle v \rangle^{\gamma+2s} \epsilon^{2(1-s)} \geq 1$ so that
\ben
\sqrt{\lambda} \sum_{k \in \mathbb{Z}^{3}}\left(\int_{\frac{1}{2\epsilon}}^{T}  A_{\epsilon}^{\prime}(t) \|\mathrm{1}_{\mathrm{D}_{1}}(1-\zeta)\widehat{h}(t, k)\|_{L^{2}}^{2} d t\right)^{\frac{1}{2}}
 \lesssim \sqrt{\lambda} \sum_{k \in \mathbb{Z}^{3}}\left(\int_{\frac{1}{2\epsilon}}^{T}   \|\widehat{h}(t, k)\|_{\epsilon,\gamma/2}^{2} d t\right)^{\frac{1}{2}} . \label{T-case2-dissipation}
\een
Note that by taking $\lambda$ small enough such that $\lambda/\lambda_{0} \ll 1$,  this can be absorbed by the dissipation. In $\mathrm{D}_{2}$, since $\langle v \rangle > T_{\epsilon} \geq T \geq t$, similar to \eqref{T-case1}, by using Theorem \ref{propagation-of-moment},
we have
\ben
&& \sum_{k \in \mathbb{Z}^{3}}\left(\int_{\frac{1}{2 \epsilon}}^{T} A_{\epsilon}^{\prime}(t) \|\mathrm{1}_{\mathrm{D}_{2}}(1-\zeta)\widehat{h}(t, k)\|_{L^{2}}^{2} d t\right)^{\frac{1}{2}}
\nonumber \\ &\lesssim&
\sum_{k \in \mathbb{Z}^{3}}\left(\int_{\frac{1}{2 \epsilon}}^{T} \exp(4\lambda t-2qt)\|(1-\zeta)\exp(q\langle v \rangle)\widehat{f}(t, k)\|_{L^{2}}^{2} d t\right)^{\frac{1}{2}}
 \lesssim C(q,\lambda) \|\exp(q\langle v \rangle)f_{0}\|_{L^{1}_{k}L^{2}}. \label{T-case2-v-for-t}
\een
Plugging \eqref{T-case2-dissipation} and \eqref{T-case2-v-for-t} into \eqref{starting-point-of-decay},
we have \eqref{conclusion} for $\frac{1}{2\epsilon} \leq T \leq T_{\epsilon}$.

{\it {Case 3, $T > T_{\epsilon}$.}}
In view of \eqref{T-case1}, \eqref{T-case2-dissipation} and \eqref{T-case2-v-for-t}, we only need to consider
\beno
\sum_{k \in \mathbb{Z}^{3}}\left(\int_{T_{\epsilon}}^{T} A_{\epsilon}^{\prime}(t) \|(1-\zeta)\widehat{h}(t, k)\|_{L^{2}}^{2} d t\right)^{\frac{1}{2}} ,
\eeno
where $A_{\epsilon}^{\prime}(t)$ gives some decay since $0< \kappa <1$,
\beno
A_{\epsilon}^{\prime}(t) = \kappa (\frac{t}{\epsilon^{2(1-s)}})^{\kappa-1} \frac{1}{\epsilon^{2(1-s)}}.
\eeno
Note that when $t \geq T_{\epsilon}, A_{\epsilon}(t)=(\frac{t}{\epsilon^{2(1-s)}})^{\kappa}$.

Divide the domain  $T_{\epsilon} \leq t \leq T, v \in \mathbb{R}^{3}$  into two parts,
\beno
\mathrm{D}_{3} :=\{ \langle v \rangle^{\gamma+2s} \geq (\frac{t}{\epsilon^{2(1-s)}})^{\kappa-1} \}, \quad \mathrm{D}_{4} :=\{ \langle v \rangle^{\gamma+2s} < (\frac{t}{\epsilon^{2(1-s)}})^{\kappa-1} \}.
\eeno
In $\mathrm{D}_{3}$, we have
\beno
\langle v \rangle^{\gamma+2s} \frac{1}{\epsilon^{2(1-s)}} \geq (\frac{t}{\epsilon^{2(1-s)}})^{\kappa-1} \frac{1}{\epsilon^{2(1-s)}} \geq A_{\epsilon}^{\prime}(t)
\eeno
so that
\ben
\sqrt{\lambda} \sum_{k \in \mathbb{Z}^{3}}\left(\int_{T_{\epsilon}}^{T}  A_{\epsilon}^{\prime}(t) \|\mathrm{1}_{\mathrm{D}_{3}}(1-\zeta)\widehat{h}(t, k)\|_{L^{2}}^{2} d t\right)^{\frac{1}{2}}
\lesssim \sqrt{\lambda} \sum_{k \in \mathbb{Z}^{3}}\left(\int_{T_{\epsilon}}^{T}   \|\widehat{h}(t, k)\|_{\epsilon,\gamma/2}^{2} d t\right)^{\frac{1}{2}} . \label{T-case3-dissipation}
\een
Note that by taking $\lambda$ small enough such that $\lambda/\lambda_{0} \ll 1$,  this  can be absorbed by the dissipation.
In $\mathrm{D}_{4}$, since $\langle v \rangle > (\frac{t}{\epsilon^{2(1-s)}})^{\kappa}$, similar to \eqref{T-case1},   by using Theorem \ref{propagation-of-moment},
we have
\ben
&&\sum_{k \in \mathbb{Z}^{3}}\left(\int_{T_{\epsilon}}^{T} A_{\epsilon}^{\prime}(t) \|\mathrm{1}_{\mathrm{D}_{4}}(1-\zeta)\widehat{h}(t, k)\|_{L^{2}}^{2} d t\right)^{\frac{1}{2}}
\nonumber \\ &=&
\sum_{k \in \mathbb{Z}^{3}}\left(\int_{T_{\epsilon}}^{T} A_{\epsilon}^{\prime}(t) \exp(2\lambda(\frac{t}{\epsilon^{2(1-s)}})^{\kappa}) \|\mathrm{1}_{\mathrm{D}_{4}}(1-\zeta)\widehat{f}(t, k)\|_{L^{2}}^{2} d t\right)^{\frac{1}{2}}
\nonumber \\ &\lesssim&
\sum_{k \in \mathbb{Z}^{3}}\left(\int_{T_{\epsilon}}^{T} A_{\epsilon}^{\prime}(t) \exp(2(\lambda-q)(\frac{t}{\epsilon^{2(1-s)}})^{\kappa})\|(1-\zeta)\exp(q\langle v \rangle)\widehat{f}(t, k)\|_{L^{2}}^{2} d t\right)^{\frac{1}{2}}
\nonumber \\ &\lesssim& C(q,\lambda,\kappa)
\sum_{k \in \mathbb{Z}^{3}} \|(1-\zeta)\exp(q\langle v \rangle)\widehat{f}(t, k)\|_{L^{\infty}_{T}L^{2}}^{2}
 \lesssim C(q,\lambda,\kappa) \|\exp(q\langle v \rangle)f_{0}\|_{L^{1}_{k}L^{2}}, \label{T-case3-v-for-t}
\een
where for $0<\kappa <1,  \lambda< q$, we have used the following estimate
\beno
\int_{T_{\epsilon}}^{T} A_{\epsilon}^{\prime}(t) \exp(2(\lambda-q)(\frac{t}{\epsilon^{2(1-s)}})^{\kappa})  d t
&=& \int_{T_{\epsilon}}^{T} \kappa (\frac{t}{\epsilon^{2(1-s)}})^{\kappa-1} \frac{1}{\epsilon^{2(1-s)}} \exp(2(\lambda-q)(\frac{t}{\epsilon^{2(1-s)}})^{\kappa})  d t
\\&\leq& \int_{0}^{\infty} z^{\kappa-1} \exp(2(\lambda-q)z^{\kappa})  d t \lesssim C(q,\lambda,\kappa).
\eeno
Plugging \eqref{T-case3-dissipation} and \eqref{T-case3-v-for-t} into \eqref{starting-point-of-decay}, we have \eqref{conclusion} for $T > T_{\epsilon}$.

Since $h^{\epsilon}(t) := \exp(\lambda A_{\epsilon}(t)) f^{\epsilon}(t)$, by \eqref{conclusion}, we get
\beno
\|f^{\epsilon}(t)\|_{L^{1}_{k}L^{2}} \lesssim \exp(-\lambda A_{\epsilon}(t))\|\exp(q\langle v \rangle)f_{0}\|_{L^{1}_{k}L^{2}},
\eeno
By the definition of $A_{\epsilon}$, we obtain \eqref{nice-decay-connection}.
\end{proof}

\section{Appendix} \label{appendix}
\begin{lem}[Lemma 4.1 of \cite{he2018sharp}] \label{aniso-from-He-Sharp-Bounds} For any function $f$ defined on $\mathbb{S}^2$, it holds
\beno
(1-s)\epsilon^{2s-2}\int_{\mathbb{S}^2\times\mathbb{S}^2}\frac{|f(\sigma)-f(\tau)|^{2}}{|\sigma-\tau|^{2+2s}}\mathrm{1}_{|\sigma-\tau| \leq \epsilon} d\sigma d\tau + |f|^{2}_{L^{2}(\mathbb{S}^{2})}
\sim |W^{\epsilon}((-\Delta_{\mathbb{S}^{2}})^{\frac{1}{2}})f|^{2}_{L^{2}(\mathbb{S}^{2})} + |f|^{2}_{L^{2}(\mathbb{S}^{2})}.
\eeno
As consequence, for any function $f$ defined on $\mathbb{R}^3$, it holds
\ben\label{anisotropic-R-3}
(1-s) \epsilon^{2s-2}\int_{\mathbb{R}_{+}\times\mathbb{S}^2\times\mathbb{S}^2}\frac{|f(r\sigma)-f(r\tau)|^{2}}{|\sigma-\tau|^{2+2s}} \mathrm{1}_{|\sigma-\tau| \leq \epsilon} r^{2}d\sigma d\tau dr + |f|^{2}_{L^{2}}
\sim |W^{\epsilon}((-\Delta_{\mathbb{S}^{2}})^{\frac{1}{2}})f|^{2}_{L^{2}} + |f|^{2}_{L^{2}}.
\een
\end{lem}

\begin{rmk} \label{another-factor-equivalent}
Lemma \ref{aniso-from-He-Sharp-Bounds} also  holds  if we replace $\mathrm{1}_{|\sigma-\tau| \leq \epsilon}$ by $\mathrm{1}_{|\sigma-\tau| \leq 2\epsilon}$.
\end{rmk}

Similar to Lemma 5.8 in \cite{he2018sharp}, we have
\begin{lem}\label{comWep} Let $\mathcal{F}$ be Fourier transform operator, then $\mathcal{F}W^\epsilon((-\triangle_{\mathbb{S}^2})^{\frac{1}{2}})=
W^\epsilon((-\triangle_{\mathbb{S}^2})^{\frac{1}{2}})\mathcal{F}$.
\end{lem}

\begin{prop}\label{symbol}
Suppose $ E^\epsilon(\xi):= \frac{1}{4 \pi}\int_{\mathbb{S}^2} b^\epsilon(\f{\xi}{|\xi|}\cdot \sigma)\min\{ |\xi|^2\sin^2(\theta/2),1\} d\sigma$. Then we have
\beno E^\epsilon(\xi)  = \mathrm{1}_{|\xi| \leq \epsilon^{-1}} |\xi|^{2} + \mathrm{1}_{|\xi| > \epsilon^{-1}} \epsilon^{2s-2} [ \frac{1}{s}(|\xi|^{2s} - \epsilon^{-2s}) +\epsilon^{-2s} ].\eeno
As a result, we have
\beno E^\epsilon(\xi) + 1  \sim  (W^{\epsilon})^{2}(\xi),\eeno
where $W^{\epsilon}$ is defined in \eqref{characteristic-function}. Here the constants in $\sim$ may depend on  $s$.
\end{prop}
\begin{proof} Recalling \eqref{definition-angular-part} and $d\sigma = 4 \sin(\theta/2) d \mathbb{S} d \sin(\theta/2)$,
we have
\beno E^\epsilon(\xi)=2(1-s)\epsilon^{2s-2} \int_0^{\pi} \sin^{-1-2s}(\theta / 2) \sin\theta \mathrm{1}_{\sin(\theta/2) \leq \epsilon} \min\{|\xi|^2\sin^2(\theta/2),1\} d \sin(\theta/2). \eeno
By the change of variable: $t=\sin(\theta/2)$, we have
\beno E^\epsilon(\xi) = 2 (1-s)\epsilon^{2s-2} \int_0^{\epsilon} t^{-1-2s} \min\{ |\xi|^2t^2,1\}dt.
\eeno
When $|\xi| \leq \epsilon^{-1}$, we have
\beno E^\epsilon(\xi) = 2 (1-s)\epsilon^{2s-2} |\xi|^2 \int_0^{\epsilon} t^{1-2s} dt = |\xi|^2.
\eeno
When $|\xi| > \epsilon^{-1}$, we have
\beno E^\epsilon(\xi) &=& 2 (1-s)\epsilon^{2s-2} |\xi|^2 \int_0^{|\xi|^{-1}} t^{1-2s} dt +
2 (1-s)\epsilon^{2s-2} \int_{|\xi|^{-1}}^{\epsilon} t^{-1-2s} dt
\\&=& \epsilon^{2s-2} [ \frac{1}{s}(|\xi|^{2s} - \epsilon^{-2s}) +\epsilon^{-2s} ].
\eeno
The proof is completed by combining the above two cases.
\end{proof}

We now recall the definition of the symbol class $S^{m}_{1,0}$.
\begin{defi} A smooth function $a(v,\xi)$ is a symbol of type $S^{m}_{1,0}$ if   $a(v,\xi)$  satisfies for any multi-indices $\alpha$ and $\beta$,
\beno |(\pa^\alpha_\xi\pa^\beta_v a)(v,\xi)|\le C_{\alpha,\beta} \langle \xi\rangle^{m-|\alpha|}, \eeno
where $C_{\alpha,\beta}$ is a constant depending only on   $\alpha$ and $\beta$.
\end{defi}

\begin{lem}[Lemma 5.3 in \cite{he2018sharp}]\label{operatorcommutator1}
Let $l, s, r \in \R, M \in S^{r}_{1,0}, \Phi \in S^{l}_{1,0}$. The estimate
$|[M(D), \Phi]f|_{H^{s}} \lesssim |f|_{H^{r+s-1}_{l-1}}$ holds.
\end{lem}

We now recall the dyadic
decomposition. Let $B_{4/3} := \{x\in\R^{3}: |x| \leq 4/3\}$ and $C := \{x\in\R^{3}: 3/4 \leq |x| \leq 8/3\}$.  Denote two
radial functions $\phi \in C_{0}^{\infty}(B_{4/3})$ and $\psi \in C_{0}^{\infty}(C)$ which satisfy
\ben \label{dyadic-decomposition-def}
 0\leq \phi, \psi \leq 1, \text{ and } \phi(x) + \sum_{j \geq 0} \psi(2^{-j}x) =1, \text{ for all } x \in \R^{3}. \een
Set $\varphi_{-1}(x) :=  \phi(x)$ and $\varphi_{j}(x) :=  \psi(2^{-j}x)$ for any $x \in \R^{3}$ and $j \geq 0$. Then the  dyadic decomposition
$ f = \sum_{j=-1}^{\infty} \varphi_{j}f $ holds
for any function defined on $\R^{3}$.

\begin{prop}(Proposition 5.2 in \cite{he2018asymptotic}) \label{fourier-transform-cross-term} It holds
	\beno
	\int_{\R^3\times\mathbb{S}^2} b(\f{u}{|u|}\cdot \sigma) h(u)(f(u^+)-f(\f{|u|}{|u^+|}u^+)) d\sigma du
	=\int_{\R^3\times\mathbb{S}^2} b(\f{\xi}{|\xi|}\cdot \sigma)  (\hat{h}(\xi^+)-\hat{h}(\f{|\xi|}{|\xi^+|}\xi^+))\bar{\hat{f}}(\xi) d\sigma d\xi.
	\eeno
\end{prop}

\begin{lem}\label{crosstermsimilar}
	Let
	$
	\mathcal{Y}^{\epsilon,\gamma}(h,f) := \int b^{\epsilon}(\frac{u}{|u|}\cdot\sigma)\langle u \rangle^{\gamma} h(u)[f(u^{+}) - f(|u|\frac{u^{+}}{|u^{+}|})] du d\sigma,
	$ where $u^{+} = \frac{u + |u|\sigma}{2}$,
	then
	\beno
	|\mathcal{Y}^{\epsilon,\gamma}(h,f)| \lesssim (|W^{\epsilon}W_{\gamma/2}h|_{L^{2}}+|W^{\epsilon}(D)W_{\gamma/2}h|_{L^{2}})
	(|W^{\epsilon}W_{\gamma/2}f|_{L^{2}}+|W^{\epsilon}(D)W_{\gamma/2}f|_{L^{2}}).
	\eeno
\end{lem}
\begin{proof}We divide the proof into two steps.
	
{\it Step 1: $\gamma = 0$.} Since the support of  $b^{\epsilon}$ is  in $\frac{u}{|u|}\cdot\sigma \geq 0$, we get $|u|/\sqrt{2} \leq |u^{+}| \leq |u|$. Recalling the function $\zeta$ in \eqref{zeta-property}, we define $\zeta_{4}(\cdot):=\zeta(\frac{\cdot}{4})$.
We apply  the following decomposition
\ben \label{decomposition-f}
 \mathcal{Y}^{\epsilon, 0}(h,f) &=& \mathcal{Y}^{\epsilon, 0}(h, \zeta(\epsilon v)f) + \mathcal{Y}^{\epsilon, 0}(h, (1-\zeta(\epsilon v))f)
\nonumber \\&=& \mathcal{Y}^{\epsilon, 0}(\zeta_{4}(\epsilon v)h, \zeta(\epsilon v)f) + \mathcal{Y}^{\epsilon, 0}(h, (1-\zeta(\epsilon v))f).
\een
Note that the second equality is ensured by the definition of $\zeta$ and the fact $|u|/\sqrt{2} \leq |u^{+}| \leq |u|$.
The first term in \eqref{decomposition-f} can be decomposed further as
\ben \label{decomposition-h}
\mathcal{Y}^{\epsilon, 0}(\zeta_{4}(\epsilon v)h, \zeta(\epsilon v)f)
= \mathcal{Y}^{\epsilon, 0}(\zeta(\epsilon D) \zeta_{4}(\epsilon v)h, \zeta(\epsilon v)f)
+\mathcal{Y}^{\epsilon, 0}((1-\zeta(\epsilon D))\zeta_{4}(\epsilon v)h, \zeta(\epsilon v)f).
\een

{\it Step 1.1: Estimate of $\mathcal{Y}^{\epsilon, 0}(\zeta(\epsilon D)\zeta_{4}(\epsilon v)h, \zeta(\epsilon v)f)$.} By Proposition \ref{fourier-transform-cross-term} and the fact $|\xi|/\sqrt{2} \leq |\xi^{+}| \leq |\xi|$, we have
\beno
&&\mathcal{Y}^{\epsilon, 0}(\zeta(\epsilon D)\zeta_{4}(\epsilon v)h, \zeta(\epsilon v)f) =\mathcal{Y}^{\epsilon, 0}(\zeta(\epsilon D)\zeta_{4}(\epsilon v)h, \zeta_{4}(\epsilon D)\zeta(\epsilon v)f)
\\&=& \int b^{\epsilon}(\frac{u}{|u|}\cdot\sigma) (\zeta(\epsilon D)\zeta_{4}(\epsilon v)h)(u) [(\zeta_{4}(\epsilon D)\zeta(\epsilon v)f)(u^{+})- (\zeta_{4}(\epsilon D)\zeta(\epsilon v)f)(|u|\frac{u^{+}}{|u^{+}|})] du d\sigma.
	\eeno
By Taylor expansion,
	\beno
	(\zeta_{4}(\epsilon D)\zeta(\epsilon v)f)(u^{+})- (\zeta_{4}(\epsilon D)\zeta(\epsilon v)f)(|u|\frac{u^{+}}{|u^{+}|}) = (1-\frac{1}{\cos\theta})\int_{0}^{1}
	(\nabla (\zeta_{4}(\epsilon D)\zeta(\epsilon v)f))(u^{+}(\kappa))\cdot u^{+} d\kappa,
	\eeno
	where $u^{+}(\kappa) = (1-\kappa)|u|\frac{u^{+}}{|u^{+}|} + \kappa u^{+}$. Then by Cauchy-Schwartz inequality, we get
\ben \label{low-xi-low-v--low-xi-low-v}
&&\mathcal{Y}^{\epsilon, 0}(\zeta(\epsilon D)\zeta_{4}(\epsilon v)h, \zeta(\epsilon v)f)
\nonumber \\&=& \int b^{\epsilon}(\frac{u}{|u|}\cdot\sigma)(1-\frac{1}{\cos\theta})
 (\zeta(\epsilon D)\zeta(\epsilon v)h)(u)
\int_{0}^{1}
	(\nabla (\zeta_{4}(\epsilon D)\zeta(\epsilon v)f))(u^{+}(\kappa))\cdot u^{+} d\kappa du d\sigma.
\nonumber \\&\lesssim& \left(\int b^{\epsilon}(\frac{u}{|u|}\cdot\sigma)\sin^{2}\frac{\theta}{2}
| (\zeta(\epsilon D)\zeta_{4}(\epsilon v)h)(u)|^{2} |u^{+}|^{2}  du d\sigma \right)^{\frac{1}{2}}
\nonumber \\ &&\times
\left(\int b^{\epsilon}(\frac{u}{|u|}\cdot\sigma)\sin^{2}\frac{\theta}{2}
	|(\nabla (\zeta_{4}(\epsilon D)\zeta(\epsilon v)f))(u^{+}(\kappa))|^{2}  du d\sigma \right)^{\frac{1}{2}}
\nonumber \\&\lesssim& |\zeta(\epsilon D)\zeta_{4}(\epsilon v)h|_{L^{2}_{1}}|\zeta_{4}(\epsilon D)\zeta(\epsilon v)f|_{H^{1}}
 \lesssim |W^{\epsilon}h|_{L^{2}} |W^{\epsilon}(D)f|_{L^{2}},
	\een
where we have used the fact that $|u^{+}| \sim |u|$,
the change of variable $u \rightarrow u^{+}(\kappa)$ and the estimate \eqref{order-2}.

{\it Step 1.2: Estimate of $\mathcal{Y}^{\epsilon, 0}( (1-\zeta(\epsilon D))\zeta_{4}(\epsilon v)h, \zeta(\epsilon v)f)$.}
By Proposition \ref{fourier-transform-cross-term} and  the dyadic decomposition in the frequency space, we have
	\ben \nonumber
&&\mathcal{Y}^{\epsilon, 0}((1-\zeta(\epsilon D))\zeta_{4}(\epsilon v)h, \zeta(\epsilon v)f)
 \\&=& \int b^{\epsilon}(\frac{\xi}{|\xi|}\cdot\sigma)  [((1-\zeta(\epsilon \xi))\widehat{\zeta_{4}(\epsilon \cdot)h})(\xi^{+})- ((1-\zeta(\epsilon \xi))\widehat{\zeta_{4}(\epsilon \cdot)h})(|\xi|\frac{\xi^{+}}{|\xi^{+}|})]
 \overline{\widehat{\zeta(\epsilon \cdot)f}}(\xi)  d\xi d\sigma
\nonumber	\\ \label{decomposition-in-frequency-space} &=& \sum_{l \geq [-\log_{2} \epsilon] - 4}  \mathcal{Y}_{l},
 \\ \mathcal{Y}_{l} &:=& \int b^{\epsilon}(\frac{\xi}{|\xi|}\cdot\sigma)[({\varphi}_{l}\widehat{\zeta_{4}(\epsilon \cdot)h})(\xi^{+})- ({\varphi}_{l}\widehat{\zeta_{4}(\epsilon \cdot)h})(|\xi|\frac{\xi^{+}}{|\xi^{+}|})] (\tilde{\varphi}_{l}\overline{\widehat{\zeta(\epsilon \cdot)f}})(\xi)  d\xi d\sigma , \nonumber
	\een
where $\tilde{\varphi}_{l} := \sum_{|j-l| \leq 2, j \geq -1} {\varphi}_{j}$.
Decompose $\mathcal{Y}_{l}= \mathcal{Y}_{l, \leq }  + \mathcal{Y}_{l, \geq }$ according to $\sin(\theta/2) \leq 2^{-l}$ and $\sin(\theta/2) \geq 2^{-l}$.
By Taylor expansion,
	$
	({\varphi}_{l}\widehat{\zeta_{4}(\epsilon \cdot)h})(\xi^{+})- ({\varphi}_{l}\widehat{\zeta_{4}(\epsilon \cdot)h})(|\xi|\frac{\xi^{+}}{|\xi^{+}|}) = (1-\frac{1}{\cos\theta})\int_{0}^{1}
	(\nabla {\varphi}_{l}\widehat{\zeta_{4}(\epsilon \cdot)h})(\xi^{+}(\kappa))\cdot \xi^{+} d\kappa,
	$
where $\xi^{+}(\kappa) = (1-\kappa)|\xi|\frac{\xi^{+}}{|\xi^{+}|} + \kappa \xi^{+}$. Then by Cauchy-Schwartz inequality, we obtain
	\ben \label{theta-leq-2-l}
\mathcal{Y}_{l, \leq} &=& \int b^{\epsilon}(\frac{\xi}{|\xi|}\cdot\sigma)\mathrm{1}_{\sin(\theta/2) \leq 2^{-l}}
(1-\frac{1}{\cos\theta}) (\int_{0}^{1}
	(\nabla {\varphi}_{l}\widehat{\zeta_{4}(\epsilon \cdot)h})(\xi^{+}(\kappa))\cdot \xi^{+} d\kappa )(\tilde{\varphi}_{l}\overline{\widehat{\zeta(\epsilon \cdot)f}})(\xi)  d\xi d\sigma
\nonumber \\&\lesssim& 2^{l} \bigg(\int b^{\epsilon}(\frac{\xi}{|\xi|}\cdot\sigma)\mathrm{1}_{\sin(\theta/2) \leq 2^{-l}}
\sin^{2}\frac{\theta}{2}
|(\nabla {\varphi}_{l}\widehat{\zeta_{4}(\epsilon \cdot)h})(\xi^{+}(\kappa))|^{2}   d\xi d\sigma \bigg)^{\frac{1}{2}}
\nonumber \\ &&\times
\bigg(\int b^{\epsilon}(\frac{\xi}{|\xi|}\cdot\sigma)\mathrm{1}_{\sin(\theta/2) \leq 2^{-l}}
\sin^{2}\frac{\theta}{2}
	|(\tilde{\varphi}_{l}\widehat{\zeta(\epsilon \cdot)f})(\xi)|^{2}  d\xi d\sigma \bigg)^{\frac{1}{2}}
\nonumber \\&\lesssim& \epsilon^{2s-2}2^{l(2s-1)}  |\nabla {\varphi}_{l}\widehat{\zeta_{4}(\epsilon \cdot)h}|_{L^{2}}
|\tilde{\varphi}_{l}\widehat{\zeta(\epsilon \cdot)f}|_{L^{2}}
\lesssim \epsilon^{s-1}2^{ls}  |\nabla {\varphi}_{l}\widehat{\zeta_{4}(\epsilon \cdot)h}|_{L^{2}}
|\tilde{\varphi}_{l}\widehat{\zeta(\epsilon \cdot)f}|_{L^{2}},\een
where we have used   the fact that $|\xi^{+}| \sim |\xi| \sim 2^{l}, 2^{l} \gtrsim \epsilon^{-1}$,
the change of variable $\xi \rightarrow \xi^{+}(\kappa)$ and the estimate
\beno
\int b^{\epsilon}(\frac{\xi}{|\xi|}\cdot\sigma)\sin^{2}\frac{\theta}{2} \mathrm{1}_{\sin(\theta/2) \leq 2^{-l}} d\sigma \lesssim (1-s) \epsilon^{2s-2} \int_{0}^{2^{-l}} t^{1-2s} dt \lesssim \epsilon^{2s-2}2^{l(2s-2)}.
\eeno
Since $|\nabla {\varphi}_{l}\widehat{\zeta_{4}(\epsilon \cdot)h}| =
|{\varphi}_{l}\widehat{\zeta_{4}(\epsilon \cdot)v h} + (\nabla {\varphi}_{l})\widehat{\zeta_{4}(\epsilon \cdot)h}| \lesssim
|{\varphi}_{l}\widehat{\zeta_{4}(\epsilon \cdot)v h}| + 2^{-l}|\tilde{\varphi}_{l}\widehat{\zeta_{4}(\epsilon \cdot)h}|
$, we have
\ben \label{sum-y-leq}
&&\sum_{l \geq [-\log_{2} \epsilon] - 4} \mathcal{Y}_{l, \leq} \lesssim
\sum_{l \geq [-\log_{2} \epsilon] - 4}
\epsilon^{s-1}2^{ls} (|{\varphi}_{l}\widehat{\zeta_{4}(\epsilon \cdot)v h}|_{L^{2}} + 2^{-2l}|\tilde{\varphi}_{l}\widehat{\zeta_{4}(\epsilon \cdot)h}|_{L^{2}})
|\tilde{\varphi}_{l}\widehat{\zeta(\epsilon \cdot)f}|_{L^{2}}
\nonumber \\ &\lesssim&
\bigg(\sum_{l \geq [-\log_{2} \epsilon] - 4}( |{\varphi}_{l}\widehat{\zeta_{4}(\epsilon \cdot)v h}|^{2}_{L^{2}} + 2^{-l}|\tilde{\varphi}_{l}\widehat{\zeta_{4}(\epsilon \cdot)h}|^{2}_{L^{2}}) \bigg)^{\frac{1}{2}}
\bigg(\sum_{l \geq [-\log_{2} \epsilon] - 4} \epsilon^{2s-2}2^{2ls}|\tilde{\varphi}_{l}\widehat{\zeta(\epsilon \cdot)f}|^{2}_{L^{2}} \bigg)^{\frac{1}{2}}
\nonumber\\ &\lesssim& |W^{\epsilon}h|_{L^{2}} |W^{\epsilon}(D)f|_{L^{2}}.
\een
By Cauchy-Schwartz inequality,
the change of variable $\xi \rightarrow \xi^{+}$ and $\xi \rightarrow |\xi|\frac{\xi^{+}}{|\xi^{+}|}$ and the estimate
\beno
\int b^{\epsilon}(\frac{\xi}{|\xi|}\cdot\sigma) \mathrm{1}_{\sin(\theta/2) \geq 2^{-l}} d\sigma \lesssim (1-s) \epsilon^{2s-2} \int_{2^{-l}}^{\epsilon} t^{-1-2s} dt \lesssim \epsilon^{2s-2}2^{2sl},
\eeno
we  have similarly
	\ben \label{y-l-geq}
\mathcal{Y}_{l, \geq} &\leq& \int b^{\epsilon}(\frac{\xi}{|\xi|}\cdot\sigma)\mathrm{1}_{\sin(\theta/2) \geq 2^{-l}}
[|({\varphi}_{l}\widehat{\zeta_{4}(\epsilon \cdot)h})(\xi^{+})|+|({\varphi}_{l}\widehat{\zeta_{4}(\epsilon \cdot)h})(|\xi|\frac{\xi^{+}}{|\xi^{+}|})|] |(\tilde{\varphi}_{l}\overline{\widehat{\zeta(\epsilon \cdot)f}})(\xi)| d\xi d\sigma
\nonumber \\&\lesssim& \epsilon^{2s-2} 2^{2ls} \|{\varphi}_{l}\widehat{\zeta_{4}(\epsilon \cdot)h}|_{L^{2}}
|\tilde{\varphi}_{l}\widehat{\zeta(\epsilon \cdot)f}|_{L^{2}}.\een
By taking the sum, we get
\ben \label{sum-y-geq}
\sum_{l \geq [-\log_{2} \epsilon] - 4} \mathcal{Y}_{l, \geq}  &\lesssim&
\bigg(\sum_{l \geq [-\log_{2} \epsilon] - 4} \epsilon^{2s-2} 2^{2ls} |{\varphi}_{l}\widehat{\zeta_{4}(\epsilon \cdot)h}|^{2}_{L^{2}} \bigg)^{\frac{1}{2}}
\bigg(\sum_{l \geq [-\log_{2} \epsilon] - 4} \epsilon^{2s-2}2^{2ls} |\tilde{\varphi}_{l}\widehat{\zeta(\epsilon \cdot)f}|_{L^{2}} \bigg)^{\frac{1}{2}}
\nonumber\\&\lesssim&  |W^{\epsilon}(D)h|_{L^{2}} |W^{\epsilon}(D)f|_{L^{2}}.
\een
By combining  \eqref{sum-y-leq} and \eqref{sum-y-geq}, \eqref{decomposition-in-frequency-space} gives
\ben \label{high-xi-low-v--low-v}
\mathcal{Y}^{\epsilon, 0}((1-\zeta(\epsilon D))\zeta_{4}(\epsilon v)h, \zeta(\epsilon v)f) \lesssim  (|W^{\epsilon}h|_{L^{2}} + |W^{\epsilon}(D)h|_{L^{2}})|W^{\epsilon}(D)f|_{L^{2}}.
\een

{\it Step 1.3: Estimate of $\mathcal{Y}^{\epsilon, 0}(h, (1-\zeta(\epsilon v))f)$.}
Note that
\ben \label{decomposition-low-high-frequency}
&&\mathcal{Y}^{\epsilon, 0}(h, (1-\zeta(\epsilon v))f) = \sum_{k \geq [-\log_{2}\epsilon]-4}
\mathcal{Y}^{\epsilon, 0}(\tilde{\varphi}_{k}h, \varphi_{k}f)
\nonumber \\
&&=\sum_{k \geq [-\log_{2}\epsilon]-4}
\mathcal{Y}^{\epsilon, 0}(\tilde{\varphi}_{k}h, \zeta(\epsilon D)\varphi_{k}f)+\sum_{k \geq [-\log_{2}\epsilon]-4}
\mathcal{Y}^{\epsilon, 0}(\tilde{\varphi}_{k}h, (1-\zeta(\epsilon D))\varphi_{k}f).
\een

We first consider $\sum_{k \geq [-\log_{2}\epsilon]-4} \mathcal{Y}^{\epsilon, 0}(\tilde{\varphi}_{k}h, \zeta(\epsilon D)\varphi_{k}f)$.
Decompose $\mathcal{Y}^{\epsilon, 0}(\tilde{\varphi}_{k}h, \zeta(\epsilon D)\varphi_{k}f) = \mathcal{Y}_{k, \leq}  + \mathcal{Y}_{k, \geq}$ according to $\sin(\theta/2) \leq 2^{-k}$ and $\sin(\theta/2) \geq 2^{-k}$.

For $\mathcal{Y}_{k, \leq}$,
by Taylor expansion to $\zeta(\epsilon D)\varphi_{k}f$, similar to \eqref{theta-leq-2-l}, we have
\beno
\mathcal{Y}_{k, \leq} &=& \int b^{\epsilon}(\frac{\xi}{|\xi|}\cdot\sigma)\mathrm{1}_{\sin(\theta/2) \leq 2^{-k}}
(1-\frac{1}{\cos\theta})  (\tilde{\varphi}_{k}h)(u)
\int_{0}^{1}
	(\nabla \zeta(\epsilon D)\varphi_{k}f (u^{+}(\kappa))\cdot u^{+} d\kappa  du d\sigma
\\&\lesssim& \epsilon^{s-1}2^{ks} |\tilde{\varphi}_{k}h|_{L^{2}}
|\nabla \zeta(\epsilon D)\varphi_{k}f|_{L^{2}}.
\eeno
By taking the sum, since $|\nabla \zeta(\epsilon D)\varphi_{k}f|_{L^{2}} \lesssim |W^{\epsilon}(D)\varphi_{k}f|_{L^{2}}$, we have
\ben \label{sum-y-k-leq}
\sum_{k \geq [-\log_{2}\epsilon]-4}\mathcal{Y}_{k, \leq} \lesssim  \sum_{k \geq [-\log_{2}\epsilon]-4} \epsilon^{s-1}2^{ks} |\tilde{\varphi}_{k}h|_{L^{2}}
|W^{\epsilon}(D) \varphi_{k}f|_{L^{2}}  \lesssim  |W^{\epsilon}h|_{L^{2}} |W^{\epsilon}(D)f|_{L^{2}},
\een
where we have used
\ben \nonumber
\sum_{k \geq -1}^{\infty}|W^{\epsilon}(D)\varphi_{k}f|^{2}_{L^{2}} &=& \sum_{k \geq -1}^{\infty}2^{-2k}|W^{\epsilon}(D)2^{k}\varphi_{k}f|^{2}_{L^{2}}
\\ \label{pa-together-norm} &\lesssim&\sum_{k \geq -1}^{\infty}2^{-2k}(|2^{k}\varphi_{k}W^{\epsilon}(D)f|^{2}_{L^{2}}+|f|^{2}_{H^{0}})
\lesssim |W^{\epsilon}(D)f|^{2}_{L^{2}},
\een
because $W^{\epsilon}\in S^{1}_{1,0}, 2^{k}\varphi_{k} \in S^{1}_{1,0}$ and Lemma \ref{operatorcommutator1}.

For $\mathcal{Y}_{k, \geq}$,  similar to \eqref{y-l-geq}, we get
$
\mathcal{Y}_{k, \geq} \lesssim  \epsilon^{2s-2}2^{2ks} |\tilde{\varphi}_{k}h |_{L^{2}} |\zeta(\epsilon D)\varphi_{k}f |_{L^{2}},
$
and
\ben \label{sum-y-k-geq}
\sum_{k \geq [-\log_{2}\epsilon]-4} \mathcal{Y}_{k, \geq} \lesssim \sum_{k \geq [-\log_{2}\epsilon]-4}  \epsilon^{2s-2}2^{2ks} |\tilde{\varphi}_{k}h |_{L^{2}} |\varphi_{k}f |_{L^{2}}
  \lesssim  |W^{\epsilon}h|_{L^{2}} |W^{\epsilon}f|_{L^{2}}.
\een
Combining \eqref{sum-y-k-leq} and \eqref{sum-y-k-geq}, we get
\ben \label{part1}
\sum_{k \geq [-\log_{2}\epsilon]-4} \mathcal{Y}^{\epsilon, 0}(\tilde{\varphi}_{k}h, \zeta(\epsilon D)\varphi_{k}f) \lesssim
|W^{\epsilon}h|_{L^{2}} |W^{\epsilon}(D)f|_{L^{2}} + |W^{\epsilon}h|_{L^{2}} |W^{\epsilon}f|_{L^{2}}.
\een

Now we consider $\sum_{k \geq [-\log_{2}\epsilon]-4}
\mathcal{Y}^{\epsilon, 0}(\tilde{\varphi}_{k}h, (1-\zeta(\epsilon D))\varphi_{k}f)$. By Proposition \ref{fourier-transform-cross-term} and  the dyadic decomposition in the frequency space, we have
\beno
&&\mathcal{Y}^{\epsilon, 0}(\tilde{\varphi}_{k}h, (1-\zeta(\epsilon D))\varphi_{k}f)
\\&=& \int b^{\epsilon}(\frac{\xi}{|\xi|}\cdot\sigma)  [(\widehat{\tilde{\varphi}_{k}h})(\xi^{+})- (\widehat{\tilde{\varphi}_{k}h})(|\xi|\frac{\xi^{+}}{|\xi^{+}|})]
 (1-\zeta(\epsilon \xi))\overline{\widehat{\varphi_{k}f}}(\xi)  d\xi d\sigma
\\&=& \sum_{l \geq [-\log_{2} \epsilon] - 4} \int b^{\epsilon}(\frac{\xi}{|\xi|}\cdot\sigma)  [(\tilde{\varphi}_{l}\widehat{\tilde{\varphi}_{k}h})(\xi^{+})- (\tilde{\varphi}_{l}\widehat{\tilde{\varphi}_{k}h})(|\xi|\frac{\xi^{+}}{|\xi^{+}|})]
 \varphi_{l}\overline{\widehat{\varphi_{k}f}}(\xi)  d\xi d\sigma
:= \sum_{l \geq [-\log_{2} \epsilon] - 4} \mathcal{Y}_{k,l}.
\eeno
Decompose $\mathcal{Y}_{k,l} = \mathcal{Y}_{k, l, \leq}  + \mathcal{Y}_{k, l, \geq}$ according to $\sin(\theta/2) \leq 2^{-(k+l)/2}$ and $\sin(\theta/2) \geq 2^{-(k+l)/2}$.

For $\mathcal{Y}_{k, l, \leq} $, by Taylor expansion for $\tilde{\varphi}_{l}\widehat{\tilde{\varphi}_{k}h}$, similar to \eqref{theta-leq-2-l}, we have
\beno
\mathcal{Y}_{k, l, \leq}   &=& \int b^{\epsilon}(\frac{\xi}{|\xi|}\cdot\sigma)\mathrm{1}_{\sin(\theta/2) \leq 2^{-(k+l)/2}}
(1-\frac{1}{\cos\theta}) (\int_{0}^{1}
	(\nabla \tilde{\varphi}_{l}\widehat{\tilde{\varphi}_{k}h} )(\xi^{+}(\kappa))\cdot \xi^{+} d\kappa )\varphi_{l}\overline{\widehat{\varphi_{k}f}}(\xi)  d\xi d\sigma
\\&\lesssim& 2^{l}\epsilon^{2s-2}2^{(s-1)(k+l)}  |\nabla \tilde{\varphi}_{l}\widehat{\tilde{\varphi}_{k}h} |_{L^{2}}
|\varphi_{l}\widehat{\varphi_{k}f}|_{L^{2}}
= \epsilon^{2s-2}2^{s(k+l)} 2^{-k}  |\nabla \tilde{\varphi}_{l}\widehat{\tilde{\varphi}_{k}h} |_{L^{2}}
|\varphi_{l}\widehat{\varphi_{k}f}|_{L^{2}}.
\eeno
Since $|\nabla \tilde{\varphi}_{l}\widehat{\tilde{\varphi}_{k}h}| =
|\tilde{\varphi}_{l}\widehat{v \tilde{\varphi}_{k}h}+ (\nabla \tilde{\varphi}_{l})\widehat{\tilde{\varphi}_{k}h}| \lesssim
|\tilde{\varphi}_{l}\widehat{v \tilde{\varphi}_{k}h}| + 2^{-l}|\tilde{\tilde{\varphi}}_{l}\widehat{\tilde{\varphi}_{k}h}|
$ where $\tilde{\tilde{\varphi}}_{l} := \sum_{|j-l| \leq 4, j \geq -1} {\varphi}_{j}$, we have
\ben \label{sum-k-l-leq}
\sum_{k,l \geq [-\log_{2}\epsilon] -4} \mathcal{Y}_{k, l, \leq} &\lesssim&
\sum_{k,l \geq [-\log_{2}\epsilon] -4}
\epsilon^{2s-2}2^{s(k+l)} 2^{-k}  (|\tilde{\varphi}_{l}\widehat{v \tilde{\varphi}_{k}h}|_{L^{2}} + 2^{-l}|\tilde{\tilde{\varphi}}_{l}\widehat{\tilde{\varphi}_{k}h}|_{L^{2}})
|\varphi_{l}\widehat{\varphi_{k}f}|_{L^{2}}
\nonumber \\ &\lesssim&
\bigg(\sum_{k,l \geq [-\log_{2}\epsilon] -4} ( \epsilon^{2s-2}2^{2sk} 2^{-2k}  |\tilde{\varphi}_{l}\widehat{v \tilde{\varphi}_{k}h}|^{2}_{L^{2}} + \epsilon^{2s-2}2^{2sk} 2^{-2k} 2^{-2l}|\tilde{\tilde{\varphi}}_{l}\widehat{\tilde{\varphi}_{k}h}|^{2}_{L^{2}}) \bigg)^{\frac{1}{2}}
\nonumber  \\&&\times
\bigg(\sum_{k,l \geq [-\log_{2}\epsilon] -4} \epsilon^{2s-2}2^{2ls}|\varphi_{l}\widehat{\varphi_{k}f}|^{2}_{L^{2}} \bigg)^{\frac{1}{2}}
\lesssim  |W^{\epsilon}h|_{L^{2}} |W^{\epsilon}(D)f|_{L^{2}}.
\een

For $\mathcal{Y}_{k, l, \geq}$,  we have
$
\mathcal{Y}_{k, l, \geq} \lesssim  \epsilon^{2s-2}2^{(k+l)s} |\tilde{\varphi}_{l}\widehat{\tilde{\varphi}_{k}h}|_{L^{2}} |\varphi_{l}\widehat{\varphi_{k}f}|_{L^{2}}.
$
Thus by \eqref{pa-together-norm},
\ben \label{sum-k-l-geq}
\sum_{k,l \geq [-\log_{2}\epsilon] -4} \mathcal{Y}_{k, l, \geq} \lesssim
\sum_{k,l \geq [-\log_{2}\epsilon] -4}  \epsilon^{2s-2}2^{(k+l)s} |\tilde{\varphi}_{l}\widehat{\tilde{\varphi}_{k}h}|_{L^{2}} |\varphi_{l}\widehat{\varphi_{k}f}|_{L^{2}}
  \lesssim  |W^{\epsilon}h|_{L^{2}} |W^{\epsilon}(D)f|_{L^{2}}.
\een
Combining \eqref{sum-k-l-leq} and \eqref{sum-k-l-geq} gives
\ben \label{part2}
\sum_{k \geq [-\log_{2}\epsilon]-4}
\mathcal{Y}^{\epsilon, 0}(\tilde{\varphi}_{k}h, (1-\zeta(\epsilon D))\varphi_{k}f) \lesssim |W^{\epsilon}h|_{L^{2}} |W^{\epsilon}(D)f|_{L^{2}}.
\een
With \eqref{part1} and \eqref{part2}, \eqref{decomposition-low-high-frequency} gives
\ben \label{h-full-high-v-f}
\mathcal{Y}^{\epsilon, 0}(h, (1-\zeta(\epsilon v))f) \lesssim |W^{\epsilon}h|_{L^{2}} |W^{\epsilon}(D)f|_{L^{2}} + |W^{\epsilon}h|_{L^{2}} |W^{\epsilon}f|_{L^{2}}.
\een

Back to \eqref{decomposition-f} and \eqref{decomposition-h},
 by combining \eqref{low-xi-low-v--low-xi-low-v}, \eqref{high-xi-low-v--low-v} and \eqref{h-full-high-v-f},
we get
\ben \label{case-gamma-0}
	|\mathcal{Y}^{\epsilon,0}(h,f)| \lesssim (|W^{\epsilon}h|_{L^{2}}+|W^{\epsilon}(D)h|_{L^{2}})^{\frac{1}{2}}
	(|W^{\epsilon}f|_{L^{2}}+|W^{\epsilon}(D)f|_{L^{2}})^{\frac{1}{2}}.
\een

{\it Step 2: $\gamma \neq 0$.}  For simplicity, denote $w = |u|\frac{u^{+}}{|u^{+}|}$, then $W_{\gamma/2}(u) = W_{\gamma/2}(w)$.  Note that
	\beno
	\langle u \rangle^{\gamma} h(u)[f(u^{+}) - f(w)] &=& (W_{\gamma/2}h)(u)[(W_{\gamma/2}f)(u^{+})-(W_{\gamma/2}f)(w)]
	\\&& +(W_{\gamma/2}h)(u)(W_{\gamma/2}f)(u^{+})(W_{\gamma/2}(w)W_{-\gamma/2}(u^{+}) - 1),
	\eeno
	 which yields
	\beno
	\mathcal{Y}^{\epsilon,\gamma}(h,f) &=& \mathcal{Y}^{\epsilon,0}(W_{\gamma/2}h, W_{\gamma/2}f) + \mathcal{A},
\\
\mathcal{A} &:=& \int b^{\epsilon}(\frac{u}{|u|}\cdot\sigma)(W_{\gamma/2}h)(u) (W_{\gamma/2}f)(u^{+})(W_{\gamma/2}(w)W_{-\gamma/2}(u^{+}) - 1) du d\sigma.
	\eeno
By Cauchy-Schwartz inequality,
$|W_{\gamma/2}(u)W_{-\gamma/2}(u^{+}) - 1| \lesssim \sin^{2}\frac{\theta}{2}$ and the estimate \eqref{order-2},
	we have
	\beno
	|\mathcal{A}| &\lesssim&  \bigg(\int b^{\epsilon}(\frac{u}{|u|}\cdot\sigma)|(W_{\gamma/2}h)(u)|^{2}|W_{\gamma/2}(w)W_{-\gamma/2}(u^{+}) - 1|du d\sigma\bigg)^{\frac{1}{2}}
	\\&&\times\bigg(\int b^{\epsilon}(\frac{u}{|u|}\cdot\sigma)|(W_{\gamma/2}f)(u^{+})|^{2}|W_{\gamma/2}(w)W_{-\gamma/2}(u^{+}) - 1| du d\sigma\bigg)^{\frac{1}{2}}
	\lesssim  |W_{\gamma/2}h|_{L^{2}}|W_{\gamma/2}f|_{L^{2}},
	\eeno
	where the change of variable $u \rightarrow u^{+}$ has been used in the estimate for $f$. This together with \eqref{case-gamma-0} completes  the proof of the lemma.
\end{proof}
\begin{rmk} \label{exact-cross-term} Set $
	\mathcal{X}^{\epsilon,\gamma}(h,f) := \int b^{\epsilon}(\frac{u}{|u|}\cdot\sigma)|u|^{\gamma}(1-\zeta(|u|)) h(u)[f(u^{+}) - f(|u|\frac{u^{+}}{|u^{+}|})] du d\sigma,
	$
Then
\beno
	|\mathcal{X}^{\epsilon,\gamma}(h,f)| \lesssim (|W^{\epsilon}W_{\gamma/2}h|_{L^{2}}+|W^{\epsilon}(D)W_{\gamma/2}h|_{L^{2}})
	(|W^{\epsilon}W_{\gamma/2}f|_{L^{2}}+|W^{\epsilon}(D)W_{\gamma/2}f|_{L^{2}}).
	\eeno
Indeed, by the identity $|u|^{\gamma}(1-\zeta(|u|))
= \langle u \rangle^{\gamma}(|u|^{\gamma}\langle u \rangle^{-\gamma}-1)(1-\zeta(|u|))+\langle u \rangle^{\gamma}(1-\zeta(|u|))$,
we have
\beno \mathcal{X}^{\epsilon,\gamma}(h,f) = \mathcal{Y}^{\epsilon,\gamma}((|\cdot|^{\gamma}\langle \cdot \rangle^{-\gamma}-1)(1-\zeta)h,f)
+\mathcal{Y}^{\epsilon,\gamma}((1-\zeta)h,f).
\eeno
 Then the estimate follows from Lemma \ref{crosstermsimilar} and Lemma \ref{operatorcommutator1} because $(|\cdot|^{\gamma}\langle \cdot \rangle^{-\gamma}-1)(1-\zeta), 1-\zeta  \in S^{0}_{1,0}$.
\end{rmk}

 \begin{lem}\label{a-technical-lemma}
It holds
\beno (1-s)\epsilon^{2s-2}\int_{\R^{3}}\int_{0}^{2\epsilon} \theta^{-1-2s}|f(v) - f(v\cos\theta)|^{2} dv d\theta
\lesssim |W^{\epsilon}(D)f|^{2}_{L^{2}} + |W^{\epsilon}f|^{2}_{L^{2}}.\eeno
\end{lem}
We omit the  proof for brevity because  the localization techniques used in Lemma \ref{crosstermsimilar}
can be applied similarly by considering $f(v)(f(v) - f(v \cos\theta))$ and $f(v \cos\theta)(f(v) - f(v \cos\theta))$ separately.
\bigskip

\noindent {\bf Acknowledgment:}\,
 The research of Renjun Duan was  supported by the NSFC/RGC Joint Research Scheme (N\_CUHK409/19) from RGC in Hong Kong and a Direct Grant from CUHK. The research of
Ling-Bing He was supported by NSF of China under the grant 11771236. The research of Tong Yang  was supported by a fellowship award from the Research Grants Council of the Hong Kong Special Administrative Region, China (Project no. SRF2021-1S01). And the research of
Yu-Long Zhou was supported by NSF of China under the grant 12001552.


\end{document}